\documentclass[11pt,reqno]{amsart}
\usepackage{amssymb,amscd,amsfonts,amsbsy,enumerate}
\usepackage{latexsym}
\usepackage{exscale}
\usepackage{amsmath,amsthm,amsfonts}
\usepackage{mathrsfs}

\usepackage[colorlinks=true, pdfstartview=FitV, linkcolor=blue, citecolor=red, urlcolor=blue, pagebackref]{hyperref}

\usepackage{color}

\numberwithin{equation}{section}

\def\Xint#1{\mathchoice
   {\XXint\displaystyle\textstyle{#1}}%
   {\XXint\textstyle\scriptstyle{#1}}%
   {\XXint\scriptstyle\scriptscriptstyle{#1}}%
   {\XXint\scriptscriptstyle\scriptscriptstyle{#1}}%
   \!\int}
\def\XXint#1#2#3{{\setbox0=\hbox{$#1{#2#3}{\int}$}
     \vcenter{\hbox{$#2#3$}}\kern-.5\wd0}}
\def\aver#1{\Xint-_{#1}}

\allowdisplaybreaks

\newtheorem{theorem}{Theorem}[section]
\newtheorem{lemma}[theorem]{Lemma}
\newtheorem{corollary}[theorem]{Corollary}
\newtheorem{proposition}[theorem]{Proposition}
\newtheorem{definition}[theorem]{Definition}

\theoremstyle{remark}

\begin{document}
\allowdisplaybreaks

\title[The {\rm BMO}-Dirichlet problem and quantitative characterizations of {\rm VMO}]
{The {\rm BMO}-Dirichlet problem for elliptic \\ systems in the upper-half 
space and \\ quantitative characterizations of {\rm VMO}}

\author{Jos\'e Mar{\'\i}a Martell}
\address{Jos\'e Mar{\'\i}a Martell
\\
Instituto de Ciencias Matem\'aticas CSIC-UAM-UC3M-UCM
\\
Consejo Superior de Investigaciones Cient{\'\i}ficas
\\
C/ Nicol\'as Cabrera, 13-15
\\
E-28049 Madrid, Spain} \email{chema.martell@icmat.es}
\address{and}
\address{Department of Mathematics
\\
University of Missouri
\\
Columbia, MO 65211, USA} 
\email{martellj@missouri.edu}

\author{Dorina Mitrea}
\address{Dorina Mitrea
\\
Department of Mathematics
\\
University of Missouri
\\
Columbia, MO 65211, USA} \email{mitread@missouri.edu}

\author{Irina Mitrea}
\address{Irina Mitrea
\\
Department of Mathematics
\\
Temple University\!
\\
1805\,N.\,Broad\,Street
\\
Philadelphia, PA 19122, USA} \email{imitrea@temple.edu}

\author{Marius Mitrea}
\address{Marius Mitrea
\\
Department of Mathematics
\\
University of Missouri
\\
Columbia, MO 65211, USA} \email{mitream@missouri.edu}

\thanks{The first author would like to express his gratitude to the University of Missouri-Columbia (USA), for its support and hospitality while he was visiting this institution. The first author acknowledges financial support from the Spanish Ministry of Economy and Competitiveness, through the ``Severo Ochoa Programme for Centres of Excellence in R\&D'' (SEV-2015-0554). He also acknowledges that
the research leading to these results has received funding from the European Research
Council under the European Union's Seventh Framework Programme (FP7/2007-2013)/ ERC
agreement no. 615112 HAPDEGMT. The second author has been supported 
in part by the Simons Foundation grant $\#\,$426669, the third author has
been supported in part by the Simons Foundation grant $\#\,$318658, while the fourth author has been
supported in part by the Simons Foundation grant $\#\,$281566, and by a University of
Missouri Research Leave grant. This work has been possible thanks to the support and hospitality
of \textit{Temple University} (USA), \textit{University of Missouri} (USA), and
\textit{ICMAT, Consejo Superior de Investigaciones Cient{\'\i}ficas} (Spain).
The authors express their gratitude to these institutions.}

\date{\today}



\begin{abstract}
We prove that for any homogeneous, second order, constant complex coefficient elliptic system
$L$ in $\mathbb{R}^{n}$, the Dirichlet problem in $\mathbb{R}^{n}_{+}$ with 
boundary data in $\mathrm{BMO}(\mathbb{R}^{n-1})$ is well-posed in the class of functions $u$ 
for which the Littlewood-Paley measure associated with $u$, namely $d\mu_u(x',t):=|\nabla u(x',t)|^2\,t\,dx'dt$, 
is a Carleson measure in $\mathbb{R}^{n}_{+}$. 

In the process we establish a Fatou type theorem guaranteeing the existence of the pointwise 
nontangential boundary trace for smooth null-solutions $u$ of such systems satisfying the said 
Carleson measure condition. In concert, these results imply that the space $\mathrm{BMO}(\mathbb{R}^{n-1})$ 
can be characterized as the collection of nontangential pointwise traces of smooth 
null-solutions $u$ to the elliptic system $L$ with the property that $\mu_u$ is a Carleson measure 
in $\mathbb{R}^{n}_{+}$. 

We also establish a regularity result for the {\rm BMO}-Dirichlet problem in the upper-half space, 
to the effect that the nontangential pointwise trace on the boundary of $\mathbb{R}^{n}_{+}$ of 
any given smooth null-solutions $u$ of $L$ in $\mathbb{R}^{n}_{+}$ satisfying the above Carleson 
measure condition actually belongs to Sarason's space $\mathrm{VMO}(\mathbb{R}^{n-1})$ if and 
only if $\mu_u(T(Q))/|Q|\to 0$ as $|Q|\to 0$, uniformly with respect to the location of the 
cube $Q\subset{\mathbb{R}}^{n-1}$ (where $T(Q)$ is the Carleson box associated with $Q$, 
and $|Q|$ denotes the Euclidean volume of $Q$). 

Moreover, we are able to establish the well-posedness of the Dirichlet problem in 
$\mathbb{R}^{n}_{+}$ 
for a system $L$ as above in the case when the boundary data are prescribed in Morrey-Campanato
spaces in $\mathbb{R}^{n-1}$. In such a scenario, the solution $u$ is required to satisfy 
a vanishing Carleson measure condition of fractional order. 

By relying on these well-posedness and regularity results we succeed in producing 
characterizations of the space $\mathrm{VMO}$ as the closure in {\rm BMO} of classes of 
smooth functions contained in {\rm BMO} within which uniform continuity may be suitably quantified 
(such as the class of smooth functions satisfying a H\"older or Lipschitz condition). 
This improves on Sarason's classical result describing $\mathrm{VMO}$ as the closure 
in {\rm BMO} of the space of uniformly continuous functions with bounded mean oscillations.
In turn, this allows us to show that any Calder\'on-Zygmund operator $T$ satisfying $T(1)=0$
extends as a linear and bounded mapping from $\mathrm{VMO}$ (modulo constants) into itself. 
In turn, this is used to describe algebras of singular integral operators on $\mathrm{VMO}$,
and to characterize the membership to $\mathrm{VMO}$ via the action of various classes of 
singular integral operators.
\end{abstract}

\maketitle

\allowdisplaybreaks

\tableofcontents
\normalsize

\section{Introduction and statement of main theorems}
\setcounter{equation}{0}
\label{S-1}

In his ground breaking 1971 article \cite{Fe}, C. Fefferman writes 
``{\it The main idea in proving {\rm [}that the dual of the Hardy space $H^1$ 
is the John-Nirenberg space ${\rm BMO}${\rm ]} is to study the Poisson integral 
of a function in ${\rm BMO}$.}'' Implicit in this quote is the fact that 
the Poisson kernel is associated with the Laplace operator, and one of the primary aims 
of the present paper is to advance this line of work by considering more general 
systems of partial differential operators than the Laplacian. 
For example, the key PDE result announced by C. Fefferman in \cite{Fe} states that 
\begin{equation}\label{L-dJHG}
\displaystyle
\parbox{11.10cm}{a measurable function $f$ with
$\displaystyle\int_{{\mathbb{R}}^{n-1}}|f(x')|(1+|x'|)^{-n}\,dx'<+\infty$ belongs to the space
${\rm BMO}({\mathbb{R}}^{n-1})$ if and only if its Poisson integral 
$u:{\mathbb{R}}^n_{+}\to{\mathbb{R}}$, with respect to the Laplace operator
in ${\mathbb{R}}^n$, satisfies $\displaystyle\sup\limits_{x'\in{\mathbb{R}}^{n-1}}\sup\limits_{r>0}
\Big\{r^{1-n}\int\limits_{|x'-y'|<r}\int\limits_0^r|(\nabla u)(y',t)|^2\,t\,dt\,dx'\Big\}<+\infty$,}
\end{equation}
and one of the main goals here is to develop machinery that permits us to replace the
Laplacian in \eqref{L-dJHG} with much more general second order elliptic systems with 
complex coefficients. In order to be more specific, we proceed to elaborate on the 
actual setting adopted in this paper. 

Let $M\in{\mathbb{N}}$ and consider a second-order, homogeneous, $M\times M$ system,
with constant complex coefficients, written (with the usual convention of summation
over repeated indices in place) as
\begin{equation}\label{L-def}
Lu:=\Bigl(\partial_r(a^{\alpha\beta}_{rs}\partial_s u_\beta)\Bigr)_{1\leq\alpha\leq M},
\end{equation}
when acting on a ${\mathscr{C}}^2$ vector-valued function $u=(u_\beta)_{1\leq\beta\leq M}$
defined in an open subset of ${\mathbb{R}}^n$. Assume that $L$ is {\tt strongly}
{\tt elliptic} in the sense that there exists $\kappa_o\in(0,\infty)$ such that
\begin{equation}\label{L-ell.X}
\begin{array}{c}
{\rm Re}\,\bigl[a^{\alpha\beta}_{rs}\xi_r\xi_s\overline{\eta_\alpha}
\eta_\beta\,\bigr]\geq\kappa_o|\xi|^2|\eta|^2\,\,\mbox{ for every}
\\[8pt]
\xi=(\xi_r)_{1\leq r\leq n}\in{\mathbb{R}}^n\,\,\mbox{ and }\,\,
\eta=(\eta_\alpha)_{1\leq\alpha\leq M}\in{\mathbb{C}}^M.
\end{array}
\end{equation}
Examples include scalar operators, such as the Laplacian $\Delta=\sum\limits_{j=1}^n\partial_j^2$ 
or, more generally, operators of the form ${\rm div}A\nabla$ with $A=(a_{rs})_{1\leq r,s\leq n}$ an 
$n\times n$ matrix with complex entries satisfying the ellipticity condition
\begin{equation}\label{YUjhv-753}
\inf_{\xi\in S^{n-1}}{\rm Re}\,\big[a_{rs}\xi_r\xi_s\bigr]>0,
\end{equation}
(where $S^{n-1}$ denotes the unit sphere in ${\mathbb{R}}^n$), as well as complex
versions of the Lam\'e system of elasticity
\begin{equation}\label{TYd-YG-76g}
Lu:=\mu\Delta u+(\lambda+\mu)\nabla{\rm div}\,u,\qquad u=(u_1,...,u_n)\in{\mathscr{C}}^2.
\end{equation}
Above, the constants $\lambda,\mu\in{\mathbb{C}}$ (typically called Lam\'e moduli),
are assumed to satisfy
\begin{equation}\label{Yfhv-8yg}
{\rm Re}\,\mu>0\,\,\mbox{ and }\,\,{\rm Re}\,(2\mu+\lambda)>0,
\end{equation}
a condition equivalent to the demand that the Lam\'e system \eqref{TYd-YG-76g} 
satisfies the Legendre-Hadamard ellipticity condition \eqref{L-ell.X}.
While the Lam\'e system is symmetric, we stress that the results in this paper require no symmetry 
for the systems involved. 

Returning to the general framework, 
with every system $L$ as in \eqref{L-def}-\eqref{L-ell.X} one may associate
a Poisson kernel, $P^L$, which is a ${\mathbb{C}}^{M\times M}$-valued function defined 
in ${\mathbb{R}}^{n-1}$ described in detail in Theorem~\ref{kkjbhV}.
This Poisson kernel has played a pivotal role in the treatment of the $L^p$-Dirichlet
boundary value problem for $L$ in the upper-half space in \cite{K-MMMM}. 

To state our main results, some notation is needed. For a function 
$\phi:{\mathbb{R}}^{n-1}\rightarrow{\mathbb{C}}$ set 
\begin{equation}\label{phisubt}
\phi_t(x'):=t^{1-n}\phi(x'/t),
\qquad\mbox{for every $x'\in\mathbb{R}^{n-1}$ and every $t>0$.}
\end{equation}
In particular, $P^L_t(x')=t^{1-n}P^L(x'/t)$ for every $x'\in{\mathbb{R}}^{n-1}$ and $t>0$.
We agree to identify the boundary of the upper-half space
\begin{equation}\label{RRR-UpHs}
{\mathbb{R}}^{n}_{+}:=\big\{x=(x',x_n)\in
{\mathbb{R}}^{n}={\mathbb{R}}^{n-1}\times{\mathbb{R}}:\,x_n>0\big\}
\end{equation}
with the horizontal hyperplane ${\mathbb{R}}^{n-1}$ via $(x',0)\equiv x'$
for any $x'\in{\mathbb{R}}^{n-1}$. The origin in ${\mathbb{R}}^{n-1}$ is denoted by $0'$.
Having fixed some background parameter $\kappa>0$, at each point $x'\in\partial{\mathbb{R}}^{n}_{+}$ 
we define the conical nontangential approach region with vertex at $x'$ as
\begin{equation}\label{NT-1}
\Gamma_\kappa(x'):=\big\{y=(y',t)\in{\mathbb{R}}^{n}_{+}:\,|x'-y'|<\kappa\,t\big\}.
\end{equation}
Whenever meaningful, the nontangential pointwise trace of a continuous 
vector-valued function $u$ defined in ${\mathbb{R}}^n_+$ is given by
\begin{equation}\label{nkc-EE-2}
\Big(u\big|^{{}^{\rm n.t.}}_{\partial{\mathbb{R}}^{n}_{+}}\Big)(x')
:=\lim_{\Gamma_{\kappa}(x')\ni y\to (x',0)}u(y)
\,\,\mbox{ for }\,\,x'\in\partial{\mathbb{R}}^{n}_{+}\equiv{\mathbb{R}}^{n-1}.
\end{equation}

For each positive integer $k$ denote by ${\mathscr{L}}^k$ the $k$-dimensional Lebesgue
measure in ${\mathbb{R}}^k$. A Borel measure $\mu$ in $\mathbb{R}^{n}_{+}$ is said 
to be a Carleson measure in $\mathbb{R}^{n}_{+}$ provided
\begin{equation}\label{defi-Carleson}
\|\mu\|_{\mathcal{C}(\mathbb{R}_+^{n})}:=\sup_{Q\subset\mathbb{R}^{n-1}} 
\frac1{|Q|}\int_{0}^{\ell(Q)}\int_Q d\mu(x',t)<\infty,
\end{equation}
where the supremum runs over all cubes $Q$ in $\mathbb{R}^{n-1}$. Here and elsewhere in the paper,
by a cube $Q$ in $\mathbb{R}^{n-1}$ we shall understand a cube with sides parallel to the 
coordinate axes, and its side-length will be denoted by $\ell(Q)$. Also, 
the ${\mathscr{L}}^{n-1}$ measure of $Q$ is denoted by $|Q|$ and if $\lambda>0$ 
then $\lambda\,Q$ denotes the cube concentric with $Q$ whose side-length is $\lambda\,\ell(Q)$. 
Call a Borel measure $\mu$ in $\mathbb{R}^{n}_{+}$ a vanishing Carleson measure whenever
$\mu$ is a Carleson measure to begin with and, in addition,  
\begin{equation}\label{defi-CarlesonVan}
\lim_{r\to 0^{+}}\left(\sup_{Q\subset\mathbb{R}^{n-1},\,\ell(Q)\leq r} 
\frac1{|Q|}\int_{0}^{\ell(Q)}\int_Q d\mu(x',t)\right)=0.
\end{equation}

Next, the Littlewood-Paley measure associated with a continuously differentiable function 
$u$ in ${\mathbb{R}}^n_{+}$ is $|\nabla u(x',t)|^2\,t\,dx'dt$, and we set  
\begin{equation}\label{ustarstar}
\|u\|_{**}:=\sup_{Q\subset\mathbb{R}^{n-1}}\left(\frac1{|Q|}\int_{0}^{\ell(Q)}
\int_Q |\nabla u(x',t)|^2\,t\,dx'dt\right)^\frac12.
\end{equation}
In particular, for a continuously differentiable function 
$u$ in ${\mathbb{R}}^n_{+}$ we have 
\begin{equation}\label{ncud}
\|u\|_{**}<\infty\,\,\Longleftrightarrow\,\,
|\nabla u(x',t)|^2\,t\,dx'dt\,\,\text{ is a Carleson measure in }\,\,{\mathbb{R}}^n_{+}.
\end{equation}

We next introduce $\mathrm{BMO}(\mathbb{R}^{n-1},\mathbb{C}^M)$, the John-Nirenberg 
space of vector-valued functions of bounded mean oscillations in ${\mathbb{R}}^{n-1}$,
as the collection of $\mathbb{C}^M$-valued functions $f=(f_\alpha)_{1\le\alpha\le M}$ 
with components in $L^1_{\rm loc}(\mathbb{R}^{n-1})$ satisfying  
\begin{equation}\label{defi-BMO}
\|f\|_{\mathrm{BMO}(\mathbb{R}^{n-1},\mathbb{C}^M)}:=
\sup_{Q\subset\mathbb{R}^{n-1}}\aver{Q}\big|f(x')-f_Q\big|\,dx'<\infty.
\end{equation}
Above, for every cube $Q$ in ${\mathbb{R}}^{n-1}$ and every function 
$h\in L^1_{\rm loc}(\mathbb{R}^{n-1},{\mathbb{C}}^M)$ we have abbreviated 
\begin{equation}\label{nota-aver}
h_Q:=\aver{Q} h(x')\,dx':=\frac1{|Q|}\int_{Q}h(x')\,dx'\in{\mathbb{C}}^M,
\end{equation}
where the last integration is performed componentwise. 
To lighten notation, when $M=1$ we simply write $\mathrm{BMO}(\mathbb{R}^{n-1})$
in place of $\mathrm{BMO}(\mathbb{R}^{n-1},\mathbb{C})$. Clearly, for every 
$f\in L^1_{\rm loc}(\mathbb{R}^{n-1},\mathbb{C}^M)$ we have
\begin{equation}\label{defi-BMO-CCC}
\begin{array}{ll}
\|f\|_{\mathrm{BMO}(\mathbb{R}^{n-1},\mathbb{C}^M)}
=\|f+C\|_{\mathrm{BMO}(\mathbb{R}^{n-1},\mathbb{C}^M)},
& \forall\,C\in{\mathbb{C}}^{M},
\\[6pt]
\|f\|_{\mathrm{BMO}(\mathbb{R}^{n-1},\mathbb{C}^M)}
=\|\tau_{z'}f\|_{\mathrm{BMO}(\mathbb{R}^{n-1},\mathbb{C}^M)}, & \forall\,z'\in{\mathbb{R}}^{n-1},
\\[6pt]
\|f\|_{\mathrm{BMO}(\mathbb{R}^{n-1},\mathbb{C}^M)}
=\|\delta_{\lambda}f\|_{\mathrm{BMO}(\mathbb{R}^{n-1},\mathbb{C}^M)}, & \forall\,\lambda\in(0,\infty),
\end{array}
\end{equation}
where $\tau_{z'}$ is the operator of translation by $z'$, i.e., 
$(\tau_{z'}f)(x'):=f(x'+z')$ for every $x'\in{\mathbb{R}}^{n-1}$, 
and $\delta_\lambda$ is the operator of dilation by $\lambda$, i.e., 
$(\delta_\lambda f)(x'):=f(\lambda x')$ for every $x'\in{\mathbb{R}}^{n-1}$.

We wish to note here that $\|\cdot\|_{\mathrm{BMO}(\mathbb{R}^{n-1},\mathbb{C}^M)}$ 
is only a seminorm, since for every function $f\in L^1_{\rm loc}(\mathbb{R}^{n-1},\mathbb{C}^M)$ we have 
\begin{equation}\label{defi-BMO-nVV}
\|f\|_{\mathrm{BMO}(\mathbb{R}^{n-1},\mathbb{C}^M)}=0\,\Longleftrightarrow\,
\text{$f$ is constant ${\mathscr{L}}^{n-1}$-a.e. in ${\mathbb{R}}^{n-1}$ 
(in ${\mathbb{C}}^M$)}.  
\end{equation}
Occasionally, we find it useful to mod out its null-space, in order to render 
the resulting quotient space Banach. Specifically, for two $\mathbb{C}^M$-valued 
Lebesgue measurable functions $f,g$ defined in $\mathbb{R}^{n-1}$ we say that $f\sim g$ 
provided $f-g$ is constant ${\mathscr{L}}^{n-1}$-a.e. in ${\mathbb{R}}^{n-1}$. 
This is an equivalence relation and we let 
\begin{align}\label{jgsyjw-AASSS}
[f]:=\big\{g:\mathbb{R}^{n-1}\to\mathbb{C}^M:\,\text{$g$ measurable and $f\sim g$}\big\}
\end{align}
denote the equivalence class of any given $\mathbb{C}^M$-valued Lebesgue 
measurable function $f$ defined in $\mathbb{R}^{n-1}$. If for each 
$f\in{\mathrm{BMO}}(\mathbb{R}^{n-1},\mathbb{C}^M)$ we now set
\begin{align}\label{defi-BMO-nbgxcr}
\big\|\,[f]\,\big\|_{\widetilde{\mathrm{BMO}}(\mathbb{R}^{n-1},\mathbb{C}^M)}
:=\|f\|_{\mathrm{BMO}(\mathbb{R}^{n-1},\mathbb{C}^M)},
\end{align}
then $\big\|\,[\cdot]\,\big\|_{\widetilde{\mathrm{BMO}}(\mathbb{R}^{n-1},\mathbb{C}^M)}$
becomes a genuine norm on the quotient space 
\begin{equation}\label{defi-BMO-tilde}
\widetilde{\mathrm{BMO}}(\mathbb{R}^{n-1},\mathbb{C}^M)
:=\big\{[f]:\,f\in{\mathrm{BMO}}(\mathbb{R}^{n-1},\mathbb{C}^M)\big\}.
\end{equation}
In fact, when equipped with the norm \eqref{defi-BMO-nbgxcr}, the space
\eqref{defi-BMO-tilde} is complete (hence Banach).

Moving on, the Sarason space of $\mathbb{C}^M$-valued functions of vanishing 
mean oscillations in ${\mathbb{R}}^{n-1}$ is defined by
\begin{align}\label{defi-VMO}
{\mathrm{VMO}}(\mathbb{R}^{n-1},\mathbb{C}^M)&:=
\Bigg\{ f\in{\mathrm{BMO}}(\mathbb{R}^{n-1},\mathbb{C}^M):
\\[6pt]
&\hskip 1.00in
\lim_{r\to 0^{+}}\left(\sup_{Q\subset\mathbb{R}^{n-1},\,\ell(Q)\leq r}\,\,
\aver{Q}\big|f(x')-f_Q\big|\,dx'\right)=0\Bigg\}.
\nonumber
\end{align}
The space ${\mathrm{VMO}}(\mathbb{R}^{n-1},\mathbb{C}^M)$ turns out to be a closed subspace
of ${\mathrm{BMO}}(\mathbb{R}^{n-1},\mathbb{C}^M)$. In fact, if 
${\mathrm{UC}}({\mathbb{R}}^{n-1},\mathbb{C}^M)$ stands for the space
of $\mathbb{C}^M$-valued uniformly continuous functions in ${\mathbb{R}}^{n-1}$, then 
\begin{align}\label{ku6ffcf-UCUC}
&\hskip -0.20in
{\rm UC}({\mathbb{R}}^{n-1},{\mathbb{C}}^M)
\cap\Big(\bigcup_{1\leq p\leq\infty}L^p({\mathbb{R}}^{n-1},{\mathbb{C}}^M)\Big)
\nonumber\\[6pt]
&\hskip 1.00in
\subset{\mathrm{UC}}(\mathbb{R}^{n-1},\mathbb{C}^M)\cap
{\mathrm{BMO}}(\mathbb{R}^{n-1},\mathbb{C}^M)
\nonumber\\[6pt]
&\hskip 1.00in
\subset{\mathrm{VMO}}(\mathbb{R}^{n-1},\mathbb{C}^M).
\end{align}
To justify the first inclusion, consider 
$f\in{\rm UC}({\mathbb{R}}^{n-1},{\mathbb{C}}^M)\cap L^p({\mathbb{R}}^{n-1},{\mathbb{C}}^M)$
for some $p\in[1,\infty]$. Then there exists $r_0\in(0,\infty)$ with the property that 
$|f(x')-f(y')|\leq 1$ whenever $x',y'\in{\mathbb{R}}^{n-1}$ are such that 
$|x'-y'|\leq r_0\sqrt{n-1}$. 
Suppose now that some arbitrary cube $Q$ in ${\mathbb{R}}^{n-1}$ has been fixed.
If $\ell(Q)\geq r_0$ with the help of H\"older's inequality we estimate 
\begin{align}\label{j8hVGVV.1}
\aver{Q}\big|f-f_Q\big|\,d{\mathscr{L}}^{n-1}
&\leq 2\aver{Q}|f|\,d{\mathscr{L}}^{n-1}\leq
\frac{2\|f\|_{L^p({\mathbb{R}}^{n-1},{\mathbb{C}}^M)}}{|Q|^{1/p}}
\nonumber\\[6pt]
&\leq\frac{2\|f\|_{L^p({\mathbb{R}}^{n-1},{\mathbb{C}}^M)}}{r_0^{(n-1)/p}},
\end{align}
whereas if $\ell(Q)<r_0$ we make use of the uniform continuity of $f$ to estimate 
\begin{equation}\label{j8hVGVV.2}
\aver{Q}\big|f-f_Q\big|\,d{\mathscr{L}}^{n-1}\leq\aver{Q}\aver{Q}|f(x')-f(y')|\,dx'\,dy'\leq 1.
\end{equation}
In turn, from \eqref{j8hVGVV.1}-\eqref{j8hVGVV.2} we then conclude that $f$ belongs to 
${\rm BMO}({\mathbb{R}}^{n-1},{\mathbb{C}}^M)$, which establishes the first inclusion in
\eqref{ku6ffcf-UCUC}. The second inclusion in \eqref{ku6ffcf-UCUC} is clear from \eqref{defi-VMO}.

As regards the second inclusion in \eqref{ku6ffcf-UCUC}, a well-known result of 
Sarason \cite[Theorem~1, p.\,392]{Sa75} implies that, in fact, 
\begin{equation}\label{ku6ffcfc}
\parbox{9.70cm}{a given function $f\in{\mathrm{BMO}}({\mathbb{R}}^{n-1},\mathbb{C}^M)$
actually belongs to the space ${\mathrm{VMO}}({\mathbb{R}}^{n-1},\mathbb{C}^M)$ if and only if 
there exists a sequence 
$\{f_j\}_{j\in{\mathbb{N}}}\subset{\mathrm{UC}}({\mathbb{R}}^{n-1},\mathbb{C}^M)
\cap{\mathrm{BMO}}({\mathbb{R}}^{n-1},\mathbb{C}^M)$ such that 
$\|f-f_j\|_{{\mathrm{BMO}}({\mathbb{R}}^{n-1},\mathbb{C}^M)}\longrightarrow 0$ as $j\to\infty$.}
\end{equation}
We shall refer to this simply by saying that Sarason's {\rm VMO} space is the closure 
of ${\rm UC}\cap{\rm BMO}$ in the space {\rm BMO}. In relation to \eqref{ku6ffcf-UCUC}
we wish to note that continuity without uniformity will not preserve the inclusion 
in \eqref{ku6ffcf-UCUC}. For example, there exist functions in
${\mathscr{C}}^\infty({\mathbb{R}})\cap L^\infty({\mathbb{R}})$ which do not belong to
${\rm VMO}({\mathbb{R}})$. To see this, consider the mutually disjoint intervals
$I_j:=[j,j+2/j]$, for each $j\in{\mathbb{N}}$, $j\geq 3$. Now let
$f:{\mathbb{R}}\to{\mathbb{R}}$ be a function with the property that,
for each $j\in{\mathbb{N}}$, $j\geq 3$, the graph of $f\big|_{I_j}$ is the line 
segment joining the point $(j,-1)$ with $(j+2/j,1)$ and otherwise 
the graph of $f$ is made up of curves joining these line segments smoothly within
the strip ${\mathbb{R}}\times[-2,2]$. By design, 
$f\in{\mathscr{C}}^\infty({\mathbb{R}})\cap L^\infty({\mathbb{R}})$.
In particular, $f\in{\rm BMO}({\mathbb{R}})$. However, for each
$j\in{\mathbb{N}}$, $j\geq 3$ we have $f_{I_j}=0$ and 
\begin{align}\label{jusfewv}
\aver{I_j}\big|f-f_{I_j}\big|\,d{\mathscr{L}}^1
=\aver{I_j}|f|\,d{\mathscr{L}}^1=\frac{1}{2}.
\end{align}
Since $|I_j|=2/j\to 0$ as $j\to\infty$, from \eqref{jusfewv} and \eqref{defi-VMO} 
it is then clear that $f\not\in{\rm VMO}({\mathbb{R}})$.

\vskip 0.10in

Another characterization of ${\mathrm{VMO}}(\mathbb{R}^{n-1},\mathbb{C}^M)$
due to Sarason (cf. \cite[Theorem~1, p.\,392]{Sa75}) is as follows:
\begin{equation}\label{defi-VMO-SSS}
\parbox{7.80cm}{a given function $f\in{\mathrm{BMO}}(\mathbb{R}^{n-1},\mathbb{C}^M)$ 
actually belongs to the space
${\mathrm{VMO}}(\mathbb{R}^{n-1},\mathbb{C}^M)$ if and only if 
$\lim\limits_{{\mathbb{R}}^{n-1}\ni z'\to 0'}
\|\tau_{z'}f-f\|_{{\mathrm{BMO}}(\mathbb{R}^{n-1},\mathbb{C}^M)}=0$.}
\end{equation}

We are now ready to state our first main result. This concerns the well-posedness of the 
$\mathrm{BMO}$-Dirichlet problem in the upper-half space for systems $L$ as in 
\eqref{L-def}-\eqref{L-ell.X}. The existence of a unique solution is established 
in the class of functions $u$ satisfying a Carleson measure condition 
(expressed in terms of the finiteness of \eqref{ustarstar}). The formulation of 
our theorem emphasizes the fact that this contains as a ``sub-problem'' the 
$\mathrm{VMO}$-Dirichlet problem for $L$ in ${\mathbb{R}}^n_{+}$
(in which scenario $u$ satisfies a vanishing Carleson measure condition). 

\begin{theorem}\label{them:BMO-Dir}
Let $L$ be an $M\times M$ elliptic constant complex coefficient system as in 
\eqref{L-def}-\eqref{L-ell.X}. Then the $\mathrm{BMO}$-Dirichlet boundary value 
problem for $L$ in $\mathbb{R}^{n}_+$, namely 
\begin{equation}\label{Dir-BVP-BMO}
\left\{
\begin{array}{l}
u\in{\mathscr{C}}^\infty(\mathbb{R}^{n}_{+},{\mathbb{C}}^M),
\\[4pt]
Lu=0\,\,\mbox{ in }\,\,\mathbb{R}^{n}_{+},
\\[4pt]
\big|\nabla u(x',t)\big|^2\,t\,dx'dt\,\,\mbox{is a Carleson measure in }\mathbb{R}^{n}_+,
\\[6pt]
u\big|^{{}^{\rm n.t.}}_{\partial{\mathbb{R}}^{n}_{+}}=f\,\,
\text{ a.e. in }\,\,{\mathbb{R}}^{n-1},\,\,
f\in\mathrm{BMO}(\mathbb{R}^{n-1},\mathbb{C}^M),
\end{array}
\right.
\end{equation}
has a unique solution. Moreover, this unique solution satisfies the following additional properties:
\begin{list}{$(\theenumi)$}{\usecounter{enumi}\leftmargin=.8cm
\labelwidth=.8cm\itemsep=0.2cm\topsep=.1cm
\renewcommand{\theenumi}{\alph{enumi}}}
\item[(i)] With $P^L$ denoting the Poisson kernel for $L$ in $\mathbb{R}^{n}_+$ from
Theorem~\ref{kkjbhV}, one has the Poisson integral representation formula
\begin{equation}\label{eqn-Dir-BMO:u}
u(x',t)=(P_t^L*f)(x'),\qquad\forall\,(x',t)\in{\mathbb{R}}^n_{+}.
\end{equation}
\item[(ii)] There exists a constant $C=C(n,L)\in(1,\infty)$ with the property that 
the solution $u$ of the Dirichlet problem \eqref{Dir-BVP-BMO} satisfies the two-sided estimate
\begin{equation}\label{Dir-BVP-BMO-Car}
C^{-1}\|f\|_{\mathrm{BMO}(\mathbb{R}^{n-1},\mathbb{C}^M)}\leq
\|u\|_{**}\leq C\,\|f\|_{\mathrm{BMO}(\mathbb{R}^{n-1},\mathbb{C}^M)}.
\end{equation}
That is, the size of the solution is comparable to the size of the boundary datum.
\vskip 0.08in
\item[(iii)] For each $\varepsilon>0$ the function $u(\cdot,\varepsilon)$ belongs to 
${\mathrm{BMO}}(\mathbb{R}^{n-1},\mathbb{C}^M)$ and there exists a constant $C=C(n,L)\in(0,\infty)$ 
independent of $u$ with the property that the following uniform {\rm BMO} estimate holds:
\begin{equation}\label{feps-BTTGB}
\sup_{\varepsilon>0}\|u(\cdot,\varepsilon)\|_{\mathrm{BMO}(\mathbb{R}^{n-1},\mathbb{C}^M)}
\leq C\,\|u\|_{**}.
\end{equation}
Moreover, 
\begin{equation}\label{eqn:conv-Bfed}
{}\hskip 0.30in
\lim_{\varepsilon\to 0^+}\|u(\cdot,\varepsilon)-f\|_{\mathrm{BMO}(\mathbb{R}^{n-1},\mathbb{C}^M)}=0
\Longleftrightarrow
\left\{
\begin{array}{l}
\big|\nabla u(x',t)\big|^2\,t\,dx'dt\,\,\mbox{is a vanishing}
\\[4pt]
\text{Carleson measure in }\,\,\mathbb{R}^{n}_{+}.
\end{array}
\right.
\end{equation}
That is, $u$ satisfies a vanishing Carleson measure condition in $\mathbb{R}^{n}_{+}$ if and only if 
$u$ converges to its boundary datum vertically in $\mathrm{BMO}(\mathbb{R}^{n-1},\mathbb{C}^M)$.
\vskip 0.08in
\item[(iv)] The following regularity results hold:
\begin{align}\label{Dir-BVP-Reg}
f\in{\mathrm{VMO}(\mathbb{R}^{n-1},\mathbb{C}^M)}
& \Longleftrightarrow
\left\{
\begin{array}{l}
\big|\nabla u(x',t)\big|^2\,t\,dx'dt\,\,\mbox{is a vanishing}
\\[4pt]
\text{Carleson measure in }\,\,\mathbb{R}^{n}_+
\end{array}
\right.
\\[6pt]
& \Longleftrightarrow
\lim_{{\mathbb{R}}^n_{+}\ni z\to 0}\|\tau_z u-u\|_{**}=0,
\label{Dir-BVP-Reg.TTT}
\end{align}
where $(\tau_z u)(x):=u(x+z)$ for each $x,z\in{\mathbb{R}}^n_{+}$.
\end{list}

As a consequence, the $\mathrm{VMO}$-Dirichlet boundary value problem for $L$ 
in $\mathbb{R}^{n}_+$, i.e.,
\begin{equation}\label{Dir-BVP-VMO}
\left\{
\begin{array}{l}
u\in{\mathscr{C}}^\infty(\mathbb{R}^{n}_{+},{\mathbb{C}}^M),
\\[4pt]
Lu=0\,\,\mbox{ in }\,\,\mathbb{R}^{n}_{+},
\\[4pt]
\big|\nabla u(x',t)\big|^2\,t\,dx'dt\,\,\mbox{is a vanishing Carleson measure in }\mathbb{R}^{n}_+,
\\[6pt]
u\big|^{{}^{\rm n.t.}}_{\partial{\mathbb{R}}^{n}_{+}}=f\,\,
\text{ a.e. in }\,\,{\mathbb{R}}^{n-1},\,\,
f\in\mathrm{VMO}(\mathbb{R}^{n-1},\mathbb{C}^M),
\end{array}
\right.
\end{equation}
is well-posed. Moreover, its unique solution is given by \eqref{eqn-Dir-BMO:u}, 
satisfies \eqref{Dir-BVP-BMO-Car}-\eqref{feps-BTTGB}, and
\begin{equation}\label{eqn:conv-BfEE}
\lim_{\varepsilon\to 0^+}\|u(\cdot,\varepsilon)-f\|_{\mathrm{BMO}(\mathbb{R}^{n-1},\mathbb{C}^M)}=0.
\end{equation}
\end{theorem}

It is reassuring to remark that replacing the boundary datum $f$ by $f+C$ where 
$C\in{\mathbb{C}}^M$ in \eqref{Dir-BVP-BMO} changes the solution $u$ into $u+C$
(given that convolution with the Poisson kernel reproduces constants from ${\mathbb{C}}^M$; 
cf. \eqref{eq:IG6gy.2PPP}).
As such, the $\widetilde{\rm BMO}$-Dirichlet problem for $L$ in ${\mathbb{R}}^n_{+}$ 
is also well-posed, if uniqueness of the solution is now understood modulo 
constants from ${\mathbb{C}}^M$. 

\medskip 

As regards the right-pointing implication in \eqref{Dir-BVP-Reg}, for suitable dense 
subspaces of {\rm VMO} we are able to precisely quantify the rate at which the Carleson 
measure $\big|\nabla u(x',t)\big|^2\,t\,dx'dt$ vanishes in $\mathbb{R}^{n}_{+}$.
For example, with $\dot{\mathscr{C}}^\eta({\mathbb{R}}^{n-1},{\mathbb{C}}^M)$ denoting 
the homogeneous H\"older space of order $\eta\in(0,1)$ of ${\mathbb{C}}^M$-valued functions defined
in ${\mathbb{R}}^{n-1}$, it follows from \eqref{hsrwWW-AA-uuu} in Proposition~\ref{prop-Dir-BMO:exis}
(cf. also \eqref{jsfd-3-fff}) that  
\begin{align}\label{eq:Car-r-restri-INTRO}
\begin{array}{c}
\text{if $f\in\dot{\mathscr{C}}^\eta(\mathbb{R}^{n-1},\mathbb{C}^M)$ with $\eta\in(0,1)$ and 
$u$ is as in \eqref{eqn-Dir-BMO:u}, then}
\\[6pt]
\displaystyle
\sup\limits_{Q\subset\mathbb{R}^{n-1},\,\ell(Q)\leq r} 
\Big(\int_{0}^{\ell(Q)}\aver{Q} 
|\nabla u(x',t)|^2\,t\,dx'dt\Big)^{\frac12}=O(r^\eta)\,\,\text{ as }\,\,r\to 0^{+},
\end{array}
\end{align}
where the multiplicative constant implicit in the big-{\it O} condition above depends only 
on $n,L$, $\eta$, and $\|f\|_{\dot{\mathscr{C}}^\eta(\mathbb{R}^{n-1},\mathbb{C}^M)}$. 
The relevance of this result stems from the fact that, for each $\eta\in(0,1)$, 
the collection of functions from ${\mathrm{BMO}}({\mathbb{R}}^{n-1},{\mathbb{C}}^M)$ which also 
belong to $\dot{\mathscr{C}}^\eta({\mathbb{R}}^{n-1},{\mathbb{C}}^M)$ 
make up a dense subspace of ${\mathrm{VMO}}({\mathbb{R}}^{n-1},{\mathbb{C}}^M)$. 
The latter density result constitutes one of the main results in this paper, 
and is formally stated in Theorem~\ref{THMVMO.i}, along with a number of 
variants and generalizations. Let us also point out here that the decay rate 
in \eqref{eq:Car-r-restri-INTRO} is in agreement with the format of the 
well-posedness result proved later in Theorem~\ref{them:BMO-Dir-frac}
(in view of \eqref{defi-BMO.2b-SP.ii} and \eqref{ustarstar-222}).

\medskip

The proof of Theorem~\ref{them:BMO-Dir} relies on a quantitative Fatou type 
theorem, which includes a Poisson integral representation formula along with a 
characterization of {\rm BMO} in terms of the traces of solutions 
to elliptic systems. This is stated next as Theorem~\ref{thm:fatou-ADEEDE}. 
Among other things, the said theorem shows
that the conditions stipulated in the first three lines of \eqref{Dir-BVP-BMO} 
imply that the pointwise nontangential limit considered in the fourth 
line of \eqref{Dir-BVP-BMO} is always meaningful, and that the boundary 
datum should necessarily be selected from the space ${\rm BMO}$. It also 
highlights the fact that it is natural to seek a solution of 
the $\mathrm{BMO}$ Dirichlet problem by taking the convolution of the boundary datum 
with the Poisson kernel of $L$ in the upper-half space. Finally, 
Theorem~\ref{thm:fatou-ADEEDE} is the key ingredient in the proof of uniqueness for the 
$\mathrm{BMO}$-Dirichlet boundary value problem formulated in \eqref{Dir-BVP-BMO}.

\begin{theorem}\label{thm:fatou-ADEEDE}
Let $L$ be an $M\times M$ elliptic system with constant complex coefficients as in
\eqref{L-def}-\eqref{L-ell.X} and consider $P^L$, the associated Poisson kernel for 
$L$ in $\mathbb{R}^{n}_+$ from Theorem~\ref{kkjbhV}. Then there exists a constant
$C=C(L,n)\in(1,\infty)$ with the property that 
\begin{eqnarray}\label{Tafva.BMO}
\left.
\begin{array}{r}
u\in{\mathscr{C}}^\infty({\mathbb{R}}^n_{+},{\mathbb{C}}^M)
\\[4pt]
Lu=0\,\mbox{ in }\,{\mathbb{R}}^n_{+}
\\[6pt]
\text{and }\,\,\|u\|_{**}<\infty
\end{array}
\right\}
\Longrightarrow
\left\{
\begin{array}{l}
u\big|^{{}^{\rm n.t.}}_{\partial{\mathbb{R}}^n_{+}}\,\mbox{ exists a.e.~in }\,
{\mathbb{R}}^{n-1},\,\mbox{ lies in }\,\mathrm{BMO}(\mathbb{R}^{n-1},\mathbb{C}^M),
\\[12pt]
u(x',t)=\Big(P^L_t\ast\big(u\big|^{{}^{\rm n.t.}}_{\partial{\mathbb{R}}^n_{+}}\big)\Big)(x')
\,\text{ for all }\,(x',t)\in{\mathbb{R}}^n_{+},
\\[12pt]
\mbox{and }\,C^{-1}\|u\|_{**}\leq
\big\|u\big|^{{}^{\rm n.t.}}_{\partial{\mathbb{R}}^n_{+}}\big\|
_{\mathrm{BMO}(\mathbb{R}^{n-1},\mathbb{C}^M)}\leq C\|u\|_{**}.
\end{array}
\right.
\end{eqnarray}
In fact, the following characterization 
of $\mathrm{BMO}(\mathbb{R}^{n-1},\mathbb{C}^M)$, adapted to the system $L$, 
holds:
\begin{equation}\label{eq:tr-sols}
\mathrm{BMO}(\mathbb{R}^{n-1},\mathbb{C}^M)
=\Big\{u\big|^{{}^{\rm n.t.}}_{\partial{\mathbb{R}}^{n}_{+}}: 
u\in\mathscr{C}^\infty(\mathbb{R}^n_+,{\mathbb{C}}^M),\,\,Lu=0 
\mbox{ in }\,\mathbb{R}^{n}_{+},\,\,
\|u\|_{**}<\infty\Big\}.
\end{equation}

Moreover, 
\begin{equation}\label{eq:tr-OP-SP}
{\mathrm{LMO}}({\mathbb{R}}^n_{+}):=
\Big\{u\in\mathscr{C}^\infty(\mathbb{R}^n_+,{\mathbb{C}}^M):\,
Lu=0\mbox{ in }\mathbb{R}^{n}_{+},\,\,\|u\|_{**}<\infty\Big\}
\end{equation}
is a linear space on which $\|\cdot\|_{**}$ is a seminorm with null-space
${\mathbb{C}}^M$, the quotient space ${\mathrm{LMO}}({\mathbb{R}}^n_{+})\big/{\mathbb{C}}^M$ 
becomes complete {\rm (}hence Banach{\rm )} when equipped with $\|\cdot\|_{**}$, and 
the nontangential pointwise trace operator acting on equivalence classes in the context 
\begin{equation}\label{eq:tr-OP}
{\mathrm{LMO}}({\mathbb{R}}^n_{+})\big/{\mathbb{C}}^M\ni[u]\longmapsto 
\big[u\big|^{{}^{\rm n.t.}}_{\partial{\mathbb{R}}^{n}_{+}}\big]
\in\widetilde{\mathrm{BMO}}(\mathbb{R}^{n-1},\mathbb{C}^M)
\end{equation}
is a well-defined linear isomorphism between Banach spaces, where
$[u]$ in \eqref{eq:tr-OP} denotes the equivalence class of $u$ in 
${\mathrm{LMO}}({\mathbb{R}}^n_{+})\big/{\mathbb{C}}^M$ and 
$\big[u\big|^{{}^{\rm n.t.}}_{\partial{\mathbb{R}}^{n}_{+}}\big]$ is interpreted as
in \eqref{jgsyjw-AASSS}.
\end{theorem}

There is a counterpart of the Fatou type result stated as Theorem~\ref{thm:fatou-ADEEDE}
emphasizing the space {\rm VMO} in place of {\rm BMO}. Specifically, we prove the following 
theorem. 

\begin{theorem}\label{thm:fatou-VMO}
Let $L$ be an $M\times M$ elliptic system with constant complex coefficients as in
\eqref{L-def}-\eqref{L-ell.X} and consider $P^L$, the associated Poisson kernel for 
$L$ in $\mathbb{R}^{n}_+$ from Theorem~\ref{kkjbhV}. Then for any 
function 
\begin{equation}\label{Dir-BVP-VMOq1}
\text{$u\in{\mathscr{C}}^\infty({\mathbb{R}}^n_{+},{\mathbb{C}}^M)$ satisfying 
$Lu=0$ in ${\mathbb{R}}^n_{+}$ and $\|u\|_{**}<\infty$}
\end{equation}
one has
\begin{equation}\label{Dir-BVP-VMOq2}
\left.
\begin{array}{l}
\big|\nabla u(x',t)\big|^2\,t\,dx'dt\,\,\mbox{is a vanishing}
\\[4pt]
\text{Carleson measure in }\,\,\mathbb{R}^{n}_+
\end{array}
\right\}
\Longrightarrow\,\,
u\big|^{{}^{\rm n.t.}}_{\partial{\mathbb{R}}^n_{+}}
\in{\mathrm{VMO}(\mathbb{R}^{n-1},\mathbb{C}^M)}.
\end{equation}

Furthermore, the following characterization of the space 
$\mathrm{VMO}(\mathbb{R}^{n-1},\mathbb{C}^M)$, adapted to the system $L$, 
holds:
\begin{align}\label{eq:tr-sols-VMO}
\mathrm{VMO}(\mathbb{R}^{n-1},\mathbb{C}^M)
=\Big\{u\big|^{{}^{\rm n.t.}}_{\partial{\mathbb{R}}^{n}_{+}}:\,\,&
u\in\mathrm{LMO}(\mathbb{R}^n_{+},\mathbb{C}^M)\,\,\text{ and }\,\,
\,\big|\nabla u(x',t)\big|^2\,t\,dx'dt
\nonumber\\[0pt] 
&\text{is a vanishing Carleson measure in }\,\,\mathbb{R}^{n}_+\Big\}.
\end{align}
\end{theorem}

The analogue of Fefferman's theorem, characterizing {\rm BMO} as in \eqref{L-dJHG}, 
in the case of elliptic systems with complex coefficients makes the topic of the 
first item of our next theorem. The second item may be viewed as a characterization 
of {\rm VMO} in the spirit of Fefferman's original result. 

\begin{theorem}\label{thm:FEFF}
Let $L$ be an $M\times M$ elliptic system with constant complex coefficients as in
\eqref{L-def}-\eqref{L-ell.X} and consider $P^L$, the associated Poisson kernel for 
$L$ in $\mathbb{R}^{n}_+$ from Theorem~\ref{kkjbhV}. Assume 
$f:\mathbb{R}^{n-1}\to\mathbb{C}^{M}$ is a Lebesgue measurable function such that
\begin{equation}\label{Di-AK}
\int_{{\mathbb{R}}^{n-1}}\frac{|f(x')|}{1+|x'|^{n}}\,dx'<\infty.
\end{equation}
Let $u$ be the Poisson integral of $f$ with respect to the system $L$, i.e., 
$u:{\mathbb{R}}^n_{+}\to{\mathbb{C}}^M$ is given by $u(x',t):=(P^L_t\ast f)(x')$ 
for each $(x',t)\in{\mathbb{R}}^n_{+}$. Then the following are true:

\begin{enumerate}
\item[(1)] $f$ belongs to the space ${\rm BMO}({\mathbb{R}}^{n-1};\mathbb{C}^{M})$ 
if and only if $\|u\|_{\ast\ast}<\infty$;
\vskip 0.08in
\item[(2)] $f$ belongs to the space ${\rm VMO}({\mathbb{R}}^{n-1};\mathbb{C}^{M})$ 
if and only if $|\nabla u(x',t)|^2\,t\,dx'dt$ is a vanishing Carleson measure in 
${\mathbb{R}}^n_{+}$.
\end{enumerate}
\end{theorem}

In our next result we shall revisit the issue of describing {\rm VMO} as the closure 
within {\rm BMO} of a subspace of functions whose pointwise oscillations vanish as the scale
decreases to zero. One such description is contained in \eqref{ku6ffcfc}. However, for a 
variety of purposes (such as the proof of the result recorded in Theorem~\ref{i87hbBV} below), 
the fact that the condition of uniform continuity is of a purely qualitative 
nature renders the space {\rm UC} difficult to work with. As such, 
it is very desirable to replace it, 
in the context of Sarason's density result recorded in \eqref{ku6ffcfc}, by smaller subspaces 
within which uniform continuity may be suitably quantified. This issue is addressed in 
Theorem~\ref{THMVMO.i} below. As a preamble, we introduce some notation. 
Pick a modulus of continuity, i.e., a function 
\begin{align}\label{UpU}  
\Upsilon:[0,\infty)\to[0,\infty]\,\,\text{ nondecreasing and such that}\,
\lim\limits_{s\to 0^{+}}\Upsilon(s)=0.
\end{align}
Given $m\in{\mathbb{N}}$, consider the space
\begin{align}\label{UpUpUp}
& \hskip -0.50in
{\mathscr{C}}^{\Upsilon}({\mathbb{R}}^{m}):=\big\{
f:{\mathbb{R}}^{m}\to\mathbb{C}:\,\text{ there exists }\,C\in(0,\infty)
\,\text{ such that}
\nonumber\\[4pt]
& \hskip 1.00in
|f(a)-f(b)|\leq C\Upsilon(|a-b|)\,\text{ for all }\,a,b\in{\mathbb{R}}^{m}\big\}
\end{align}
and define $\|f\|_{{\mathscr{C}}^{\Upsilon}({\mathbb{R}}^{m})}$ to be the smallest constant $C$ 
intervening above. In the sequel, the space of ${\mathbb{C}}^M$-valued functions with components in 
${\mathscr{C}}^{\Upsilon}({\mathbb{R}}^{m})$ will be denoted 
by ${\mathscr{C}}^{\Upsilon}({\mathbb{R}}^{m},{\mathbb{C}}^M)$.
Clearly, ${\mathscr{C}}^{\Upsilon}({\mathbb{R}}^{m})\subseteq
{\mathrm{UC}}({\mathbb{R}}^{m})$ and, in fact, 
\begin{equation}\label{i7y554}
{\mathrm{UC}}({\mathbb{R}}^{m})
=\bigcup\limits_{\stackrel{\Upsilon\,\text{modulus of}}{\text{\tiny{continuity}}}}
{\mathscr{C}}^{\Upsilon}({\mathbb{R}}^{m}).
\end{equation}
To see the left-to-right inclusion in \eqref{i7y554}, observe that if 
$f\in{\mathrm{UC}}({\mathbb{R}}^{m})$ is arbitrary and we define
$\Upsilon_f(s):=\sup\{|f(x)-f(y)|:\,x,y\in{\mathbb{R}}^m,\,|x-y|\leq s\}$ for each 
$s\in[0,\infty)$, then $\Upsilon_f$ is a modulus of continuity and 
$|f(a)-f(b)|\leq \Upsilon_f(|a-b|)$ for all $a,b\in{\mathbb{R}}^{m}$, hence
$f\in{\mathscr{C}}^{\Upsilon_f}({\mathbb{R}}^{m})$, as wanted.

Examples of interest are obtained by taking $\eta\in(0,1]$ and defining $\Upsilon_\eta(s):=s^\eta$ 
for every $s\geq 0$. Then the space ${\mathscr{C}}^{\Upsilon_\eta}({\mathbb{R}}^{m})$ becomes 
precisely $\dot{\mathscr{C}}^\eta({\mathbb{R}}^{m})$, the space of functions satisfying 
a homogeneous H\"older condition of order $\eta$ in ${\mathbb{R}}^{m}$ in the case when $\eta\in(0,1)$, 
and becomes ${\mathrm{Lip}}({\mathbb{R}}^{m})$, the space of Lipschitz functions in ${\mathbb{R}}^{m}$,
in the case when $\eta=1$.

Here is the theorem advertised earlier, which may be regarded as a quantitative description of {\rm VMO}, 
improving on Sarason's classical result \eqref{ku6ffcfc}. 

\begin{theorem}\label{THMVMO.i}
Consider the function $\Upsilon_{\!\#}:[0,\infty)\to [0,\infty)$ given 
at each $s\geq 0$ by 
\begin{equation}\label{decay-infty:Ups0}
\Upsilon_{\!\#}(s):=\min\{1,s\}+\max\{0,\ln s\}
=\left\{
\begin{array}{ll}
s, &\,\mbox{if }\,s\leq 1,
\\[6pt]
1+\ln s, &\,\mbox{if }\,s>1.
\end{array}
\right.
\end{equation}
Then for every modulus of continuity $\Upsilon$ with the property that $\Upsilon_{\!\#}\leq C\Upsilon$
on $[0,\infty)$ for some finite constant $C>0$, the following density result holds for each $n\in{\mathbb{N}}$:
\begin{equation}\label{UpUpUp.2}
\parbox{8.60cm}{for every function $f\in{\mathrm{VMO}}({\mathbb{R}}^{n})$ 
there exists a sequence $\{f_j\}_{j\in{\mathbb{N}}}\subset
{\mathscr{C}}^\Upsilon({\mathbb{R}}^{n})\cap{\mathscr{C}}^\infty({\mathbb{R}}^{n})
\cap{\mathrm{BMO}}({\mathbb{R}}^{n})$ 
such that $\|f-f_j\|_{{\mathrm{BMO}}({\mathbb{R}}^{n})}\longrightarrow 0$ as $j\to\infty$.}
\end{equation}
In short, ${\mathscr{C}}^\Upsilon({\mathbb{R}}^{n})\cap{\mathscr{C}}^\infty({\mathbb{R}}^{n})
\cap{\mathrm{BMO}}({\mathbb{R}}^{n})$ is dense in ${\mathrm{VMO}}({\mathbb{R}}^{n})$.
In fact, 
\begin{equation}\label{iy65ffvgH}
\parbox{11.80cm}{the smaller space, consisting of 
$f\in{\mathscr{C}}^\Upsilon({\mathbb{R}}^{n})\cap{\mathscr{C}}^\infty({\mathbb{R}}^{n})
\cap{\mathrm{BMO}}({\mathbb{R}}^{n})$ such that $\partial^{\alpha}f\in{\mathscr{C}}^\Upsilon({\mathbb{R}}^{n})
\cap L^\infty({\mathbb{R}}^{n})$ for every $\alpha\in{\mathbb{N}}_0^{n}$ with 
$|\alpha|\geq 1$, is also dense in ${\mathrm{VMO}}({\mathbb{R}}^{n})$.}
\end{equation}
\end{theorem}

The proof of Theorem~\ref{THMVMO.i} (stated with $n-1$ in place of $n$) relies on the 
fact that, given any $f\in{\mathrm{BMO}}({\mathbb{R}}^{n-1},{\mathbb{C}}^M)$, we have 
(as seen from \eqref{eqn-Dir-BMO:u} and \eqref{eqn:conv-Bfed}-\eqref{Dir-BVP-Reg})
\begin{equation}\label{P-qi7GVV}
P^L_t\ast f\longrightarrow f\,\,\text{ in }\,\,{\mathrm{BMO}}({\mathbb{R}}^{n-1},{\mathbb{C}}^M)
\,\,\text{ as }\,\,t\to 0^{+}\,\Longleftrightarrow\,
f\in{\mathrm{VMO}}({\mathbb{R}}^{n-1},{\mathbb{C}}^M),
\end{equation}
for some (or any) $M\times M$ elliptic system $L$  with constant complex coefficients as in
\eqref{L-def}-\eqref{L-ell.X}. A posteriori, once the density result in Theorem~\ref{THMVMO.i} has been 
established, we can considerably enlarge the class of approximations to the identity 
for which a result as in \eqref{P-qi7GVV} holds, as described below. 

\begin{theorem}\label{ndyRE}
Suppose $\varphi:{\mathbb{R}}^{n}\to{\mathbb{C}}^{M\times M}$ has the property 
that there exist $C\in(0,\infty)$ and $\varepsilon\in(0,1]$ such that
\begin{equation}\label{Bgstwy-2-new}
|\varphi(x)|\leq C(1+|x|)^{-n-\varepsilon}\quad\mbox{for every }\,\,x\in\mathbb{R}^{n}\setminus\{0\},
\end{equation}
and
\begin{equation}\label{Bgstwy-2-new2}
|\varphi(x+h)-\varphi(x)|\leq\frac{C|h|^\varepsilon}{|x|^{n+\varepsilon}}\quad
\mbox{for all }\,\,x\in\mathbb{R}^{n}\setminus\{0\},\,\,h\in{\mathbb{R}}^n,\,\,|h|<|x|/2.
\end{equation}
In addition, assume
\begin{equation}\label{Bgstwy}
\int_{\mathbb{R}^{n}}\varphi(x)\,dx=I_{M\times M}
\end{equation}
{\rm (}the $M\times M$ identity matrix{\rm )}. Then
\begin{equation}\label{Bgstwy-3}
\begin{array}{c}
\text{for each $f\in{\mathrm{VMO}}({\mathbb{R}}^{n},{\mathbb{C}}^M)$ there holds}
\\[6pt]
\varphi_t\ast f\longrightarrow f\,\text{ in }\,
{\mathrm{BMO}}({\mathbb{R}}^{n},{\mathbb{C}}^M),\,\text{ as }\,t\to 0^{+},
\end{array}
\end{equation}
where, in the present context, $\varphi_t(x):=t^{-n}\varphi(x/t)$
for each $x\in{\mathbb{R}}^n$ and each $t>0$.

As a consequence, given $\varphi\in{\mathscr{C}}^1\big({\mathbb{R}}^{n},{\mathbb{C}}^{M\times M}\big)$
such that \eqref{Bgstwy} holds and such that there exists $C\in(0,\infty)$ for which 
\begin{equation}\label{Bgstwy-2aaa}
|\varphi(x)|+|(\nabla\varphi)(x)|
\leq C(1+|x|)^{-n-1}\,\,\mbox{ for every }\,\,x\in\mathbb{R}^{n},
\end{equation}
one has the following real-variable characterization of the membership to ${\rm VMO}$:
\begin{equation}\label{Bgstwy-3TRG}
\begin{array}{c}
\text{for every function $f\in{\mathrm{BMO}}({\mathbb{R}}^{n},{\mathbb{C}}^M)$ there holds}
\\[6pt]
\varphi_t\ast f\rightarrow f\,\text{ in }\,
{\mathrm{BMO}}({\mathbb{R}}^{n},{\mathbb{C}}^M)\,\text{ as }\,t\to 0^{+}
\Longleftrightarrow\,f\in{\mathrm{VMO}}({\mathbb{R}}^{n},{\mathbb{C}}^M).
\end{array}
\end{equation}
\end{theorem}

\vskip 0.10in

Several density results, of independent interest, are obtained by 
specializing Theorem~\ref{THMVMO.i} to moduli of continuity of the form 
$\Upsilon_\eta(s):=s^\eta$ for $s\geq 0$, with $\eta\in(0,1]$
(simply by observing that there exists some finite constant $C_\eta>0$ with the property that
$\Upsilon_{\!\#}\leq C_\eta\Upsilon_\eta$ on $[0,\infty)$). To state these, recall that the 
inhomogeneous H\"older space of order $\eta\in(0,1)$ in ${\mathbb{R}}^n$ is defined as
\begin{equation}\label{ku7tgg}
{\mathscr{C}}^\eta({\mathbb{R}}^{n}):=\dot{\mathscr{C}}^\eta({\mathbb{R}}^{n})
\cap L^\infty({\mathbb{R}}^{n}).
\end{equation}

\begin{corollary}\label{Cbna-j77h}
For each $\eta\in(0,1)$, 
\begin{equation}\label{iy65ffvgH-111}
\parbox{11.80cm}{the space consisting of 
$f\in\dot{\mathscr{C}}^\eta({\mathbb{R}}^{n})\cap{\mathscr{C}}^\infty({\mathbb{R}}^{n})
\cap{\mathrm{BMO}}({\mathbb{R}}^{n})$ such that 
$\partial^{\alpha}f\in{\mathscr{C}}^\eta({\mathbb{R}}^{n})$ 
for every $\alpha\in{\mathbb{N}}_0^{n}$ with $|\alpha|\geq 1$ is dense in 
${\mathrm{VMO}}({\mathbb{R}}^{n})$.}
\end{equation}
Consequently, for each $\eta\in(0,1)$,
\begin{equation}\label{UpUpUp.2c}
\begin{array}{c}
\dot{\mathscr{C}}^\eta({\mathbb{R}}^{n})\cap{\mathscr{C}}^\infty({\mathbb{R}}^{n})
\cap{\mathrm{BMO}}({\mathbb{R}}^{n})
\\[6pt]
\text{is a dense subspace of }\,{\mathrm{VMO}}({\mathbb{R}}^{n}).
\end{array}
\end{equation}
In particular, for each $\eta\in(0,1)$ the space 
$\dot{\mathscr{C}}^\eta({\mathbb{R}}^{n})\cap{\mathrm{BMO}}({\mathbb{R}}^{n})$ is dense in
${\mathrm{VMO}}({\mathbb{R}}^{n})$. Moreover, 
\begin{equation}\label{iy65ffvgH-122e}
\parbox{12.40cm}{the space consisting of functions
$f\in{\mathrm{Lip}}({\mathbb{R}}^{n})\cap{\mathscr{C}}^\infty({\mathbb{R}}^{n})
\cap{\mathrm{BMO}}({\mathbb{R}}^{n})$ such that 
$\partial^{\alpha}f\in{\mathrm{Lip}}({\mathbb{R}}^{n})\cap L^\infty({\mathbb{R}}^{n})$ 
for every $\alpha\in{\mathbb{N}}_0^{n}$ with $|\alpha|\geq 1$ is dense in 
${\mathrm{VMO}}({\mathbb{R}}^{n})$.}
\end{equation}
In particular, 
\begin{equation}\label{UpUpUp.2b}
\begin{array}{c}
{\mathrm{Lip}}({\mathbb{R}}^{n})\cap{\mathscr{C}}^\infty({\mathbb{R}}^{n})
\cap{\mathrm{BMO}}({\mathbb{R}}^{n})
\\[6pt]
\text{is a dense subspace of }\,{\mathrm{VMO}}({\mathbb{R}}^{n}).
\end{array}
\end{equation}
\end{corollary}

An interesting feature of Theorem~\ref{THMVMO.i} is that even though the conclusions 
are of a purely real-variable nature, its proof makes essential use of the PDE-rooted 
results established earlier (such as the well-posedness of the {\rm BMO}-Dirichlet 
problem for, say, the Laplacian in ${\mathbb{R}}^n_{+}$). See \S\ref{Pf-mainThms} 
for details. Theorem~\ref{THMVMO.i} should be contrasted with the following 
negative result. 

\begin{theorem}\label{THMVMO.CCC}
The space ${\mathrm{UC}}({\mathbb{R}}^{n})\cap L^\infty({\mathbb{R}}^{n})$ 
is not dense in ${\mathrm{VMO}}({\mathbb{R}}^{n})$. 
\end{theorem}

An example of an unbounded function belonging to ${\mathrm{VMO}}({\mathbb{R}}^{n})$ is 
\begin{equation}\label{kjdg}
f(x):=\left\{
\begin{array}{ll}
\ln\ln(1/|x|) & \text{if }\,\,|x|\leq 1/e,
\\[4pt]
0& \text{if }\,\,|x|>1/e,
\end{array}
\right.
\qquad\forall\,x\in{\mathbb{R}}^n.
\end{equation}

In the context of the main density result presented in Theorem~\ref{THMVMO.i}, the function 
$\Upsilon_{\!\#}$ defined in \eqref{decay-infty:Ups0} exhibits an optimal behavior both at 
small and large scales, which cannot be improved, in the following precise sense: 
{\it If $\Upsilon$ is a modulus of continuity with the property that} 
\begin{equation}\label{UpUpUp.2cEE}
\text{\it either }\,\,\Upsilon(s)/s=o(1)\,\,\text{\it as }\,\,s\to 0^{+},
\,\,\,\text{\it or }\,\,\Upsilon(s)=O(1)\,\,\text{\it as }\,\,s\to\infty,
\end{equation}
{\it then }
\begin{equation}\label{UpUpUp.2cSGB}
{\mathscr{C}}^\Upsilon({\mathbb{R}}^{n})\cap{\mathrm{BMO}}({\mathbb{R}}^{n})
\,\text{\it is not dense in }\,{\mathrm{VMO}}({\mathbb{R}}^{n}).
\end{equation}
Indeed, \eqref{UpUpUp.2cSGB} is clear when the first eventuality in \eqref{UpUpUp.2cEE}
materializes since the space ${\mathscr{C}}^\Upsilon({\mathbb{R}}^{n})$ reduces to just 
constants in this case. Also, in the scenario when the second possibility in 
\eqref{UpUpUp.2cEE} takes place, ${\mathscr{C}}^\Upsilon({\mathbb{R}}^{n})$ becomes 
a subspace of ${\mathrm{UC}}({\mathbb{R}}^{n})\cap L^\infty({\mathbb{R}}^{n})$, 
in which case the desired conclusion follows from Theorem~\ref{THMVMO.CCC}. 

Among other things, the density result stated in Corollary~\ref{Cbna-j77h} permits us to 
quantify the proximity of a Littlewood-Paley type measure to the class of 
vanishing Carleson measures in the upper-half space. This result, of a purely real variable nature, 
is formally stated in the theorem below. 

\begin{theorem}\label{ndyRE-NNN}
Let $\psi\in{\mathscr{C}}^1\big({\mathbb{R}}^{n}\big)$ be a function with the property 
that there exists $C\in(0,\infty)$ such that
\begin{equation}\label{Bgstwy-2aaa-NNN}
\begin{array}{c}
\displaystyle
|\psi(x)|\leq\frac{C}{(1+|x|)^{n+1}}\,\text{ and }\,
|(\nabla\psi)(x)|\leq\frac{C}{(1+|x|)^{n+2}}\,\mbox{ for every }\,x\in\mathbb{R}^{n},
\\[8pt]
\displaystyle
\text{as well as }\,\,\int_{\mathbb{R}^{n}}\psi(x)\,dx=0.
\end{array}
\end{equation}
For each $x\in{\mathbb{R}}^n$ and $t>0$ set $\psi_t(x):=t^{-n}\psi(x/t)$. 
Then for each function $f\in{\mathrm{BMO}}({\mathbb{R}}^{n})$
\begin{equation}\label{Bgstwy-3TRG-NNN}
\mu_f(x,t):=|(\psi_t\ast f)(x)|^2\frac{dx\,dt}{t}
\end{equation}
is a Carleson measure in ${\mathbb{R}}^{n+1}_{+}$ satisfying 
\begin{align}\label{defi-Carleson-Niii}
\lim_{r\to 0^{+}}\Bigg\{\sup\limits_{\substack{Q\subset\mathbb{R}^{n}\\ \ell(Q)\leq r}}
\frac{1}{|Q|}\int_{0}^{\ell(Q)}\int_Q|(\psi_t\ast f)(x)|^2\frac{dx\,dt}{t}\Bigg\}
\leq C\,{\rm dist}\,\big(f\,,\,{\mathrm{VMO}}({\mathbb{R}}^{n})\big)^2
\end{align}
where ${\rm dist}\,\big(f\,,\,{\mathrm{VMO}}({\mathbb{R}}^{n})\big)
:=\inf\big\{\|f-g\|_{{\mathrm{BMO}}({\mathbb{R}}^{n})}:
\,g\in{\mathrm{VMO}}({\mathbb{R}}^{n})\big\}$.

As a corollary, 
\begin{equation}\label{Bgstwy-2aaa-Njjjj}
\parbox{9.45cm}{if $\psi\in{\mathscr{C}}^1\big({\mathbb{R}}^{n}\big)$ is a function satisfying 
the conditions in \eqref{Bgstwy-2aaa-NNN} and $f\in{\mathrm{VMO}}({\mathbb{R}}^{n})$, it follows 
that $\mu_f(x,t)$, defined as in \eqref{Bgstwy-3TRG-NNN}, is a vanishing Carleson measure 
in ${\mathbb{R}}^{n+1}_{+}$.}
\end{equation}
\end{theorem}

Theorem~\ref{ndyRE-NNN} allows us to establish the result stated below which may be
regarded as a quantified version of the equivalence \eqref{Dir-BVP-Reg} 
in Theorem~\ref{them:BMO-Dir}.

\begin{theorem}\label{jhdwtRD}
Let $L$ be an $M\times M$ elliptic constant complex coefficient system as in 
\eqref{L-def}-\eqref{L-ell.X}. Then there exists a constant $C=C(n,L)\in(0,\infty)$ with the 
property that for any given $f\in{\mathrm{BMO}}({\mathbb{R}}^{n-1},\mathbb{C}^M)$ 
the unique solution $u$ of the $\mathrm{BMO}$-Dirichlet boundary value 
problem \eqref{Dir-BVP-BMO} for $L$ in $\mathbb{R}^{n}_+$ with boundary datum $f$ satisfies
\begin{align}\label{XeeTT}
\lim_{r\to 0^{+}}\left\{\sup_{Q\subset\mathbb{R}^{n-1},\,\ell(Q)\leq r} 
\int_{0}^{\ell(Q)}\aver{Q}|\nabla u(x',t)|^2\,t\,dx'dt\right\}
\leq C\,{\rm dist}\big(f,{\mathrm{VMO}}({\mathbb{R}}^{n-1},\mathbb{C}^M)\big)^2,
\end{align}
where ${\rm dist}\big(f,{\mathrm{VMO}}({\mathbb{R}}^{n-1},\mathbb{C}^M)\big)
:=\inf\limits_{g\in{\mathrm{BMO}}({\mathbb{R}}^{n-1},\mathbb{C}^M)}
\|f-g\|_{{\mathrm{BMO}}({\mathbb{R}}^{n-1},\mathbb{C}^M)}$.
\end{theorem}

\vskip 0.08in

Moving on, if in analogy with \eqref{defi-BMO-tilde} we also define 
\begin{equation}\label{defi-VMO-tilde}
\widetilde{\mathrm{VMO}}(\mathbb{R}^n)
:=\big\{[f]:\,f\in{\mathrm{VMO}}(\mathbb{R}^n)\big\},
\end{equation}
then $\widetilde{\mathrm{VMO}}(\mathbb{R}^n)$ becomes a closed subspace of 
the Banach space $\Big({\mathrm{BMO}}(\mathbb{R}^n)\,,\,
\big\|\,[\cdot]\,\big\|_{\widetilde{\mathrm{BMO}}(\mathbb{R}^n)}\Big)$.
In particular, $\Big(\widetilde{\mathrm{VMO}}(\mathbb{R}^n)\,,\,
\big\|\,[\cdot]\,\big\|_{\widetilde{\mathrm{BMO}}(\mathbb{R}^n)}\Big)$
is itself a Banach space. Likewise, for each $\eta\in(0,1)$ let us introduce the 
quotient space\footnote{Observe that since we are presently dealing with continuous functions, 
$f\sim g$ means that $f-g$ is everywhere equal to a constant} 
\begin{equation}\label{defi-CCC-tilde}
\dot{\mathscr{C}}^\eta(\mathbb{R}^n)\big/_{\sim}
:=\big\{[f]:\,f\in\dot{\mathscr{C}}^\eta(\mathbb{R}^n)\big\}
\end{equation}
and equip it with the norm 
\begin{align}\label{jgsyjw-VCb}
\big\|\,[f]\,\big\|_{\dot{\mathscr{C}}^\eta(\mathbb{R}^n)/_{\sim}}
:=\|f\|_{\dot{\mathscr{C}}^\eta(\mathbb{R}^n)},
\qquad\forall\,[f]\in\dot{\mathscr{C}}^\eta(\mathbb{R}^n)\big/_{\sim}.
\end{align}
Then $\Big(\dot{\mathscr{C}}^\eta(\mathbb{R}^n)\big/_{\sim}\,,\,
\big\|\,[\cdot]\,\big\|_{\dot{\mathscr{C}}^\eta(\mathbb{R}^n)/_{\sim}}\Big)$ 
becomes a Banach space. 

Regarding $\widetilde{\mathrm{VMO}}({\mathbb{R}}^{n})$ as a Banach space in the 
fashion described above, Corollary~\ref{Cbna-j77h} readily implies the following 
density result. 

\begin{corollary}\label{Cbna-j77h-TTT}
For each $\eta\in(0,1)$ the set $\big(\dot{\mathscr{C}}^\eta({\mathbb{R}}^{n})\big/_{\sim}\big)
\cap\widetilde{\mathrm{BMO}}({\mathbb{R}}^{n})$ is dense in 
$\widetilde{\mathrm{VMO}}({\mathbb{R}}^{n})$.
\end{corollary}

The quantitative characterizations of the Sarason space provided in 
Theorem~\ref{THMVMO.i}, Corollary~\ref{Cbna-j77h}, and Corollary~\ref{Cbna-j77h-TTT} 
have important consequences as far as the mapping properties of Calder\'on-Zygmund 
operators on {\rm VMO} are concerned. To elaborate on this aspect, we first recall 
the definition of the latter class of operators. 

\begin{definition}\label{h6f43sSSSA} 
Given $n\in{\mathbb{N}}$, for each $\gamma\in(0,1]$ denote by ${\rm SCZ}(n,\gamma)$ the collection of all 
linear and continuous mappings $T:{\mathscr{S}}({\mathbb{R}}^n)\to{\mathscr{S}}'({\mathbb{R}}^n)$
which extend to a bounded operator on $L^2({\mathbb{R}}^n)$ and have the property that 
there exist $C',C''\in(0,\infty)$ such that the Schwartz kernel $K(\cdot,\cdot)$ of $T$ satisfies 
\begin{align}\label{ytgfff.iiiii}
K\in L^1_{\rm loc}({\mathbb{R}}^n\times{\mathbb{R}}^n\setminus\text{\rm diag})
\end{align}
and, for every $x,y\in{\mathbb{R}}^n$ with $x\not=y$, and each 
$z\in{\mathbb{R}}^n$ with $|x-z|<\tfrac12|x-y|$, 
\begin{align}\label{ytgfff}
|K(x,y)|\leq\frac{C'}{|x-y|^n}\,\,\text{ and }\,\,|K(x,y)-K(z,y)|\leq C''\frac{|x-z|^\gamma}{|x-y|^{n+\gamma}}.
\end{align}
Simply call $T$ a semi-Calder\'on-Zygmund operator in ${\mathbb{R}}^n$ if 
$T\in\bigcup_{0<\gamma\leq 1}{\rm SCZ}(n,\gamma)$.

Also, for each $\gamma\in(0,1]$ introduce 
${\rm CZ}(n,\gamma):=\big\{T\in{\rm SCZ}(n,\gamma):\,T^\top\in{\rm SCZ}(n,\gamma)\big\}$
{\rm (}where $T^\top:{\mathscr{S}}({\mathbb{R}}^n)\to{\mathscr{S}}'({\mathbb{R}}^n)$ is the transposed 
of $T$, with Schwartz kernel $K^\top(x,y):=K(y,x)${\rm )}, and refer to the operators in 
$\bigcup_{0<\gamma\leq 1}{\rm CZ}(n,\gamma)$ as being Calder\'on-Zygmund operators in ${\mathbb{R}}^n$.
\end{definition}

Fix a semi-Calder\'on-Zygmund operator $T$ in ${\mathbb{R}}^n$. A classical result in harmonic analysis 
(see, e.g., the proof of \cite[Theorem~3, p.\,114]{Stein93} which readily adapts to the present setting)
is the fact that $T^\top$ maps the Hardy space $H^1$ boundedly into the space of absolutely integrable 
functions, i.e., 
\begin{align}\label{ytgfff.9ytrdf.455}
T^\top:H^1({\mathbb{R}}^n)\longrightarrow L^1({\mathbb{R}}^n)
\end{align}
is a well-defined, linear, and bounded operator. In particular, this allows us to define $T(1)$ as a
functional in $\widetilde{\rm BMO}({\mathbb{R}}^n)=\big(H^1({\mathbb{R}}^n)\big)^\ast$ acting on 
any $h\in H^1({\mathbb{R}}^n)$ according to 
\begin{align}\label{ytgfff.9ytrdf}
\big\langle T(1),h\rangle:=\int_{{\mathbb{R}}^n}T^\top h\,d{\mathscr{L}}^n.
\end{align}
In particular, with the notion of $H^1$-atom as in \eqref{defi-atom}, 
\begin{align}\label{ytgfff.9ytrdf-tDDD}
T(1)=0\,\,\text{ in }\,\,\widetilde{\rm BMO}({\mathbb{R}}^n)\,\Longleftrightarrow\,
\int_{{\mathbb{R}}^n}T^\top a\,d{\mathscr{L}}^n=0\,\,\text{ for each $H^1$-atom $a$}.
\end{align}
Via interpolation and duality we have
\begin{align}\label{ytgfff.9ytrdf.455-ff}
\parbox{11.00cm}{if $T$ is a semi-Calder\'on-Zygmund operator then $T$ is bounded on 
$L^p({\mathbb{R}}^n)$ for each $p\in(2,\infty)$; as a consequence, if $T$ 
is a Calder\'on-Zygmund operator then $T$ is bounded on $L^p({\mathbb{R}}^n)$ for $p\in(1,\infty)$.}
\end{align}

In this vein, we wish to remark that 
(recall that a function $\Theta:{\mathbb{R}}^{n}\setminus\{0\}\to{\mathbb{C}}$ is said to be 
positive homogeneous of degree $m$ provided $\Theta(\lambda x)=\lambda^{m}\Theta(x)$ for each 
$x\in{\mathbb{R}}^{n}\setminus\{0\}$ and each $\lambda\in(0,\infty)$)
\begin{align}\label{ytgfff.9ytrdf.ygfg}
\parbox{11.50cm}{a principal-value convolution type operator 
$T_\Theta:{\mathscr{S}}({\mathbb{R}}^n)\to{\mathscr{S}}'({\mathbb{R}}^n)$ 
given by $\displaystyle 
T_\Theta f(x):=\lim\limits_{\varepsilon\to 0^{+}}\int_{y\in{\mathbb{R}}^n\setminus B(x,\varepsilon)}
\Theta(x-y)f(y)\,dy$, for $f\in{\mathscr{S}}({\mathbb{R}}^n)$ and $x\in{\mathbb{R}}^n$,  
with a kernel $\Theta\in{\mathscr{C}}^1({\mathbb{R}}^n\setminus\{0\})$ which is positive 
homogenous of degree $-n$ and such that $\displaystyle\int_{S^{n-1}}\Theta(\omega)\,d\omega=0$,
is a Calder\'on-Zygmund operator in ${\mathbb{R}}^n$ (in the sense of Definition~\ref{h6f43sSSSA} with 
$\gamma=1$, $C'=\|\Theta\|_{L^\infty(S^{n-1})}$, and $C''=\|\nabla\Theta\|_{L^\infty(S^{n-1})}$) 
which satisfies $T_\Theta(1)=(T_\Theta)^\top(1)=0$ in $\widetilde{\rm BMO}({\mathbb{R}}^n)$. Moreover, 
if we define $\widetilde{\Theta}(x):=\Theta(-x)$ for each $x\in{\mathbb{R}}^n\setminus\{0\}$, then 
the transposed of $T_\Theta$ acting on $L^p({\mathbb{R}}^n)$ with $1<p<\infty$ 
is the operator $T_{\widetilde{\Theta}}$ acting on $L^{p'}({\mathbb{R}}^n)$ where $1/p+1/p'=1$.}
\end{align}
This is a consequence of the fact that such an operator $T_\Theta$ is a multiplier
(cf., e.g., \cite[Theorem~4.96, pp.\,172-173]{DM}), i.e., it satisfies
$\widehat{T_\Theta\varphi}=m_\Theta\widehat{\varphi}$ for each $\varphi\in{\mathscr{S}}({\mathbb{R}}^n)$, 
where `hat' stands for the Fourier transform. The symbol $m_\Theta$ is the Fourier transform of 
the tempered distribution ${\rm P.V.}\,\Theta$, defined as (cf. \cite[(4.4.2), p.\,136]{DM})
\begin{equation}\label{iu7HBBam}
\big\langle{\rm P.V.}\,\Theta\,,\,\varphi\big\rangle:=\lim_{\varepsilon\to 0^{+}}
\int\limits_{x\in{\mathbb{R}}^n,\,|x|>\varepsilon}\Theta(x)\varphi(x)\,dx,\qquad
\forall\varphi\in{\mathscr{S}}({\mathbb{R}}^n),
\end{equation}
i.e., 
\begin{align}\label{PhBB-1}
m_{\Theta}=\widehat{{\rm P.V.}\,\Theta}\,\,\text{ in }\,{\mathscr{S}}'({\mathbb{R}}^n).
\end{align}   
From \cite[Theorem~4.71, p.\,142]{DM} it is known that 
\begin{align}\label{PhBB-1nv}
m_{\Theta}(\xi) &=-\int_{S^{n-1}}\Theta(\omega)\log\big(i\,\xi\cdot\omega\big)\,d\omega
\nonumber\\[6pt]
&=-\int_{S^{n-1}}\Theta(\omega)\Big(\ln\Big|\tfrac{\xi}{|\xi|}\cdot\omega\Big|
+i\frac{\pi}{2}{\rm sgn}(\xi\cdot\omega)\Big)\,d\omega
\,\,\,\text{ for each }\,\xi\in{\mathbb{R}}^n\setminus\{0\},
\end{align}   
where the last equality uses the vanishing moment property of $\Theta$ (see \cite[(4.5.15), p.\,143]{DM}). 
From this representation it is then apparent (reasoning as in \cite[Step~II, pp.\,349-350]{DM}) that 
\begin{align}\label{PhBB-1gd}
\parbox{11.00cm}{the restriction of $m_\Theta$ to ${\mathbb{R}}^n\setminus\{0\}$ is a function having 
the same order of differentiability as $\Theta$, is positive homogeneous of degree zero
bounded, satisfies $\overline{m_{\widetilde{\Theta}}}=m_{\overline{\Theta}}$ and
$\int_{S^{n-1}}m_{\Theta}(\omega)\,d\omega=0$, as well as $m_{\widetilde{\Theta}}(\xi)=m_\Theta(-\xi)$ 
for each $\xi\in{\mathbb{R}}^n\setminus\{0\}$.}
\end{align}   
Let us also note that, starting with \eqref{PhBB-1nv} and making use of \cite[Proposition~13.46, p.\,439]{DM} 
it is not difficult to see that 
\begin{align}\label{PhBB-1nv-ugg}
\parbox{7.50cm}{for each $p\in(1,\infty]$ there exists $C_{n,p}\in[0,\infty)$ such that 
$\|m_{\Theta}\|_{L^\infty({\mathbb{R}}^n)}\leq C_{n,p}\|\Theta\|_{L^p(S^{n-1})}$.}
\end{align}   

In turn, via Parseval's formula these properties imply that $T_\Theta$ extends to a linear and bounded 
operator on $L^2({\mathbb{R}}^n)$ which satisfies
\begin{align}\label{PhBB-1nv-ugg-hgv}
\widehat{T_\Theta f}=m_\Theta\widehat{f}\,\,\text{ for each }\,\,f\in L^2({\mathbb{R}}^n).
\end{align}
In addition, for each $f,g\in L^2({\mathbb{R}}^n)$, we have
\begin{align}\label{uttGV}
\int_{{\mathbb{R}}^n}(T_\Theta f)(x)g(x)\,dx
&=(2\pi)^{-n}\int_{{\mathbb{R}}^n}\widehat{T_\Theta f}(\xi)\widehat{g}(-\xi)\,d\xi
=(2\pi)^{-n}\int_{{\mathbb{R}}^n}m_{\Theta}(\xi)\widehat{f}(\xi)\widehat{g}(-\xi)\,d\xi
\nonumber\\[6pt]
&=(2\pi)^{-n}\int_{{\mathbb{R}}^n}\widehat{f}(\xi)\widehat{T_{\widetilde{\Theta}}g}(-\xi)\,d\xi
=\int_{{\mathbb{R}}^n}f(x)(T_{\widetilde{\Theta}}g)(x)\,dx,
\end{align}
from which we ultimately conclude that the transposed of $T_\Theta$ is $T_{\widetilde{\Theta}}$.
Moreover, for each given $H^1$-atom $a$, the fact that $T_\Theta a$ belongs to $L^1({\mathbb{R}}^n)$ 
(cf. \eqref{ytgfff.9ytrdf.455}) implies that $\widehat{T_\Theta a}$ is a continuous function satisfying 
$\int_{{\mathbb{R}}^n}T_\Theta a\,d{\mathscr{L}}^n=\widehat{T_\Theta a}(0)=\lim_{\xi\to 0}m_\Theta(\xi)\widehat{a}(\xi)=0$
since $m_\Theta$ is bounded, $\widehat{a}$ is continuous (given that $a\in L^1({\mathbb{R}}^n)$), 
and $\widehat{a}(0)=\int_{{\mathbb{R}}^n}a\,d{\mathscr{L}}^n=0$ thanks to the vanishing moment 
property of the atom. In light of \eqref{ytgfff.9ytrdf-tDDD}, this shows that $T_\Theta(1)=0$. 
Finally, in a similar fashion, $(T_\Theta)^\top(1)=0$.

Natural examples of operators of the sort discussed in \eqref{ytgfff.9ytrdf.ygfg} are 
offered by the Riesz transforms in ${\mathbb{R}}^n$. These are defined as the 
family $(R_j)_{1\leq j\leq n}$ where, for each $j\in\{1,\dots,n\}$  and
each $f\in L^p({\mathbb{R}}^n)$ with $1\leq p<\infty$ we have set
\begin{equation}\label{R-87ygbg}
\begin{array}{c}
\displaystyle
(R_jf)(x):=\lim_{\varepsilon\to 0^{+}}\int_{y\in{\mathbb{R}}^n\setminus B(x,\varepsilon)}
K_j(x-y)f(y)\,dy,\qquad x\in{\mathbb{R}}^n,
\\[14pt]
\displaystyle
\text{with $K_j(z):=\frac{\Gamma\big(\frac{n+1}{2}\big)}{\pi^{\frac{n+1}{2}}}\frac{z_j}{|z|^{n+1}}$
for each $z\in{\mathbb{R}}^n\setminus\{0\}$.}
\end{array}
\end{equation}
These are singular integral operators of convolution type involving odd kernels. 
A prominent example of a singular integral operator of convolution type involving an even kernel
(with vanishing integral on the unit sphere) is offered by the Beurling (or Beurling-Ahlfors) 
transform in the complex plane 
\begin{equation}\label{R-87ygbg.BEBE}
(Sf)(z):=-\lim_{\varepsilon\to 0^{+}}\frac{1}{\pi}
\int\limits_{\substack{\zeta\in{\mathbb{C}}\\ |z-\zeta|>\varepsilon}}
\frac{f(\zeta)}{(z-\zeta)^2}\,d{\mathscr{L}}^2(\zeta),\qquad z\in{\mathbb{C}}.
\end{equation}
This has the basic property that 
\begin{equation}\label{R-87ygbg.BEBE.2tfV}
S(\partial_{\overline{z}}f)=\partial_z f\,\,\text{ for each Schwartz function }\,\,f\in{\mathscr{S}}({\mathbb{C}}),
\end{equation}
where $\partial_{\overline{z}}:=(1/2)(\partial_x-(1/i)\partial_y)$ and
$\partial_z:=(1/2)(\partial_x+(1/i)\partial_y)$ are, respectively, the Cauchy-Riemann operator and 
its complex conjugate.

To state the result pertaining to the boundedness of semi-Calder\'on-Zygmund operators 
on the space of functions of vanishing mean oscillations advertised earlier, recall 
that the quotient space $\widetilde{\mathrm{VMO}}({\mathbb{R}}^n)$ has been defined 
in \eqref{defi-VMO-tilde}.

\begin{theorem}\label{i87hbBV}
Consider a semi-Calder\'on-Zygmund operator $T$ in ${\mathbb{R}}^n$ satisfying $T(1)=0$.
Extend $T$ to a linear and bounded operator $\widetilde{T}$ from 
$\widetilde{\mathrm{BMO}}({\mathbb{R}}^n)$ into itself by setting 
{\rm (}with $\langle\cdot,\cdot\rangle$ denoting the 
$\widetilde{\rm BMO}$-$H^1$ duality pairing; cf. item {\it (iv)} of Proposition~\ref{jhsdgf}{\rm )}
\begin{equation}\label{dkegfs.4XXX}
\begin{array}{c}
\widetilde{T}:\widetilde{\rm BMO}({\mathbb{R}}^n)
\longrightarrow\widetilde{\rm BMO}({\mathbb{R}}^n)
\\[4pt]
\big\langle\widetilde{T}[f],g\big\rangle:=\big\langle[f],T^\top g\big\rangle,
\quad\forall\,[f]\in\widetilde{\rm BMO}({\mathbb{R}}^n),\,\,\forall\,g\in H^1({\mathbb{R}}^n).
\end{array}
\end{equation}

Then $\widetilde{\mathrm{VMO}}({\mathbb{R}}^n)$ is an invariant subspace of $\widetilde{T}$. In particular, 
its restriction to $\widetilde{\rm VMO}({\mathbb{R}}^n)$, 
\begin{equation}\label{dkegfs.4YYY}
\begin{array}{c}
\widetilde{T}\big|_{{}_{\rm VMO}}:\widetilde{\rm VMO}({\mathbb{R}}^n)
\longrightarrow\widetilde{\rm VMO}({\mathbb{R}}^n)
\\[6pt]
\big(\widetilde{T}\big|_{{}_{\rm VMO}}\big)[f]:=\widetilde{T}[f]
\,\,\text{ for each }\,\,[f]\in\widetilde{\rm VMO}({\mathbb{R}}^n),
\end{array}
\end{equation}
is a well-defined, linear and bounded operator. Moreover, 
$\widetilde{T}\big|_{{}_{\rm VMO}}$ is compatible with the action of $T$ on Lebesgue spaces
in the sense that for each $p\in[2,\infty)$ one has
\begin{equation}\label{dkegfs.6.uuu}
\Big(\widetilde{T}\big|_{{}_{\rm VMO}}\Big)[f]=[Tf],\qquad
\forall\,f\in{\rm VMO}({\mathbb{R}}^n)\cap L^p({\mathbb{R}}^n).
\end{equation} 
\end{theorem}

\vskip 0.08in
\noindent{\bf Example~1:} 
In view of \eqref{ytgfff.9ytrdf.ygfg}, Theorem~\ref{i87hbBV} applies directly to the 
Riesz transforms in ${\mathbb{R}}^n$, as well as the Beurling transform in ${\mathbb{C}}$. 
More generally, given any principal-value convolution type operator $T_\Theta$ as in \eqref{ytgfff.9ytrdf.ygfg}, 
its realization as a linear and bounded mapping from the space $\widetilde{\mathrm{BMO}}({\mathbb{R}}^n)$ 
into itself, via the transposition formula
\begin{equation}\label{dkegfs.gG-jgV}
\begin{array}{c}
\widetilde{T}_{\Theta}:\widetilde{\rm BMO}({\mathbb{R}}^n)
\longrightarrow\widetilde{\rm BMO}({\mathbb{R}}^n)
\\[4pt]
\big\langle\widetilde{T}_{\Theta}[f],g\big\rangle:=\big\langle[f],T_{\widetilde{\Theta}}g\big\rangle,
\quad\forall\,[f]\in\widetilde{\rm BMO}({\mathbb{R}}^n),\,\,\forall\,g\in H^1({\mathbb{R}}^n),
\end{array}
\end{equation}
where $\langle\cdot,\cdot\rangle$ stands for the $\widetilde{\rm BMO}$-$H^1$ duality pairing, 
and $\widetilde{\Theta}(x):=\Theta(-x)$ for each $x\in{\mathbb{R}}^n\setminus\{0\}$, induces
a well-defined, linear and bounded operator
\begin{equation}\label{dkegfs.4YYY.gG}
\widetilde{T}_\Theta\big|_{{}_{\rm VMO}}:\widetilde{\rm VMO}({\mathbb{R}}^n)
\longrightarrow\widetilde{\rm VMO}({\mathbb{R}}^n).
\end{equation}

\vskip 0.08in
\noindent{\bf Example~2:}
Recall that, for a given Lipschitz function $A:{\mathbb{R}}\to{\mathbb{C}}$, the Calder\'on commutator 
of order $m\in{\mathbb{N}}_0$ is the principal value singular integral operator $C_m$ on the real line 
whose kernel is given by 
\begin{equation}\label{ttfVV.Ka-1}
K_m(x,y):=\frac{\big(A(x)-A(y)\big)^m}{(x-y)^{m+1}},\qquad x,y\in{\mathbb{R}},\,\,x\not=y.
\end{equation}
It is then a basic fact that each $C_m$ is a Calder\'on-Zygmund operator 
(e.g., $C_0$ is, up to normalization, the Hilbert transform on the real line). 
In particular, they all extend to well-defined and bounded linear operators from 
$L^\infty({\mathbb{R}})$ into ${\rm BMO}({\mathbb{R}})$. Retaining the same notation for the said 
extensions, a well-known trick (based on integration by parts) then yields the following remarkable 
recursive identity (cf. \cite[(2.14), p.\,266]{Meyer}) 
\begin{equation}\label{ttfVV.Ka-2}
C_m(1)=C_{m-1}(A')\,\,\text{ for each }\,\,m\in{\mathbb{N}}.
\end{equation}

In relation to the above family of operators, for each $m\in{\mathbb{N}}$ let us consider 
the principal value singular integral operator $T_m$ on the real line associated with 
the modified kernel
\begin{align}\label{ttfVV.Ka-3}
k_m(x,y) &:=K_{m}(x,y)-K_{m-1}(x,y)A'(y)
\nonumber\\[6pt]
&=\frac{\big(A(x)-A(y)\big)^{m-1}}{(x-y)^{m+1}}\Big\{A(x)-A(y)-(x-y)A'(y)\big\},
\qquad x,y\in{\mathbb{R}},\,\,x\not=y.
\end{align}
Since, generally speaking, the function $A'$ is only essentially bounded, the operator $T_m$ is 
only semi-Calder\'on-Zygmund (as opposed to $C_m$ which is a genuine Calder\'on-Zygmund operator). 
This being said, in contrast with \eqref{ttfVV.Ka-2} we presently have $T_m(1)=C_m(1)-C_{m-1}(A')=0$. 
Granted these, Theorem~\ref{i87hbBV} applies and gives that 
\begin{align}\label{ttfVV.Ka-4}
\parbox{10.35cm}{$T_m$, the modified Calder\'on commutator of order $m\in{\mathbb{N}}$ on the real line, 
associated with the kernel $k_m$ defined in \eqref{ttfVV.Ka-3}, induces a bounded operator from 
the space $\widetilde{\rm VMO}({\mathbb{R}})$ into itself.}
\end{align}

\vskip 0.08in
\noindent{\bf Example~3:}
Consider the principal value Cauchy singular integral operator ${\mathcal{C}}$ on a curve 
$\Sigma\subseteq{\mathbb{C}}$ which is the graph of a Lipschitz function $A:{\mathbb{R}}\to{\mathbb{R}}$. 
That is, $\Sigma:=\{z=x+iA(x):\,x\in{\mathbb{R}}\}$ and ${\mathcal{C}}$ acts on a function 
$f:\Sigma\to{\mathbb{C}}$ according to 
\begin{equation}\label{ugvVVfv-geed-1}
{\mathcal{C}}f(z):=\lim_{\varepsilon\to 0^{+}}\frac{1}{2\pi i}\int\limits_{\zeta\in\Sigma\setminus B(z,\varepsilon)}
\frac{f(\zeta)}{\zeta-z}\,d\zeta,\qquad z\in\Sigma.
\end{equation}
Making the bi-Lipschitz change of variables ${\mathbb{R}}\ni x\mapsto x+iA(x)\in\Sigma$ and 
identifying $f$ with the function $g(x):=f\big(x+iA(x)\big)$ for $x\in{\mathbb{R}}$, this 
becomes (after adjusting the truncation; cf. \cite[Lemma~B.1]{HMT15} in this regard) 
the principal value singular integral operator on the real line 
\begin{equation}\label{ugvVVfv-geed-2}
Tg(x):=\lim_{\varepsilon\to 0^{+}}\frac{1}{2\pi i}
\int\limits_{y\in{\mathbb{R}}\setminus(x-\varepsilon,x+\varepsilon)}
\frac{(1+iA'(y))g(y)}{y-x+i(A(y)-A(x))}\,dy,\qquad x\in{\mathbb{R}}.
\end{equation}
While the above integral kernel is, generally speaking, lacking smoothness in the $y$ variable, 
$T$ is nonetheless a semi-Calder\'on-Zygmund operator on ${\mathbb{R}}$, and we claim that $T(1)=0$.
To justify this claim, pick an arbitrary $H^1$ atom $a$ on the real line and observe that if 
\begin{equation}\label{ugvVVfv-geed-3}
b:\Sigma\to{\mathbb{C}}\,\,\text{ is defined as }\,\,
b\big(x+iA(x)\big):=\frac{a(x)}{1+iA'(x)}\,\,\text{ for }\,\,x\in{\mathbb{R}},
\end{equation}
then $\int_{\Sigma}b(z)\,dz=\int_{\mathbb{R}}a\,d{\mathscr{L}}^1=0$ and
\begin{equation}\label{ugvVVfv-geed-4}
\int_{\mathbb{R}}T^\top a\,d{\mathscr{L}}^1
=-\int_{\Sigma}({\mathcal{C}}b)(z)\,dz
=-\int_{\Sigma}\big((\tfrac{1}{2}I+{\mathcal{C}})b\big)(z)\,dz=0.
\end{equation}
The last equality above relies on Cauchy's vanishing formula (cf. \cite{MMM17}) 
applied to the function defined in the domain $\Omega\subseteq{\mathbb{C}}$ lying above the graph 
$\Sigma$ by $u(z):=\frac{1}{2\pi i}\int_{\Sigma}b(\zeta)(\zeta-z)^{-1}\,d\zeta$ for each $z\in\Omega$, 
which has an integrable nontangential maximal function on $\Sigma=\partial\Omega$ and whose nontangential 
boundary trace is precisely $(\tfrac{1}{2}I+{\mathcal{C}})b$ at a.e. point on $\Sigma=\partial\Omega$. 
In view of \eqref{ytgfff.9ytrdf-tDDD}, we conclude from \eqref{ugvVVfv-geed-4} that, indeed, $T(1)=0$. 

With the knowledge that $T$ is a semi-Calder\'on-Zygmund operator on ${\mathbb{R}}$ satisfying $T(1)=0$, 
Theorem~\ref{i87hbBV} applies and gives that 
\begin{align}\label{ugvVVfv-geed-5}
\parbox{10.35cm}{the principal value Cauchy singular integral operator, 
defined on the real line as in \eqref{ugvVVfv-geed-2}, induces a well-defined, linear and
bounded operator from the space $\widetilde{\rm VMO}({\mathbb{R}})$ into itself.}
\end{align}
This result may be further generalized to higher dimensions by considering the principal value Cauchy-Clifford 
singular integral operator on a Lipschitz surface as in \cite{Mi}.

\vskip 0.08in
\noindent{\bf Example~4:}
Having fixed $n\in{\mathbb{N}}$, recall the principal value harmonic double layer ${\mathcal{K}}$, 
defined on a surface $\Sigma\subseteq{\mathbb{R}}^{n+1}$ which is the graph of a Lipschitz function 
$A:{\mathbb{R}}^{n}\to{\mathbb{R}}$. Specifically, 
$\Sigma:=\{X=(x,A(x))\in{\mathbb{R}}^{n+1}:\,x\in{\mathbb{R}}^n\}$, 
and ${\mathcal{K}}$ maps a function $f:\Sigma\to{\mathbb{C}}$ into 
\begin{equation}\label{ugvVVfv-geed-DDD1}
{\mathcal{K}}f(X):=\lim_{\varepsilon\to 0^{+}}\frac{1}{\omega_{n}}\int\limits_{Y\in\Sigma\setminus B(X,\varepsilon)}
\frac{\langle\nu(Y),Y-X\rangle}{|X-Y|^{n+1}}\,d\sigma(Y),\qquad X\in\Sigma,
\end{equation}
where $\nu$ and $\sigma$, the unit normal and surface measure to $\Sigma$, are given by 
\begin{equation}\label{ugvVVfv-geed-DDD1.bis}
\nu\big(x,A(x)\big)=\frac{\big(\nabla A(x),-1\big)}{\sqrt{1+|\nabla A(x)|^2}},\quad
d\sigma\big(x,A(x)\big)=\sqrt{1+|\nabla A(x)|^2}\,dx,\quad x\in{\mathbb{R}}^n.
\end{equation}
Much as in the case of the Cauchy operator considered earlier, make the bi-Lipschitz 
change of variables ${\mathbb{R}}^n\ni x\mapsto\big(x,A(x)\big)\in\Sigma$ and identify 
$f$ with the function $g(x):=f\big(x,A(x)\big)$ for $x\in{\mathbb{R}}^n$. This permits us to 
identify the harmonic double layer ${\mathcal{K}}$ with the principal value singular integral 
operator in ${\mathbb{R}}^n$ given by 
\begin{equation}\label{ugvVVfv-geed-2DDD}
Tg(x):=\lim_{\varepsilon\to 0^{+}}\frac{1}{\omega_n}\int\limits_{y\in{\mathbb{R}}^n\setminus B(x,\varepsilon)}
\frac{A(x)-A(y)-\langle x-y,\nabla A(y)\rangle}{\big(|x-y|^2+(A(x)-A(y)^2\big)^{\frac{n+1}{2}}}\,g(y)\,dy,
\quad x\in{\mathbb{R}}^n.
\end{equation}
We remark that the integral kernel above does not, generally speaking, possess any smoothness in the $y$ variable.
Nonetheless, $T$ is bounded on $L^2({\mathbb{R}}^n)$ (cf. \cite[Th\'eor\`eme~11, p.\,320]{Meyer}), hence 
$T$ is a semi-Calder\'on-Zygmund operator on ${\mathbb{R}}^n$. We claim that $T(1)=0$.
To see that this is the case, pick an arbitrary $H^1$ atom $a$ in ${\mathbb{R}}^n$ and note that if 
\begin{equation}\label{ugvVVfv-geed-3DDD}
b:\Sigma\to{\mathbb{C}}\,\,\text{ is defined as }\,\,
b\big(x,A(x)\big):=\frac{a(x)}{\sqrt{1+|\nabla A(x)|^2}}\,\,\text{ for }\,\,x\in{\mathbb{R}}^n,
\end{equation}
then $\int_{\Sigma}b\,d\sigma=\int_{{\mathbb{R}}^n}a\,d{\mathscr{L}}^n=0$. Also, if we denote by ${\mathcal{K}}^\top$
the transposed of ${\mathcal{K}}$ on $L^2(\Sigma)$, then
\begin{equation}\label{ugvVVfv-geed-4DDD}
\int_{{\mathbb{R}}^n}T^\top a\,d{\mathscr{L}}^n=\int_{\Sigma}{\mathcal{K}}^\top b\,d\sigma
=\int_{\Sigma}\big(-\tfrac{1}{2}I+{\mathcal{K}}^\top\big)b\,d\sigma=0.
\end{equation}
The last equality above relies on the version of Divergence Formula established in \cite{MMM17}, 
currently used for the vector field defined in the domain $\Omega\subseteq{\mathbb{R}}^{n+1}$ 
lying above the surface $\Sigma$ by 
\begin{equation}\label{ugvVVfv-geed-4DDD.bix}
\vec{F}(X):=\frac{1}{\omega_{n}}\int_{\Sigma}\frac{X-Y}{|X-Y|^{n+1}}\,b(Y)\,d\sigma(Y),
\qquad\forall\,X\in\Omega, 
\end{equation}
which is smooth and divergence-free in $\Omega$, has an integrable nontangential maximal function, 
and whose nontangential boundary trace $\vec{F}\big|^{{}^{\rm n.t.}}_{\partial\Omega}$
satisfies $\nu\cdot\big(\vec{F}\big|^{{}^{\rm n.t.}}_{\partial\Omega}\big)
=\big(-\tfrac{1}{2}I+{\mathcal{K}}^\top\big)b$ at $\sigma$-a.e. point on $\Sigma=\partial\Omega$
(see \cite{MMM17} for more details). Having proved \eqref{ugvVVfv-geed-4DDD} we then 
conclude from \eqref{ytgfff.9ytrdf-tDDD} that $T(1)=0$, as wanted. 
Given that $T$ is a semi-Calder\'on-Zygmund operator in ${\mathbb{R}}^n$ satisfying $T(1)=0$, from 
Theorem~\ref{i87hbBV} we may then conclude that 
\begin{align}\label{ugvVVfv-geed-5DDD}
\parbox{10.35cm}{the principal value harmonic double layer operator, 
defined in ${\mathbb{R}}^n$ as in \eqref{ugvVVfv-geed-2DDD}, induces a well-defined, linear and
bounded operator from the space $\widetilde{\rm VMO}({\mathbb{R}}^n)$ into itself.}
\end{align}
To close, we mention that similar results are valid for the pull-back from a Lipschitz graph
to the Euclidean space of any double layer potential operator associated with a homogeneous 
second order elliptic system.   

\vskip 0.10in

Moving on, we note that he argument which proves Theorem~\ref{i87hbBV} is indicative 
of a more general principle at play here, to the effect that, regardless of its actual format, 
\begin{equation}\label{i87hbBV-jt4}
\parbox{10.00cm}{any linear operator which is bounded both on ${\rm BMO}$ and on a (homogeneous) 
H\"older space is also bounded on ${\mathrm{VMO}}$.}
\end{equation}
In relation to \eqref{i87hbBV-jt4}, it is also worth pointing out that the class of operators
which are simultaneously bounded on ${\rm BMO}$ as well as on some common (homogeneous) H\"older
space is considerably larger than the class of the semi-Calder\'on-Zygmund operators considered in 
Theorem~\ref{i87hbBV} since, as opposed to the latter, the former is stable under composition 
hence, in particular, constitutes an algebra. This being said, by additionally hypothesizing 
a suitable cancellation condition for the transposed, one can identify a (maximal) subfamily of 
Calder\'on-Zygmund operators which do make up an algebra. To facilitate stating such a result, 
for any given Banach space ${\mathcal{X}}$ we agree to denote by ${\mathscr{B}}({\mathcal{X}})$ 
the Banach algebra of linear and bounded operators from ${\mathcal{X}}$ into itself (with respect 
to the ordinary addition and composition of operators, and ordinary operator norm). 

\begin{theorem}\label{i87hbBV-MY}
Fix $n\in{\mathbb{N}}$ arbitrary. Then the family ${\mathscr{A}}^0_{{}_{\widetilde{\rm CZ}}}$ 
consisting of all operators $\widetilde{T}\big|_{{}_{\rm VMO}}$ as in \eqref{dkegfs.4YYY}, 
where $T$ is a Calder\'on-Zygmund operator in ${\mathbb{R}}^n$ satisfying $T(1)=T^\top(1)=0$, 
is a sub-algebra of ${\mathscr{B}}\big(\,\widetilde{\mathrm{VMO}}({\mathbb{R}}^n)\big)$. 
\end{theorem}

The family of principal-value convolution type operators $T_\Theta$ associated as in \eqref{ytgfff.9ytrdf.ygfg} 
with kernels $\Theta$ which are actually of class ${\mathscr{C}}^\infty$ in ${\mathbb{R}}^n\setminus\{0\}$ 
also gives rise to an algebra of linear and bounded operators on $\widetilde{\rm VMO}({\mathbb{R}}^n)$,
of the sort described in our next theorem. 

\begin{theorem}\label{i87hbBV-ALG}
Fix $n\in{\mathbb{N}}$ arbitrary. Associate with each complex-valued function 
\begin{align}\label{ytgfff.9ytrdf.ygfg.222222}
\begin{array}{c}
\text{$\Theta\in{\mathscr{C}}^\infty({\mathbb{R}}^n\setminus\{0\})$, 
positive homogenous of degree $-n$,}
\\[6pt]
\text{and with the cancellation property $\int_{S^{n-1}}\Theta(\omega)\,d\omega=0$},
\end{array}
\end{align}
the principal-value convolution type singular integral operator $T_\Theta$ defined as in 
\eqref{ytgfff.9ytrdf.ygfg}, and denote by $\widetilde{T}_{\Theta}$ its realization as a linear and 
bounded mapping from the space $\widetilde{\mathrm{BMO}}({\mathbb{R}}^n)$ into itself as in 
\eqref{dkegfs.gG-jgV}. Then, with $I$ denoting the identity operator, the following properties hold:
\begin{enumerate}
\item[(a)] The set
\begin{equation}\label{gabbIHka}
{\mathscr{A}}_{{}_{\widetilde{\rm SIO}}}:=\Big\{cI+\widetilde{T}_{\Theta}\big|_{{}_{\rm VMO}}
:\widetilde{\mathrm{VMO}}({\mathbb{R}}^n)\to\widetilde{\mathrm{VMO}}({\mathbb{R}}^n)
:\,c\in{\mathbb{C}},\,\text{ and $\Theta$ as in \eqref{ytgfff.9ytrdf.ygfg.222222}}\Big\}
\end{equation}
is a commutative unital sub-algebra of ${\mathscr{B}}\big(\,\widetilde{\mathrm{VMO}}({\mathbb{R}}^n)\big)$.
In ${\mathscr{A}}_{{}_{\widetilde{\rm SIO}}}$ the following composition law holds: if $c\in{\mathbb{C}}$ 
and the functions $\Theta_1,\dots,\Theta_N,{\Theta'}_{\!\!1},\dots,{\Theta'}_{\!\!N},\Theta$ are as in 
\eqref{ytgfff.9ytrdf.ygfg.222222} and satisfy 
\begin{equation}\label{u7GBB.Za.5.XXX}
\sum_{j=1}^N m_{\widetilde{\Theta'}_{\!\!j}}\,m_{\widetilde{\Theta}_j}=c+m_{\widetilde{\Theta}}
\,\,\text{ in }\,\,{\mathbb{R}}^n\setminus\{0\},
\end{equation}
then
\begin{equation}\label{u7GBB.Za.8.YYY}
\sum_{j=1}^N\big(\widetilde{T}_{{\Theta'}_{\!\!j}}\big|_{{}_{\rm VMO}}\big)
\big(\widetilde{T}_{{\Theta}_j}\big|_{{}_{\rm VMO}}\big)=cI+\widetilde{T}_{{\Theta}}\big|_{{}_{\rm VMO}}
\,\,\text{ in }\,\,{\mathscr{B}}\big(\,\widetilde{\rm VMO}({\mathbb{R}}^n)\big).
\end{equation}

\vskip 0.08in
\item[(b)] With `bar' denoting closure in ${\mathscr{B}}\big(\,\widetilde{\mathrm{VMO}}({\mathbb{R}}^n)\big)$, 
\begin{equation}\label{gab-ii-GGVa}
\overline{{\mathscr{A}}_{{}_{\widetilde{\rm SIO}}}}
=\overline{\text{span}}\,\Big\{\widetilde{R}_j\big|_{{}_{\rm VMO}}\Big\}_{1\leq j\leq n},
\end{equation}
that is, $\overline{{\mathscr{A}}_{{}_{\widetilde{\rm SIO}}}}$ coincides with the smallest closed subalgebra of  
${\mathscr{B}}\big(\,\widetilde{\mathrm{VMO}}({\mathbb{R}}^n)\big)$ containing the Riesz transforms, 
$\widetilde{R}_j\big|_{{}_{\rm VMO}}\in{\mathscr{B}}\big(\,\widetilde{\mathrm{VMO}}({\mathbb{R}}^n)\big)$ 
with $1\leq j\leq n$.

\vskip 0.08in
\item[(c)] Whenever the function $\Theta$ is as in \eqref{ytgfff.9ytrdf.ygfg.222222} and
\begin{equation}\label{gab-ii-GG-yt55}
c\in{\mathbb{C}}\setminus\big\{-m_{\widetilde{\Theta}}(\xi):\,\xi\in{\mathbb{R}}^n\setminus\{0\}\big\}
\end{equation}
it follows that $cI+\widetilde{T}_{\Theta}\big|_{{}_{\rm VMO}}$ has an inverse in 
${\mathscr{A}}_{{}_{\widetilde{\rm SIO}}}$. More specifically, whenever $\Theta$ is as 
in \eqref{ytgfff.9ytrdf.ygfg.222222} and $c$ is as in \eqref{gab-ii-GG-yt55}, the operator 
$cI+\widetilde{T}_{\Theta}\in{\mathscr{B}}\big(\,\widetilde{\mathrm{BMO}}({\mathbb{R}}^n)\big)$
has an inverse in $\widetilde{\mathrm{BMO}}({\mathbb{R}}^n)$ of the form 
$c_0I+\widetilde{T}_{\Theta_0}\in{\mathscr{B}}\big(\,\widetilde{\mathrm{BMO}}({\mathbb{R}}^n)\big)$
for some $c_0\in{\mathbb{C}}$ and $\Theta_0$ as in \eqref{ytgfff.9ytrdf.ygfg.222222}, with the property 
that $c_0I+\widetilde{T}_{\Theta_0}\big|_{{}_{\rm VMO}}$ is the inverse of 
$cI+\widetilde{T}_{\Theta}\big|_{{}_{\rm VMO}}$ in ${\mathscr{A}}_{{}_{\widetilde{\rm SIO}}}$.

\vskip 0.08in
\item[(d)] Suppose $\Theta$ is as in \eqref{ytgfff.9ytrdf.ygfg.222222} and $c$ is as in 
\eqref{gab-ii-GG-yt55}. Then for each $f\in{\mathrm{BMO}}({\mathbb{R}}^n)$ one has
\begin{equation}\label{u7GBB}
f\in{\mathrm{VMO}}({\mathbb{R}}^n)\Longleftrightarrow
(cI+\widetilde{T}_{\Theta})[f]\in\widetilde{\mathrm{VMO}}({\mathbb{R}}^n).
\end{equation}

More generally, let $N\in{\mathbb{N}}$ be a given integer and assume $\Theta_1,\dots,\Theta_N$ is a family 
of functions, each of which as in \eqref{ytgfff.9ytrdf.ygfg.222222}. Also, fix 
\begin{equation}\label{gab-ii-GG-yt55.ZZZ}
(c_1,\dots,c_N)\in{\mathbb{C}}^N\setminus\Big\{\big(-m_{\widetilde{\Theta}_j}(\xi)\big)_{1\leq j\leq N}:
\,\xi\in{\mathbb{R}}^n\setminus\{0\}\Big\}.
\end{equation}
Then for each given function $f\in{\mathrm{BMO}}({\mathbb{R}}^n)$ one has
\begin{equation}\label{u7GBB.ZZZ}
f\in{\mathrm{VMO}}({\mathbb{R}}^n)\Longleftrightarrow
(c_jI+\widetilde{T}_{\Theta_j})[f]\in\widetilde{\mathrm{VMO}}({\mathbb{R}}^n)
\,\,\text{ for each }\,\,j\in\{1,\dots,N\}.
\end{equation}

\item[(e)] Items {\it (a)}, {\it (c)}, and the first part of {\it (d)}, have natural versions in 
the case when the functions involved are vector-valued and the kernels of the singular integral 
operators are matrix-valued. The specifics of this more general setting are as follows. Given a 
finite dimensional complex vector space ${\mathscr{V}}$, consider ${\mathscr{V}}$-valued functions 
whose scalar components {\rm (}with respect to some fixed basis of ${\mathscr{V}}${\rm )} are 
from $\widetilde{\mathrm{VMO}}({\mathbb{R}}^n)$ {\rm (}or $\widetilde{\mathrm{BMO}}({\mathbb{R}}^n)$, 
depending on the context{\rm )}. Also, consider principal-value convolution type operators $T_\Theta$ 
defined as in \eqref{ytgfff.9ytrdf.ygfg}, associated with kernels $\Theta$ as in 
\eqref{ytgfff.9ytrdf.ygfg.222222} taking values in ${\rm Hom}\,({\mathscr{V}},{\mathscr{V}})$. 
In particular, $T_\Theta$ may be viewed as a matrix of ordinary scalar, principal-value, convolution 
type operators, and extending each individual entry in this matrix as in \eqref{dkegfs.gG-jgV} then 
yields a linear and bounded operator $\widetilde{T}_{\Theta}$ from 
$\widetilde{\mathrm{BMO}}({\mathbb{R}}^n)\otimes{\mathscr{V}}$ into itself which leaves 
the subspace $\widetilde{\mathrm{VMO}}({\mathbb{R}}^n)\otimes{\mathscr{V}}$ invariant. 

The version of item {\it (a)} in this setting is that if one now defines ${\mathscr{A}}_{{}_{\widetilde{\rm SIO}}}$ 
as in \eqref{gabbIHka}, but with the intervening singular integral operators as just described above and 
with $cI$ now replaced by $c\in{\rm Hom}\,({\mathscr{V}},{\mathscr{V}})$ arbitrary, 
then ${\mathscr{A}}_{{}_{\widetilde{\rm SIO}}}$ becomes a {\rm (}typically non-commutative{\rm )} 
sub-algebra of ${\mathscr{B}}\big(\,\widetilde{\mathrm{VMO}}({\mathbb{R}}^n)\otimes{\mathscr{V}}\big)$.
Finally, in the case of item {\it (c)} and the first part of item {\it (d)}, condition \eqref{gab-ii-GG-yt55} 
is now replaced by 
\begin{equation}\label{gab-ii-GG-yt55-pj}
c+m_{\widetilde{\Theta}}(\xi)\,\,\text{ is invertible in }\,\,{\rm Hom}\,({\mathscr{V}},{\mathscr{V}})
\,\,\text{ for each }\,\,\xi\in{\mathbb{R}}^n\setminus\{0\}.
\end{equation}
\end{enumerate}
\end{theorem}

Theorem~\ref{i87hbBV-ALG}, whose proof is presented in \S\ref{NEWSSS}, has many consequences of 
independent interest, which we shall now explore. We begin by stating a version of the first claim 
in item {\it (d)} of Theorem~\ref{i87hbBV-ALG} for kernels taking values in a finite dimensional algebra
(again, proved in \S\ref{NEWSSS}).

\begin{corollary}\label{i87hbBV-ALG-CCC}
Let $A=(A,+,\odot,1)$ be a finite dimensional {\rm (}complex{\rm )} unital associative algebra. 
Fix $n\in{\mathbb{N}}$ arbitrary and, associate with each $A$-valued function 
\begin{align}\label{ytgfff.9ytrdf.ygfg.222222-CCC}
\begin{array}{c}
\text{$\Theta:{\mathbb{R}}^n\setminus\{0\}\longrightarrow A$ which is of class ${\mathscr{C}}^\infty$,} 
\\[6pt]
\text{positive homogenous of degree $-n$, and with}
\\[6pt]
\text{the cancellation property $\int_{S^{n-1}}\Theta(\omega)\,d\omega=0$},
\end{array}
\end{align}
consider the principal-value convolution type operator $T_\Theta$ acting on 
$A$-valued Schwartz functions $f\in{\mathscr{S}}({\mathbb{R}}^n)\otimes A$ according to 
$T_\Theta f(x):=\lim\limits_{\varepsilon\to 0^{+}}\int_{y\in{\mathbb{R}}^n\setminus B(x,\varepsilon)}
\Theta(x-y)\odot f(y)\,dy$ for $x\in{\mathbb{R}}^n$.

Denote by $\widetilde{T}_{\Theta}$ the realization of the operator $T_\Theta$ as a linear and bounded 
mapping from the space $\widetilde{\mathrm{BMO}}({\mathbb{R}}^n)\otimes A$ into itself, obtained by extending 
each scalar component of $T_\Theta$ to $\widetilde{\mathrm{BMO}}({\mathbb{R}}^n)$ as in \eqref{dkegfs.gG-jgV}.
Also, fix some 
\begin{equation}\label{gab-ii-GG-yt55-CCC}
\parbox{6.60cm}{$c\in A$ such that $c+m_{\widetilde{\Theta}}(\xi)$ is invertible 
in $A$ from the right for each $\xi\in{\mathbb{R}}^n\setminus\{0\}$.}
\end{equation}
Then, with $I$ denoting the identity operator, for each $f\in{\mathrm{BMO}}({\mathbb{R}}^n)\otimes A$ one has
\begin{equation}\label{u7GBB-CCC}
f\in{\mathrm{VMO}}({\mathbb{R}}^n)\otimes A\Longleftrightarrow
(cI+\widetilde{T}_{\Theta})[f]\in\widetilde{\mathrm{VMO}}({\mathbb{R}}^n)\otimes A.
\end{equation}
\end{corollary}

\vskip 0.10in

Historically, the Riesz transforms have been successfully employed in 
characterizing the regularity of functions in the Euclidean space. 
For example, it is well-known (cf., e.g., \cite[(4.11), p.\,284]{GCRF85}) 
that the Hardy space $H^1({\mathbb{R}}^n)$ may be described as
\begin{equation}\label{u6gGVV.1}
H^1({\mathbb{R}}^n)=\big\{f\in L^1({\mathbb{R}}^n):\,R_jf\in L^1({\mathbb{R}}^n)
\,\,\text{ for }\,\,1\leq j\leq n\big\}.
\end{equation}
Also, if for each $j\in\{1,\dots,n\}$ we denote by $\widetilde{R}_j$ the extension of 
the $j$-th Riesz transform, originally acting on $L^2({\mathbb{R}}^n)$ as in \eqref{R-87ygbg}, 
to a bounded operator on $\widetilde{\mathrm{BMO}}({\mathbb{R}}^n)$ defined as in \eqref{dkegfs.gG-jgV},
then the following characterization of the space $\widetilde{\rm BMO}({\mathbb{R}}^n)$ 
may be deduced from \cite[Theorem~2, p.\,587]{Fe}:
\begin{equation}\label{u6gGVV.2}
\widetilde{\rm BMO}({\mathbb{R}}^n)=\Big\{[g_0]+\sum_{j=1}^n{\widetilde{R}}_j[g_j]:\,
g_0,g_1,\dots,g_n\in L^\infty({\mathbb{R}}^n)\Big\}.
\end{equation}
In a similar vein, a characterization of the space ${\rm VMO}({\mathbb{R}})$ as 
(where $H$ is the Hilbert transform on the real line) 
\begin{equation}\label{u6gGVV.2rrr}
{\rm VMO}({\mathbb{R}})=\Big\{u+Hv:\,u,v\in L^\infty({\mathbb{R}})\cap{\rm UC}({\mathbb{R}})\Big\}
\end{equation}
has been given by Sarason in \cite[Theorem~1, p.\,392]{Sa75}. Let us also mention that regularity 
results of a geometric flavor involving the Riesz transforms have been established in \cite{MMV}. 
Here is a result along this line of work, providing characterizations of the Sarason space 
${\rm VMO}$ in terms of Riesz and Beurling transforms in the complex plane.
\begin{corollary}\label{nhxtrE.2D}
Work in the two dimensional setting ${\mathbb{R}}^2\equiv{\mathbb{C}}$ and 
consider the complex Riesz transform 
\begin{equation}\label{R-87ygbg.RRR}
R_{{}_{\mathbb{C}}}f(z):=\lim_{\varepsilon\to 0^{+}}\frac{1}{2\pi}
\int_{\zeta\in{\mathbb{C}}\setminus B(z,\varepsilon)}\frac{z-\zeta}{|z-\zeta|^{3}}
f(\zeta)\,d{\mathscr{L}}^2(\zeta),\qquad z\in{\mathbb{C}}.
\end{equation}
Denote $\widetilde{R}_{{}_{\mathbb{C}}}$ the extension of the complex Riesz 
transform, originally considered as in \eqref{R-87ygbg.RRR} on $L^2({\mathbb{C}})$ 
{\rm (}cf. \eqref{ytgfff.9ytrdf.ygfg}{\rm )}, to a linear and bounded operator on 
$\widetilde{\mathrm{BMO}}({\mathbb{C}})$ {\rm (}cf. \eqref{dkegfs.gG-jgV}{\rm )}.
Analogously, denote by $\widetilde{S}$ the extension of the Beurling transform defined 
as in \eqref{R-87ygbg.BEBE} on $L^2({\mathbb{C}})$ to a linear and bounded operator on 
$\widetilde{\mathrm{BMO}}({\mathbb{C}})$. Finally, fix an arbitrary number $c\in{\mathbb{C}}$
such that $|c|\not=1$. 

Then for each given function $f\in{\mathrm{BMO}}({\mathbb{C}})$ the following conditions are equivalent:
\begin{enumerate}
\item[(i)] $f$ belongs to the Sarason space ${\mathrm{VMO}}({\mathbb{C}})$;
\item[(ii)] $\big(cI+\widetilde{R}_{{}_{\mathbb{C}}}\big)[f]$ belongs to $\widetilde{\mathrm{VMO}}({\mathbb{C}})$;
\item[(iii)] $(cI+\widetilde{S})[f]$ belongs to $\widetilde{\mathrm{VMO}}({\mathbb{C}})$.
\end{enumerate}
\end{corollary}

The key ingredient in the proof of Corollary~\ref{nhxtrE.2D}, presented in \S\ref{NEWSSS}, is 
Theorem~\ref{i87hbBV}. In turn, the equivalence of {\it (i)}-{\it (iv)} in Corollary~\ref{nhxtrE.2D} 
may be generalized to higher-dimensions using Clifford algebras as a substitute for the field of 
complex numbers. Specifically, given any $n\in{\mathbb{N}}$, denote by $(\mathcal{C}\!\ell_n,+,\odot)$ 
the (complex) Clifford algebra generated by $n$ anti-commuting imaginary units, denoted by 
$(e_j)_{1\leq j\leq n}$. Hence,
\begin{equation}\label{im-e}
e_j\odot e_j=-1\quad\mbox{and}\quad e_j\odot e_k=-e_k\odot e_j
\,\,\mbox{ whenever }\,1\leq j\neq k\leq n.
\end{equation}
The Euclidean ambient ${\mathbb{R}}^n$ embeds canonically into $\mathcal{C}\!\ell_n$ by 
identifying $(e_j)_{1\leq j\leq n}$ with the standard orthonormal basis in ${\mathbb{R}}^n$, 
i.e., 
\begin{equation}\label{im-e-yyy}
{\mathbb{R}}^n\ni x=(x_1,\dots,x_n)\equiv x:=\sum_{j=1}^nx_je_j\in\mathcal{C}\!\ell_n.
\end{equation}
Under this embedding, \eqref{im-e} implies that 
\begin{equation}\label{X-sqr}
x\odot x=-|x|^2\,\,\text{ for each }\,\,x\in{\mathbb{R}}^n\hookrightarrow{\mathcal{C}}\!\ell_{n}.
\end{equation}
More information on this topic may be found in \cite{Mi} and the references therein.
Here is the higher-dimensional version of the portion of Corollary~\ref{nhxtrE.2D} 
dealing with the complex Riesz transform.
\begin{corollary}\label{nhxtrE.ND}
Consider the Clifford-Riesz transform acting on ${\mathcal{C}}\!\ell_{n}$-valued functions $f$ defined in 
${\mathbb{R}}^n$ according to 
\begin{equation}\label{R-87ygbg.RRR.NNN}
R_{{}_{{\mathcal{C}}\!\ell}}f(x):=\lim_{\varepsilon\to 0^{+}}\frac{\Gamma\big(\frac{n+1}{2}\big)}{\pi^{\frac{n+1}{2}}}
\int_{y\in{\mathbb{R}}^n\setminus B(x,\varepsilon)}\frac{x-y}{|x-y|^{n+1}}\odot
f(y)\,dy,\qquad x\in{\mathbb{R}}^{n},
\end{equation}
and denote by $\widetilde{R}_{{}_{{\mathcal{C}}\!\ell}}$ its extension to a bounded operator on 
$\widetilde{\mathrm{BMO}}({\mathbb{R}}^n)\otimes{\mathcal{C}}\!\ell_{n}$. 
Also, consider 
\begin{equation}\label{utgVVavbnU9k}
\parbox{7.90cm}{$c\in{\mathcal{C}}\!\ell_{n}$ such that $c+i\omega$ is invertible  
in ${\mathcal{C}}\!\ell_{n}$ from the right
for each vector $\omega\in S^{n-1}\subseteq{\mathbb{R}}^n\hookrightarrow{\mathcal{C}}\!\ell_{n}$.}
\end{equation}

Then for each given function $f\in{\mathrm{BMO}}({\mathbb{R}}^n)\otimes{\mathcal{C}}\!\ell_{n}$ one has
\begin{equation}\label{ytGVVV}
f\in{\mathrm{VMO}}({\mathbb{R}}^n)\otimes{\mathcal{C}}\!\ell_{n}\Longleftrightarrow
\big(cI+\widetilde{R}_{{}_{{\mathcal{C}}\!\ell}}\big)[f]\in
\widetilde{\mathrm{VMO}}({\mathbb{R}}^n)\otimes{\mathcal{C}}\!\ell_{n}.
\end{equation}
\end{corollary}

As discussed in \S\ref{NEWSSS}, the above result is readily implied by Corollary~\ref{i87hbBV-ALG-CCC}.
We single out another immediate consequence of Theorem~\ref{i87hbBV-ALG} formulated in terms 
of scalar-valued functions.

\begin{corollary}\label{nhxtrE}
For each $j\in\{1,\dots,n\}$ denote by $\widetilde{R}_j$ the extension of 
the $j$-th Riesz transform, originally acting on $L^2({\mathbb{R}}^n)$ as in \eqref{R-87ygbg}, 
to a bounded operator on $\widetilde{\mathrm{BMO}}({\mathbb{R}}^n)$ defined as in \eqref{dkegfs.gG-jgV}. 
Then for each complex-valued function $f\in{\mathrm{BMO}}({\mathbb{R}}^n)$ and each 
$(c_1,\dots,c_n)\in{\mathbb{C}}^n\setminus iS^{n-1}$ one has
\begin{equation}\label{jhxfdw-R.jhghg}
f\in{\mathrm{VMO}}({\mathbb{R}}^n)\Longleftrightarrow
\big(c_jI+\widetilde{R}_j\big)[f]\in\widetilde{\mathrm{VMO}}({\mathbb{R}}^n)
\,\text{ for each }\,j\in\{1,\dots,n\}.
\end{equation}
In particular, corresponding to the special case when $c_1=\cdots=c_n=0$, for each complex-valued function 
$f\in{\mathrm{BMO}}({\mathbb{R}}^n)$ one has\footnote{The first author would like to express his gratitude 
to L. Escauriaza for some conversations pertaining to the one dimensional case of \eqref{jhxfdw-R}.}
\begin{equation}\label{jhxfdw-R}
f\in{\mathrm{VMO}}({\mathbb{R}}^n)\,\,\text{ if and only if }\,\,
\widetilde{R}_j[f]\in\widetilde{\mathrm{VMO}}({\mathbb{R}}^n)
\,\,\text{ for each }\,\,j\in\{1,\dots,n\}.
\end{equation}
\end{corollary}

Finally, we note that it is also possible to extend the characterizations of the membership to 
{\rm VMO} given in the two-dimensional setting in Corollary~\ref{nhxtrE.2D} to higher dimensions 
and differential forms by introducing suitable higher-dimensional versions of the Beurling and 
Riesz transforms acting on differential forms. To describe them, we need a some standard notation 
from differential geometry (see, e.g., \cite[\S2.1]{MMMT16}). For each $\ell\in\{0,1,\dots,n\}$ 
let $\Lambda^\ell$ denote the space of differential forms of degree $\ell$ in ${\mathbb{R}}^n$, and 
set $\Lambda:=\bigoplus_{\ell=0}^n\Lambda^\ell$ for the space of differential forms of arbitrary 
mixed degrees in ${\mathbb{R}}^n$. The exterior derivative operator $d$ and its formal adjoint 
$\delta$ in ${\mathbb{R}}^n$ are defined, respectively, as
\begin{equation}\label{8hhBBBa.AAA}
df:=\sum_{j=1}^n dx_j\wedge(\partial_j f),\quad
\delta f:=-\sum_{j=1}^n dx_j\vee(\partial_j f),\quad\forall\,f\in{\mathcal{D}}'({\mathbb{R}}^n)\otimes\Lambda,
\end{equation}
where $\wedge,\vee$ stand for the exterior and interior product on $\Lambda$, and where the 
partial derivatives are applied to the individual components of the differential form $f$. 
For each $\theta\in{\mathbb{C}}\setminus\{0\}$ consider then the $\theta$-Beurling transform 
in ${\mathbb{R}}^n$ defined (on the frequency side) as
\begin{equation}\label{8hhBBBa.1}
S_{\theta}:=\big(\theta\,d\delta-\theta^{-1}\delta d\big)\Delta^{-1}
:{\mathscr{S}}({\mathbb{R}}^n)\otimes\Lambda\rightarrow{\mathscr{S}}'({\mathbb{R}}^n)\otimes\Lambda.
\end{equation}
In the particular case when $\theta=1$, this operator appears in \cite[(12.71), p.\,326]{IM01}.
The reason for which it is reasonable to think of $S_{\theta}$ above as some kind of generalization 
of the classical Beurling transform defined in the complex plane in \eqref{R-87ygbg.BEBE} is as follows. 
If for each $\theta\in{\mathbb{C}}\setminus\{0\}$ we also introduce the first order differential operators 
\begin{equation}\label{8hhBBBa.2}
D_{\theta}:=i(\theta\,d-\theta^{-1}\delta),
\end{equation}
then $(D_\theta)^2=d\delta+\delta d=-\Delta$, so each $D_\theta$ may be regarded as a 
square-root of the negative Laplacian. Hence, each $D_{\theta}$ is a Dirac type operator, much as 
the Cauchy-Riemann operator $\partial_{\overline{z}}$ and its complex conjugate $\partial_z$
in the complex plane. Moreover, a simple computation (which makes use of the fact that $d^2=0$, 
$\delta^2=0$, and $\Delta=-d\delta-\delta d$) shows that
\begin{equation}\label{8hhBBBa.3}
S_{\theta_1}D_{\theta_2}=i\,D_{i\theta_1\cdot\theta_2}
\,\,\text{ for each }\,\,\theta_1,\theta_2\in{\mathbb{C}}\setminus\{0\},
\end{equation}
which may be viewed as an extension of the classical intertwining properties recorded in 
\eqref{R-87ygbg.BEBE.2tfV}. 

An alternative representation of $S_\theta$ as an operator on $L^2({\mathbb{R}}^n)\otimes\Lambda$, which is visible 
from \eqref{8hhBBBa.AAA}-\eqref{8hhBBBa.1} (upon recalling that the $j$-th Riesz transform on $L^2({\mathbb{R}}^n)$ 
is the multiplier with symbol $-i\xi_j/|\xi|$), is 
\begin{equation}\label{8hhBBBa.UUU}
S_{\theta}f=-\theta\,R\wedge(R\vee f)+\theta^{-1}\,R\vee(R\wedge f),\qquad 
f\in L^2({\mathbb{R}}^n)\otimes\Lambda,
\end{equation}
with the understanding that, in analogy to \eqref{8hhBBBa.AAA}, 
\begin{equation}\label{8hhBBBa.AAA.xxx}
R\wedge f:=\sum_{j=1}^n dx_j\wedge(R_j f),\quad
R\vee f:=-\sum_{j=1}^n dx_j\vee(R_j f),\quad\forall\,f\in L^2({\mathbb{R}}^n)\otimes\Lambda,
\end{equation}
where the Riesz transforms $R_j$'s act on the individual components of the differential form $f$. 
In particular, if for each $f\in\widetilde{\mathrm{BMO}}({\mathbb{R}}^n)\otimes\Lambda$ 
we also define (with similar conventions as above) 
\begin{equation}\label{8hhBBBa.AAA.xxx.2}
\widetilde{R}\wedge[f]:=\sum_{j=1}^n dx_j\wedge(\widetilde{R}_j[f]),\qquad
\widetilde{R}\vee[f]:=-\sum_{j=1}^n dx_j\vee(\widetilde{R}_j[f]),
\end{equation}
then Theorem~\ref{i87hbBV} permits us to extend the $\theta$-Beurling transform, originally 
considered as in \eqref{8hhBBBa.UUU}, to a linear and bounded operator $\widetilde{S}_{\theta}$ 
from $\widetilde{\mathrm{BMO}}({\mathbb{R}}^n)\otimes\Lambda$ into itself given by 
\begin{equation}\label{8hhBBBa.UUU.xxxx}
\widetilde{S}_{\theta}[f]:=-\theta\,\widetilde{R}\wedge(\widetilde{R}\vee[f])
+\theta^{-1}\,\widetilde{R}\vee(\widetilde{R}\wedge[f]),\qquad 
[f]\in\widetilde{\mathrm{BMO}}({\mathbb{R}}^n)\otimes\Lambda.
\end{equation}

In this vein, let us also introduce the $\theta$-Riesz transforms
(once again, on the frequency side) as
\begin{equation}\label{8hhBBBa.2rrrr}
R_{\theta}:=\frac{D_\theta}{\sqrt{-\Delta}}=i\theta\,\frac{d}{\sqrt{-\Delta}}
-i\theta^{-1}\frac{\delta}{\sqrt{-\Delta}},
\qquad\forall\,\theta\in{\mathbb{C}}\setminus\{0\},
\end{equation}
and note that they induce linear and bounded mappings on $L^2({\mathbb{R}}^n)\otimes\Lambda$ according to 
\begin{equation}\label{8hhBBBa.VVV}
R_{\theta}f=-i\big(\theta\,R\wedge f+\theta^{-1}\,R\vee f\big),\qquad 
f\in L^2({\mathbb{R}}^n)\otimes\Lambda.
\end{equation}
Thanks to Theorem~\ref{i87hbBV}, the $\theta$-Riesz transforms above may further be extended to 
linear and bounded operators $\widetilde{R}_{\theta}$ on 
$\widetilde{\mathrm{BMO}}({\mathbb{R}}^n)\otimes\Lambda$ according to 
\begin{equation}\label{8hhBBBa.scV}
\widetilde{R}_{\theta}[f]=-i\big(\theta\,\widetilde{R}\wedge[f]
+\theta^{-1}\,\widetilde{R}\vee[f]\big),\qquad [f]\in\widetilde{\mathrm{BMO}}({\mathbb{R}}^n)\otimes\Lambda.
\end{equation}

In relation to the $\theta$-Beurling transforms in \eqref{8hhBBBa.UUU.xxxx}
and the $\theta$-Riesz transforms in \eqref{8hhBBBa.scV}, we have the following result, 
akin to the characterization of the membership to {\rm VMO} in the two-dimensional case
given in Corollary~\ref{nhxtrE.2D}:

\begin{corollary}\label{nhxtrE-SSS}
For each $j,k\in\{1,\dots,n\}$ introduce 
\begin{equation}\label{utggvVV.T1}
\Theta_{jk}(x):=\frac{-nx_jx_k+\delta_{jk}|x|^2}{|x|^{n+2}},\qquad\forall\,x\in{\mathbb{R}}^n\setminus\{0\},
\end{equation}
and note that 
\begin{equation}\label{utggvVV.T2}
\begin{array}{c}
\displaystyle
\Theta_{jk}\in{\mathscr{C}}^\infty({\mathbb{R}}^n\setminus\{0\}),\,\,\,\Theta_{kj}=\Theta_{jk},\,\,\,
\int_{S^{n-1}}\Theta_{jk}(\omega)\,d\omega=0,\,\,\text{ and}
\\[6pt]
\text{$\Theta_{jk}$ is even and positive homogeneous of degree $-n$ in ${\mathbb{R}}^n\setminus\{0\}$}.
\end{array}
\end{equation}
In particular, these permit introducing the principal-value singular integral operators of 
convolution type $T_{\Theta_{jk}}$ associated with the $\Theta_{jk}$'s as in \eqref{ytgfff.9ytrdf.ygfg}.
Then for each $\theta\in{\mathbb{C}}\setminus\{0\}$ the operator $S_\theta$ is symmetric on 
$L^2({\mathbb{R}}^n)\otimes\Lambda$ and for each $f\in L^2({\mathbb{R}}^n)\otimes\Lambda$ one has 
{\rm (}with $T_{\Theta_{jk}}$ acting on the differential form $f$ componentwise{\rm )}
\begin{align}\label{i7yggg.B4}
S_\theta f &=-\frac{\theta}{\omega_{n-1}}\sum_{j,k=1}^n
dx_j\wedge\big(dx_k\vee\big(T_{\Theta_{jk}}f\big)\big)
+\frac{\theta^{-1}}{\omega_{n-1}}
\sum_{j,k=1}^n dx_j\vee\big(dx_k\wedge\big(T_{\Theta_{jk}}f\big)\big)
\nonumber\\[6pt]
&\quad-\frac{\theta}{n}\sum_{j=1}^n dx_j\wedge(dx_j\vee f)
+\frac{\theta^{-1}}{n}\sum_{j=1}^n dx_j\vee(dx_j\wedge f),
\end{align}
while for each $[f]\in\widetilde{\rm BMO}({\mathbb{R}}^n)\otimes\Lambda$ one has
{\rm (}with similar conventions as above{\rm )}
\begin{align}\label{i7yggg.B5555}
\widetilde{S}_\theta[f] &=-\frac{\theta}{\omega_{n-1}}\sum_{j,k=1}^n
dx_j\wedge\big(dx_k\vee\big(\widetilde{T}_{\Theta_{jk}}[f]\big)\big)
+\frac{\theta^{-1}}{\omega_{n-1}}
\sum_{j,k=1}^n dx_j\vee\big(dx_k\wedge\big(\widetilde{T}_{\Theta_{jk}}[f]\big)\big)
\nonumber\\[6pt]
&\quad-\frac{\theta}{n}\sum_{j=1}^n dx_j\wedge(dx_j\vee[f])
+\frac{\theta^{-1}}{n}\sum_{j=1}^n dx_j\vee(dx_j\wedge[f]).
\end{align}

Moreover, for each given differential form $f\in{\mathrm{BMO}}({\mathbb{R}}^n)\otimes\Lambda$ 
the following three conditions are equivalent:
\begin{enumerate}
\item[(i)] $f$ belongs to the space ${\mathrm{VMO}}({\mathbb{R}}^n)\otimes\Lambda$; 
\item[(ii)] $\big(cI+\widetilde{S}_\theta\big)[f]\in\widetilde{\mathrm{VMO}}({\mathbb{R}}^n)\otimes\Lambda$
for some {\rm (}or every{\rm )} $\theta\in{\mathbb{C}}\setminus\{0\}$ 
and $c\in{\mathbb{C}}\setminus\big\{\theta\,,\,-\theta^{-1}\big\}$;
\item[(iii)] $\big(cI+\widetilde{R}_\theta\big)[f]\in\widetilde{\mathrm{VMO}}({\mathbb{R}}^n)\otimes\Lambda$
for some {\rm (}or every{\rm )} $\theta\in{\mathbb{C}}\setminus\{0\}$
and $c\in{\mathbb{C}}\setminus\{\pm 1\}$.
\end{enumerate}
\end{corollary}

\vskip 0.10in

This paper is part of a larger program aimed at treating Dirichlet boundary value 
problems for $M\times M$ systems with constant complex coefficients as in
\eqref{L-def}-\eqref{L-ell.X} in the upper half-space ${\mathbb{R}}^n_{+}$ 
with boundary datum in various function spaces on ${\mathbb{R}}^{n-1}$.
The space ${\mathrm{BMO}}$, presently considered,  
lies at the cross-roads of several fundamental scales of function spaces in analysis. 
For one thing, ${\mathrm{BMO}}({\mathbb{R}}^{n-1},{\mathbb{C}}^M)$ may be regarded as a 
natural (right-most) end-point of the Lebesgue scale $L^p({\mathbb{R}}^{n-1},{\mathbb{C}}^M)$ 
with $p\in(1,\infty)$. The Dirichlet boundary value problem for elliptic systems $L$ as
in \eqref{L-def}-\eqref{L-ell.X} in the upper-half space with data from the latter scale 
of spaces has been recently treated in \cite{K-MMMM}, where the size of the solution 
$u:{\mathbb{R}}^{n}_{+}\to{\mathbb{C}}^M$ is measured using the nontangential 
maximal operator defined as 
\begin{equation}\label{NT-Fct}
\big({\mathcal{N}}u\big)(x'):=
\big({\mathcal{N}}_\kappa u\big)(x')
:=\sup\big\{|u(y)|:\,y\in\Gamma_\kappa(x')\big\},\qquad\forall\,x'\in{\mathbb{R}}^{n-1}.
\end{equation}
In this endeavor, the crux of the matter is the pointwise inequality (see \eqref{exTGFVC}) 
\begin{equation}\label{NT-FctMMM}
\begin{array}{l}
\big({\mathcal{N}}u\big)(x')\leq C({\mathcal{M}}f)(x')\,\,\text{ at each point }\,\,
x'\in{\mathbb{R}}^{n-1}
\\[4pt]
\text{if }\,\,
u(x',t):=(P^L_t\ast f)(x')\,\,\text{ for every }\,\,(x',t)\in{\mathbb{R}}^n_{+}\,\,
\\[4pt]
\text{and for some function }\,\,f\in L^1\big({\mathbb{R}}^{n-1},\frac{1}{1+|x'|^{n}}\,dx'\big)^M,
\end{array}
\end{equation}
where ${\mathcal{M}}$ is the Hardy-Littlewood maximal operator on ${\mathbb{R}}^{n-1}$
(cf. \eqref{MMax}). 

In fact, estimate \eqref{NT-FctMMM} permitted the treatment in \cite{K-MMMM} of  
a much larger variety of function lattice spaces. Indeed, one of the main results 
established in \cite{K-MMMM} is that the boundedness of the Hardy-Littlewood maximal operator 
on a K\"othe function space ${\mathbb{X}}$ and on its K\"othe dual ${\mathbb{X}}'$ 
(both considered in ${\mathbb{R}}^{n-1}$) is actually equivalent to the well-posedness of 
the $\mathbb{X}$-Dirichlet and $\mathbb{X}'$-Dirichlet problems in $\mathbb{R}^{n}_{+}$ 
in the class of all second-order, homogeneous, elliptic systems, with constant 
complex coefficients. As a consequence, in \cite{K-MMMM} the Dirichlet problem
for such systems has been shown to be well-posed for boundary data in Lebesgue spaces, variable
exponent Lebesgue spaces, Lorentz spaces, Zygmund spaces, as well as their weighted versions
with weights in the Muckenhoupt class.

This being said, the John-Nirenberg space ${\mathrm{BMO}}({\mathbb{R}}^{n-1})$ is not a lattice 
space (in the sense that a nonnegative measurable function with a pointwise majorant
in {\rm BMO} does not necessarily belong to {\rm BMO}), so a fresh 
look at the corresponding Dirichlet problem is warranted. In particular, the nature 
of the space of solutions (which should be suitably tailored to the specific space 
of boundary data) now involves a Carleson measure condition in place of the 
nontangential maximal operator \eqref{NT-Fct} which has been extensively used in \cite{K-MMMM}. 

Another point of view places the John-Nirenberg space 
${\mathrm{BMO}}({\mathbb{R}}^{n-1},{\mathbb{C}}^M)$ as a (left-most) endpoint 
for the scale of homogeneous H\"older spaces 
$\dot{\mathscr{C}}^\eta({\mathbb{R}}^{n-1},{\mathbb{C}}^M)$ with $\eta\in(0,1)$
(for pertinent definitions and basic properties regarding this scale see 
the discussion in the first part of \S\ref{auxiliary}).
Bearing this in mind, it is possible to formulate (a significant portion of) 
Theorem~\ref{them:BMO-Dir} in a manner that reflects the aforementioned feature 
of ${\rm BMO}$. To elaborate on this idea, given $\eta\in[0,1)$ and $p\in[1,\infty)$, for every 
$f\in L^1_{\rm loc}({\mathbb{R}}^{n-1},{\mathbb{C}}^M)$
define 
\begin{equation}\label{defi-BMO.2b}
\|f\|_{\ast}^{(\eta,p)}:=\sup_{Q\subset\mathbb{R}^{n-1}}\ell(Q)^{-\eta}
\Big(\,\,\aver{Q}\big|f(x')-f_Q\big|^p\,dx'\Big)^{\frac{1}{p}},
\end{equation}
and introduce the Morrey-Campanato space (which may be regarded as a fractional {\rm BMO} space, 
$L^p$-based, of order $\eta$) by setting 
\begin{equation}\label{defi-BMO.2b-SP}
{\mathscr{E}}^{\eta,p}({\mathbb{R}}^{n-1},\mathbb{C}^M):=\big\{
f\in L^1_{\rm loc}({\mathbb{R}}^{n-1},{\mathbb{C}}^M):\,
\|f\|_{\ast}^{(\eta,p)}<\infty\big\}.
\end{equation}
By the John-Nirenberg inequality it follows that, corresponding to the 
end-point case $\eta=0$ we have
\begin{equation}\label{defi-BMO.2b-SP.22}
{\mathscr{E}}^{0,p}({\mathbb{R}}^{n-1},\mathbb{C}^M)
={\mathrm{BMO}}({\mathbb{R}}^{n-1},\mathbb{C}^M),
\end{equation}
and it is clear from definitions that, in the regime $\eta>0$, the 
vanishing mean oscillation condition \eqref{defi-VMO} holds 
(this time, at a precisely quantified rate of decay) for every function 
$f\in{\mathscr{E}}^{\eta,p}({\mathbb{R}}^{n-1},\mathbb{C}^M)$.
Going further, for every $u\in{\mathscr{C}}^1({\mathbb{R}}^n_{+},{\mathbb{C}}^M)$ set
\begin{equation}\label{ustarstar-222}
\|u\|^{(\eta,p)}_{**}:=\sup_{Q\subset\mathbb{R}^{n-1}}\ell(Q)^{-\eta}
\left(\aver{Q}\Big(\int_{0}^{\ell(Q)}|\nabla u(x',t)|^2\,t\,dt\Big)^\frac{p}{2}
dx'\right)^\frac{1}{p}.
\end{equation}
The finiteness demand 
$\|u\|^{(\eta,p)}_{**}<\infty$ may be viewed (compare with \eqref{ncud}) as a 
fractional Carleson measure condition ($L^p$-based, of order $\eta$). 
In particular, it implies that the measure $d\mu(x',t):=|\nabla u(x',t)|^2 t\,dtdx'$ 
satisfies the vanishing condition \eqref{defi-CarlesonVan}, with a precisely 
quantified rate of decay.

Here is the statement of the theorem advertised earlier 
which deals with the larger, more inclusive context considered above and which 
complements the end-point case $\eta=0$ corresponding to the portion of 
Theorem~\ref{them:BMO-Dir} pertaining to the well-posedness of the 
${\rm BMO}$-Dirichlet boundary value problem. 

\begin{theorem}\label{them:BMO-Dir-frac}
Let $L$ be an $M\times M$ elliptic constant complex coefficient system as in 
\eqref{L-def}-\eqref{L-ell.X}, and fix $\eta\in(0,1)$ along with $p,q\in[1,\infty)$.
Then the Morrey-Campanato-Dirichlet boundary value 
problem for $L$ in $\mathbb{R}^{n}_+$, formulated as
\begin{equation}\label{Dir-BVP-BMO-frac}
\left\{
\begin{array}{l}
u\in{\mathscr{C}}^\infty(\mathbb{R}^{n}_{+},{\mathbb{C}}^M),
\\[4pt]
Lu=0\,\,\mbox{ in }\,\,\mathbb{R}^{n}_{+},
\\[4pt]
\|u\|^{(\eta,q)}_{**}<\infty,
\\[6pt]
u\big|^{{}^{\rm n.t.}}_{\partial{\mathbb{R}}^{n}_{+}}=f\,\,
\text{ a.e. in }\,\,{\mathbb{R}}^{n-1},\,\,
f\in{\mathscr{E}}^{\eta,p}(\mathbb{R}^{n-1},\mathbb{C}^M),
\end{array}
\right.
\end{equation}
has a unique solution. The solution $u$ of \eqref{Dir-BVP-BMO-frac} is given by
\eqref{eqn-Dir-BMO:u} and there exists a constant $C=C(n,L,\eta,p,q)\in(1,\infty)$ 
with the property that 
\begin{equation}\label{Dir-BVP-BMO-Car-frac}
C^{-1}\|f\|_\ast^{(\eta,p)}\leq\|u\|^{(\eta,q)}_{**}\leq C\,\|f\|_\ast^{(\eta,p)}.
\end{equation}
Moreover, $u$ belongs to $\dot{\mathscr{C}}^\eta({\mathbb{R}}^n_{+},{\mathbb{C}}^M)
=\dot{\mathscr{C}}^\eta(\overline{{\mathbb{R}}^n_{+}},{\mathbb{C}}^M)$ and, 
with $C\in(1,\infty)$ as above, 
\begin{equation}\label{Dir-BVP-BMO-Car-frac22}
C^{-1}\|f\|_\ast^{(\eta,p)}\leq\|u\|_{\dot{\mathscr{C}}^\eta({\mathbb{R}}^n_{+},{\mathbb{C}}^M)}
\leq C\,\|f\|_\ast^{(\eta,p)}.
\end{equation}
\end{theorem}

As a consequence of Theorem~\ref{them:BMO-Dir-frac} and its proof
(cf. also \eqref{eqn:BMO-decay-EEE.8T56.ww})
we obtain that, in fact, 
\begin{equation}\label{defi-BMO.2b-SP.ii}
{\mathscr{E}}^{\eta,p}({\mathbb{R}}^{n-1},\mathbb{C}^M)
=\dot{\mathscr{C}}^\eta({\mathbb{R}}^{n-1},{\mathbb{C}}^M)
\end{equation}
as vector spaces, with equivalent norms (the left-to-right inclusion 
is understood in the sense that if $f\in{\mathscr{E}}^{\eta,p}({\mathbb{R}}^{n-1},\mathbb{C}^M)$
then there exists some $g\in\dot{\mathscr{C}}^\eta({\mathbb{R}}^{n-1},{\mathbb{C}}^M)$
such that $f=g$ a.e. in ${\mathbb{R}}^{n-1}$). This offers a new proof (of a PDE flavor) 
of an old embedding result of N.~Meyers \cite{Mey}. An inspection of the proof of 
Theorem~\ref{them:BMO-Dir-frac} also reveals that there is a Fatou type result 
naturally accompanying the well-posedness result for the boundary value 
problem \eqref{Dir-BVP-BMO-frac}. 

\medskip

We shall now succinctly comment on the literature dealing with 
Dirichlet boundary value problems for elliptic operators in the upper-half space. 
As already noted, the nature of these problems strongly depends on the choice of the 
function space from which the boundary datum $f$ is selected, the specific way in 
which the size of the solution $u$ is measured, and the very manner in which its 
boundary trace is considered. To illustrate these distinctions, recall first that
there is a vast body of work targeting the case when the solution $u$ is sought in 
various Sobolev spaces in ${\mathbb{R}}^n_{+}$, the boundary datum $f$ belongs 
to suitable Besov spaces on ${\mathbb{R}}^{n-1}$, and the boundary trace of $u$ is 
considered in the sense of Sobolev space theory. Classical references in this regard include 
\cite{ADNI}, \cite{ADNII}, \cite{LionsMagenes}, \cite{MazShap}, \cite{Taylor}, and
the reader is also referred to the literature cited therein.      

The scenario in which the size of $u$ is measured in terms of the 
nontangential maximal operator \eqref{NT-Fct} and when the trace of $u$ on 
the boundary of ${\mathbb{R}}^n_{+}$ is taken in a nontangential pointwise sense 
(cf. \eqref{nkc-EE-2}) has been treated in \cite{K-MMMM} for the general class of 
$M\times M$ systems $L$ with constant complex coefficients as in
\eqref{L-def}-\eqref{L-ell.X}. This extends classical work carried out in the particular 
case when $L=\Delta$, the Laplacian in ${\mathbb{R}}^n$, treated 
in a number of monographs, including \cite{ABR}, \cite{GCRF85}, \cite{St70}, \cite{Stein93}, 
and \cite{StWe71}. The corresponding higher-order regularity Dirichlet problem 
in a similar framework has been recently considered in \cite{Esc-MMMM}. See also 
\cite{Sem-MMMM} for related work, emphasizing semigroup techniques. 

There is also a significant amount of work focused on the 
classical Dirichlet problem for the Laplacian in the upper-half space with a 
continuous boundary datum $f$. In such a case, one seeks a harmonic function 
$u\in{\mathscr{C}}^\infty({\mathbb{R}}^n_{+})\cap{\mathscr{C}}^0(\overline{{\mathbb{R}}^n_{+}})$
satisfying $u|_{\partial{\mathbb{R}}^n_{+}}=f$. A remarkable feature
(noted by Helms in \cite[p.\,42 and p.\,158]{He}) is that even in the case when the 
boundary datum $f$ is a bounded continuous function in ${\mathbb{R}}^{n-1}$ 
the solution $u$ of this classical Dirichlet problem is not unique. 
To ensure uniqueness in such a setting one typically specifies the behavior 
of $u(x',t)$ as $t\to\infty$. A case in point is \cite{ST96}, where uniqueness 
is established in the class of harmonic functions 
$u\in{\mathscr{C}}^\infty({\mathbb{R}}^n_{+})\cap{\mathscr{C}}^0(\overline{{\mathbb{R}}^n_{+}})$ 
satisfying $u(x)=o(|x|\sec^\gamma\theta)$ as $|x|\to\infty$ (where $\theta:=\arccos(x_n/|x|)$ and 
$\gamma\in{\mathbb{R}}$ is arbitrary), by proving a Phragmen-Lindel\"of principle under the
latter growth condition. This builds on the work of \cite{Si88}, \cite{Wolf41}, \cite{Yo96}, 
and others. The works just cited crucially rely on positivity and other various highly specialized 
properties of the Laplace operator, so the techniques employed there do not adapt  
to the considerably more general class of elliptic systems considered in the present paper. 

In relation to the context just described above it is instructive to make the following 
observations. First, the collection of uniformly continuous functions belonging to 
${\mathrm{BMO}}({\mathbb{R}}^{n-1},{\mathbb{C}}^M)$ is a dense subspace of 
${\mathrm{VMO}}({\mathbb{R}}^{n-1},{\mathbb{C}}^M)$ (see \eqref{ku6ffcfc}). 
Second, in the last part 
of Theorem~\ref{them:BMO-Dir} we have succeeded in proving the well-posedness
of the ${\mathrm{VMO}}$-Dirichlet problem in the class of null-solutions $u$ 
of a given elliptic system $L$ as in \eqref{L-def}-\eqref{L-ell.X} which satisfy 
a vanishing Carleson measure condition. This is a natural condition from the point 
of view of harmonic analysis which replaces the demand that the solution extends 
continuously on $\overline{{\mathbb{R}}^n_{+}}$, required in the formulation of the classical 
Dirichlet problem with continuous data. 

Apparently, the closest results in the literature to some of the work carried out in this
paper are those of E.~Fabes, R.~Johnson, and U.~Neri from 1976. Indeed, in \cite{FJN}
these authors have dealt with the ${\rm BMO}$-Dirichlet problem for the Laplacian in 
the upper-half space in the class of harmonic functions satisfying a Carleson measure 
condition (this being said, we would like to point out that there are certain gaps 
in some of the key steps of the treatment in \cite{FJN}, such as the proof 
of Lemma~1.3 on pp.\,161--162\footnote{The second equality in the first formula displayed on page 
162 is questionable, given that this involves the global gradient in ${\mathbb{R}}^{n+1}$, which  
includes the transversal variable $t$.}, and the proof of 
estimate (1.5) on page~163\footnote{Here the authors rely on the implication $3(iii)\Rightarrow\,2$ from 
\cite[pp.\,147--148]{FS} which is only established under the additional membership to $L^2({\mathbb{R}}^n)$.}). 
The portion of Theorem~\ref{them:BMO-Dir} dealing with \eqref{Dir-BVP-BMO} is
a significant generalization of their work, which is thereby extended to a much larger 
class of systems. Similar attributes are shared by our 
Theorem~\ref{them:BMO-Dir-frac} in relation to the work in \cite{FJN} dealing with
harmonic functions in the upper-half space with traces in Morrey-Campanato spaces. 
Generalizations of these results appear in \cite{Duong-Yan-Zhang} for the 
Schr\"odinger operator of the form $-\Delta+V$ with $V$ being a non-negative 
potential belonging to some reverse H\"older class (hence $0<V<\infty$ a.e.).

We also wish to mention here the work of 
B.~Dahlberg and C.~Kenig who have treated the ${\rm BMO}$-Dirichlet problem for the Laplacian 
in \cite[Theorem~4.18, p.\,463]{DaKe87} in bounded Lipschitz domains via layer potentials, 
building on the earlier work of E.~Fabes and U.~Neri from \cite{FaNe80} who employed 
harmonic measure techniques. For related work see also \cite{DiKePi}.

The techniques employed in \cite{DaKe87}, \cite{Duong-Yan-Zhang}, \cite{DiKePi}, \cite{FJN}, \cite{FaNe80},  
are largely restricted to scalar equations (as they make essential use of positivity and/or 
maximum principles). Also, the fact that in \cite{DaKe87}, \cite{DiKePi}, \cite{FaNe80},
the underlying domain is bounded makes the task of proving uniqueness considerably 
more manageable. In addition, the consideration 
of PDE's for which the well-posedness of the $L^2$-Dirichlet problem is known in arbitrary 
Lipschitz subdomains allows these authors to successfully employ a variety of localizations 
arguments. By way of contrast, most of these key features cease to be effective in 
the geometric/analytic context considered in this paper. In proving the
solvability of the $\mathrm{BMO}$-Dirichlet boundary value 
problem for an elliptic system $L$ in $\mathbb{R}^{n}_+$ as formulated in \eqref{Dir-BVP-BMO}, 
our approach makes essential use of the existence and properties of the Poisson kernel 
associated with $L$ from the work of \cite{ADNI}-\cite{ADNII}. Uniqueness
is derived with the help of the Fatou type result recorded in Theorem~\ref{thm:fatou-ADEEDE}.
A considerable amount of effort then goes into establishing the latter theorem, 
with square-function estimates (cf. Proposition~\ref{prop:SFE}), elements of 
tent-space theory (Lemma~\ref{lemma:tent}), interior estimates (cf. Theorem~\ref{ker-sbav}), 
and certain estimates near the boundary from \cite{MaMiSh} for null-solutions of $L$
vanishing on the boundary (cf. Proposition~\ref{c1.2}), among the
tools playing a key role in this regard.

We conclude with a brief overview of the contents of the sections of this paper.
Useful background material and auxiliary results are collected in \S\ref{auxiliary}.
The proofs of the existence statements in Theorem~\ref{them:BMO-Dir}, both for the 
{\rm BMO}-Dirichlet problem and the {\rm VMO}-Dirichlet problem, are carried out in 
\S\ref{section:exist-BMO}. Next, \S\ref{section:Fatou-BMO} is reserved for establishing 
a Fatou result for smooth null-solutions of $L$ satisfying a Carleson measure condition,
as well as uniqueness in the {\rm BMO}-Dirichlet problem, in the upper-half space. 
Finally, the proofs of Theorems~\ref{them:BMO-Dir}-\ref{ndyRE} as well as 
Theorems~\ref{THMVMO.CCC}-\ref{jhdwtRD} are given in \S\ref{Pf-mainThms}, 
the proof of Theorem~\ref{them:BMO-Dir-frac} is contained 
in \S\ref{S-4}, while the proofs of Theorems~\ref{i87hbBV}-\ref{i87hbBV-ALG} and 
Corollaries~\ref{i87hbBV-ALG-CCC}-\ref{nhxtrE-SSS}
are presented in \S\ref{NEWSSS}.

\section{Background material and preliminary results}
\setcounter{equation}{0}
\label{auxiliary}

In this section we collect a number of preliminary results that are useful in the sequel.
Throughout, we let ${\mathbb{N}}$ stand for the collection of all
positive integers, and set ${\mathbb{N}}_0:={\mathbb{N}}\cup\{0\}$. In this way 
$\mathbb{N}_0^k$ stands for the set of multi-indices $\alpha=(\alpha_1,\dots,\alpha_k)$ 
with $\alpha_j\in\mathbb{N}_0$ for $1\le j\le k$. Also, fix $n\in{\mathbb{N}}$ with $n\geq 2$.
For an arbitrary multi-index $\alpha=(\alpha_1,\dots,\alpha_n)\in{\mathbb{N}}_0^n$ we use the 
standard notation $\partial^\alpha:=\partial^{\alpha_1}_{x_1}\dots\partial^{\alpha_n}_{x_n}$
and we occasionally abbreviate $\partial_{x_j}$ by simply $\partial_j$ for $j\in\{1,\dots,n\}$.
The length of the multi-index $\alpha=(\alpha_1,\dots,\alpha_n)$ is defined as
$|\alpha|:=\alpha_1+\cdots+\alpha_n$. We agree to let $\{e_j\}_{1\leq j\leq n}$ stand for the 
standard orthonormal basis in ${\mathbb{R}}^n$. Occasionally, we canonically identify $e_j$ with 
a multi-index in ${\mathbb{N}}_0$ (of length $1$). Given an arbitrary set $E\subseteq{\mathbb{R}}^{n-1}$ 
we denote by ${\mathbf 1}_E$ the characteristic function of $E$.

Generally speaking, given a metric space $(X,d)$, 
corresponding to each subset $E$ of $X$ (of cardinality $\geq 2$) and number $\eta>0$,
we associate the homogeneous H\"older space or order $\eta$, 
denoted by $\dot{\mathscr{C}}^\eta(E,{\mathbb{C}}^M)$, 
as the collection of functions $w:E\to{\mathbb{C}}^M$ satisfying 
\begin{equation}\label{eqn:BMO-decay-EEE.8T56}
\|w\|_{\dot{\mathscr{C}}^\eta(E,{\mathbb{C}}^M)}:=
\sup_{\stackrel{x,y\in X}{x\not=y}}\frac{|w(x)-w(y)|}{d(x,y)^\eta}<\infty.
\end{equation}
Whenever $E\subseteq F\subseteq X$ (with $E$ having cardinality $\geq 2$) we then have
\begin{equation}\label{eqn:BMO-decay-EEE.8T56.ww}
\begin{array}{c}
\dot{\mathscr{C}}^\eta(E,{\mathbb{C}}^M)=\dot{\mathscr{C}}^\eta(\overline{E},{\mathbb{C}}^M)
\,\,\text{ isometrically, and}
\\[4pt]
\dot{\mathscr{C}}^\eta(F,{\mathbb{C}}^M)\ni w\mapsto w\big|_{E}\in
\dot{\mathscr{C}}^\eta(E,{\mathbb{C}}^M)\,\,\text{ continuously}.
\end{array}
\end{equation}
Also, 
\begin{equation}\label{eqn:BMO-decay-EEE.8T56.ww3}
\dot{\mathscr{C}}^\eta(E,{\mathbb{C}}^M)\subseteq{\mathrm{UC}}(E,{\mathbb{C}}^M),
\end{equation}
where the latter denotes the space of ${\mathbb{C}}^M$-valued functions which are 
uniformly continuous on the set $E$. Finally, we agree to drop the dependence on the range
when $M=1$, and denote by ${\mathrm{Lip}}(E)$ the homogeneous H\"older space on $E$ of order
$\eta=1$.

Moving on, we denote by $\mathcal{M}$ the Hardy-Littlewood maximal operator 
on $\mathbb{R}^{n-1}$ which acts on vector-valued functions with components 
in $L^1_{\rm loc}({\mathbb{R}}^{n-1})$ according to
\begin{equation}\label{MMax}
\big(\mathcal{M}f\big)(x'):=\sup_{Q\ni x'}\aver{Q}|f(y')|\,dy',\qquad 
\forall\,x'\in\mathbb{R}^{n-1},
\end{equation}
where the supremum runs over all cubes $Q$ in $\mathbb{R}^{n-1}$ containing $x'$. 

We will often work with the weighted Lebesgue space of the form 
\begin{align}\label{jdgswd}
& L^1\Big({\mathbb{R}}^{n-1}\,,\,\frac{dx'}{1+|x'|^a}\Big)
\\[6pt]
&\hskip 0.50in
:=\left\{f:{\mathbb{R}}^{n-1}\to{\mathbb{C}}\,\,\text{ Lebesgue measurable}:\,
\int_{{\mathbb{R}}^{n-1}}\frac{|f(x')|}{1+|x'|^a}\,dx'<\infty\right\},
\nonumber
\end{align}
where $a\in(0,\infty)$, and we shall denote by 
$L^1\Big({\mathbb{R}}^{n-1}\,,\,\frac{dx'}{1+|x'|^a}\Big)^M$ the
space of ${\mathbb{C}}^M$-valued functions with components in \eqref{jdgswd}. 
Clearly,
\begin{equation}\label{eq:aaAabgr}
L^1\Big({\mathbb{R}}^{n-1}\,,\,\frac{dx'}{1+|x'|^a}\Big)^M
\subset L^1_{\rm loc}({\mathbb{R}}^{n-1},{\mathbb{C}}^M),\qquad\forall\,a>0.
\end{equation}

Next, we record several useful properties of mean oscillations
(recall the piece of notation introduced in \eqref{nota-aver}).
First we note that if $Q$ and $Q'$ are cubes in ${\mathbb{R}}^{n-1}$
with the property that $Q'\subseteq Q$, then for any 
$f\in L^1_{\rm loc}({\mathbb{R}}^{n-1},{\mathbb{C}}^M)$ and any $p\in[1,\infty)$
we have 
\begin{equation}\label{aver-fq}
\Big(\aver{Q'}|f(y')-f_{Q'}|^p\,dy'\Big)^{\frac{1}{p}}
\leq 2\Big(\frac{\ell(Q)}{\ell(Q')}\Big)^{\frac{n-1}{p}}
\Big(\aver{Q}|f(y')-f_{Q}|^p\,dy'\Big)^{\frac{1}{p}}
\end{equation}
and  
\begin{equation}\label{aver-fq-BBBB}
\Big(\aver{Q}|f(y')-f_{Q'}|^p\,dy'\Big)^{\frac{1}{p}}
\leq\Big[1+\Big(\frac{\ell(Q)}{\ell(Q')}\Big)^{\frac{n-1}{p}}\Big]
\Big(\aver{Q}|f(y')-f_{Q}|^p\,dy'\Big)^{\frac{1}{p}}.
\end{equation}
Also, 
\begin{equation}\label{aver-fq-Cava}
\frac{1}{2}\Big(\aver{Q}|f(y')-f_{Q}|^p\,dy'\Big)^{\frac{1}{p}}
\leq\inf_{c\in{\mathbb{C}}^M}\Big(\aver{Q}|f(y')-c|^p\,dy'\Big)^{\frac{1}{p}}
\leq\Big(\aver{Q}|f(y')-f_{Q}|^p\,dy'\Big)^{\frac{1}{p}}.
\end{equation}

Second, we recall the John-Nirenberg inequality asserting that there exist 
two dimensional constants $C_1,C_2\in(0,\infty)$ with the following significance. 
Consider an arbitrary cube $Q\subset{\mathbb{R}}^{n-1}$ along with a function 
$f\in L^1(Q)$ with the property that 
\begin{equation}\label{JY6GTFF}
N_Q(f):=\sup_{Q'\subseteq Q}\aver{Q'}|f(y')-f_{Q'}|\,dy'<\infty,
\end{equation}
where the above supremum involves cubes $Q'\subset{\mathbb{R}}^{n-1}$ contained in $Q$.
Then there holds (see, e.g., \cite[Corollary~2, p.\,154]{Stein93})
\begin{equation}\label{k765rffc}
{\mathscr{L}}^{n-1}\big(\{y'\in Q:\,|f(y')-f_Q|>\lambda\}\big)\leq C_1\,
e^{-\big(\frac{C_2}{N_Q(f)}\big)\lambda}\,|Q|,\qquad\forall\,\lambda>0.
\end{equation}
Third, as a corollary of the John-Nirenberg inequality, we obtain that for every 
$p\in(0,\infty)$ there exists a constant $C_{n,p}\in(0,\infty)$ with the property that
for every cube $Q\subset{\mathbb{R}}^{n-1}$ and every function $f\in L^1(Q,{\mathbb{C}}^M)$ 
we have 
\begin{equation}\label{JY6GTFF.2}
\Big(\aver{Q}|f(y')-f_{Q}|^p\,dy'\Big)^{\frac{1}{p}}\leq C_{n,p}
\sup_{Q'\subseteq Q}\aver{Q'}|f(y')-f_{Q'}|\,dy'.
\end{equation}
 
To proceed, for each $p\in[1,\infty)$, $r\in(0,\infty)$, and 
$f\in L^1_{\rm loc}({\mathbb{R}}^{n-1},{\mathbb{C}}^M)$
define the $L^p$-based mean oscillations of $f$ at a given scale 
$r\in(0,\infty)$ as
\begin{equation}\label{osc-1}
{\rm osc}_p(f;r):=\sup_{Q\subset\mathbb{R}^{n-1},\,\ell(Q)\leq r}
\Big(\aver{Q}|f(x')-f_Q|^pdx'\Big)^{\frac{1}{p}}\in[0,\infty].
\end{equation}
Some of the main properties of this function are summarized next.

\begin{lemma}\label{jsfsQAXT}
For each $f\in L^1_{\rm loc}({\mathbb{R}}^{n-1},{\mathbb{C}}^M)$ the following 
properties hold. 
\begin{list}{$(\theenumi)$}{\usecounter{enumi}\leftmargin=.8cm
\labelwidth=.8cm\itemsep=0.2cm\topsep=.1cm
\renewcommand{\theenumi}{\alph{enumi}}}
\item[(a)] Fix $p\in[1,\infty)$. Then, as a function of $r$, the quantity 
${\rm osc}_p(f;r)$ is nondecreasing in $r$. 
\item[(b)] For every $p,q\in[1,\infty)$ there exists a constant $C=C(p,q,n)\in(1,\infty)$, 
independent of $f$, with the property that 
\begin{equation}\label{jsfd-1a}
C^{-1}{\rm osc}_p(f;r)\leq{\rm osc}_q(f;r)
\leq C\,{\rm osc}_p(f;r)\qquad\text{ for every }\,\,r\in(0,\infty).
\end{equation}
\item[(c)] The function $f$ actually belongs to ${\rm BMO}({\mathbb{R}}^{n-1},{\mathbb{C}}^M)$ 
if and only if ${\rm osc}_p(f;r)$ as a function in $r$ is bounded on $(0,\infty)$ 
for some, or any, $p\in[1,\infty)$. Moreover, for each $p\in[1,\infty)$ there exists a 
constant $C=C(n,p)\in(1,\infty)$, independent of $f$, with the property that 
\begin{equation}\label{jsfd-1}
C^{-1}\|f\|_{{\rm BMO}({\mathbb{R}}^{n-1},{\mathbb{C}}^M)}
\leq\sup_{r>0}\,{\rm osc}_p(f;r)
=\lim_{r\to \infty}{\rm osc}_p(f;r)
\leq C\|f\|_{{\rm BMO}({\mathbb{R}}^{n-1},{\mathbb{C}}^M)}.
\end{equation}
\item[(d)] The function $f$ actually belongs to 
${\rm VMO}({\mathbb{R}}^{n-1},{\mathbb{C}}^M)$ 
if and only if for some, or any exponent $p\in[1,\infty)$ one has
\begin{equation}\label{itr54ffd}
\lim\limits_{r\to 0^{+}}{\rm osc}_p(f;r)=0\,\,\text{ and }\,\,
\lim\limits_{r\to \infty}{\rm osc}_p(f;r)<\infty.
\end{equation}
\item[(e)] For every $\eta\in[0,1)$ and $p\in[1,\infty)$ we have 
{\rm (}recall \eqref{defi-BMO.2b}{\rm )}
\begin{equation}\label{jsfd-3vcgfd}
{\rm osc}_p(f;r)\leq r^\eta\|f\|_{\ast}^{(\eta,p)},\qquad\forall\,r\in(0,\infty).
\end{equation}
\item[(f)] If actually $f$ belongs to 
${\mathscr{C}}^{\Upsilon}({\mathbb{R}}^{n-1},{\mathbb{C}}^M)$ for some 
modulus of continuity $\Upsilon$ {\rm (}recall \eqref{UpU}-\eqref{UpUpUp}{\rm )}, then 
for each $p\in[1,\infty)$ one has
\begin{equation}\label{jsfd-3}
{\rm osc}_p(f;r)\leq\|f\|_{{\mathscr{C}}^{\Upsilon}({\mathbb{R}}^{n-1},{\mathbb{C}}^M)}
\Upsilon(\sqrt{n}\,r),\qquad\forall\,r\in(0,\infty).
\end{equation}
In particular, for each $p\in[1,\infty)$ and $\eta\in(0,1)$ there exists $C\in(0,\infty)$ 
such that for every function $f\in\dot{\mathscr{C}}^\eta({\mathbb{R}}^{n-1},{\mathbb{C}}^M)$ 
one has
\begin{equation}\label{jsfd-3-fff}
{\rm osc}_p(f;r)\leq Cr^\eta\|f\|_{\dot{\mathscr{C}}^\eta({\mathbb{R}}^{n-1},{\mathbb{C}}^M)},
\qquad\forall\,r\in(0,\infty).
\end{equation}
\end{list}
\end{lemma}

\begin{proof}
The claim made in part {\it (a)} follows  directly from \eqref{osc-1}. 
The claim in part {\it (b)} is a direct consequence 
of H\"older's inequality and John-Nirenberg's inequality (see \eqref{JY6GTFF.2}). 
The latter also implies the claims made in part {\it (c)}. The claim in part {\it (d)} 
is a consequence of {\it (a)}-{\it (c)} and \eqref{defi-VMO}.
Estimate \eqref{jsfd-3vcgfd} is immediate from \eqref{osc-1} and \eqref{defi-BMO.2b}.
Finally, if actually $f\in{\mathscr{C}}^{\Upsilon}({\mathbb{R}}^{n-1},{\mathbb{C}}^M)$ 
then for each $p\in[1,\infty)$ and each cube 
$Q$ in ${\mathbb{R}}^{n-1}$ H\"older's inequality gives
\begin{align}\label{jsfd-3b}
\Big(\aver{Q}|f(x')-f_Q|^pdx'\Big)^{\frac{1}{p}}
&\leq\Big(\aver{Q}\aver{Q}|f(x')-f(y')|^pdy'\,dx'\Big)^{\frac{1}{p}}
\nonumber\\[4pt]
&\leq \|f\|_{{\mathscr{C}}^{\Upsilon}({\mathbb{R}}^{n-1},{\mathbb{C}}^M)}
\Upsilon\big(\sqrt{n}\,\ell(Q)\big).
\end{align}
Then \eqref{jsfd-3} follows from \eqref{jsfd-3b} given that $\Upsilon$ is nondecreasing.
\end{proof}

Next, we discuss the manner in which global integrability properties of a given 
function are related to the behavior at infinity of its mean oscillation function. 

\begin{lemma}\label{lemma:BMO-decay}
Fix $\varepsilon>0$ arbitrary. Then there exists a constant $C_{n,\varepsilon}\in(0,\infty)$ 
such that for each function $f\in L^1_{\rm loc}(\mathbb{R}^{n-1},\mathbb{C}^M)$ and 
each cube $Q\subset{\mathbb{R}}^{n-1}$, with center $x'_Q\in{\mathbb{R}}^{n-1}$, there holds
\begin{align}\label{eqn:BMO-decay.88}
\int_{\mathbb{R}^{n-1}}\frac{|f(y')-f_{Q}|}{\big[\ell(Q)+|x'_Q-y'|\big]^{n-1+\varepsilon}}\,dy'
&\leq\frac{C_{n,\varepsilon}}{\ell(Q)^{\varepsilon}}
\int_{1}^\infty\Big(\aver{\lambda Q}|f(y')-f_{\lambda Q}|\,dy'\Big)
\frac{d\lambda}{\lambda^{1+\varepsilon}}
\nonumber\\[6pt]
&\leq\frac{C_{n,\varepsilon}}{\ell(Q)^{\varepsilon}}
\int_{1}^\infty{\rm osc}_1\big(f;\lambda\ell(Q)\big)\,\frac{d\lambda}{\lambda^{1+\varepsilon}}.
\end{align}

As a consequence, for each $f\in L^1_{\rm loc}(\mathbb{R}^{n-1},\mathbb{C}^M)$ one has
\begin{equation}\label{jGVVCc-1jmn}
\int_{1}^\infty{\rm osc}_1(f;\lambda)\,
\frac{d\lambda}{\lambda^{1+\varepsilon}}<\infty\,\Longrightarrow\,
f\in L^1\Big({\mathbb{R}^{n-1}}\,,\,\frac{dx'}{1+|x'|^{n-1+\varepsilon}}\Big)^M
\end{equation}
and there exists a constant $C_{n,\varepsilon}\in(0,\infty)$ with the property that
\begin{equation}\label{eq:aaAabgr-33}
\int_{{\mathbb{R}}^{n-1}}\frac{|f(x')|}{1+|x'|^{n-1+\varepsilon}}\,dx'
\leq C_{n,\varepsilon}\,\int_{1}^\infty{\rm osc}_1(f;\lambda)\,
\frac{d\lambda}{\lambda^{1+\varepsilon}}+C_{n,\varepsilon}\aver{Q_0}|f(x')|\,dx'
\end{equation}
where $Q_0:=(-1/2,1/2)^{n-1}$ is the cube centered at the origin $0'$ 
of $\mathbb{R}^{n-1}$ with side-length $1$. 

In particular, we have
\begin{equation}\label{eq:aaAabgr-22}
{\rm BMO}\big({\mathbb{R}}^{n-1},{\mathbb{C}}^M\big)\subset
L^1\Big({\mathbb{R}}^{n-1}\,,\,\frac{dx'}{1+|x'|^{n-1+\varepsilon}}\Big)^M,
\qquad\forall\,\varepsilon>0,
\end{equation}
and for each $p\in[1,\infty)$ {\rm (}recall \eqref{defi-BMO.2b-SP}{\rm )}
\begin{equation}\label{eq:aaAabgr-22BB}
{\mathscr{E}}^{\eta,p}({\mathbb{R}}^{n-1},\mathbb{C}^M)\subset
L^1\Big({\mathbb{R}}^{n-1}\,,\,\frac{dx'}{1+|x'|^{n-1+\varepsilon}}\Big)^M,
\qquad\forall\,\varepsilon>0,\,\,\forall\eta\in[0,\varepsilon),
\end{equation}
while in view of \eqref{jsfd-3-fff} and \eqref{jGVVCc-1jmn} we obtain
\begin{equation}\label{Gsyus}
\dot{\mathscr{C}}^\eta({\mathbb{R}}^{n-1},\mathbb{C}^M)\subset
L^1\Big({\mathbb{R}}^{n-1}\,,\,\frac{dx'}{1+|x'|^{n-1+\varepsilon}}\Big)^M,
\qquad\forall\eta\in(0,\varepsilon).
\end{equation}
\end{lemma}

\begin{proof}
Given $f\in L^1_{\rm loc}(\mathbb{R}^{n-1},\mathbb{C}^M)$ and 
a cube $Q\subset{\mathbb{R}}^{n-1}$ with center $x'_Q\in{\mathbb{R}}^{n-1}$, 
breaking up the domain of integration allows us to estimate 
\begin{align}\label{eqn-BMO-decay:1-EEE.88}
&\int_{\mathbb{R}^{n-1}}\frac{|f(y')-f_{Q}|}{\big[\ell(Q)+|x'_Q-y'|\big]^{n-1+\varepsilon}}\,dy'
\nonumber\\[4pt]
&\hskip 0.50in
\leq\ell(Q)^{-n+1-\varepsilon}\int_{Q}|f(y')-f_{Q}|\,dy'
+\sum_{k=0}^\infty\int_{2^{k+1}Q\setminus 2^kQ}
\frac{|f(y')-f_{Q}|}{|x'_Q-y'|^{n-1+\varepsilon}}\,dy'
\nonumber\\[4pt]
&\hskip 0.50in
\leq\ell(Q)^{-\varepsilon}\aver{Q}|f(y')-f_{Q}|\,dy'
\nonumber\\[4pt]
&\hskip 0.50in
\quad +2^{2(n-1)+\varepsilon}
\ell(Q)^{-\varepsilon}\sum_{k=0}^\infty 2^{-k\varepsilon}\,\aver{2^{k+1}Q}|f(y')-f_{Q}|\,dy'.
\end{align}
Next, for each $k\in{\mathbb{N}}_0$ we have
\begin{align}\label{eqn-BMO-decay:2jj.88}
\aver{2^{k+1}Q}|f(y')-f_{Q}|\,dy' 
&\leq\aver{2^{k+1}Q}|f(y')-f_{2^{k+1}Q}|\,dy'
+\sum_{j=0}^k|f_{2^j Q}-f_{2^{j+1}Q}|
\nonumber\\[4pt]
&\leq\aver{2^{k+1}Q}|f(y')-f_{2^{k+1}Q}|\,dy'
\nonumber\\[4pt]
&\quad
+2^{n-1}\,\sum_{j=0}^k\aver{2^{j+1}Q}|f(y')-f_{2^{j+1}Q}|\,dy',
\end{align}
hence,  
\begin{align}\label{eqn-BMO-deRRF.1}
\sum_{k=0}^\infty 2^{-k\varepsilon}\,\aver{2^{k+1}Q}|f(y')-f_{Q}|\,dy'
& \leq\sum_{k=0}^\infty 2^{-k\varepsilon}\,\aver{2^{k+1}Q}|f(y')-f_{2^{k+1}Q}|\,dy'
\nonumber\\[4pt]
&\quad
+2^{n-1}\,\sum_{k=0}^\infty 2^{-k\varepsilon}\Big\{
\sum_{j=0}^k\aver{2^{j+1}Q}|f(y')-f_{2^{j+1}Q}|\,dy'\Big\}
\nonumber\\[4pt]
& =\sum_{k=0}^\infty 2^{-k\varepsilon}\,\aver{2^{k+1}Q}|f(y')-f_{2^{k+1}Q}|\,dy'
\nonumber\\[4pt]
&\quad
+\frac{2^{n-1}}{1-2^{-\varepsilon}}\,\sum_{j=0}^\infty
2^{-j\varepsilon}\aver{2^{j+1}Q}|f(y')-f_{2^{j+1}Q}|\,dy'
\nonumber\\[4pt]
&=\Big(1+\frac{2^{n-1}}{1-2^{-\varepsilon}}\Big)\sum_{k=0}^\infty
2^{-k\varepsilon}\aver{2^{k+1}Q}|f(y')-f_{2^{k+1}Q}|\,dy',
\end{align}
where the first equality has been obtained by interchanging the sums in $k$ and $j$. 
Collectively, \eqref{eqn-BMO-decay:1-EEE.88} and \eqref{eqn-BMO-deRRF.1} permit us to conclude 
that
\begin{align}\label{eqn-BMO-dePHGR}
&\int_{\mathbb{R}^{n-1}}\frac{|f(y')-f_{Q}|}{\big[\ell(Q)+|x'_Q-y'|\big]^{n-1+\varepsilon}}\,dy'
\nonumber\\[4pt]
&\hskip 0.50in
\leq 4^{n-1+\varepsilon}\Big(1+\frac{2^{n-1}}{1-2^{-\varepsilon}}\Big)
\ell(Q)^{-\varepsilon}\sum_{k=0}^\infty
2^{-k\varepsilon}\aver{2^{k}Q}|f(y')-f_{2^{k}Q}|\,dy'.
\end{align}
To proceed, observe that \eqref{aver-fq} yields 
\begin{equation}\label{aver-fq-5tgb}
\begin{array}{c}
\displaystyle
\aver{2^kQ}|f(y')-f_{2^kQ}|\,dy'\leq 2^n\aver{\lambda Q}|f(y')-f_{\lambda Q}|\,dy',
\\[10pt]
\text{for each $k\in{\mathbb{N}}_0$ and each $\lambda\in[2^k,2^{k+1}]$}.
\end{array} 
\end{equation}
This, in turn, implies that for each $k\in{\mathbb{N}}_0$ we have
\begin{equation}\label{aver-fq-5tgb.22}
2^{-k\varepsilon}\aver{2^kQ}|f(y')-f_{2^kQ}|\,dy'
\leq \frac{2^n\varepsilon}{1-2^{-\varepsilon}}\int_{2^k}^{2^{k+1}}
\Big(\aver{\lambda Q}|f(y')-f_{\lambda Q}|\,dy'\Big)\,\frac{d\lambda}{\lambda^{1+\varepsilon}}.
\end{equation}
Availing ourselves of this estimate back into \eqref{eqn-BMO-dePHGR} 
then establishes the first inequality in \eqref{eqn:BMO-decay.88} for the choice 
\begin{equation}\label{aver-fq-5tgb.23}
C_{n,\varepsilon}:=2^n\,4^{n-1+\varepsilon}\Big(1+\frac{2^{n-1}}{1-2^{-\varepsilon}}\Big)
\cdot\frac{\varepsilon}{1-2^{-\varepsilon}}.
\end{equation}
The second inequality in \eqref{eqn:BMO-decay.88} is a direct 
consequence of this and \eqref{osc-1}. Going further, \eqref{jGVVCc-1jmn}-\eqref{eq:aaAabgr-33} 
follow from the second inequality in \eqref{eqn:BMO-decay.88}
with $Q:=(-1/2,1/2)^{n-1}$. 
In turn, \eqref{eq:aaAabgr-33} together with part {\it (c)} in Lemma~\ref{jsfsQAXT} 
give \eqref{eq:aaAabgr-22}, while 
\eqref{eq:aaAabgr-33} together with part {\it (e)} in Lemma~\ref{jsfsQAXT} 
give \eqref{eq:aaAabgr-22BB}.
\end{proof}

Poisson kernels for elliptic operators in a half-space have a long history 
(see, e.g., \cite{ADNI}, \cite{ADNII}, \cite{Sol1}, \cite{Sol2}).
Here we record the following useful existence and uniqueness result.
In its statement (as well as elsewhere in the paper), we make the convention 
that the convolution between two functions, which are matrix-valued and 
vector-valued, respectively, takes into account the algebraic 
multiplication between a matrix and a vector in a natural fashion.

\begin{theorem}\label{kkjbhV}
Let $L$ be an $M\times M$ elliptic system with constant complex coefficients as in
\eqref{L-def}-\eqref{L-ell.X}. Then there exists a matrix-valued function
$P^L=\big(P^L_{\alpha\beta}\big)_{1\leq\alpha,\beta\leq M}:
\mathbb{R}^{n-1}\to\mathbb{C}^{M\times M}$
{\rm (}called the Poisson kernel for $L$ in $\mathbb{R}^{n}_+${\rm )}
satisfying the following properties:
\begin{list}{$(\theenumi)$}{\usecounter{enumi}\leftmargin=.8cm
\labelwidth=.8cm\itemsep=0.2cm\topsep=.1cm
\renewcommand{\theenumi}{\alph{enumi}}}
\item[(1)] There exists $C\in(0,\infty)$ such that
\begin{equation}\label{eq:IG6gy}
|P^L(x')|\leq\frac{C}{(1+|x'|^2)^{\frac{n}2}}\quad\mbox{for each }\,\,x'\in\mathbb{R}^{n-1}.
\end{equation}
\item[(2)] The function $P^L$ is Lebesgue measurable and
\begin{equation}\label{eq:IG6gy.2}
\int_{\mathbb{R}^{n-1}}P^L(x')\,dx'=I_{M\times M},
\end{equation}
where $I_{M\times M}$ denotes the $M\times M$ identity matrix. In particular, for every 
constant vector $C=\big(C_{\alpha}\big)_{1\leq\alpha\leq M}\in\mathbb{C}^M$
one has
\begin{equation}\label{eq:IG6gy.2PPP}
\int_{\mathbb{R}^{n-1}}\sum_{1\leq\beta\leq M}
\big(P^L_{\alpha\beta}\big)_t(x'-y')C_\beta\,dy'
=C_{\alpha},\qquad\forall\,(x',t)\in{\mathbb{R}}^n_{+}.
\end{equation}
\item[(3)] If one sets
\begin{equation}\label{eq:Gvav7g5}
K^L(x',t):=P^L_t(x')=t^{1-n}P^L(x'/t)\quad\mbox{for each }\,\,x'\in\mathbb{R}^{n-1}
\,\,\,\mbox{ and }\,\,t>0,
\end{equation}
then the function $K^L=\big(K^L_{\alpha\beta}\big)_{1\leq\alpha,\beta\leq M}$
satisfies {\rm (}in the sense of distributions{\rm )}
\begin{equation}\label{uahgab-UBVCX}
LK^L_{\cdot\beta}=0\,\,\mbox{ in }\,\,\mathbb{R}^{n}_{+}
\,\,\mbox{ for each }\,\,\beta\in\{1,\dots,M\},
\end{equation}
where $K^L_{\cdot\beta}:=\big(K^L_{\alpha\beta}\big)_{1\leq\alpha\leq M}$
is the $\beta$-th column in $K^L$.
\end{list}

Moreover, $P^L$ is unique in the class of $\mathbb{C}^{M\times M}$-valued
functions defined in ${\mathbb{R}}^{n-1}$ and satisfying (1)-(3) above, 
and has the following additional properties:

\begin{list}{$(\theenumi)$}{\usecounter{enumi}\leftmargin=.8cm
\labelwidth=.8cm\itemsep=0.2cm\topsep=.1cm
\renewcommand{\theenumi}{\alph{enumi}}}
\item[(4)] One has $P^L\in{\mathscr{C}}^\infty(\mathbb{R}^{n-1})$ and
$K^L\in{\mathscr{C}}^\infty\big(\overline{{\mathbb{R}}^n_{+}}\setminus B(0,\varepsilon)\big)$
for every $\varepsilon>0$. Consequently, \eqref{uahgab-UBVCX}
holds in a pointwise sense.
\item[(5)] There holds $K^L(\lambda x)=\lambda^{1-n}K^L(x)$ for all $x\in{\mathbb{R}}^n_{+}$
and $\lambda>0$. In particular, for each multi-index $\alpha\in{\mathbb{N}}_0^n$
there exists $C_\alpha\in(0,\infty)$ with the property that
\begin{equation}\label{eq:Kest}
\big|(\partial^\alpha K^L)(x)\big|\leq C_\alpha\,|x|^{1-n-|\alpha|},\qquad
\forall\,x\in{\overline{{\mathbb{R}}^n_{+}}}\setminus\{0\}.
\end{equation}
\item[(6)] For each $\kappa>0$ there exists a finite constant $C_\kappa>0$ 
with the property that for each $x'\in\mathbb{R}^{n-1}$,
\begin{equation}\label{exTGFVC}
\sup_{|x'-y'|<\kappa t}\big|(P^L_t\ast f)(y')\big|\leq C_\kappa\,\mathcal{M}f(x'),
\qquad\forall\,f\in L^1\Big({\mathbb{R}}^{n-1}\,,\,\frac{1}{1+|x'|^n}\,dx'\Big)^M.
\end{equation}
\item[(7)] Fix an arbitrary $\kappa>0$ and a function
\begin{equation}\label{eq:aaAa}
f=(f_\beta)_{1\leq\beta\leq M}\in L^1\Big({\mathbb{R}}^{n-1}\,,\,\frac{1}{1+|x'|^n}\,dx'\Big)^M.
\end{equation}
Then the function $u(x',t):=(P^L_t\ast f)(x')$ for each $(x',t)\in{\mathbb{R}}^n_{+}$, 
is meaningfully defined, via an absolutely convergent integral, satisfies
\begin{equation}\label{exist:u2****}
u\in\mathscr{C}^\infty(\mathbb{R}^n_{+},{\mathbb{C}}^M),\qquad
Lu=0\,\,\mbox{ in }\,\,\mathbb{R}^{n}_{+},
\end{equation}
and, at every Lebesgue point $x'_0\in{\mathbb{R}}^{n-1}$ of $f$,
\begin{equation}\label{exTGFVC.2s}
\Big(u\big|^{{}^{\rm n.t.}}_{\partial{\mathbb{R}}^{n}_{+}}\Big)(x'_0):=
\lim_{\substack{(x',\,t)\to(x'_0,0)\\ |x'-x'_0|<\kappa t}}(P^L_t\ast f)(x')=f(x'_0).
\end{equation}
\item[(8)] The function $P^L$ satisfies the semigroup property
\begin{equation}\label{eq:re4fd}
P^L_{t_1}\ast P^L_{t_2}=P^L_{t_1+t_2}\,\,\,\mbox{ for every }\,\,t_1,t_2>0.
\end{equation}
\end{list}
\end{theorem}

Concerning Theorem~\ref{kkjbhV}, we note that the existence part follows from
the classical work of S.\,Agmon, A.\,Douglis, and L.\,Nirenberg in \cite{ADNII}.
The uniqueness property has been recently proved in \cite{K-MMMM},
where \eqref{exTGFVC}, \eqref{exist:u2****}, \eqref{exTGFVC.2s}, 
as well as the semigroup property \eqref{eq:re4fd} have also been established.
\vskip 0.08in

Next, we record the following versatile version of interior estimates for
higher-order elliptic systems. A proof may be found in \cite[Theorem~11.9, p.\,364]{DM}.

\begin{theorem}\label{ker-sbav}
Assume the system $L$ is as in \eqref{L-def}-\eqref{L-ell.X}.
Then for each null-solution $u$ of $L$ in a ball $B(x,R)$
{\rm (}where $x\in{\mathbb{R}}^n$ and $R>0${\rm )}, $p\in(0,\infty)$, $\lambda\in(0,1)$,
$\ell\in{\mathbb{N}}_0$, and $r\in(0,R)$, one has
\begin{equation}\label{detraz}
\sup_{z\in B(x,\lambda r)}|\nabla^\ell u(z)|
\leq\frac{C}{r^\ell}\left(\aver{B(x,r)}|u|^p\,d{\mathscr{L}}^n\right)^{\frac{1}{p}},
\end{equation}
where $C=C(L,p,\ell,\lambda,n)>0$ is a finite constant.
\end{theorem}

To proceed we need to introduce some additional terminology. Let
\begin{align}\label{w12bd}
W^{1,2}_{\rm bd}({\mathbb{R}}^n_{+}):=\big\{ &
w\text{ Lebesgue measurable in }{\mathbb{R}}^n_{+}:\,
\nonumber\\[4pt]
&\qquad\quad
w,\,\nabla w\in L^2(\mathbb{R}^n_+\cap B(x,r)),\,\,\,
\forall\,x\in\mathbb{R}_+^{n},\,\,\forall\,r\in(0,\infty)
\big\}
\end{align}
In the sequel, the space of ${\mathbb{C}}^M$-valued functions with components in 
$W^{1,2}_{\rm bd}({\mathbb{R}}^n_{+})$ will be denoted 
by $W^{1,2}_{\rm bd}({\mathbb{R}}^n_{+},{\mathbb{C}}^M)$.
Also, (whenever meaningful) the Sobolev trace ${\rm Tr}$ 
is defined as:
\begin{equation}\label{Veri-S2TG.3}
\big({\rm Tr}\,w\big)(x'):=
\lim\limits_{r\to 0^{+}}\aver{B((x',0),r)\cap{\mathbb{R}}^n_{+}}
w\,d{\mathscr{L}}^n,\qquad x'\in\mathbb{R}^{n-1}.
\end{equation}

The following result can be found in \cite[Corollary~2.4]{MaMiSh}, and it is a
consequence of the {\it a priori} regularity estimates obtained in \cite{ADNII}
and Sobolev embeddings.

\begin{proposition}\label{c1.2}
Let $L$ be an $M\times M$ elliptic system as in \eqref{L-def}-\eqref{L-ell.X} and consider 
a vector-valued function $w\in W^{1,2}_{\rm bd}({\mathbb{R}}^n_{+},{\mathbb{C}}^M)$ such that
\begin{equation}\label{eJB-iY}
\left\{
\begin{array}{ll}
Lw=0 & \mbox{ in }\,\,\mathbb{R}^n_+,
\\[6pt]
{\rm Tr}\,w=0 & \mbox{ ${\mathscr{L}}^{n-1}$-a.e. on }\,\,\mathbb{R}^{n-1}.
\end{array}
\right.
\end{equation}

Then $w\in{\mathscr{C}}^\infty({\mathbb{R}}^n_{+},{\mathbb{C}}^M)$, and
for each $z\in\overline{\mathbb{R}^n_+}$ and $\rho>0$ one has
\begin{equation}\label{eq1.18}
\sup_{\mathbb{R}^n_+\cap B(z,\rho)}|\nabla w|
\leq C\,\rho^{-1}\sup_{\mathbb{R}^n_+\cap B(z,2\rho)}|w|
\end{equation}
where $C\in(0,\infty)$ is a constant independent of the scale $\rho$, the point $z$, 
and the function $w$.
\end{proposition}

We will also need an $L^p$-Fatou type result obtained in \cite[Corollary 6.3]{K-MMMM}. 
To state it, the reader is invited to recall the nontangential maximal operator 
from \eqref{NT-Fct}.

\begin{corollary}\label{tuFatou.Lp}
Assume $L$ is an elliptic $M\times M$ system as in \eqref{L-def}-\eqref{L-ell.X}. 
Then for each $p\in[1,\infty)$,
\begin{eqnarray}\label{Tafva.Lp}
\left.
\begin{array}{r}
u\in{\mathscr{C}}^\infty({\mathbb{R}}^n_{+},{\mathbb{C}}^M)
\\[4pt]
Lu=0\,\mbox{ in }\,{\mathbb{R}}^n_{+}
\\[6pt]
{\mathcal{N}}u\in L^p({\mathbb{R}}^{n-1})
\end{array}
\right\}
\Longrightarrow
\left\{
\begin{array}{l}
u\big|^{{}^{\rm n.t.}}_{\partial{\mathbb{R}}^n_{+}}\,\mbox{ exists a.e.~in }\,
{\mathbb{R}}^{n-1},\,\mbox{ belongs to $L^p({\mathbb{R}}^{n-1},{\mathbb{C}}^M)$,}
\\[12pt]
\mbox{and }\,u(x',t)=\Big(P^L_t\ast\big(u\big|^{{}^{\rm n.t.}}_{\partial{\mathbb{R}}^n_{+}}\big)\Big)(x'),
\quad\forall\,(x',t)\in{\mathbb{R}}^n_{+},
\end{array}
\right.
\end{eqnarray}
where $P^L$ is the Poisson kernel for $L$ in ${\mathbb{R}}^n_{+}$ 
from Theorem~\ref{kkjbhV}.
\end{corollary}

Our last auxiliary result, of a purely real-variable nature, 
can be found in \cite[Lemma~3.3]{K-MMMM}.

\begin{lemma}\label{lennii}
Fix $M\in{\mathbb{N}}$ and let $P=\big(P_{\alpha\beta}\big)_{1\leq\alpha,\beta\leq M}:
\mathbb{R}^{n-1}\to\mathbb{C}^{M\times M}$ be a Lebesgue measurable
function satisfying, for some $c\in(0,\infty)$,
\begin{equation}\label{vznv.ADF}
|P(x')|\leq\frac{c}{(1+|x'|^2)^{\frac{n}2}}
\,\,\,\mbox{ for each }\,\,\,x'\in\mathbb{R}^{n-1}.
\end{equation}
Recall that $P_t(x')=t^{1-n}P(x'/t)$ for each $x'\in\mathbb{R}^{n-1}$ and $t\in(0,\infty)$.

Then, for each $t\in(0,\infty)$ fixed, the operator
\begin{equation}\label{eq:Eda}
L^1\Big({\mathbb{R}}^{n-1}\,,\,\frac{1}{1+|x'|^n}\,dx'\Big)^M\ni f\mapsto P_t\ast f\in
L^1\Big({\mathbb{R}}^{n-1}\,,\,\frac{1}{1+|x'|^n}\,dx'\Big)^M
\end{equation}
is well-defined, linear and bounded, with operator norm controlled by $C(t+1)$. Moreover,
for every $\kappa>0$ there exists a finite constant $C_\kappa>0$ with the property that
for each $x'\in\mathbb{R}^{n-1}$,
\begin{equation}\label{exTGFVC:app}
\sup_{|x'-y'|<\kappa t}\big|(P_t\ast f)(y')\big|\leq C_\kappa\,\mathcal{M} f(x'),
\qquad\forall\,f\in L^1\Big({\mathbb{R}}^{n-1}\,,\,\frac{1}{1+|x'|^n}\,dx'\Big)^M.
\end{equation}
Finally, given any function
\begin{equation}\label{eq:aaAa:app}
f=(f_\beta)_{1\leq\beta\leq M}\in L^1\Big({\mathbb{R}}^{n-1}\,,\,\frac{1}{1+|x'|^n}dx'\Big)^M
\subset L^1_{\rm loc}({\mathbb{R}}^{n-1},{\mathbb{C}}^M),
\end{equation}
at every Lebesgue point $x'_0\in{\mathbb{R}}^{n-1}$ of $f$ there holds
\begin{equation}\label{exTGFVC.2s:app}
\lim_{\substack{(x',\,t)\to(x'_0,0)\\ |x'-x'_0|<\kappa t}}(P_t\ast f)(x')
=\left(\int_{\mathbb{R}^{n-1}}P(x')\,dx'\right)f(x'_0),
\end{equation}
and the function
\begin{equation}\label{exTefef}
{\mathbb{R}}^n_{+}\ni(x',t)\mapsto(P_t\ast f)(x')\in{\mathbb{C}}^M
\,\,\,\mbox{ is locally integrable in }\,\,{\mathbb{R}}^n_{+}.
\end{equation}
\end{lemma}

\section{Proof of the existence statements in Theorem~\ref{them:BMO-Dir}}
\setcounter{equation}{0}
\label{section:exist-BMO}

This section is devoted to proving Proposition~\ref{prop-Dir-BMO:exis}, 
dealing with with the issue of existence for the $\mathrm{BMO}$-Dirichlet 
boundary value problem \eqref{Dir-BVP-BMO}, the upper estimate in \eqref{Dir-BVP-BMO-Car}, 
and the issue of existence for the $\mathrm{VMO}$-Dirichlet boundary value 
problem \eqref{Dir-BVP-VMO}.

In this regard, we find it useful to adopt a more general point of view, by going beyond the 
class ${\rm BMO}$ through the consideration of convolutions of the Poisson kernel with functions 
$f$ from the weighted Lebesgue space $L^1\Big({\mathbb{R}}^{n-1}\,,\,\frac{dx'}{1+|x'|^n}\Big)^M$
(recall the inclusion in \eqref{eq:aaAabgr-22}). The aforementioned convolutions are then shown 
to satisfy a variety of Carleson measure-like conditions, which only require 
(recall \eqref{osc-1})
\begin{equation}\label{hsrwWW-FF45-uuu222}
\int_{1}^{\infty}{\rm osc}_1\big(f;\lambda\big)\,
\frac{d\lambda}{\lambda^{2}}<\infty.
\end{equation}
Note that this permits the oscillations ${\rm osc}_1\big(f;\lambda\big)$ of the given function $f$ 
to grow with the scale $\lambda$. In particular, this allows us to simultaneously treat
several scales of spaces of interest, including H\"older spaces 
$\dot{\mathscr{C}}^\eta({\mathbb{R}}^{n-1},{\mathbb{C}}^M)$ with $\eta\in(0,1)$, 
the Morrey-Campanato space ${\mathscr{E}}^{\eta,p}({\mathbb{R}}^{n-1},\mathbb{C}^M)$ 
with $\eta\in(0,1)$ and $p\in[1,\infty)$, as well as the John-Nirenberg space 
${\rm BMO}({\mathbb{R}}^{n-1},{\mathbb{C}}^M)$. 

An example of a function $f\in\dot{\mathscr{C}}^\eta({\mathbb{R}}^{n-1},{\mathbb{C}}^M)$ 
with $\eta\in(0,1)$ which does not belong to 
${\rm BMO}({\mathbb{R}}^{n-1},{\mathbb{C}}^M)$ is offered by 
\begin{equation}\label{u76ggv}
f(x'):=|x'|^\eta,\qquad\forall\,x'\in{\mathbb{R}}^{n-1}.
\end{equation}
Indeed, $\|f\|_{\dot{\mathscr{C}}^\eta({\mathbb{R}}^{n-1},{\mathbb{C}}^M)}\leq 1$ and
since $\delta_\lambda f=\lambda^\eta f$, it follows from the last line in \eqref{defi-BMO-CCC}
that necessarily $\|f\|_{{\rm BMO}({\mathbb{R}}^{n-1},{\mathbb{C}}^M)}=\infty$. 
Incidentally, for $f$ as in \eqref{u76ggv}, we have 
${\rm osc}_1\big(f;\lambda\big)=O(\lambda^\eta)$ as $\lambda\to\infty$, 
hence \eqref{hsrwWW-FF45-uuu222} holds in this case.

Here is the formal statement of the result just advertised above. 

\begin{proposition}\label{prop-Dir-BMO:exis}
Let $L$ be an $M\times M$ elliptic system with constant complex coefficients as in
\eqref{L-def}-\eqref{L-ell.X} and let $P^L$ be the Poisson kernel for $L$ in 
$\mathbb{R}^{n}_+$ from Theorem~\ref{kkjbhV}. Select
$f\in L^1\Big({\mathbb{R}}^{n-1}\,,\,\frac{dx'}{1+|x'|^n}\Big)^M$ and set
\begin{equation}\label{eqn-Dir-BMO:u:prop}
u(x',t):=(P_t^L*f)(x'),\qquad\forall\,(x',t)\in{\mathbb{R}}^n_{+}.
\end{equation}

Then $u$ is meaningfully defined via an absolutely convergent integral and satisfies:
\begin{equation}\label{exist:u2}
u\in\mathscr{C}^\infty(\mathbb{R}^n_{+},{\mathbb{C}}^M),\quad
Lu=0\,\,\mbox{ in }\,\,\mathbb{R}^{n}_{+},\,\text{ and }\,
u\big|_{\partial\mathbb{R}^{n}_{+}}^{{}^{\rm n.t.}}
=f\,\,\mbox{ a.e.~in }\,\,\mathbb{R}^{n-1}.
\end{equation}
In addition, $u$ enjoys the following properties:
\begin{list}{(\theenumi)}{\usecounter{enumi}\leftmargin=.8cm
\labelwidth=.8cm\itemsep=0.2cm\topsep=.1cm
\renewcommand{\theenumi}{\alph{enumi}}}
\item For each integer $\ell\geq 1$ there exists a constant $C\in(0,\infty)$ with the property that 
the following pointwise estimate holds for every $(x',t)\in{\mathbb{R}}^n_{+}$:
\begin{align}\label{eqn:B-knB-yTFV}
\big|(\nabla^\ell u)(x',t)\big|\leq\frac{C}{t^\ell}
\int_{1}^\infty{\rm osc}_1\big(f;\lambda\,t\big)\,\frac{d\lambda}{\lambda^{1+\ell}}.
\end{align}
In particular, there exists $C\in(0,\infty)$ such that
\begin{align}\label{eqn:B-knB}
\big|(\nabla u)(x',t)\big|\leq\frac{C}{t}
\int_{1}^\infty{\rm osc}_1\big(f;\lambda\,t\big)\,\frac{d\lambda}{\lambda^{2}},
\qquad\forall\,(x',t)\in{\mathbb{R}}^n_{+}.
\end{align}
\item There exists a constant $C\in(0,\infty)$ such that for every cube $Q$ in ${\mathbb{R}}^{n-1}$ 
the following ``cube-by-cube'' Carleson measure estimates hold:
\begin{align}\label{hsrwWW-AA-Vdewe-uuu}
\left(\int_{0}^{\ell(Q)}\aver{Q}|(\nabla u)(x',t)|^2\,t\,dx'dt\right)^{\frac12}
&\leq C\int_{1}^\infty\Big(\aver{\lambda Q}|f(y')-f_{\lambda Q}|\,dy'\Big)
\frac{d\lambda}{\lambda^{2}}
\nonumber\\[6pt]
& \quad 
+C\,\sup_{Q'\subseteq 4Q}\aver{Q'}|f(y')-f_{Q'}|\,dy'
\end{align}
and
\begin{align}\label{hsr-u54367yhg-uuu}
\left(\int_{0}^{\ell(Q)}\aver{Q}|(\nabla u)(x',t)|^2\,t\,dx'dt\right)^{\frac12}
\leq C\int_{1}^\infty{\rm osc}_1\big(f;\lambda\ell(Q)\big)
\frac{d\lambda}{\lambda^{2}}.
\end{align}
\item There exists $C\in(0,\infty)$ such that 
the following local Carleson measure estimate holds for every scale $r\in(0,\infty)$:
\begin{equation}\label{hsrwWW-AA-uuu}
\sup_{Q\subset\mathbb{R}^{n-1},\,\ell(Q)\leq r}
\left(\int_{0}^{\ell(Q)}\aver{Q}|(\nabla u)(x',t)|^2\,t\,dx'dt\right)^{\frac12}
\leq C\int_{1}^{\infty}{\rm osc}_1\big(f;r\lambda\big)\,
\frac{d\lambda}{\lambda^{2}}.
\end{equation}
\item Whenever $f$ satisfies
\begin{equation}\label{hsrwWW-FF45-uuu}
\int_{1}^{\infty}{\rm osc}_1\big(f;\lambda\big)\,
\frac{d\lambda}{\lambda^{2}}<\infty,
\end{equation}
the global weighted Carleson measure estimate 
\begin{equation}\label{hsrwWW-FF46-uuu}
\sup_{Q\subset\mathbb{R}^{n-1}}\left\{
\left(\int_{1}^{\infty}{\rm osc}_1\big(f;\lambda\ell(Q)\big)\,
\frac{d\lambda}{\lambda^{2}}\right)^{-1}
\left(\int_{0}^{\ell(Q)}\aver{Q}|(\nabla u)(x',t)|^2\,t\,dx'dt\right)^{\frac12}\right\}
\leq C
\end{equation}
holds for some $C\in(0,\infty)$ independent of $f$. 
\item There exists a constant $C\in(0,\infty)$ such that the following 
global Carleson measure estimate holds:
\begin{equation}\label{exist:u2-carleso}
\|u\|_{**}=\sup_{Q\subset\mathbb{R}^{n-1}}
\left(\int_{0}^{\ell(Q)}\aver{Q}|(\nabla u)(x',t)|^2\,t\,dx'dt\right)^{\frac12}
\leq C\|f\|_{\mathrm{BMO}(\mathbb{R}^{n-1},{\mathbb{C}}^M)}.
\end{equation}
In particular, thanks to \eqref{eq:aaAabgr-22}, estimate \eqref{exist:u2-carleso}
holds for every $f\in\mathrm{BMO}(\mathbb{R}^{n-1},\mathbb{C}^M)$.
\item Whenever $f$ satisfies
\begin{equation}\label{hsrwWW-FF45.bb-uuu}
\int_{1}^{\infty}{\rm osc}_1(f;\lambda)\,
\frac{d\lambda}{\lambda^{2}}<\infty\,\,\text{ and }\,\,
\lim_{r\to 0^{+}}{\rm osc}_1(f;r)=0,
\end{equation}
the following vanishing Carleson measure condition holds:
\begin{equation}\label{hsrwWW-VMO-uuu}
\lim_{r\to 0^{+}}\left\{\sup_{Q\subset\mathbb{R}^{n-1},\,\ell(Q)\leq r}
\left(\int_{0}^{\ell(Q)}\aver{Q}|(\nabla u)(x',t)|^2\,t\,dx'dt\right)^{\frac12}
\right\}=0.
\end{equation}
In particular, in the case when actually $f\in\mathrm{VMO}(\mathbb{R}^{n-1},\mathbb{C}^M)$ 
to begin with, $u$ has the additional property that 
\begin{equation}\label{Dkjnbvb}
\big|\nabla u(x',t)\big|^2\,t\,dx'dt\,\,
\mbox{ is a vanishing Carleson measure in }\mathbb{R}^{n}_{+}.
\end{equation}
\end{list}
\end{proposition}

Ultimately, the proof of Proposition~\ref{prop-Dir-BMO:exis} relies on square-function 
estimates. For now, assuming a suitable $L^2$ bound (implicit in 
\eqref{y6bTVV} below) we may establish some versatile Carleson measure 
estimates (of local and global nature), as well as vanishing Carleson measure 
properties for integral operators (modeled upon the gradient of the convolution 
with the Poisson kernel) acting on function spaces larger than the standard {\rm BMO}. 
This is made precise in the following proposition.

\begin{proposition}\label{prop:SFE}
Let $\theta:\mathbb{R}^{n}_+\times\mathbb{R}^{n-1}\longrightarrow{\mathbb{C}}^{M\times M}$ 
be a matrix-valued Lebesgue measurable function, with the property that there exist 
$\varepsilon\in(0,\infty)$ and $C\in(0,\infty)$ such that
\begin{equation}\label{SFE-est-theta}
|\theta(x',t;y')|\leq\frac{C\,t^\varepsilon}{|(x'-y',t)|^{n-1+\varepsilon}},\quad
\forall\,(x',t)\in\mathbb{R}^{n}_{+},\quad\forall\,y'\in\mathbb{R}^{n-1},
\end{equation}
and the following cancellation condition holds:
\begin{equation}\label{SFE-vanish-theta}
\int_{\mathbb{R}^{n-1}}\theta(x',t;y')\,dy'=0\in{\mathbb{C}}^{M\times M},
\qquad\forall\,(x',t)\in\mathbb{R}^{n}_{+}.
\end{equation}
In relation to the kernel $\theta$, one may then consider the integral operator 
$\Theta$ acting on arbitrary functions 
$f\in L^1\Big({\mathbb{R}}^{n-1}\,,\,\frac{dx'}{1+|x'|^{n-1+\varepsilon}}\Big)^M$ 
according to {\rm (}the absolutely convergent integral{\rm )}
\begin{equation}\label{defi:Theta}
(\Theta f)(x',t):=\int_{\mathbb{R}^{n-1}}\theta(x',t;y')\,f(y')\,dy', 
\qquad\forall\,(x',t)\in\mathbb{R}^{n}_+.
\end{equation}

Then, under the assumption that the operator 
\begin{equation}\label{y6bTVV}
\Theta:L^2({\mathbb{R}}^{n-1},{\mathbb{C}}^M)\longrightarrow L^2\Big({\mathbb{R}}^n_{+}\,,\,
\frac{dx'dt}{t}\Big)^M\,\,\text{ is bounded},
\end{equation}
the following properties hold:
\begin{list}{(\theenumi)}{\usecounter{enumi}\leftmargin=.8cm
\labelwidth=.8cm\itemsep=0.2cm\topsep=.1cm
\renewcommand{\theenumi}{\alph{enumi}}}
\item There exists a constant $C\in(0,\infty)$ such that for every
$f\in L^1\Big({\mathbb{R}}^{n-1}\,,\,\frac{dx'}{1+|x'|^{n-1+\varepsilon}}\Big)^M$ 
and every cube $Q$ in ${\mathbb{R}}^{n-1}$ 
the following ``cube-by-cube'' Carleson measure estimates hold:
\begin{align}\label{hsrwWW-AA-Vdewe}
\left(\int_{0}^{\ell(Q)}\aver{Q}|(\Theta f)(x',t)|^2\,\frac{dx'dt}{t}\right)^{\frac12}
&\leq C\int_{1}^\infty\Big(\aver{\lambda Q}|f(y')-f_{\lambda Q}|\,dy'\Big)
\frac{d\lambda}{\lambda^{1+\varepsilon}}
\nonumber\\[6pt]
& \quad 
+C\,\sup_{Q'\subseteq 4Q}\aver{Q'}|f(y')-f_{Q'}|\,dy'
\end{align}
and
\begin{align}\label{hsr-u54367yhg}
\left(\int_{0}^{\ell(Q)}\aver{Q}|(\Theta f)(x',t)|^2\,\frac{dx'dt}{t}\right)^{\frac12}
\leq C\int_{1}^\infty{\rm osc}_1\big(f;\lambda\ell(Q)\big)
\frac{d\lambda}{\lambda^{1+\varepsilon}}.
\end{align}
\item There exists $C\in(0,\infty)$ such that for every function
$f\in L^1\Big({\mathbb{R}}^{n-1}\,,\,\frac{dx'}{1+|x'|^{n-1+\varepsilon}}\Big)^M$ 
the following local Carleson measure estimate holds for every scale $r\in(0,\infty)$:
\begin{equation}\label{hsrwWW-AA}
\sup_{Q\subset\mathbb{R}^{n-1},\,\ell(Q)\leq r}
\left(\int_{0}^{\ell(Q)}\aver{Q}|(\Theta f)(x',t)|^2\,\frac{dx'dt}{t}\right)^{\frac12}
\leq C\int_{1}^{\infty}{\rm osc}_1\big(f;r\lambda\big)\,
\frac{d\lambda}{\lambda^{1+\varepsilon}}.
\end{equation}
\item There exists $C\in(0,\infty)$ such that for any given function
$f\in L^1_{\rm loc}({\mathbb{R}}^{n-1},{\mathbb{C}}^M)$ with the property that
\begin{equation}\label{hsrwWW-FF45}
\int_{1}^{\infty}{\rm osc}_1\big(f;\lambda\big)\,
\frac{d\lambda}{\lambda^{1+\varepsilon}}<\infty
\end{equation}
{\rm (}which necessarily places $f$ into the space 
$L^1\Big({\mathbb{R}}^{n-1}\,,\,\frac{dx'}{1+|x'|^{n-1+\varepsilon}}\Big)^M$
by \eqref{eq:aaAabgr-33}{\rm )} the following global weighted Carleson measure 
estimate holds:
\begin{equation}\label{hsrwWW-FF46}
\sup_{Q\subset\mathbb{R}^{n-1}}\left\{
\left(\int_{1}^{\infty}{\rm osc}_1\big(f;\lambda\ell(Q)\big)\,
\frac{d\lambda}{\lambda^{1+\varepsilon}}\right)^{-1}
\left(\int_{0}^{\ell(Q)}\aver{Q}|(\Theta f)(x',t)|^2\,\frac{dx'dt}{t}\right)^{\frac12}\right\}
\leq C.
\end{equation}
\item There exists a constant $C\in(0,\infty)$ such that for every
$f\in\mathrm{BMO}(\mathbb{R}^{n-1},{\mathbb{C}}^M)$ the following 
global Carleson measure estimate holds:
\begin{equation}\label{hsrwWW}
\sup_{Q\subset\mathbb{R}^{n-1}}
\left(\int_{0}^{\ell(Q)}\aver{Q}|(\Theta f)(x',t)|^2\,\frac{dx'dt}{t}\right)^{\frac12}
\leq C\|f\|_{\mathrm{BMO}(\mathbb{R}^{n-1},{\mathbb{C}}^M)}.
\end{equation}
\item Whenever $f\in L^1_{\rm loc}({\mathbb{R}}^{n-1},{\mathbb{C}}^M)$ 
is such that
\begin{equation}\label{hsrwWW-FF45.bb}
\int_{1}^{\infty}{\rm osc}_1(f;\lambda)\,
\frac{d\lambda}{\lambda^{1+\varepsilon}}<\infty\,\,\text{ and }\,\,
\lim_{r\to 0^{+}}{\rm osc}_1(f;r)=0,
\end{equation}
then $f\in L^1\Big({\mathbb{R}}^{n-1}\,,\,\frac{dx'}{1+|x'|^{n-1+\varepsilon}}\Big)^M$
and the following vanishing Carleson measure condition holds:
\begin{equation}\label{hsrwWW-VMO}
\lim_{r\to 0^{+}}\left\{\sup_{Q\subset\mathbb{R}^{n-1},\,\ell(Q)\leq r}
\left(\int_{0}^{\ell(Q)}\aver{Q}|(\Theta f)(x',t)|^2\,\frac{dx'dt}{t}\right)^{\frac12}
\right\}=0.
\end{equation}
In particular, \eqref{hsrwWW-VMO} holds for every function 
$f\in\mathrm{VMO}(\mathbb{R}^{n-1},{\mathbb{C}}^M)$.
\end{list}
\end{proposition}

\begin{proof}
Start by fixing an arbitrary cube $Q$ in $\mathbb{R}^{n-1}$ 
and denote by $x'_Q$ its center. Given a function 
$f\in L^1\Big({\mathbb{R}}^{n-1}\,,\,\frac{dx'}{1+|x'|^{n-1+\varepsilon}}\Big)^M$, 
use \eqref{SFE-vanish-theta} in order to write
\begin{align}\label{est-u**}
\left(\int_{0}^{\ell(Q)}\aver{Q}|(\Theta f)(x',t)|^2\,\frac{dx'dt}{t}\right)^\frac12
&=\left(\int_{0}^{\ell(Q)}\aver{Q}|(\Theta(f-f_Q))(x',t)|^2\,\frac{dx'dt}{t}
\right)^\frac12
\nonumber\\[4pt]
&\leq I+II,
\end{align}
where
\begin{equation}\label{Twazv-5-BB}
I:=\left(\int_{0}^{\ell(Q)}\aver{Q}
\big|\Theta\big((f-f_Q)\,{\mathbf 1}_{4Q}\big)(x',t)\big|^2\,\frac{dx'dt}{t}\right)^\frac12
\end{equation}
and 
\begin{equation}\label{Twazv-5-KK}
II:=\left(\int_{0}^{\ell(Q)}\aver{Q} 
|\Theta\big((f-f_Q)\,{\mathbf 1}_{{\mathbb{R}}^{n-1}\setminus 4Q}\big)(x',t)|^2\,
\frac{dx'dt}{t}\right)^\frac12.
\end{equation}
To estimate $I$, invoke \eqref{y6bTVV}, \eqref{aver-fq-BBBB} with $p=2$, and
\eqref{JY6GTFF.2} to estimate 
\begin{align}\label{estbxchd}
I & \leq\frac1{|Q|^{\frac12}}
\left(\int_{\mathbb{R}_+^n}|\Theta\big((f-f_Q)\,{\mathbf 1}_{4Q}\big)(x',t)|^2\,\frac{dx'dt}{t}
\right)^\frac12
\nonumber\\[4pt]
& \leq C\,\left(\aver{4Q}|f(y')-f_{Q}|^2\,dy'\right)^{\frac12}
\leq C\,\left(\aver{4Q}|f(y')-f_{4Q}|^2\,dy'\right)^{\frac12}
\nonumber\\[4pt]
&\leq C\sup_{Q'\subseteq 4Q}\aver{Q'}|f(y')-f_{Q'}|\,dy',
\end{align}
where $C\in(0,\infty)$ is independent of $f$ and $Q$. 
To proceed, observe that there exists a purely dimensional constant $c_n\in(0,\infty)$ 
(e.g., the choice $c_n:=3/(6+2\sqrt{n-1}\,)$ will do) with the property that
\begin{equation}\label{utrf-532fv}
|x'-y'|\geq c_n\big(\ell(Q)+|x'_Q-y'|\big)
\,\,\text{ for each }\,x'\in Q,\,\,y'\in{\mathbb{R}}^{n-1}\setminus 4Q. 
\end{equation}
Based on this, \eqref{defi:Theta}, and \eqref{SFE-est-theta}, we may then estimate 
\begin{equation}\label{Twazv-5-KK.44}
\begin{array}{c}
\displaystyle
|\Theta\big((f-f_Q)\,{\mathbf 1}_{{\mathbb{R}}^{n-1}\setminus 4Q}\big)(x',t)|
\leq Ct^\varepsilon\int_{{\mathbb{R}}^{n-1}}
\frac{|f(y')-f_Q|}{\big[\ell(Q)+|x'_Q-y'|\big]^{n-1+\varepsilon}}\,dy',
\\[18pt]
\text{for every point $x'\in Q$ and every number $t>0$},
\end{array}
\end{equation}
for some $C\in(0,\infty)$ depending only on $n$ and the constant appearing in 
\eqref{SFE-est-theta}. In turn, from \eqref{Twazv-5-KK.44} and \eqref{eqn:BMO-decay.88} 
we conclude that
\begin{align}\label{eqn:BhFva}
II &\leq C\left(\int_{0}^{\ell(Q)}\Big(\frac{t}{\ell(Q)}\Big)^{2\varepsilon}
\,\frac{dt}{t}\right)^{\frac12}
\cdot\int_{1}^\infty\Big(\aver{\lambda Q}|f(y')-f_{\lambda Q}|\,dy'\Big)
\frac{d\lambda}{\lambda^{1+\varepsilon}}
\nonumber\\[6pt]
&=C\int_{1}^\infty\Big(\aver{\lambda Q}|f(y')-f_{\lambda Q}|\,dy'\Big)
\frac{d\lambda}{\lambda^{1+\varepsilon}}.
\end{align}
At this stage, \eqref{est-u**}, \eqref{estbxchd}, and \eqref{eqn:BhFva} 
combine to give \eqref{hsrwWW-AA-Vdewe}. In turn, \eqref{hsr-u54367yhg} 
readily follows from \eqref{hsrwWW-AA-Vdewe} and part {\it (a)} in 
Lemma~\ref{jsfsQAXT} which allows us to estimate 
\begin{align}\label{li8hbb}
{\rm osc}_1\big(f;4\ell(Q)\big) &\leq\varepsilon\big(4^{-\varepsilon}-5^{-\varepsilon}\big)^{-1}
\int_{4}^5{\rm osc}_1\big(f;\lambda\ell(Q)\big)\frac{d\lambda}{\lambda^{1+\varepsilon}}
\nonumber\\[4pt]
&\leq C_\varepsilon
\int_{1}^\infty{\rm osc}_1\big(f;\lambda\ell(Q)\big)\frac{d\lambda}{\lambda^{1+\varepsilon}}.
\end{align}
In concert with part {\it (a)} in Lemma~\ref{jsfsQAXT}, estimate \eqref{hsr-u54367yhg}
immediately gives \eqref{hsrwWW-AA}. Estimate \eqref{hsr-u54367yhg} also implies
the global weighted Carleson measure estimate formulated in \eqref{hsrwWW-FF46}.
From \eqref{hsrwWW-AA} and part {\it (c)} in Lemma~\ref{jsfsQAXT}, 
the global Carleson measure estimate stated in \eqref{hsrwWW} follows. 

Going further, assume the function $f\in L^1_{\rm loc}({\mathbb{R}}^{n-1},{\mathbb{C}}^M)$ 
satisfies the properties listed in \eqref{hsrwWW-FF45.bb}. Then
$f\in L^1\Big({\mathbb{R}}^{n-1}\,,\,\frac{dx'}{1+|x'|^{n-1+\varepsilon}}\Big)^M$
by \eqref{jGVVCc-1jmn}. Also, thanks to \eqref{hsrwWW-FF45.bb} and 
part {\it (a)} in Lemma~\ref{jsfsQAXT}, Lebesgue's Dominated Convergence Theorem 
applies and yields
\begin{equation}\label{hsrwWW-FF45.ss}
\lim_{r\to 0^{+}}\int_{1}^{\infty}{\rm osc}_1(f;r\lambda)\,
\frac{d\lambda}{\lambda^{1+\varepsilon}}=0.
\end{equation}
Together with \eqref{hsrwWW-AA}, this ultimately proves the 
vanishing Carleson measure condition stated in \eqref{hsrwWW-VMO}.
Finally, that any function $f\in\mathrm{VMO}(\mathbb{R}^{n-1},{\mathbb{C}}^M)$
actually satisfies the properties listed in \eqref{hsrwWW-FF45.bb} is clear 
from \eqref{defi-VMO}, \eqref{osc-1}, and part {\it (c)} in Lemma~\ref{jsfsQAXT}.
This completes the proof of Proposition~\ref{prop:SFE}.
\end{proof}

Next the goal is to identify a class of integral kernels $\theta$ 
satisfying \eqref{SFE-est-theta}-\eqref{SFE-vanish-theta}
with the property that the operator $\Theta$ associated with $\theta$ 
as in \eqref{defi:Theta} enjoys the $L^2$-boundedness condition formulated 
in \eqref{y6bTVV}. We adopt a broader point of view by considering a larger 
variety of spaces, which turns out to be useful later. 
To set the stage, let us recall the definition of the Hardy space 
$H^1(\mathbb{R}^{n-1})$ using $(1,\infty)$-atoms. 
Specifically, a Lebesgue measurable function $a:\mathbb{R}^{n-1}\rightarrow\mathbb{C}$ 
is said to be a $(1,\infty)$-atom provided there exists a cube 
$Q\subset\mathbb{R}^{n-1}$ such that the following localization, normalization, and
cancellation properties hold:
\begin{equation}\label{defi-atom}
\mathrm{supp}\,a\subseteq Q,\qquad\|a\|_{L^\infty(\mathbb{R}^{n-1})}
\leq |Q|^{-1},\qquad\int_{\mathbb{R}^{n-1}}a(y')\,dy'=0.
\end{equation}
The space $H^1(\mathbb{R}^{n-1})$ is then defined as the collection of all Lebesgue measurable 
functions $f$ defined in ${\mathbb{R}}^{n-1}$ such that
\begin{equation}\label{eq:f-H1-atoms}
f=\sum_{j=1}^\infty\lambda_j\,a_j\,\,\mbox{ a.e. in }\,\mathbb{R}^{n-1},
\end{equation}
with the $a_j$'s being $(1,\infty)$-atoms, and where the sequence 
$\{\lambda_j\}_{j\in{\mathbb{N}}}\subset\mathbb{C}$ satisfies $\sum\limits_{j=1}^\infty|\lambda_j|<\infty$. 
The norm in $H^1(\mathbb{R}^{n-1})$ is defined as
\begin{equation}\label{eq:H1-norm}
\|f\|_{H^1(\mathbb{R}^{n-1})}:=\inf\sum_{j=1}^\infty|\lambda_j|,
\end{equation}
where the infimum runs over all the atomic decompositions of $f$ as in \eqref{eq:f-H1-atoms}. 
In particular, the series in \eqref{eq:f-H1-atoms} converges in $H^1(\mathbb{R}^{n-1})$.
Let us also write $H^1(\mathbb{R}^{n-1},\mathbb{C}^M)$ for the collection of 
all $\mathbb{C}^M$-valued functions $f=(f_\alpha)_{1\leq\alpha\leq M}$ with components 
in $H^1(\mathbb{R}^{n-1})$. In such a scenario, we set
\begin{equation}\label{eq:H1-norm-vv}
\|f\|_{H^1(\mathbb{R}^{n-1},\mathbb{C}^M)}
:=\sum_{\alpha=1}^M\|f_\alpha\|_{H^1(\mathbb{R}^{n-1})}.
\end{equation}

Here are the square-function estimates alluded to earlier. For more background
and relevant references the reader is referred to the recent exposition in \cite{HMMM}.

\begin{proposition}\label{prop:SFE-early}
Let $\theta$ and $\Theta$ be as in \eqref{SFE-est-theta}-\eqref{defi:Theta}
with $\varepsilon=1$ and $M=1$. In addition, assume $\theta$ is of class ${\mathscr{C}}^1$ 
in the variable $y'\in{\mathbb{R}}^{n-1}$ and suppose there exists some $C\in(0,\infty)$ 
with the property that
\begin{equation}\label{SFE-est-theta-bis}
|\nabla_{y'}\theta(x',t;y')|\leq\frac{C\,t}{|(x'-y',t)|^{n+1}},
\qquad\forall\,(x',t)\in\mathbb{R}^{n}_{+},\quad\forall\,y'\in\mathbb{R}^{n-1}.
\end{equation}
Fix a background parameter $\kappa>0$ and, with the nontangential cone 
$\Gamma_\kappa(x')$ as in \eqref{NT-1} for each $x'\in{\mathbb{R}}^{n-1}$, 
define the square function operator $S_\Theta$ by setting
\begin{equation}\label{bhxdhyswt}
(S_{\Theta}f)(x'):=\Big(\int_{\Gamma_\kappa(x')}
|(\Theta f)(y',t)|^2\,\frac{dy'\,dt}{t^n}\Big)^{\frac12},
\qquad\forall\,x'\in\mathbb{R}^{n-1}.
\end{equation}

Then the following are well-defined, linear, and bounded operators:
\begin{align}\label{hdgswf}
& \Theta:L^2({\mathbb{R}}^{n-1})\longrightarrow 
L^2\Big({\mathbb{R}}^n_{+}\,,\,\frac{dx'\,dt}{t}\Big),
\\[6pt]
& S_{\Theta}:L^p(\mathbb{R}^{n-1})\longrightarrow L^p(\mathbb{R}^{n-1}),
\quad\forall\,p\in(1,\infty),
\label{hdgswf.DD}
\\[6pt]
& S_{\Theta}:L^1(\mathbb{R}^{n-1})\longrightarrow L^{1,\infty}(\mathbb{R}^{n-1}),
\label{hdgswf.DD2}
\\[6pt]
& S_{\Theta}:H^1(\mathbb{R}^{n-1})\longrightarrow L^1(\mathbb{R}^{n-1}).
\label{hdgswf.DD3}
\end{align}
\end{proposition}

\begin{proof}
We are going to use \cite[Theorem~20, p.\,69]{Ch90} (see also \cite{ChJ}).
First observe that \eqref{SFE-est-theta} with $\varepsilon=1$ presently implies
\begin{equation}\label{jdgsw}
|\theta(x',t; y')|\leq C\,\frac{t}{(t+|x'-y'|)^{n}}
\qquad\forall\,x',y'\in{\mathbb{R}}^{n-1},\,\,\forall\,t>0.
\end{equation}
Second, if $x'$, $y'$, $z'\in{\mathbb{R}}^{n-1}$ and $t>0$ are such that 
$|y'-z'|\leq(t+|x'-y'|)/2$, the Mean Value Theorem and \eqref{SFE-est-theta-bis} imply
(here and elsewhere, $[a,b]$ denotes the line segment with end-points 
$a,b\in{\mathbb{R}}^{n-1}$)
\begin{align}\label{SFE-Lip}
|\theta(x',t; y')-\theta(x',t; z')|
& \leq |y'-z'|\,\sup_{w'\in[y',z']}|\nabla_{y'}\theta(x',t;w')|
\nonumber\\[4pt]
& \leq C|y'-z'|\,\sup_{w'\in[y',z']}\frac{t}{(t+|x'-w'|)^{n+1}}
\nonumber\\[4pt]
& \leq C\,\frac{|y'-z'|\,t}{(t+|x'-y'|)^{n+1}}.
\end{align}
This proves that the family of kernels $\{\theta(x',t;y')\}_{t\in (0,\infty)}$ 
is a standard family in $\mathbb{R}^{n-1}$ as in \cite[Definition~19,\,p.69]{Ch90}.
Third, \eqref{SFE-vanish-theta} implies that $\Theta 1(x',t)=0$ for every
$(x',t)\in{\mathbb{R}}^n_{+}$.  
We can therefore apply \cite[Theorem~20, p.\,69]{Ch90} to conclude that 
the operator in \eqref{hdgswf} is well-defined, linear and bounded. In particular, 
the boundedness of the operator in \eqref{hdgswf} implies that there exists a constant 
$C\in(0,\infty)$ such that for every function $f\in L^2(\mathbb{R}^{n-1})$ 
there holds
\begin{equation}\label{SFE-L2}
\|S_{\Theta}f\|_{L^2(\mathbb{R}^{n-1})}^2
=C_{n,\kappa}\int_{\mathbb{R}^n_+}|(\Theta f)(y',t)|^2\,\frac{dy'\,dt}{t}
\leq C\,\int_{\mathbb{R}^{n-1}}|f(x')|^2\,dx'.
\end{equation}
Proving the boundedness of the operator in \eqref{hdgswf.DD2} 
comes down to establishing the weak-type $(1,1)$ estimate for $S_\Theta$.
In a first stage, we claim that there exists some constant $C\in(0,\infty)$ with the
property that for any cube $Q$ in $\mathbb{R}^{n-1}$ and any function $h$ satisfying 
\begin{equation}\label{fdsnn}
h\in L^1(\mathbb{R}^{n-1}),\quad
\mathrm{supp}\,h\subseteq Q,\quad\mbox{and}\quad
\int_{\mathbb{R}^{n-1}}h(y')\,dy'=0
\end{equation} 
we have
\begin{equation}\label{est-atom-h}
|(S_\Theta h)(x')|\leq C\,\|h\|_{L^1(\mathbb{R}^{n-1})}\,\frac{\ell(Q)}{|x'-x'_Q|^n}
\,\,\,\,\text{ for every }\,x'\in{\mathbb{R}}^{n-1}\setminus 2Q,
\end{equation}
where $x'_Q$ is the center of the cube $Q$.
To justify the claim, given $x',z'\in\mathbb{R}^{n-1}$ with $x'\neq z'$ we first use 
\eqref{SFE-est-theta-bis} combined with natural changes of variables to write
\begin{align}\label{wswaufLL}
&\int_{\Gamma_\kappa(x')}|\nabla_{y'}\theta(y',t;z')|^2\,\frac{dy'\,dt}{t^n}
\nonumber\\[6pt]
&\qquad\qquad
\leq C\,\int_{|y'-x'|<\kappa t}\frac{t^2}{|(y'-z',t)|^{2\,(n+1)}}\,\frac{dy'\,dt}{t^n}
\nonumber\\[6pt]
&\qquad\qquad
=C\,\int_0^\infty\int_{|y'|<\kappa} 
\frac{t^{2}}{|(y't+(x'-z'),t)|^{2\,(n+1)}}\,\frac{dy'\,dt}{t}
\nonumber\\[6pt]
&\qquad\qquad
=C\,\int_0^\infty\int_{|y'|<\kappa} 
\frac{t^{-2n}}{|(y'+(x'-z')/t,1)|^{2\,(n+1)}}\,\frac{dy'\,dt}{t}
\nonumber\\[6pt]
&\qquad\qquad
\leq C\,|x'-z'|^{-2\,n}\,
\int_0^\infty\int_{|y'|<\kappa}
\frac{t^{-2n}}{\big|\big(y'+(x'-z')/(t|x'-z'|),1\big)\big|^{2\,(n+1)}}\,\frac{dy'\,dt}{t}
\nonumber\\[6pt]
&\qquad\qquad\leq
C\,|x'-z'|^{-2\,n}\,\sup_{|v'|=1}
\int_0^\infty\int_{|y'|<\kappa}\frac{t^{-2n}}{|y'+v'/t|^{2\,(n+1)}+1}\,\frac{dy'\,dt}{t}.
\end{align}
Next, fix $v'\in\mathbb{R}^{n-1}$ with $|v'|=1$. 
If $t<1/(2\kappa)$ and $|y'|<\kappa$ we have
\begin{equation}\label{Gew}
t^{-1}=|v'|/t\leq|y'+v'/t|+|y'|<|y'+v'/t|+1/(2\,t)
\end{equation}
and therefore $|y'+v'/t|>1/(2\,t)$. Thus,
\begin{align}\label{gsew}
\int_0^{1/(2\kappa)}\int_{|y'|<\kappa} 
\frac{t^{-2n}}{|y'+v'/t|^{2\,(n+1)}+1}\,\frac{dy'\,dt}{t}
\leq C\,\int_0^{1/(2\kappa)}\int_{|y'|<\kappa} t\,dy'\,dt\leq C.
\end{align}
Also, it is immediate that
\begin{equation}\label{yrwee}
\int_{1/(2\kappa)}^\infty\int_{|y'|<\kappa}
\frac{t^{-2n}}{|y'+v'/t|^{2\,(n+1)}+1}\,\frac{dy'\,dt}{t}
\leq C\,\int_{1/(2\kappa)}^\infty\int_{|y'|<\kappa}t^{-2\,n}\,\frac{dy'\,dt}{t}\leq C.
\end{equation}
Combining \eqref{wswaufLL}, \eqref{gsew}, and \eqref{yrwee}
we may conclude that
\begin{equation}\label{est-kernel-con}
\Big(\int_{\Gamma_\kappa(x')} |\nabla_{y'}\theta(y',t;z')|^2\,\frac{dy'\,dt}{t^n}\Big)^{\frac12}
\leq\frac{C}{|x'-z'|^{n}},\quad\text{ if }\,\, x'\neq z'.
\end{equation}

At this point we return to the proof of \eqref{est-atom-h}. 
Fix $x'\in{\mathbb{R}}^{n-1}\setminus 2\,Q$ and consider $h$ as in \eqref{fdsnn}. 
Making use of the last property of $h$ recorded in \eqref{fdsnn}, the Fundamental 
Theorem of Calculus, Minkowski's inequality, and \eqref{est-kernel-con}
we may compute
\begin{align}\label{ysre}
(S_{\Theta} h)(x')=&\Big(\int_{\Gamma_\kappa(x')} 
|(\Theta h)(y',t)|^2\,\frac{dy'\,dt}{t^n}\Big)^{\frac12}
\nonumber\\[6pt]
&\leq\Big(\int_{\Gamma_\kappa(x')}\Big(\int_{\mathbb{R}^{n-1}} 
|\theta(y',t;z')-\theta(y',t;x_Q')|\,|h(z')|\,dz'\Big)^2\,\frac{dy'\,dt}{t^n}\Big)^{\frac12}
\nonumber\\[6pt]
&\leq\int_{0}^1\int_{\mathbb{R}^{n-1}}\Big(\int_{\Gamma_\kappa(x')} 
|\nabla_{z'}\theta(y',t;x_Q'+s\,(z'-x_Q'))|^2
\frac{dy'\,dt}{t^n}\Big)^{\frac12}\,|h(z')|\,|z'-x_Q'|\,dz'\,ds
\nonumber\\[6pt]
&\leq C\,\int_{0}^1\int_{Q}
\frac{1}{\big|x'-(x_Q'+s\,(z'-x_Q'))\big|^{n}}\,|h(z')|\,|z'-x_Q'|\,dz'\,ds
\nonumber\\[6pt]
&\leq C\,\|h\|_{L^1(\mathbb{R}^{n-1})}\frac{\ell(Q)}{|x'-x'_Q|^n}.
\end{align}
For the last inequality in \eqref{ysre} we used the observation that for every 
$s\in(0,1)$ and every $z'\in Q$ one has (keeping in mind that 
$x'\in{\mathbb{R}}^{n-1}\setminus 2\,Q$ and $x_Q'+s\,(z'-x_Q')\in Q$)
\begin{align}\label{Twazv}
|x'-x_Q'|&\leq\big|x'-(x_Q'+s\,(z'-x_Q'))\big|+\frac{\sqrt{n-1}\,\ell(Q)}{2}
\nonumber\\[4pt]
& \leq\big|x'-(x_Q'+s\,(z'-x_Q'))\big|+\sqrt{n-1}\,\big|x'-(x_Q'+s\,(z'-x_Q'))\big|
\nonumber\\[4pt]
& =(1+\sqrt{n-1})\big|x'-(x_Q'+s\,(z'-x_Q'))\big|.
\end{align}
This finishes the proof of \eqref{est-atom-h}.

Let us momentarily digress to show that 
\begin{equation}\label{Mdgds}
\int_{{\mathbb{R}}^{n-1}\setminus Q}\frac{\ell(Q)}{|x'-x_Q'|^n}\,dx'
\leq\sum_{k=0}^\infty 
\int_{2^{k+1}Q\setminus 2^kQ}\,\frac{\ell(Q)}{(\ell(2^kQ)/2)^n}\,dx'
\leq 2^{2n-1}\sum_{k=0}^\infty 2^{-k}=4^n.
\end{equation}

We are now ready to show that $S_\Theta$ maps $L^1(\mathbb{R}^{n-1})$ continuously 
into $L^{1,\infty}(\mathbb{R}^{n-1})$. Following \cite[p.\,140]{GCRF85}, given a function 
$f\in L^1(\mathbb{R}^{n-1})$ and fixed $\lambda>0$, let $\{Q_j\}_j\subset{\mathbb{R}}^{n-1}$ 
be the non-overlapping cubes of the Calder\'on-Zygmund decomposition of $|f|$ at height $\lambda$. 
That is, the $Q_j$'s are the maximal dyadic cubes for which $|Q_j|^{-1}\int_{Q_j}|f(y')|\,dy'>\lambda$. Set 
\begin{equation}\label{Twazv-2}
\Omega_\lambda=\bigcup_{j=1}^\infty Q_j 
\end{equation}
and observe that we have the following properties
\begin{equation}\label{iu66tg}
{\mathscr{L}}^{n-1}(\Omega_\lambda)\leq \lambda^{-1}\|f\|_{L^1({\mathbb{R}}^{n-1})},
\end{equation}
\begin{equation}\label{Twazv-2hgff}
\lambda<\aver{Q_j}|f(y')|\,dy'\leq 2^{n-1}\,\lambda,\qquad\forall\,j\in{\mathbb{N}},
\end{equation}
and
\begin{equation}\label{Twazv-515dede}
|f(x')|\le \lambda,
\qquad \text{for ${\mathscr{L}}^{n-1}$-a.e. $x'\in\mathbb{R}^{n-1}\setminus\Omega_\lambda$}.
\end{equation}
Finally, split $f=g+b$ where (cf. \cite[p.\,198]{GCRF85})
\begin{equation}\label{Twazv-3}
\begin{array}{c}
g:=f\,{\mathbf 1}_{\mathbb{R}^{n-1}\setminus\Omega_\lambda}
+\sum\limits_{j=1}^\infty f_{Q_j}\,{\mathbf 1}_{Q_j},\qquad b=\sum\limits_{j=1}^\infty b_j,
\\[8pt]
\text{with }\,\,b_j:=\big(f-f_{Q_j})\,{\mathbf 1}_{Q_j}\,\,\text{ for each }\,\,j\in{\mathbb{N}}.
\end{array}
\end{equation}
In particular, \eqref{Twazv-2}-\eqref{Twazv-3} imply (cf. \cite[p.\,198]{GCRF85}) 
that for some constant $C\in(0,\infty)$ 
independent of $f$ and $\lambda$ we have
\begin{equation}\label{Twazv-3ttyy}
\|g\|_{L^2({\mathbb{R}}^{n-1})}^2\leq 2^{n-1}\lambda\|f\|_{L^1({\mathbb{R}}^{n-1})}.
\end{equation}

Making use of \eqref{Twazv-3}, \eqref{iu66tg}, 
\eqref{Twazv-3ttyy}, \eqref{SFE-L2}-\eqref{est-atom-h}, \eqref{Mdgds}
(used with $Q_j$ in place of $Q$),
and bearing in mind that for each $j\in{\mathbb{N}}$ we have ${\rm supp}\,b_j\subset Q_j$ 
and $\int_{\mathbb{R}^{n-1}}b_j(y')\,dy'=0$, we may then estimate
\begin{align}\label{Twazv-4}
&\hskip -0.10in
\lambda\,{\mathscr{L}}^{n-1}\big(\{x'\in\mathbb{R}^{n-1}:\,(S_\Theta f)(x')>\lambda\}\big)
\nonumber\\[6pt]
&\hskip 0.40in
\leq\lambda\,{\mathscr{L}}^{n-1}\big(\{x'\in\mathbb{R}^{n-1}:\,(S_\Theta g)(x')>\lambda/2\}\big)
\nonumber\\[6pt]
&\hskip 0.60in
+\lambda\,{\mathscr{L}}^{n-1}\big(\{x'\in\mathbb{R}^{n-1}:\,(S_\Theta b)(x')>\lambda/2\}\big)
\nonumber\\[6pt]
&\hskip 0.40in
\leq\frac{4}{\lambda}\,\int_{\mathbb{R}^{n-1}}|(S_\Theta g)(x')|^2\,dx'
+\lambda\,{\mathscr{L}}^{n-1}\Big(\bigcup_{j=1}^\infty 2\,Q_j\Big)
\nonumber\\[6pt]
&\hskip 0.60in
+2\,\sum_{j=1}^\infty\int_{\mathbb{R}^{n-1}\setminus 2Q_j}\,|(S_\Theta b_j)(x')|\,dx'
\nonumber\\[6pt]
&\hskip 0.40in
\leq\frac{C}{\lambda}\int_{\mathbb{R}^{n-1}}|g(x')|^2\,dx'
+C\|f\|_{L^1(\mathbb{R}^{n-1})}
\nonumber\\[6pt]
&\hskip 0.60in
+C\,\sum_{j=1}^\infty\|b_j\|_{L^1(\mathbb{R}^{n-1})}
\int_{\mathbb{R}^{n-1}\setminus 2Q_j}\,\frac{\ell(Q_j)}{|x'-x'_{Q_j}|^n}\,dx'
\nonumber\\[6pt]
&\hskip 0.40in
\leq C\,\|f\|_{L^1(\mathbb{R}^{n-1})}
+C\,\sum_{j=1}^\infty\|b_j\|_{L^1(\mathbb{R}^{n-1})}
\nonumber\\[6pt]
&\hskip 0.40in
\leq C\,\|f\|_{L^1(\mathbb{R}^{n-1})}
+C\,\Big(\sum_{j=1}^\infty\int_{Q_j}|f(y')|\,dy'\Big)
\nonumber\\[6pt]
&\hskip 0.40in
\leq C\,\|f\|_{L^1(\mathbb{R}^{n-1})}.
\end{align}
This proves that the operator in \eqref{hdgswf.DD2} is 
well-defined, linear and bounded. The latter and Marcinkewicz's Interpolation Theorem 
imply the boundedness of the operator in \eqref{hdgswf.DD} when $1<p\leq 2$. 
We may handle the full range $1<p<\infty$ by invoking \cite[Theorem~1.1, p.\,6]{HMMM}, 
applied with ${\mathscr{X}}:=\overline{{\mathbb{R}}^{n}_{+}}$ equipped with the standard 
Euclidean distance and Lebesgue measure, $E:={\mathbb{R}}^{n-1}\times\{0\}$, 
$m=n$, $d=n-1$, $\upsilon=1$, $\alpha=1$, $\sigma:={\mathscr{L}}^{n-1}$, 
and the integral operator with kernel $t^{-1}\theta(x',t;y')$. The fact that 
\eqref{hdgswf} holds implies that \cite[(1.25), p.\,6]{HMMM} is satisfied. 
As such, \cite[Theorem~1.1, p.\,6]{HMMM} guarantees the validity of \cite[(1.34), p.\,7]{HMMM} 
which, in our current notation, implies that the operator in \eqref{hdgswf.DD} 
is bounded for every $p\in(1,\infty)$.

Next we consider $S_\Theta$ in the context of \eqref{hdgswf.DD3}.
In this regard, we shall first show that there exists some constant 
$C\in(0,\infty)$ such that for every $(1,\infty)$-atom $a$ one has
\begin{equation}\label{SFE-H1}
\|S_\Theta\,a\|_{L^1(\mathbb{R}^{n-1})}\leq C.
\end{equation}
To justify \eqref{SFE-H1} fix an arbitrary function $a$ satisfying the conditions 
listed in \eqref{defi-atom}. On the one hand, based on H\"older's inequality, \eqref{SFE-L2},
and the first two properties in \eqref{defi-atom} we may write
\begin{equation}\label{SFE-H1-local}
\int_{2Q}|(S_\Theta\,a)(x')|\,dx'
\leq C\,|Q|^{\frac12}\,\|S_{\Theta}\,a\|_{L^2(\mathbb{R}^{n-1})}
\leq C\,|Q|^{\frac12}\,\|a\|_{L^2(\mathbb{R}^{n-1})}\leq C,
\end{equation}
for some finite constant $C>0$ independent of $a$. On the other hand, \eqref{defi-atom} 
allows us to make use of \eqref{est-atom-h}
(with $a$ in place of $h$), which we combine with the second property in 
\eqref{defi-atom} and \eqref{Mdgds} to obtain
\begin{equation}\label{SFE-H1-global}
\int_{\mathbb{R}^{n-1}\setminus 2Q}|(S_\Theta\,a)(x')|\,dx'
\leq C\,\|a\|_{L^1(\mathbb{R}^{n-1})}\,\int_{\mathbb{R}^{n-1}\setminus 2Q}
\frac{\ell(Q)}{|x'-x'_Q|^n}\,dx'\leq C,
\end{equation}
with $C\in(0,\infty)$ independent of the atom $a$. 
Combining \eqref{SFE-H1-local} and \eqref{SFE-H1-global} 
then proves that \eqref{SFE-H1} holds.

Here is the end-game in the proof of the fact 
that $S_\Theta$ maps $H^1(\mathbb{R}^{n-1})$ boundedly into 
$L^1(\mathbb{R}^{n-1})$. Let $f\in H^1(\mathbb{R}^{n-1})$ be arbitrary and consider
an atomic decomposition $f=\sum\limits_{j=1}^\infty\lambda_j\,a_j$ convergent in $H^1(\mathbb{R}^{n-1})$, 
where the $a_j$'s are $(1,\infty)$-atoms, which is quasi-optimal in the sense that 
$\sum\limits_{j=1}^\infty|\lambda_j|\approx\|f\|_{H^1(\mathbb{R}^{n-1})}$,
where the proportionality constants do not depend on $f$. 
In particular, this forces $f=\sum\limits_{j=1}^\infty\lambda_j\,a_j$ in $L^1(\mathbb{R}^{n-1})$ 
and the weak-type $(1,1)$ estimate for $S_\Theta$ then implies
$S_\Theta f=\sum\limits_{j=1}^\infty\lambda_j\,S_\Theta\,a_j$ in $L^{1,\infty}(\mathbb{R}^{n-1})$. 
Then the sequence of partial sums associated with the latter series has a sub-sequence 
which converges a.e. to $S_\Theta f$. In turn, this allows us to conclude that
\begin{equation}\label{Twazv-a}
|(S_\Theta f)(x')|\leq\sum_{j=1}^\infty|\lambda_j|\,|(S_\Theta\,a_j)(x')|
\quad\mbox{ for a.e. }\,x'\in\mathbb{R}^{n-1}.
\end{equation}
In concert, \eqref{Twazv-a} and \eqref{SFE-H1} then imply
\begin{equation}\label{Twazv-b}
\|S_\Theta\,f\|_{L^1(\mathbb{R}^{n-1})}\leq\sum_{j=1}^\infty|\lambda_j|\,
\|S_\Theta\,a_j\|_{L^1(\mathbb{R}^{n-1})}\leq C\,\sum_{j=1}^\infty|\lambda_j|
\leq C\,\|f\|_{H^1(\mathbb{R}^{n-1})},
\end{equation}
as desired, for some constant $C\in(0,\infty)$ independent of $f$.
\end{proof}

We now have all the ingredients to proceed with the proof of 
Proposition~\ref{prop-Dir-BMO:exis}.

\vskip 0.08in
\begin{proof}[Proof of Proposition~\ref{prop-Dir-BMO:exis}]
Fix an arbitrary $f\in L^1\Big({\mathbb{R}}^{n-1}\,,\,\frac{dx'}{1+|x'|^n}\Big)^M$ and
define $u$ as in \eqref{eqn-Dir-BMO:u:prop}. Part {\it (7)} in Theorem~\ref{kkjbhV} then 
ensures that this function satisfies all properties listed in \eqref{exist:u2}. 
 
As in \eqref{eq:Gvav7g5}, write $K^L(x',t)=P_t^L(x')$ for each $(x',t)\in\mathbb{R}^n_+$. 
To proceed, fix an arbitrary point $(x',t)\in{\mathbb{R}}^n_{+}$ and
denote by $Q_{x',t}$ the cube in $\mathbb{R}^{n-1}$ centered at $x'$ with side-length $t$.
Making use of \eqref{eq:IG6gy.2PPP} we obtain
\begin{equation}\label{XXdgsr}
\int_{{\mathbb{R}}^{n-1}}\nabla^\ell\big[ P^L_t(x'-y')\big]\,dy'=0,
\qquad\forall\,x'\in{\mathbb{R}}^{n-1},\,\,\forall\,t>0,\,\,\,\forall\,\ell\in{\mathbb{N}}.
\end{equation}
Based on this, for each $\ell\in{\mathbb{N}}$ we may then write 
\begin{align}\label{eqn:BMO-decay-EEE.3}
(\nabla^\ell u)(x',t) &=\int_{{\mathbb{R}}^{n-1}}\nabla^\ell\big[P^L_t(x'-y')\big]f(y')\,dy'
\nonumber\\[4pt]
&=\int_{{\mathbb{R}}^{n-1}}\nabla^\ell\big[P^L_t(x'-y')\big]\big[f(y')-f_{Q_{x',t}}\big]\,dy'
\nonumber\\[4pt]
&=\int_{{\mathbb{R}}^{n-1}}(\nabla^\ell K^L)(x'-y',t)\big[f(y')-f_{Q_{x',t}}\big]\,dy'.
\end{align}
Combining \eqref{eqn:BMO-decay-EEE.3}, \eqref{eq:Kest}, and \eqref{eqn:BMO-decay.88} 
(with $\varepsilon=\ell$), we may now estimate 
\begin{align}\label{eqn:BMO-decay-EEE.AB}
\big|(\nabla^\ell u)(x',t)\big|\leq
C\int_{\mathbb{R}^{n-1}}\frac{|f(y')-f_{Q_{x',t}}|}{\big|(x'-y',t)\big|^{n-1+\ell}}\,dy'
\leq \frac{C}{t^\ell}\int_{1}^\infty{\rm osc}_1\big(f;\lambda\,t\big)\,\frac{d\lambda}{\lambda^{1+\ell}},
\end{align}
from which the claims in part {\it (a)} of the statement follow. 

Moving on, fix an arbitrary $j\in\{1,\dots,n\}$ and, for each $\alpha,\beta\in\{1,\dots,M\}$, set 
\begin{equation}\label{est-theJKBB}
\theta_{\alpha\beta}^j(x',t;y'):=t\,\partial_j K_{\alpha\beta}^L(x'-y',t),\qquad
\forall\,x',y'\in\mathbb{R}^{n-1}_{+},\quad\forall\,t>0. 
\end{equation}
In this regard, first observe that \eqref{eq:Kest} in Theorem~\ref{kkjbhV} implies
\begin{equation}\label{est-theta-K}
|\theta_{\alpha\beta}^j(x',t;y')|
=t\,|\partial_j K_{\alpha\beta}^L(x'-y',t)|
\leq C\,t\,|(x'-y',t)|^{-n}
\end{equation}
and
\begin{equation}\label{est-theta-K-nabla}
|\nabla_{y'}\theta_{\alpha\beta}^j(x',t;y')|
\leq t\,|\nabla^2 K_{\alpha\beta}^L(x'-y',t)|
\leq C\,t|(x'-y',t)|^{-n-1}.
\end{equation}
Hence, \eqref{SFE-est-theta} (with $\varepsilon=1$) and \eqref{SFE-est-theta-bis}
hold for $\theta_{\alpha\beta}^j$. Moreover, 
\begin{align}\label{est-theta-vanish}
\int_{\mathbb{R}^{n-1}}\theta_{\alpha\beta}^j(x',t;y')\,dy' 
=\int_{\mathbb{R}^{n-1}}t\,\partial_j K_{\alpha\beta}^L(x'-y',t)\,dy'
=t\,\partial_j\int_{\mathbb{R}^{n-1}}K_{\alpha\beta}^L(y',t)\,dy'=0
\end{align}
since $\partial_j\int_{\mathbb{R}^{n-1}}(P_{\alpha\beta}^L)_t(y')\,dy'=0$ 
by \eqref{XXdgsr}. Writing $\Theta_{\alpha\beta}^j$ 
for the operator associated with the kernel $\theta_{\alpha\beta}^j$ 
(in place of $\theta$) as in \eqref{defi:Theta},
it follows from \eqref{est-theta-K}, \eqref{est-theta-K-nabla}, \eqref{est-theta-vanish},
and Proposition~\ref{prop:SFE-early} that each matrix integral operator 
$\Theta^j:=\big(\Theta_{\alpha\beta}^j\big)_{1\leq\alpha,\beta\leq M}$
satisfies all hypotheses in Proposition~\ref{prop:SFE} (including \eqref{y6bTVV}).
In addition, 
\begin{align}\label{Twazv-c}
\left(\int_{0}^{\ell(Q)}\aver{Q}|\nabla u(x',t)|^2\,t\,dx'dt\right)^\frac12
&=\left(\int_{0}^{\ell(Q)}\aver{Q}\big|t\nabla\big(P_t^L* f\big)(x')\big|^2 
\frac{dx'dt}{t}\right)^{\frac12}
\nonumber\\[6pt]
& \leq\sum_{j=1}^n\left(\int_{0}^{\ell(Q)}\aver{Q}
\Big|t\,\big(\partial_j K^L(\cdot,t)*f\big)(x')\Big|^2\frac{dx'dt}{t}\right)^{\frac12}
\nonumber\\[6pt]
&=\sum_{j=1}^n\left(\int_{0}^{\ell(Q)}
\aver{Q}\big|(\Theta^j f)(x',t)\big|^2\frac{dx'dt}{t}\right)^{\frac12}.
\end{align}
Granted this, all remaining conclusions in parts {\it (b)}-{\it (f)} of the statement become 
direct consequences of Proposition~\ref{prop:SFE}. 
\end{proof}

\section{A Fatou result and uniqueness in the {\rm BMO}-Dirichlet problem}
\setcounter{equation}{0}
\label{section:Fatou-BMO}

The main result in this section is the following Poisson representation formula and
Fatou theorem.

\begin{proposition}\label{prop-Dir-BMO:uniq-A}
Let $L$ be an $M\times M$ elliptic system with constant complex coefficients as in
\eqref{L-def}-\eqref{L-ell.X}. Assume that
\begin{equation}\label{uniq:u2-A}
u\in\mathscr{C}^\infty(\mathbb{R}^n_{+},{\mathbb{C}}^M),\quad
Lu=0\,\,\mbox{ in }\,\,\mathbb{R}^{n}_{+},\,\,\text{ and }\,\,\|u\|_{**}<\infty.
\end{equation}
Then there exists a unique function 
$f\in L^1\Big({\mathbb{R}}^{n-1},\frac{1}{1+|x'|^n}\,dx'\Big)^M$ such that 
\begin{equation}\label{eqn-Dir-BMO:u-A}
u(x',t)=(P_t^L*f)(x'),\qquad\forall\,(x',t)\in{\mathbb{R}}^n_{+},
\end{equation}
where $P^L$ is the Poisson kernel for $L$ in $\mathbb{R}^{n}_+$ from
Theorem~\ref{kkjbhV}.

In fact, $u\big|_{\partial\mathbb{R}^{n}_{+}}^{{}^{\rm n.t.}}$ exists 
at a.e. point in ${\mathbb{R}}^{n-1}$, belongs to 
$\mathrm{BMO}(\mathbb{R}^{n-1},\mathbb{C}^M)$, and 
$f=u\big|_{\partial\mathbb{R}^{n}_{+}}^{{}^{\rm n.t.}}$. Moreover, there exists
a constant $C=C(n,L)\in(1,\infty)$ such that 
\begin{equation}\label{jxdgsvgf}
C^{-1}\|f\|_{\mathrm{BMO}(\mathbb{R}^{n-1},\mathbb{C}^M)}
\leq\|u\|_{**}\leq C\|f\|_{\mathrm{BMO}(\mathbb{R}^{n-1},\mathbb{C}^M)}.
\end{equation}
\end{proposition}

Also, as a corollary of Proposition~\ref{prop-Dir-BMO:uniq-A} we have the following 
result which, in view of \eqref{ncud}, implies the uniqueness statements for the
{\rm BMO}-Dirichlet problem and the {\rm VMO}-Dirichlet problem from 
Theorem~\ref{them:BMO-Dir}.

\begin{proposition}\label{prop-Dir-BMO:uniq}
Let $L$ be an $M\times M$ elliptic system with constant complex coefficients as in
\eqref{L-def}-\eqref{L-ell.X}. Assume that
\begin{equation}\label{uniq:u2}
\begin{array}{c}
u\in\mathscr{C}^\infty(\mathbb{R}^n_{+},{\mathbb{C}}^M),\quad
Lu=0\,\,\mbox{ in }\,\,\mathbb{R}^{n}_{+},\quad\|u\|_{**}<\infty,
\\[6pt]
u\big|_{\partial\mathbb{R}^{n}_{+}}^{{}^{\rm n.t.}}\,\,\text{ exists and vanishes at a.e.
point in }\,\,\mathbb{R}^{n-1}.
\end{array}
\end{equation}
Then necessarily $u\equiv 0$ in $\mathbb{R}^n_+$.
\end{proposition}

The proof of Proposition~\ref{prop-Dir-BMO:uniq-A} occupies the bulk of this section.
To set the stage, we first prove some auxiliary lemmas. The first such lemma contains 
Bloch-like estimates for smooth null-solutions of $L$ satisfying a Carleson measure 
condition in the upper-half space. To place things in perspective, recall that a holomorphic function 
$u$ in the upper-half plane is said to satisfy a Bloch estimate provided
\begin{equation}\label{i7h6tg}
\sup_{x\in{\mathbb{R}},\,y>0}\,\big(y|u'(x+iy)|\big)<\infty.
\end{equation}

\begin{lemma}\label{lemma:decay-Du:Carl}
Let $L$ be an $M\times M$ elliptic system with constant complex coefficients as in
\eqref{L-def}-\eqref{L-ell.X}. Then for every multi-index $\alpha\in{\mathbb{N}}_0^n$ with $|\alpha|\geq 1$ 
there exists a finite constant $C=C(n,L,\alpha)>0$ with the property that for every 
function $u\in\mathscr{C}^\infty(\mathbb{R}^n_+,{\mathbb{C}}^M)$ satisfying 
$Lu=0$ in $\mathbb{R}_+^n$ and $\|u\|_{**}<\infty$ one has
\begin{equation}\label{decay-Du:Carl-alpha}
\sup_{(x',t)\in\mathbb{R}^n_+}\Big\{t^{|\alpha|}\,\big|(\partial^\alpha u)(x',t)\big|\Big\}\leq C\,\|u\|_{**}.
\end{equation}
In particular, there exists a finite constant $C=C(n,L)>0$ with the property that for every 
function $u\in\mathscr{C}^\infty(\mathbb{R}^n_+,{\mathbb{C}}^M)$ such that 
$Lu=0$ in $\mathbb{R}_+^n$ and $\|u\|_{**}<\infty$ one has
\begin{equation}\label{decay-Du:Carl}
\sup_{(x',t)\in\mathbb{R}^n_+} t\,|\nabla u(x',t)|\leq C\,\|u\|_{**}.
\end{equation}
\end{lemma}

\begin{proof}
Given a multi-index $\alpha\in{\mathbb{N}}_0^n$ with $|\alpha|\geq 1$, select 
$j\in\{1,\dots,n\}$ and $\beta\in{\mathbb{N}}_0^n$ such that $\alpha=\beta+e_j$. 
Fix $x=(x',t)\in\mathbb{R}^n_+$ and write 
$R_x$ for the cube in $\mathbb{R}^{n}$ centered at $x$ with side-length $t$. 
Also, let $Q_{x'}$ be the cube in $\mathbb{R}^{n-1}$ centered at $x'$ with side-length $t$. 
Since the function $\partial_j u$ is a null-solution of the system $L$, we may invoke 
Theorem~\ref{ker-sbav} (with $\partial_j u$ in place of $u$ and $p=2$) in order to conclude
\begin{align}\label{Twaz-GFk.1}
\big|\partial^\beta(\partial_j u)(x',t)\big|
\leq\,\frac{C_\beta}{t^{|\beta|}}\,\Big(\aver{R_x}|\partial_j u|^2\,d\mathscr{L}^n\Big)^\frac12.
\end{align}
Hence, 
\begin{align}\label{Twazvee}
t^{|\alpha|}\,\big|(\partial^\alpha u)(x',t)\big|
&\leq C\,t\,\Big(\aver{R_x}|\nabla u|^2\,d\mathscr{L}^n\Big)^\frac12
\nonumber\\[4pt]
&\leq C\,\Big(\int_{t/2}^{3\,t/2}\,\aver{Q_{x'}} 
|\nabla u(y',s)|^2\,s\,dy'ds\Big)^\frac12
\nonumber\\[4pt]
&\leq C\,\|u\|_{**},
\end{align}
proving \eqref{decay-Du:Carl-alpha}. Estimate \eqref{decay-Du:Carl} is a particular case of 
\eqref{decay-Du:Carl-alpha}. 
\end{proof}

We continue by discussing two purely real-variable results.
To state the first one, recall the function $\Upsilon_{\!\#}:[0,\infty)\to [0,\infty)$ 
from \eqref{decay-infty:Ups0}. In relation to this, we make two observations. First, 
\begin{equation}\label{PP-876tf.2hhn}
\begin{array}{c}
\text{for each $\varepsilon\in(0,\infty)$ there exists $C_\varepsilon\in(1,\infty)$ such that}
\\[6pt]
\text{$C_\varepsilon^{-1}\Upsilon_{\!\#}(s)\leq\Upsilon_{\!\#}(s/\varepsilon)\leq C_\varepsilon\Upsilon_{\!\#}(s)$ 
for every $s\in[0,\infty)$.}
\end{array}
\end{equation}
Second, since for every $\eta\in(0,1]$ 
there exists a constant $C=C_\eta\in(0,\infty)$ with the property that 
\begin{equation}\label{PP-876tf}
\Upsilon_{\!\#}(s)\leq Cs^\eta,\qquad\forall\,s\geq 0,
\end{equation}
we have (cf. \eqref{UpUpUp})
\begin{equation}\label{PP-876tf.2}
\mathscr{C}^{\Upsilon_{\!\#}}({\mathbb{R}}^{n-1},{\mathbb{C}}^M)\subset
{\mathrm{Lip}}({\mathbb{R}}^{n-1},{\mathbb{C}}^M)\cap\Big(
\bigcap_{0<\eta<1}\dot{\mathscr{C}}^\eta({\mathbb{R}}^{n-1},{\mathbb{C}}^M)\Big).
\end{equation}
This is going to be relevant later on, in the proof of Lemma~\ref{lemma:u-lift:props}. 
For now, here is the first real-variable result advertised above. 

\begin{lemma}\label{lemma:decay-infty}
Recall $\Upsilon_{\!\#}$ from \eqref{decay-infty:Ups0} and
let $u\in\mathscr{C}^1(\mathbb{R}_+^n,{\mathbb{C}}^M)$ be such that
\begin{equation}\label{decay-Du}
C_u:=\sup_{(x',t)\in\mathbb{R}^n_+} t\,|\nabla u(x',t)|<\infty.
\end{equation}
Then, for every $(x',t)$ and $(y',t)$ in $\mathbb{R}^n_+$ one has
\begin{equation}\label{decay-infty}
|u(x',t)-u(y',t)|\leq 2\,C_u\,\Upsilon_{\!\#}\Big(\frac{|x'-y'|}{t}\Big).
\end{equation}
\end{lemma}

\begin{proof}
The proof follows the argument in \cite{FJN}. Fix $(x',t)$ and $(y',t)$ in $\mathbb{R}^n_+$.
Based on the Mean Value Theorem and \eqref{decay-Du} we may estimate
\begin{equation}\label{hdtsr}
|u(x',t)-u(y',t)|\leq\sup_{\xi\in[x',y']}|\nabla u(\xi,t)|\,|x'-y'|
\leq C_u\,\frac{|x'-y'|}{t}.
\end{equation}
Suppose now that $|x'-y'|>t$ and set $r:=|x'-y'|$. Applying the 
Fundamental Theorem of Calculus, \eqref{hdtsr}, and \eqref{decay-Du} we obtain
\begin{align}\label{eqn:u-3}
|u(x',t)-u(y',t)|
&\leq|u(x',t)-u(x',r)|+|u(x',r)-u(y',r)|+|u(y',r)-u(y',t)|
\nonumber\\[4pt]
&\leq\int_t^r |\partial_n u(x',\lambda)|\,d\lambda
+C_u+\int_t^r |\partial_n u(y',\lambda)|\,d\lambda
\nonumber\\[4pt]
&\leq C_u+2\,C_u\,\int_t^r\frac1{\lambda}\,d\lambda
\nonumber\\[4pt]
&\leq 2\,C_u\,\Big(1+\ln\frac{|x'-y'|}{t}\Big).
\end{align}
With this in hand, \eqref{decay-infty} follows from \eqref{eqn:u-3} (which is valid for $|x'-y'|>t$)
and \eqref{hdtsr} used for $|x'-y'|\leq t$.
\end{proof}

The second real-variable result mentioned earlier reads as follows.  

\begin{lemma}\label{lemma:est-int-Psi0}
Let $\Upsilon_{\!\#}$ be the function defined in \eqref{decay-infty:Ups0}. 
Then for every $a>0$ one has
\begin{equation}\label{F0(a)}
\Psi(a):=\int_0^\infty\frac{s^{n-1}}{(a+s)^n}\,\Upsilon_{\!\#}(s)\,\frac{ds}{s}
\leq 3\,\left\{
\begin{array}{ll}
1+\ln\frac{1}{a}, &\,\mbox{ if }\,a\leq 1,
\\[6pt]
\frac{1+\ln a}{a}, &\,\mbox{if }\,\,a>1.
\end{array}
\right.
\end{equation}
In particular, $\Psi(a)\leq 3\,\Big(1+\log^+\frac1a\Big)$, where $\log^+ s:=\max\{\ln s,0\}$
for every $s\in(0,\infty)$.
\end{lemma}

\begin{proof}
If $a\geq 1$, we use that $s^{n-2}\,\Upsilon_{\!\#}(s)$ is increasing and elementary 
calculus to obtain
\begin{align}\label{Twazv-dd}
\Psi(a) & \leq a^{-n}\int_0^a s^{n-2}\,\Upsilon_{\!\#}(s)\,ds
+\int_a^\infty\frac{1+\ln s}{s^2}\,ds
\nonumber\\[4pt]
&\leq a^{-n}\,a^{n-2}\,\Upsilon_{\!\#}(a)\,a
+a^{-1}+\Big[\frac{-1-\ln s}{s}\Big]_{s=a}^{s=\infty}
\leq 3\,\frac{1+\ln a}{a}.
\end{align}
On the other hand, if $a< 1$ then
\begin{align}\label{Twazv-ff}
\Psi(a) & \leq\int_0^{a}\frac{s^{n-1}}{a^n}\,s\,\frac{ds}{s}
+\int_{a}^1\frac{s^{n-1}}{s^n}\,s\,\frac{ds}{s}
+\int_1^\infty\frac{s^{n-1}}{s^n}\,(1+\ln s)\,\frac{ds}{s}
\nonumber\\[4pt]
& =\frac{1}{n}+\ln\frac{1}{a}+2\leq 3\Big(1+\ln\frac{1}{a}\Big).
\end{align}
Collectively, \eqref{Twazv-dd}-\eqref{Twazv-ff} prove the lemma.
\end{proof}

Having dealt with Lemmas~\ref{lemma:decay-infty}-\ref{lemma:est-int-Psi0}, in our next two lemmas 
we study the boundary behavior of the vertical shifts of a smooth null-solution of $L$ 
which satisfies a Carleson measure condition in the upper-half space. 

\begin{lemma}\label{lemma:u-lift:props}
Let $L$ be an $M\times M$ elliptic system with constant complex coefficients as in
\eqref{L-def}-\eqref{L-ell.X} and consider $P^L$, the associated Poisson kernel for 
$L$ in $\mathbb{R}^{n}_+$ from Theorem~\ref{kkjbhV}. 
Suppose $u\in\mathscr{C}^\infty(\mathbb{R}^n_+,{\mathbb{C}}^M)$ satisfies $Lu=0$ in 
$\mathbb{R}_+^n$ and $\|u\|_{**}<\infty$. For each $\varepsilon>0$
define $u_\varepsilon(x',t):=u(x', t+\varepsilon)$, for every $(x',t)\in\mathbb{R}^n_+$ 
and $f_\varepsilon(x'):=u(x',\varepsilon)$ for every $x'\in{\mathbb{R}}^{n-1}$. 
Then there exists a constant $C\in(0,\infty)$ such that for every 
$\varepsilon>0$ the following properties are valid:
\begin{list}{(\theenumi)}{\usecounter{enumi}\leftmargin=.8cm
\labelwidth=.8cm\itemsep=0.2cm\topsep=.1cm
\renewcommand{\theenumi}{\alph{enumi}}}
\item The function $u_\varepsilon$ belongs to 
$\mathscr{C}^\infty(\overline{\mathbb{R}^n_+},{\mathbb{C}}^M)$ and 
$Lu_\varepsilon=0$ in $\mathbb{R}_+^n$.
\item One has $\|u_\varepsilon\|_{**}\leq C\,\|u\|_{**}$. In fact, for every multi-index 
$\alpha\in{\mathbb{N}}_0^n$ there exists a constant $C_{\alpha}\in(0,\infty)$, independent of $u$ and $\varepsilon$,
with the property that $\|\partial^\alpha u_\varepsilon\|_{**}\leq C_{\alpha}\varepsilon^{-|\alpha|}\|u\|_{**}$.
\item For every multi-index $\alpha\in{\mathbb{N}}_0^n$ with $|\alpha|\geq 1$ there exists a
constant $C_\alpha\in(0,\infty)$, independent of $u$, with the property that
$\|\partial^\alpha u_{\varepsilon}\|_{L^\infty(\mathbb{R}^{n}_{+})}
\leq C_\alpha\,\varepsilon^{-|\alpha|}\,\|u\|_{**}$.
\item The function $f_\varepsilon$ belongs to 
$\mathscr{C}^\infty({\mathbb{R}}^{n-1},{\mathbb{C}}^M)\cap
\mathscr{C}^{\Upsilon_{\!\#}}({\mathbb{R}}^{n-1},{\mathbb{C}}^M)$ where 
$\Upsilon_{\!\#}$ is as in \eqref{decay-infty:Ups0}. In particular, 
\begin{equation}\label{PP-876tf.2RED}
f_\varepsilon\in{\mathrm{Lip}}({\mathbb{R}}^{n-1},{\mathbb{C}}^M)\cap\Big(
\bigcap_{0<\eta<1}\dot{\mathscr{C}}^\eta({\mathbb{R}}^{n-1},{\mathbb{C}}^M)\Big)
\end{equation}
hence also $f_\varepsilon\in{\mathrm{UC}}({\mathbb{R}}^{n-1},{\mathbb{C}}^M)$ and
\begin{equation}\label{i86gg}
f_\varepsilon\in L^1\Big({\mathbb{R}}^{n-1},\frac{1}{1+|x'|^n}\,dx'\Big)^M.
\end{equation}
Moreover, 
\begin{equation}\label{i86gg-EXP}
\begin{array}{c}
\text{for every }\,\,\alpha'\in{\mathbb{N}}_0^{n-1}\,\,\text{ with }\,\,|\alpha'|\geq 1\,\,
\text{ one has}
\\[6pt]
\partial^{\alpha'}f_\varepsilon\in L^\infty({\mathbb{R}}^{n-1},{\mathbb{C}}^M)\cap
\mathscr{C}^{\Upsilon_{\!\#}}({\mathbb{R}}^{n-1},{\mathbb{C}}^M).
\end{array}
\end{equation}
\item The function $v_\varepsilon(x',t):=\big(P_t^L*f_\varepsilon\big)(x')$ is 
well-defined for all $(x',t)\in\mathbb{R}_+^n$ via an absolutely convergent integral and
\begin{equation}\label{props-lifting}
v_\varepsilon\in\mathscr{C}^\infty({\mathbb{R}}^n_{+},{\mathbb{C}}^M),\quad
Lv_\varepsilon=0\,\,\mbox{ in }\,\,\mathbb{R}^{n}_{+},\quad
v_\varepsilon\big|_{\partial\mathbb{R}^{n}_{+}}^{{}^{\rm n.t.}}
=f_\varepsilon\,\,\mbox{ everywhere in }\,\,\mathbb{R}^{n-1}.
\end{equation}
\item For every $(y',t)\in\mathbb{R}_+^n$ one has
\begin{equation}\label{estimates-veps}
|v_\varepsilon(y',t)-f_\varepsilon(y')|+t\,|\nabla v_\varepsilon(y',t)|
\leq C\,\|u\|_{**}\,(t/\varepsilon)\,\big(1+\log^+ (\varepsilon/t)\big).
\end{equation}
\end{list}
\end{lemma}

\begin{proof}
The claim in part {\it (a)} is clear from definitions. To prove the estimate in part {\it (b)},
fix a cube $Q\subset\mathbb{R}^{n-1}$. Consider first the case $\ell(Q)\geq\varepsilon$ 
in which scenario a change of variables yields
\begin{align}\label{ueps-u:**:1}
\frac1{|Q|}\int_{0}^{\ell(Q)}\int_Q |\nabla u(x',t+\varepsilon)|^2\,t\,dx'dt
&\leq\frac1{|Q|}\int_{\varepsilon}^{\ell(Q)+\varepsilon}\int_Q|\nabla u(x',t)|^2\,t\,dx'dt
\nonumber\\[4pt]
&\leq 2^{n-1}\frac{1}{|2Q|}\int_{0}^{2\,\ell(Q)}\int_{2Q}|\nabla u(x',t)|^2\,t\,dx'dt
\nonumber\\[4pt]
&\leq 2^{n-1}\,\|u\|_{**}^2.
\end{align}
In the case $\ell(Q)<\varepsilon$, use Lemma~\ref{lemma:decay-Du:Carl} to conclude that
\begin{align}\label{ueps-u:**:2}
\frac1{|Q|}\int_{0}^{\ell(Q)}\int_Q|\nabla u(x',t+\varepsilon)|^2\,t\,dx'dt
& \leq C^2\,\|u\|_{**}^2\,
\frac1{|Q|}\int_{0}^{\ell(Q)}\int_Q\frac{1}{(t+\varepsilon)^2}\,t\,dx'dt
\nonumber\\[4pt]
&\leq C^2\,\|u\|_{**}^2\,\varepsilon^{-2}\,\int_0^{\ell(Q)}t\,dt
\leq\frac{C^2}{2}\|u\|_{**}^2,
\end{align}
for some $C\in(0,\infty)$ independent of $u$ and $\varepsilon$.
Combining \eqref{ueps-u:**:1} and \eqref{ueps-u:**:2} and taking the supremum 
over all cubes $Q$ then proves the first estimate in part {\it (b)} for some $C\in(0,\infty)$ 
independent of $u$ and $\varepsilon$.

To justify the second estimate in part {\it (b)}, it suffices to consider the case when 
the multi-index $\alpha\in{\mathbb{N}}_0^n$ has length $|\alpha|\geq 1$. Assume that this 
is the case and pick an arbitrary cube $Q\subset\mathbb{R}^{n-1}$. 
Making use of \eqref{decay-Du:Carl-alpha} and bearing in mind that $|\alpha|\geq 1$ we may then estimate 
\begin{align}\label{ue-h654rff}
\frac1{|Q|}\int_{0}^{\ell(Q)}\int_Q\big|\nabla\big[(\partial^\alpha u_\varepsilon)(x',t)\big]\big|^2\,t\,dx'dt
&=\frac1{|Q|}\int_{0}^{\ell(Q)}\int_Q\big|\nabla\big[(\partial^\alpha u)(x',t+\varepsilon)\big]\big|^2\,t\,dx'dt
\nonumber\\[4pt]
&\leq\frac{C_\alpha\|u\|_{**}^2}{|Q|}\int_{0}^{\ell(Q)}\int_Q\frac{1}{(t+\varepsilon)^{2|\alpha|+1}}\,dx'dt
\nonumber\\[4pt]
&\leq C_\alpha\|u\|_{**}^2\int_{0}^{\infty}\frac{1}{(t+\varepsilon)^{2|\alpha|+1}}\,dt
\nonumber\\[4pt]
&\leq C_\alpha\|u\|_{**}^2\,\varepsilon^{-2|\alpha|},
\end{align}
from which the desired conclusion readily follows. 

Consider next the claim in part {\it (c)}. Given a multi-index $\alpha\in{\mathbb{N}}_0^n$ with 
$|\alpha|\geq 1$ we may invoke Lemma~\ref{lemma:decay-Du:Carl} (keeping in mind the conclusions in part {\it (a)})
in order to conclude that there exists a constant $C_\alpha\in(0,\infty)$, independent of $u$, such that
\begin{align}\label{Twazv-gzsre}
\sup_{(x',t)\in\mathbb{R}_+^n}|(\partial^\alpha u_\varepsilon)(x',t)|
&\leq\sup_{(x',t)\in\mathbb{R}_+^n}\Big[(t+\varepsilon)^{|\alpha|}\big|(\partial^\alpha u)(x',t+\varepsilon)\big|\Big]
\cdot\sup_{(x',t)\in\mathbb{R}_+^n}(t+\varepsilon)^{-|\alpha|}
\nonumber\\[4pt]
&\leq C_\alpha\,\|u\|_{**}\,\varepsilon^{-|\alpha|}.
\end{align}

We now turn to proving the claims in part {\it (d)}. 
First, since $f_\varepsilon(y')=u(y',\varepsilon)$ for every $y'\in{\mathbb{R}}^{n-1}$ 
we have $f_\varepsilon\in{\mathscr{C}}^\infty({\mathbb{R}}^{n-1},{\mathbb{C}}^M)$.
Second, by using \eqref{decay-infty} (with $t=\varepsilon$), \eqref{decay-Du:Carl}, 
and \eqref{PP-876tf.2hhn}, for each $x',y'\in{\mathbb{R}}^{n-1}$ we may estimate 
\begin{equation}\label{PP-876-g56y}
|f_\varepsilon(x')-f_\varepsilon(y')|\leq C\|u\|_{**}\Upsilon_{\!\#}\Big(\frac{|x'-y'|}{\varepsilon}\Big)
\leq C_{n,L,u,\varepsilon}\,\Upsilon_{\!\#}(|x'-y'|).
\end{equation}
This places $f_\varepsilon$ in $\mathscr{C}^{\Upsilon_{\!\#}}({\mathbb{R}}^{n-1},{\mathbb{C}}^M)$.
With this in hand, the conclusions in \eqref{PP-876tf.2RED} follow with the help of 
\eqref{PP-876tf.2}. As a H\"older function, $f_\varepsilon$ also belongs to 
$L^1\Big({\mathbb{R}}^{n-1},\frac{1}{1+|x'|^n}\,dx'\Big)^M$ (see \eqref{Gsyus}), 
proving \eqref{i86gg}.

Next, fix a multi-index $\alpha'\in{\mathbb{N}}_0^{n-1}$ of length $|\alpha'|\geq 1$.
Then $\partial^{(\alpha',0)}u_{\varepsilon}\in{\mathscr{C}}^\infty
(\overline{{\mathbb{R}}^n_{+}},{\mathbb{C}}^M)$
is a null-solution of $L$ in ${\mathbb{R}}^n_{+}$ and $\partial^{\alpha'}f_{\varepsilon}
=\big(\partial^{(\alpha',0)}u_{\varepsilon}\big)\big|_{\partial{\mathbb{R}}^{n}_{+}}$. 
Now the claim in \eqref{i86gg-EXP} is a consequence of parts {\it (c)} and {\it (b)}, 
bearing in mind that $\big\|\partial^{(\alpha',0)}u_{\varepsilon}\big\|_{**}<\infty$ 
(hence, the same type of argument that 
placed $f_\varepsilon$ in $\mathscr{C}^{\Upsilon_{\!\#}}({\mathbb{R}}^{n-1},{\mathbb{C}}^M)$ 
now ensures the membership of $\partial^{\alpha'}f_{\varepsilon}$ to the latter space). 

Moving on, the claim made in part {\it (e)} is a consequence of the current part {\it (d)} 
together with part {\it (7)} in Theorem~\ref{kkjbhV} and the fact 
that since $f_\varepsilon\in{\mathscr{C}}^\infty({\mathbb{R}}^{n-1},{\mathbb{C}}^M)$
all points in ${\mathbb{R}}^{n-1}$ are Lebesgue points for $f_\varepsilon$.

Finally, consider the claim in part {\it (f)}. Fix $(y',t)\in\mathbb{R}_+^{n}$. 
Then the properties of the Poisson kernel recalled in Theorem~\ref{kkjbhV},
together with Lemmas~\ref{lemma:decay-Du:Carl}, \ref{lemma:decay-infty}, 
and \ref{lemma:est-int-Psi0} permit us to estimate (bearing in mind that 
$f_\varepsilon=u(\cdot,\varepsilon)$)
\begin{align}\label{est-veps-1}
|v_\varepsilon(y',t)-f_\varepsilon(y')|
&\leq\int_{\mathbb{R}^{n-1}}
|P^L(z')|\,|f_\varepsilon(y'-tz')-f_\varepsilon(y')|\,dz'
\nonumber\\[6pt]
& \leq C\,\|u\|_{**}
\int_{\mathbb{R}^{n-1}}\frac1{(1+|z'|)^n}\,\Upsilon_{\!\#}\Big(\frac{t\,|z'|}{\varepsilon}\Big)\,dz'
\nonumber\\[6pt]
&\leq C\,\|u\|_{**}\,\int_0^\infty\frac{r^{n-1}}{(1+r)^n}\, 
\Upsilon_{\!\#}\Big(\frac{t\,r}{\varepsilon}\Big)\,\frac{dr}{r}
\nonumber\\[6pt]
&\leq C\,\|u\|_{**}\,(t/\varepsilon)\,\int_0^\infty 
\frac{s^{n-1}}{(t/\varepsilon+s)^n}\,\Upsilon_{\!\#}(s)\,\frac{ds}{s}
\nonumber\\[6pt]
&=C\,\|u\|_{**}\,(t/\varepsilon)\,\Psi\big(t/\varepsilon)
\nonumber\\[6pt]
& \leq C\,\|u\|_{**}\,(t/\varepsilon)\,\big(1+\log^+ (\varepsilon/t)\big).
\end{align}
This suits our current purposes.

Consider next the task of estimating $\nabla v_\varepsilon$. 
Using the properties of the Poisson kernel, Theorem~\ref{kkjbhV} 
(recall \eqref{eq:Gvav7g5} and \eqref{eq:Kest})
and Lemmas~\ref{lemma:decay-Du:Carl}, \ref{lemma:decay-infty}, \ref{lemma:est-int-Psi0},
we write
\begin{align}\label{est-veps-2}
|\nabla v_\varepsilon(x',t)|
&=\big|\nabla\big(P_t^L*\big(f_\varepsilon(\cdot)-f_\varepsilon(x')\big)\big)(x')\big|
\nonumber\\[6pt]
&\leq\int_{\mathbb{R}^{n-1}}|\nabla K^L(x'-y',t)|\,|f_\varepsilon(y')-f_\varepsilon(x')|\,dy'
\nonumber\\[6pt]
&\leq C\,\|u\|_{**}\int_{\mathbb{R}^{n-1}}
\frac1{(t+|x'-y'|)^n}\,\Upsilon_{\!\#}\Big(\frac{|x'-y'|}{\varepsilon}\Big)\,dy'
\nonumber\\[6pt]
&\leq C\,\|u\|_{**}\,\varepsilon^{-1}\int_0^\infty 
\frac{s^{n-1}}{(s+t/\varepsilon)^n}\,\Upsilon_{\!\#}(s)\,\frac{ds}{s}
\nonumber\\[6pt]
&=C\,\|u\|_{**}\,\varepsilon^{-1}\,\Psi\big(t/\varepsilon)
\nonumber\\[6pt]
& \leq C\,\|u\|_{**}\,\varepsilon^{-1}\,\big(1+\log^+ (\varepsilon/t)\big).
\end{align}
Collectively, \eqref{est-veps-1} and \eqref{est-veps-2} prove \eqref{estimates-veps}.
\end{proof}

We are now ready to prove that each vertical shift of a smooth null-solution of $L$ 
which satisfies a Carleson measure condition in the upper-half space has a Poisson 
integral representation formula. 

\begin{lemma}\label{lemma:u-lift:uniq}
Let $L$ be an $M\times M$ elliptic system with constant complex coefficients as in
\eqref{L-def}-\eqref{L-ell.X} and consider $P^L$, the associated Poisson kernel for 
$L$ in $\mathbb{R}^{n}_+$ from Theorem~\ref{kkjbhV}. 
Let $u\in\mathscr{C}^\infty(\mathbb{R}^n_+,{\mathbb{C}}^M)$ satisfy 
$Lu=0$ in $\mathbb{R}_+^n$ and $\|u\|_{**}<\infty$. For each given $\varepsilon>0$, 
define $u_\varepsilon(x',t):=u(x',t+\varepsilon)$ for every $(x',t)\in\mathbb{R}^n_+$. 

Then for every $\varepsilon>0$ one has 
$u_\varepsilon\in{\mathscr{C}}^\infty(\overline{\mathbb{R}_+^n},{\mathbb{C}}^M)$, 
the restriction $u_\varepsilon\bigl|_{\partial\mathbb{R}^{n}_{+}}$ belongs to the space
$L^1\Big({\mathbb{R}}^{n-1},\frac{1}{1+|x'|^n}\,dx'\Big)^M$, 
and the following Poisson integral representation formula holds:
\begin{equation}\label{Poisson:u-lift}
u_\varepsilon(x',t)
=\Big(P_t^L*\big(u_\varepsilon\bigl|_{\partial\mathbb{R}^{n}_{+}}\big)\Big)(x'),
\qquad\forall\,(x',t)\in\mathbb{R}^n_+.
\end{equation}
\end{lemma}

\begin{proof}
For each $\varepsilon>0$ set 
$f_\varepsilon:=u_\varepsilon\bigl|_{\partial\mathbb{R}^{n}_{+}}$ and 
note that by part {\it (d)} in Lemma~\ref{lemma:u-lift:props} we have that
$f_\varepsilon$ belongs to $L^1\Big({\mathbb{R}}^{n-1},\frac{1}{1+|x'|^n}\,dx'\Big)^M\cap
\mathscr{C}^\infty({\mathbb{R}}^{n-1},{\mathbb{C}}^M)$. Next, for each $\varepsilon>0$
define $v_\varepsilon(x',t):=\big(P_t^L*f_\varepsilon\big)(x')$
for every $(x',t)\in\mathbb{R}_+^{n}$.
The goal is to show that $w_\varepsilon:=v_\varepsilon-u_\varepsilon\equiv 0$ in 
$\mathbb{R}^{n}_+$. A key ingredient in this regard is Proposition~\ref{c1.2}.  

Notice first that $w_\varepsilon\in\mathscr{C}^\infty(\mathbb{R}^n_+,{\mathbb{C}}^M)$ 
and $Lw_\varepsilon=0$ in $\mathbb{R}_+^{n}$ by parts {\it (a)} and {\it (e)} in
Lemma~\ref{lemma:u-lift:props}. Next, we propose to show that ${\rm Tr}\,w_\varepsilon=0$,
where ${\rm Tr}$ is as introduced in \eqref{Veri-S2TG.3}.
Since by part {\it (a)} in Lemma~\ref{lemma:u-lift:props} we have 
${\rm Tr}\,u_\varepsilon=f_\varepsilon$, there remains to prove
${\rm Tr}\,v_\varepsilon=f_\varepsilon$ in ${\mathbb{R}}^{n-1}$. To this end,
given $x'\in\mathbb{R}^{n-1}$, we use part {\it (f)} in Lemma~\ref{lemma:u-lift:props}, 
the fact that $f_\varepsilon(x')=u(x',\varepsilon)$, 
Lemma~\ref{lemma:decay-Du:Carl} and Lemma~\ref{lemma:decay-infty} 
(recall that $\Upsilon_{\!\#}$ is defined in \eqref{decay-infty:Ups0}) to write
\begin{align}\label{Tegfsh}
&\Big|\aver{B((x',0),r)\cap{\mathbb{R}}^n_{+}}
v_\varepsilon\,d{\mathscr{L}}^n- f_\varepsilon(x')\Big|
\leq\aver{B((x',0),r)\cap{\mathbb{R}}^n_{+}} |v_\varepsilon(y',t)-f_{\varepsilon}(x')|\,dy'\,dt
\nonumber\\[4pt]
&\qquad\leq\aver{B((x',0),r)\cap{\mathbb{R}}^n_{+}}
|v_\varepsilon(y',t)-f_\varepsilon(y')|\,dy'\,dt
+\aver{B((x',0),r)\cap{\mathbb{R}}^n_{+}} |f_\varepsilon(y')-f_\varepsilon(x')|\,dy'\,dt
\nonumber\\[4pt]
&\qquad\leq C\,\|u\|_{**}\,\aver{B((x',0),r)\cap{\mathbb{R}}^n_{+}}
(t/\varepsilon)\,\big(1+\log^+ (\varepsilon/t)\big)\,dy'\,dt
\nonumber\\[4pt]
&\qquad\qquad\qquad
+C\,\|u\|_{**}\,\aver{B((x',0),r)\cap{\mathbb{R}}^n_{+}}\Upsilon_{\!\#}(|x'-y'|/\varepsilon)\,dy'\,dt
\nonumber\\[4pt]
&\qquad\leq C\,\|u\|_{**}\,\frac{r}{\varepsilon}\,\big(1+\log^+ (\varepsilon/r)\big)
+C\,\|u\|_{**}\,\Upsilon_{\!\#}(r/\varepsilon)\longrightarrow 0,\qquad\mbox{as }r\to 0^+.
\end{align}
Thus we conclude that ${\rm Tr}\,v_\varepsilon(x')=f_\varepsilon(x')$ 
for every $x'\in\mathbb{R}^{n-1}$ as desired.

Next we claim that $w_\varepsilon\in W^{1,2}_{\rm bd}({\mathbb{R}}^n_{+},{\mathbb{C}}^M)$
(recall the latter space from \eqref{w12bd}). 
By parts {\it (a)} and {\it (c)} in Lemma~\ref{lemma:u-lift:props} we have that 
$u_\varepsilon\in W^{1,2}_{\rm bd}({\mathbb{R}}^n_{+},{\mathbb{C}}^M)$. 
For $v_\varepsilon$, fix $R>0$ arbitrary and rely on \eqref{estimates-veps} to estimate
\begin{align}\label{v-ep:L2-bd}
\|v_\varepsilon\|_{L^2(B(0,R)\cap\mathbb{R}_+^n)}
&\leq\Big(\int_0^R\int_{|x'|\leq R} 
|v_\varepsilon(x',t)-f_\varepsilon(x')|^2\,dx'\,dt\Big)^{\frac12}
\nonumber\\[4pt]
&\hskip 1.00in
+\Big(\int_0^R\int_{|x'|\le R} |f_\varepsilon(x')|^2\,dx'\,dt\Big)^{\frac12}
\nonumber\\[4pt]
&\leq C\,\|u\|_{**}\,(R/\varepsilon)\,
\big(1+\log^+(\varepsilon/R)\big)\,R^{\frac{n}{2}}
\nonumber\\[4pt]
&\hskip 1.00in
+R^\frac12\,\|f_\varepsilon\|_{L^2(B_{n-1}(0',R))}<\infty,
\end{align}
since $f_\varepsilon\in C^\infty(\mathbb{R}^{n-1},{\mathbb{C}}^M)$.
Above and elsewhere in the paper we make the convention that
\begin{equation}\label{balnminus1}
\text{$B_{n-1}(x',R)$ denotes the ball in ${\mathbb{R}}^{n-1}$ 
centered at $x'\in{\mathbb{R}}^{n-1}$ and of radius $R$.}
\end{equation}
As regards $\nabla v_\varepsilon$, we use \eqref{estimates-veps} to write
\begin{align}\label{grad-v-ep:L2-bd}
\|\nabla v_\varepsilon\|_{L^2(B(0,R)\cap\mathbb{R}_+^n)}
&\leq C\,R^{\frac{n-1}2}\,\|u\|_{**}\,\varepsilon^{-1}\,
\Big(\int_0^R\big(1+\log^+(\varepsilon/t)\big)^2\,dt\Big)^{\frac12}
\nonumber\\[4pt]
&=C\,R^{\frac{n-1}2}\,\|u\|_{**}\,\varepsilon^{-1}\,
\Big(\int_0^\varepsilon\big(1+\ln(\varepsilon/t)\big)^2\,dt
+\int_\varepsilon^R\,dt\Big)^{\frac12}
\nonumber\\[4pt]
&\leq C\,R^{\frac{n-1}2}\,\|u\|_{**}\,\varepsilon^{-1}\,
\Big(\varepsilon\,\int_0^1\big(1+\ln(1/s)\big)^2\,ds+R\Big)^\frac12
\nonumber\\[4pt]
&\leq C\,R^{\frac{n-1}2}\,\|u\|_{**}\,\varepsilon^{-1}\,
(\varepsilon+R)^\frac12<\infty.
\end{align}
From \eqref{v-ep:L2-bd} and \eqref{grad-v-ep:L2-bd} we conclude that $v_\varepsilon$ 
and, therefore, $w_\varepsilon$ belongs to $W^{1,2}_{\rm bd}({\mathbb{R}}^n_{+},{\mathbb{C}}^M)$.

Having established these, we may apply Proposition~\ref{c1.2} and obtain that 
for every $z\in\overline{\mathbb{R}_+^n}$ and $\rho>0$
\begin{align}\label{est-sups}
\sup_{\mathbb{R}^n_+\cap B(z,\rho)}|\nabla w_\varepsilon|
&\leq C\,\rho^{-1}\sup_{\mathbb{R}^n_+\cap B(z,2\rho)}|w_\varepsilon|
=C\rho^{-1}\sup_{\mathbb{R}^n_+\cap B(z,2\rho)}|u_\varepsilon-v_\varepsilon|
\nonumber\\[4pt]
&\leq C\rho^{-1}\sup_{(y',t)\in\mathbb{R}^n_+\cap B(z,2\rho)}
|u_\varepsilon(y',t)-f_\varepsilon(y')|
\nonumber\\[4pt]
&\qquad
+C\rho^{-1}\sup_{(y',t)\in\mathbb{R}^n_+\cap B(z,2\rho)}
|v_\varepsilon(y',t)-f_\varepsilon(y')|.
\end{align}
Let $(y',t)\in\mathbb{R}^n_+\cap B(z,2\rho)$ and note that
 Lemma~\ref{lemma:decay-Du:Carl} implies
\begin{align}\label{est-sups:1}
|u_\varepsilon(y',t)-f_\varepsilon(y')|
&=|u(y',t+\varepsilon)-u(y',\varepsilon)|
\leq\int_\varepsilon^{t+\varepsilon} |\partial_n u(y',\lambda)|\,d\lambda
\nonumber\\[4pt]
&\leq C\,\|u\|_{**}\int_\varepsilon^{t+\varepsilon}\frac1\lambda\,d\lambda
=C\,\|u\|_{**}\ln\frac{t+\varepsilon}{\varepsilon}.
\end{align}
Proceeding as in \eqref{est-veps-1}, Lemma~\ref{lemma:est-int-Psi0} implies 
that for every $t>\varepsilon$ we have
\begin{equation}\label{est-sups:2}
|v_\varepsilon(y',t)-f_\varepsilon(y')|
\leq C\,\|u\|_{**}\,(t/\varepsilon)\,\Psi(t/\varepsilon)
\leq C\,\|u\|_{**}\,\big(1+\ln(t/\varepsilon)\big).
\end{equation}
Returning with \eqref{est-sups:1}, \eqref{est-sups:2} and \eqref{estimates-veps} 
back to \eqref{est-sups} we obtain that for every $z\in\partial{\mathbb{R}}^n_{+}$ 
and every $\rho>\varepsilon$
\begin{align}\label{hdfts-2}
\sup_{\mathbb{R}^n_+\cap B(z,\rho)}|\nabla w_\varepsilon|
&\leq C\,\|u\|_{**}
\Big(\rho^{-1}\sup_{0<t<2\rho}\ln\frac{t+\varepsilon}{\varepsilon}
+\rho^{-1}\sup_{0<t<\varepsilon}(t/\varepsilon)\,\big(1+\log^+(\varepsilon/t)\big)
\nonumber\\[4pt]
&\hskip6cm+\rho^{-1}\sup_{\varepsilon<t<2\rho} 
\big(1+\ln(t/\varepsilon)\big)\Big)
\nonumber\\[4pt]
&\leq C\,\|u\|_{**}
\Big(\rho^{-1}\ln\frac{2\rho+\varepsilon}{\varepsilon}
+\rho^{-1}+\rho^{-1}\big(1+\ln(2\,\rho/\varepsilon)\big)\Big).
\end{align}
Since the last expression converges to $0$ as $\rho\to\infty$ we obtain that 
$\nabla w_\varepsilon\equiv 0$ in $\mathbb{R}^n_+$. As we have already shown that 
$w_\varepsilon\in\mathscr{C}^\infty(\mathbb{R}_+^n,{\mathbb{C}}^M)$ this forces
$w_\varepsilon$ to be constant in $\mathbb{R}_+^n$. In concert with the fact that 
${\rm Tr}\,w_\varepsilon=0$ this ultimately implies 
$w_\varepsilon\equiv 0$ in $\mathbb{R}_+^n$ as desired.
\end{proof}

Moving on, in Lemmas~\ref{lemma:phit-g}-\ref{lemma:SFE-H1} below we develop tools which
are essential in the proof of Proposition~\ref{prop:Car->BMO}, where we prove a partial 
converse to part {\it (e)} in Proposition~\ref{prop-Dir-BMO:exis}. Concretely, there we show that 
if $f\in L^1\Big({\mathbb{R}}^{n-1}\,,\,\frac{1}{1+|x'|^n}\,dx'\Big)^M$ has the property that
the Littlewood-Paley measure $|\nabla u(x',t)|^2\,t\,dx'dt$ 
associated with the function $u$ defined as in \eqref{eqn-Dir-BMO:u:prop} is a Carleson measure 
in ${\mathbb{R}}^n_{+}$ then necessarily $f$ belongs to ${\mathrm{BMO}}(\mathbb{R}^{n-1},\mathbb{C}^M)$.

We begin by introducing some notation. Specifically, consider 
\begin{equation}\label{Car-g-dual}
H^1_a(\mathbb{R}^{n-1}):=
\Big\{g\in L^\infty_{\rm comp}(\mathbb{R}^{n-1}): 
\,\int_{\mathbb{R}^{n-1}}g\,d{\mathscr{L}}^{n-1}=0\Big\}.
\end{equation}
where $L^\infty_{\rm comp}({\mathbb{R}}^{n-1})$ stands for the space of essentially 
bounded functions with compact support in ${\mathbb{R}}^{n-1}$. In particular, 
since any $g\in H^1_a(\mathbb{R}^{n-1})$ is a scalar multiple of a $(1,\infty)$-atom
(recall \eqref{defi-atom}), it follows that 
\begin{equation}\label{h1adense}
\text{$H^1_a(\mathbb{R}^{n-1})$ is a dense subspace of $H^1(\mathbb{R}^{n-1})$.}
\end{equation} 

In the lemma below we prove a pointwise decay estimate for the vertical maximal operator  
acting on functions from $H^1_a(\mathbb{R}^{n-1})$. Recall the definition from \eqref{phisubt}.

\begin{lemma}\label{lemma:phit-g}
Let $\phi=\big(\phi_{\alpha\beta}\big)_{1\leq\alpha,\beta\leq M}:
\mathbb{R}^{n-1}\to\mathbb{C}^{M\times M}$ be a matrix-valued function with 
differentiable entries satisfying the property that there exists $C\in(0,\infty)$ such that
\begin{equation}\label{phi-Poisson}
|\phi(x')|+|\nabla\phi(x')|\leq\frac{C}{1+|x'|^n},
\qquad\mbox{for every }x'\in\mathbb{R}^{n-1}.
\end{equation}
Pick a function $g=(g_\alpha)_{1\leq\alpha\leq M}$ with components in 
$H^1_a(\mathbb{R}^{n-1})$. Then there exists a constant $C_g\in(0,\infty)$, 
depending on $g$, such that
\begin{equation}\label{phit*g}
\sup_{t>0}\Big|\big(\phi_t*g\big)(x')\Big|\leq\frac{C_g}{1+|x'|^n}
\quad\mbox{ for every }\,x'\in\mathbb{R}^{n-1}.
\end{equation}
\end{lemma}

\begin{proof}
Take $R=R_g\ge 1$ sufficiently large so that (recall \eqref{balnminus1})
$\mathrm{supp\,} g\subset B_{n-1}(0',R)=:B$. 
In the case when $x'\in 2B$, for each $t>0$ we have
\begin{align}\label{hdfts-3}
\Big|\big(\phi_t*g\big)(x')\Big|
&\leq\|g\|_{L^\infty(\mathbb{R}^{n-1})}\,\|\phi_t\|_{L^1(\mathbb{R}^{n-1})}
=\|g\|_{L^\infty(\mathbb{R}^{n-1})}\,\|\phi\|_{L^1(\mathbb{R}^{n-1})}
\nonumber\\[6pt]
&\leq\frac{\|g\|_{L^\infty(\mathbb{R}^{n-1})}\,\|\phi\|_{L^1(\mathbb{R}^{n-1})}
(1+(2R)^n)}{1+|x'|^n}.
\end{align}
Corresponding to $x'\notin 2\,B$, first we use that $g$ has vanishing integral
and its support condition to write
\begin{align}\label{eqn-phi-g-off}
\Big|\big(\phi_t*g\big)(x')\Big|
&\leq\int_{\mathbb{R}^{n-1}}\big|\phi_t(x'-y')-\phi_t(x')\big|\,|g(y')|\,dy'
\nonumber\\[4pt]
&\leq\int_{B}\big|\phi_t(x'-y')-\phi_t(x')\big|\,|g(y')|\,dy'.
\end{align}
Next, we estimate the integrand in the right hand-side. By recalling \eqref{phisubt},
an application of the Mean Value Theorem combined with \eqref{phi-Poisson}, 
for each $x'\notin 2\,B$, $y'\in B$, and $t>0$, allows us to write 
\begin{align}\label{eqn-phi-g-aa}
\big|\phi_t(x'-y')-\phi_t(x')\big|
&\leq t^{1-n}\,\frac{|y'|}{t}\,\sup_{\theta\in[0,1]}
\big|\nabla\phi\big((x'-\theta\,y')/t\big)\big|
\nonumber\\[4pt]
&\leq C|y'|\,\sup_{\theta\in[0,1]}\frac{1}{|x'-\theta\,y'|^n}.
\end{align}
Moreover, whenever $x'\notin 2\,B$ and $y'\in B$, for each $\theta\in[0,1]$ we have
\begin{equation}\label{hdfts-5}
|x'|\leq |x'-\theta\,y'|+\theta\,|y'|
\leq|x'-\theta\,y'|+R
\leq|x'-\theta\,y'|+\frac12\,|x'|,
\end{equation}
which implies $|x'-\theta\,y'|\geq(1/2)|x'|\geq(1/3)(1+|x'|)$. 
The latter when used in \eqref{eqn-phi-g-aa} in combination with \eqref{eqn-phi-g-off}
implies
\begin{equation}\label{hdfts-56}
\big|(\phi_t*g)(x')\big|\leq R\,\frac{C\|g\|_{L^1(\mathbb{R}^{n-1})}}{1+|x'|^n},
\qquad\forall\,x'\in{\mathbb{R}}^{n-1}\setminus(2B),\,\,\forall\,t>0.
\end{equation}
Now the desired conclusion follows from \eqref{hdfts-3} and \eqref{hdfts-56}
by taking 
\begin{equation}\label{GFwwf}
C_g:=\max\big\{\|g\|_{L^\infty(\mathbb{R}^{n-1})}\,\|\phi\|_{L^1(\mathbb{R}^{n-1})}
(1+(2R)^n)\,,\,CR\|g\|_{L^1(\mathbb{R}^{n-1})}\big\}.
\end{equation}
The proof of the lemma is therefore complete. 
\end{proof}

Our next preparatory lemma is needed in the proof of Proposition~\ref{prop:Car->BMO}.

\begin{lemma}\label{prop:Pt-Delta}
Let $L$ be an $M\times M$ elliptic system with constant complex coefficients as in
\eqref{L-def}-\eqref{L-ell.X} and consider $P^L$, the associated Poisson kernel 
for $L$ in $\mathbb{R}^{n}_+$ from Theorem~\ref{kkjbhV}, 
as well as $K^L$, defined in \eqref{eq:Gvav7g5}. Write 
\begin{equation}\label{Jhxg}
\Phi(x'):=(\partial_n K^L)(x',1)\,\,\text{ for every $x'\in{\mathbb{R}}^{n-1}$},
\end{equation}
and, whenever $0<a<b<\infty$, also set
\begin{equation}\label{def:Psi-a-b}
\Psi_{a,b}(x'):=4\,\int_a^b(\Phi_t*\Phi_t)(x')\,\frac{dt}{t},\qquad
\forall\,x'\in\mathbb{R}^{n-1}.
\end{equation}
Then, whenever $0<a<b<\infty$ there holds
\begin{equation}\label{Phi-CRF}
\Psi_{a,b}(x')=\Phi_{2b}(x')-P_{2b}^L(x')-\Phi_{2a}(x')+P_{2a}^L(x'),
\qquad\forall\,x'\in\mathbb{R}^{n-1}.
\end{equation}
\end{lemma}

\begin{proof}
Since $\nabla K^L$ is homogeneous of order $-n$ (recall item {\it (5)} 
in Theorem~\ref{kkjbhV}), for every $(x',t)\in{\mathbb{R}}^n_{+}$ we may write 
\begin{equation}\label{jsgag-bbA}
\Phi_t(x')=t^{1-n}\Phi(x'/t)=t^{1-n}(\partial_n K^L)(x'/t,1)
=t(\partial_n K^L)(x',t)=t\partial_t K^L(x',t).
\end{equation}
Consequently, in view of definition \eqref{eq:Gvav7g5}, in the current notation we have
\begin{equation}\label{jsgag-bb}
\Phi_t(x')=t\partial_t\big[P_t^L(x')],\qquad
\forall\,(x',t)\in{\mathbb{R}}^n_{+}. 
\end{equation}
Fix $h\in\mathscr{C}^\infty_0({\mathbb{R}}^{n-1},{\mathbb{C}}^M)$. 
Observe that
\begin{equation}\label{hsfdae}
\big(\Psi_{a,b}*h\big)(x')=4\,\int_a^b\int_{\mathbb{R}^{n-1}}\int_{\mathbb{R}^{n-1}}
\Phi_t(x'-z'-y')\,\Phi_t(z')\,h(y')\,\,dz'\,dy'\,\frac{dt}{t}
\end{equation}
since the triple integral is absolutely convergent in view of the assumptions 
made on $h$ and \eqref{eq:Kest}. Set $u(x',t):=\big(P_t^L*h\big)(x')$ for each 
$(x',t)\in\mathbb{R}^n_{+}$ and in light of \eqref{jsgag-bb}
further write \eqref{hsfdae} in the form
\begin{align}\label{eqn:Pt-delta}
\big(\Psi_{a,b}*h\big)(x')
&=4\,\int_a^b\int_{\mathbb{R}^{n-1}}\int_{\mathbb{R}^{n-1}}
\partial_n K^L(x'-z'-y',t)\,\partial_n K^L(z',t)\,h(y')\,\,dz'\,dy'\,t\,dt
\nonumber\\[4pt]
&=4\,\int_a^b\int_{\mathbb{R}^{n-1}}\int_{\mathbb{R}^{n-1}}
\partial_n K^L(z',t)\,\partial_n K^L(x'-z'-y',t)\,h(y')\,\,dy'\,dz'\,t\,dt
\nonumber\\[4pt]
&=4\,\int_a^b\int_{\mathbb{R}^{n-1}}\partial_n K^L(z',t)\,\partial_n u(x'-z',t)\,dz'\,t\,dt.
\end{align}
Next, for every $(x',t)\in\mathbb{R}_+^n$, define $v(x',t):=(\partial_n u)(x',t)$. 
By part {\it (7)} in Theorem~\ref{kkjbhV} we have that  
$u\in\mathscr{C}^\infty(\mathbb{R}_+^n,{\mathbb{C}}^M)$ and $Lu=0$ in ${\mathbb{R}}^n_{+}$. 
In turn, these imply $v\in\mathscr{C}^\infty(\mathbb{R}_+^n,{\mathbb{C}}^M)$ 
and $Lv=0$ in ${\mathbb{R}}^n_{+}$.

Moving on, for each $s>0$ set $v_s(x',t):=v(x',t+s)$ for every $(x',t)\in\mathbb{R}^n_+$. 
Then we have $v_s\in\mathscr{C}^\infty(\overline{\mathbb{R}_+^n},{\mathbb{C}}^M)$ and $Lv_s=0$
in ${\mathbb{R}}^n_{+}$. Now recall \eqref{NT-Fct}. 
For $\kappa\in(0,\infty)$ arbitrary, if $x'\in\mathbb{R}^{n-1}$ is fixed, 
Theorem~\ref{ker-sbav} allows us to estimate
\begin{align}\label{nduist}
|v_s(y',t)|
&=\big|(\partial_n u)(y',t+s)\big|
\leq\frac{C}{s}\,\aver{B((y',t+s),\kappa s/\sqrt{1+\kappa^2})} |u|\,d{\mathscr{L}}^n
\nonumber\\[4pt]
&\leq\frac{C}{s}\,{\mathcal{N}}u(x'),\qquad\forall\,(y',t)\in\Gamma_\kappa(x'),
\end{align}
where for the last inequality we have used that 
$B\big((y',t+s),\kappa s/\sqrt{1+\kappa^2}\big)\subset\Gamma_\kappa(x')$. Hence, 
\eqref{nduist} combined with \eqref{exTGFVC} yields
\begin{equation}\label{nduist-2}
(\mathcal{N}v_s)(x')\leq\frac{C}{s}\,(\mathcal{N}u)(x')
\leq\frac{C}{s}\,(\mathcal{M}h)(x'),\qquad\forall\,x'\in{\mathbb{R}}^{n-1}.
\end{equation}
Upon recalling that $h\in\mathscr{C}^\infty_0({\mathbb{R}}^{n-1},{\mathbb{C}}^M)$ and that 
the Hardy-Littlewood maximal operator is bounded on $L^p(\mathbb{R}^{n-1})$
for $p\in(1,\infty)$, from \eqref{nduist-2} we may infer that 
$\mathcal{N}v_s\in L^p(\mathbb{R}^{n-1})$ for every $p\in(1,\infty)$. 
In view of all these, we may apply Corollary~\ref{tuFatou.Lp} to $v_s$ and obtain that 
for each $s\in(0,\infty)$
\begin{align}\label{nduist-3}
v_s(x',t)
&=\big(P^L_t\ast(v_s\bigl|_{\partial\mathbb{R}^{n}_{+}})\big)(x')
=\int_{\mathbb{R}^{n-1}} P_t^L(z')\,v_s(x'-z',0)\,dz'
\nonumber\\[4pt]
&=\int_{\mathbb{R}^{n-1}} K^L(z',t)\,v_s(x'-z',0)\,dz'
\nonumber\\[4pt]
&=\int_{\mathbb{R}^{n-1}} K^L(z',t)\,(\partial_n u)(x'-z',s)\,dz',
\qquad\forall\,(x',t)\in\mathbb{R}_+^n.
\end{align}
Thus, for every $(x',t)\in\mathbb{R}_+^n$ and every $s>0$ we have
\begin{equation}\label{nduist-4}
(\partial_n^2 u)(x',t+s)=\partial_n v_s(x',t)
=\int_{\mathbb{R}^{n-1}}\partial_n K^L(z',t)\,(\partial_n u)(x'-z',s)\,dz'.
\end{equation}
Applying \eqref{nduist-4} with $s=t$, substituting the resulting equality into 
\eqref{eqn:Pt-delta}, and making use of \eqref{jsgag-bb} we obtain
\begin{align}\label{eqn:Pt-delta-inte}
\big(\Psi_{a,b}*h\big)(x')
&=4\,\int_a^b(\partial_n^2 u)(x',2t)\,t\,dt
=4\,\Big[\frac{t}2 (\partial_n u)(x',2t)-\frac14 u(x',2t)\Big]_{t=a}^{t=b}
\nonumber\\[4pt]
&=\Big[\Phi_{2t}*h(x')-P_{2t}^L*h(x')\Big]_{t=a}^{t=b}.
\end{align}
This readily yields 
\begin{equation}\label{nduist-6}
\big(\Psi_{a,b}*h\big)(x')=\big(\Phi_{2b}*h\big)(x')
-\big(P_{2b}^L*h\big)(x')-\big(\Phi_{2a}*h\big)(x')+\big(P_{2a}^L*h\big)(x')
\end{equation}
for every $x'\in\mathbb{R}^{n-1}$. Note that \eqref{nduist-6} holds for every 
$h\in{\mathscr{C}}^\infty_0({\mathbb{R}}^{n-1},{\mathbb{C}}^M)$ 
and therefore \eqref{Phi-CRF} holds for a.e. $x'\in\mathbb{R}^{n-1}$.
In addition, by Theorem~\ref{kkjbhV} and the fact that $0<a<b<\infty$ we 
see that both sides of \eqref{Phi-CRF} are continuous functions in $\mathbb{R}^{n-1}$. 
Consequently, the desired equality holds everywhere. The proof of the lemma is complete.
\end{proof}

Given a Lebesgue measurable function 
$F:\mathbb{R}^n_+\rightarrow\mathbb{C}$, for every $x'\in\mathbb{R}^{n-1}$ introduce
the Lusin area-function 
\begin{equation}\label{eq:def:AF}
(\mathcal{A}F)(x'):=\Big(\int_{\Gamma_\kappa(x')} 
|F(y',t)|^2\,\frac{dy'\,dt}{t^n}\Big)^{\frac12}
\end{equation}
and the Carleson operator
\begin{equation}\label{eq:def:CF}
(\mathcal{C}F)(x'):=\sup_{Q\ni x'}\Big(
\int_0^{\ell(Q)}\aver{Q}|F(y',t)|^2\,\frac{dy'\,dt}{t}\Big)^{\frac12}.
\end{equation}
%


In relation to these operators we recall a result from \cite[Theorem~1, p.\,313]{CMS}.

\begin{lemma}\label{lemma:tent}
There exists some constant $C\in(0,\infty)$, which depends only on $n$ and $\kappa$, 
with the property that for any Lebesgue measurable functions 
$F,G:\mathbb{R}^n_+\rightarrow\mathbb{C}$ there holds
\begin{equation}\label{eqn:CMS}
\int_{{\mathbb{R}}^n_{+}}|F(x',t)\,G(x',t)|\,\frac{dx'dt}{t}
\leq C\,\int_{\mathbb{R}^{n-1}}\mathcal{C}F(x')\,\mathcal{A}G(x')\,dx'.
\end{equation}
\end{lemma}

Strictly speaking, the statement in \cite{CMS} contains as assumptions the additional requirements 
$\mathcal{C}F\in L^\infty({\mathbb{R}}^{n-1})$ and $\mathcal{A}G\in L^1({\mathbb{R}}^{n-1})$. However, these extra assumptions may be eliminated a posteriori via a suitable limiting argument. Specifically, for each $N\in{\mathbb{N}}$ introduce
\begin{equation}\label{tgvRTF}
D_N:=\big\{(x',t)\in{\mathbb{R}}^n_{+}:\,|(x',t)|<N,\,\,t>1/N\big\}
\end{equation}
and for a generic function $f:\mathbb{R}^n_{+}\rightarrow\mathbb{C}$ define 
$f_N:\mathbb{R}^n_+\rightarrow\mathbb{C}$ by setting $f_N(x):=f(x)$ if $x\in D_N$ and
$|f(x)|\leq N$ and $f_N(x):=0$ if either $x\in{\mathbb{R}}^n\setminus D_N$ or $|f(x)|>N$, 
for each $x\in{\mathbb{R}}^n_{+}$. Then, given $F,G:\mathbb{R}^n_+\rightarrow\mathbb{C}$ 
arbitrary Lebesgue measurable functions, for each $N\in{\mathbb{N}}$ the functions 
$F_N$, $G_N$ are Lebesgue measurable and bounded. It is also immediate from definitions 
that $\mathcal{C}F_N\in L^\infty({\mathbb{R}}^{n-1})$ and 
$\mathcal{A}G_N\in L^\infty_{\rm comp}({\mathbb{R}}^{n-1})\subset L^1({\mathbb{R}}^{n-1})$.
Based on \cite[Theorem~1, p.\,313]{CMS} and the monotonicity of the operators $\mathcal{C}$ 
and $\mathcal{A}$ (with respect to the absolute value of the function to which they are applied) 
we may write
\begin{align}\label{eqnhc-S}
\int_{{\mathbb{R}}^n_{+}}|F_N(x',t)\,G_N(x',t)|\,\frac{dx'dt}{t}
&\leq C\,\int_{\mathbb{R}^{n-1}}\mathcal{C}F_N(x')\,\mathcal{A}G_N(x')\,dx'
\nonumber\\[4pt]
&\leq C\,\int_{\mathbb{R}^{n-1}}\mathcal{C}F(x')\,\mathcal{A}G(x')\,dx'.
\end{align}
Now \eqref{eqn:CMS} follows by taking the limit as $N\to\infty$ of the inequality 
resulting from \eqref{eqnhc-S} and applying Lebesgue's Monotone Convergence Theorem.

For further reference we also prove the following companion to Lemma~\ref{lemma:tent}.

\begin{lemma}\label{lexbdg}
There exists some constant $C\in(0,\infty)$ {\rm (}depending only on $n$ and $\kappa${\rm )} 
such that for any two Lebesgue measurable functions $F,G:\mathbb{R}^n_+\rightarrow\mathbb{C}$ 
one has
\begin{equation}\label{ebchd}
\int_{{\mathbb{R}}^n_{+}}|F(x',t)\,G(x',t)|\,\frac{dx'dt}{t}
\leq C\int_{{\mathbb{R}}^{n-1}}\mathcal{A}F(x')\mathcal{A}G(x')\,dx'.
\end{equation}
\end{lemma}

\begin{proof}
The idea is to estimate the expression 
\begin{align}\label{kcgds-treer}
I:=\int_{{\mathbb{R}}^{n-1}}\Big(\int_{\Gamma_{\kappa}(x')}
|F(y',t)\,G(y',t)|\,\frac{dy'dt}{t^n}\Big)\,dx'
\end{align}
in two ways. On the one hand, using Fubini's Theorem we may write
\begin{align}\label{kcgds}
I &=\int_{{\mathbb{R}}^n_{+}}|F(y',t)||G(y',t)|
\Big(\int_{{\mathbb{R}}^{n-1}}{\mathbf{1}}_{\Gamma_{\kappa}(x')}(y',t)\,dx'\Big)\frac{dy'\,dt}{t^n}
\nonumber\\[6pt]
&=C_{\kappa,n}\int_{{\mathbb{R}}^n_{+}}|F(y',t)||G(y',t)|\frac{dy'\,dt}{t}.
\end{align}
On the other hand, based on Cauchy-Schwarz' inequality we may estimate
\begin{align}\label{kcgds-treer.2}
I &\leq\int_{{\mathbb{R}}^{n-1}}\Big(\int_{\Gamma_{\kappa}(x')}
|F(y',t)|^2\,\frac{dy'dt}{t^n}\Big)^{1/2}
\Big(\int_{\Gamma_{\kappa}(x')}|G(y',t)|^2\,\frac{dy'dt}{t^n}\Big)^{1/2}\,dx'
\nonumber\\[6pt]
&=\int_{{\mathbb{R}}^{n-1}}\mathcal{A}F(x')\mathcal{A}G(x')\,dx'.
\end{align}
Now, \eqref{ebchd} follows from \eqref{kcgds} and \eqref{kcgds-treer.2}.
\end{proof}

To state the final preparatory lemma required in the proof of Proposition~\ref{prop:Car->BMO},
one more piece of notation is needed. In the sequel, $A^\top$ denotes the transpose of a given matrix $A$.

\begin{lemma}\label{lemma:SFE-H1}
Let $L$ be an $M\times M$ elliptic system with constant complex coefficients as in
\eqref{L-def}-\eqref{L-ell.X} and consider $P^L$, the associated Poisson kernel 
for $L$ in $\mathbb{R}^{n}_+$ from Theorem~\ref{kkjbhV},
as well as $K^L$ as in \eqref{eq:Gvav7g5}. Recall $\Phi$ from \eqref{Jhxg} and 
for each $x'\in\mathbb{R}^{n-1}$ set $\widetilde{\Phi}(x'):=\Phi^{\top}(-x')$.
Furthermore fix $\kappa\in(0,\infty)$ arbitrary and, given a function 
$f=(f_\beta)_{1\leq\beta\leq M}:\mathbb{R}^{n-1}\rightarrow \mathbb{C}^M$
with Lebesgue measurable entries, define for each $x'\in{\mathbb{R}}^{n-1}$
\begin{align}\label{hxewq}
(S_{\widetilde{\Phi}}f)(x')
& :=\Bigg(\int_{\Gamma_\kappa(x')}\big|(\widetilde{\Phi}_t*f)(y')\big|^2
\,\frac{dy'\,dt}{t^n}\Bigg)^{\frac12}
\nonumber\\[4pt]
&=\left(\Bigg(\int_{\Gamma_\kappa(x')} 
\sum_{\beta=1}^M\big|
\big((\widetilde{\Phi}_t)_{\alpha\beta}*f_\beta\big)(y')\big|^2
\,\frac{dy'\,dt}{t^n}\Bigg)^{\frac12}\right)_{1\leq\alpha\leq M}.
\end{align}

Then $S_{\widetilde{\Phi}}$ is a bounded operator from $H^1(\mathbb{R}^{n-1},\mathbb{C}^M)$ 
into $L^1(\mathbb{R}^{n-1})$.
\end{lemma}

\begin{proof}
For each $\alpha,\beta\in\{1,\dots,M\}$, write 
$\theta_{\alpha\beta}(x',t;y'):=t\,\partial_n K_{\beta\alpha}^L(y'-x',t)$ for every
$x',y'\in{\mathbb{R}}^{n-1}$ and $t>0$, and denote by  
$\Theta_{\alpha\beta}$ the integral operator as in \eqref{defi:Theta}
corresponding to $\theta_{\alpha\beta}$ in place of $\theta$. 
Notice that \eqref{est-theta-K}, \eqref{est-theta-K-nabla} and \eqref{est-theta-vanish} 
(with $j=n$ and the roles of $\alpha$ and $\beta$ reversed) 
allow us to apply Proposition~\ref{prop:SFE-early} and write
\begin{align}\label{hsrewTT}
\|S_{\widetilde{\Phi}}f\|_{L^1(\mathbb{R}^{n-1})}
&\leq\sum_{1\leq\alpha,\beta\le M}\|S_{\Theta_{\alpha\beta}} f_\beta\|_{L^1(\mathbb{R}^{n-1})}
\nonumber\\[4pt]
&\leq C\sum_{1\le\beta\le M}\|f_\beta\|_{H^1(\mathbb{R}^{n-1})}
=C\|f\|_{H^1(\mathbb{R}^{n-1},\mathbb{C}^M)}.
\end{align}
The desired conclusion now follows from \eqref{hsrewTT}.
\end{proof}

We have seen in Proposition~\ref{prop-Dir-BMO:exis} part {\it (e)} that if 
$f\in\mathrm{BMO}(\mathbb{R}^{n-1},\mathbb{C}^M)$ then the Littlewood-Paley measure 
$|\nabla u(x',t)|^2\,t\,dx'dt$ associated with the function $u$ defined as in 
\eqref{eqn-Dir-BMO:u:prop} is a Carleson measure in ${\mathbb{R}}^n_{+}$ 
(cf. \eqref{defi-Carleson}). In the proposition below we shall establish the converse 
implication along with the estimate which naturally accompanies this statement. 
In the proof, Lemmas~\ref{lemma:phit-g}-\ref{lemma:SFE-H1} as well as the fundamental 
duality result from \cite{FS} asserting that
\begin{equation}\label{jcgsfSEW}
\left(H^1(\mathbb{R}^{n-1},\mathbb{C}^M)\right)^*
=\widetilde{\rm BMO}(\mathbb{R}^{n-1},\mathbb{C}^M)
\end{equation}
are going to play a key role. 

\begin{proposition}\label{prop:Car->BMO}
Let $L$ be an $M\times M$ elliptic system with constant complex coefficients as in
\eqref{L-def}-\eqref{L-ell.X} and consider $P^L$, the associated Poisson kernel 
for $L$ in $\mathbb{R}^{n}_+$ from Theorem~\ref{kkjbhV}, together with $K^L$ 
as in \eqref{eq:Gvav7g5}. Recall $\Phi$ from \eqref{Jhxg}.
Let $f\in L^1\Big({\mathbb{R}}^{n-1}\,,\,\frac{1}{1+|x'|^n}\,dx'\Big)^M$
and consider the measure in $\mathbb{R}_+^{n}$ defined by
\begin{equation}\label{Car-BMO:mu}
d\mu(x',t):=\big|\big(\Phi_t*f\big)(x')\big|^2\,\frac{dx'\,dt}{t}.
\end{equation}
Then whenever $\mu$ is a Carleson measure, that is,
\begin{equation}\label{Car-BMO:Car}
\|\mu\|_{\mathcal{C}(\mathbb{R}_+^{n})}
=\sup_{Q\subset\mathbb{R}^{n-1}}\int_{0}^{\ell(Q)}
\aver{Q}\big|\big(\Phi_t*f\big)(x')\big|^2\,\frac{dx'\,dt}{t}<\infty,
\end{equation}
one necessarily has $f\in{\mathrm{BMO}}(\mathbb{R}^{n-1},\mathbb{C}^M)$ and
\begin{equation}\label{Car-BMO:BMO}
\|f\|_{\mathrm{BMO}(\mathbb{R}^{n-1},\mathbb{C}^M)}^2
\leq C\,\|\mu\|_{\mathcal{C}(\mathbb{R}_+^{n})}
\end{equation}
for some constant $C\in(0,\infty)$ independent of $f$. 
\end{proposition}

\begin{proof}
Fix a function $f$ as in the hypotheses of the proposition and suppose $\mu$ 
satisfies \eqref{Car-BMO:Car}. Let $g\in H^1_a(\mathbb{R}^{n-1})$ (see \eqref{Car-g-dual}) 
and for some arbitrary $\alpha_0\in\{1,\dots,M\}$ define 
\begin{equation}\label{jxgsrr}
h:=(g\,\delta_{\alpha\alpha_0})_{1\leq\alpha\leq M}\in\big[H^1_a(\mathbb{R}^{n-1})\big]^M
\subset H^1({\mathbb{R}}^{n-1},{\mathbb{C}}^M),
\end{equation}
where $\delta_{\alpha\alpha_0}$ denotes the standard Kronecker symbol.

Next, recall the expression of the classical harmonic 
Poisson kernel (that is, the Poisson kernel associated with the Laplacian $\Delta$)
\begin{equation}\label{Uah-TTT}
P^{\Delta}(x'):=\frac{2}{\omega_{n-1}}\frac{1}{\big(1+|x'|^2\big)^{\frac{n}{2}}},
\qquad\forall\,x'\in{\mathbb{R}}^{n-1},
\end{equation}
where $\omega_{n-1}$ stands for the area of the unit sphere in ${\mathbb{R}}^n$. 
Then the definition of $\Phi$, \eqref{eq:Kest} in Theorem~\ref{kkjbhV}, 
and \eqref{Uah-TTT} imply
\begin{equation}\label{bdgsfs}
\big|\Phi_t(x')\big|\leq C P_t^\Delta(x'),
\quad\forall\,x'\in{\mathbb{R}}^{n-1},\,\,\forall\,t\in(0,\infty).
\end{equation}
Also, by the semigroup property 
(cf., e.g., \cite[(vi), p.\,62]{St70}, or part {\it (8)} in Theorem~\ref{kkjbhV}), 
for every $\varepsilon\in(0,1)$ and every $t\in(\varepsilon,\varepsilon^{-1})$ we have
\begin{equation}\label{nduist-7}
P^\Delta_t*P^\Delta_t=P^\Delta_{2t}\leq C_\varepsilon\,P^\Delta.
\end{equation}
Combining \eqref{bdgsfs} and \eqref{nduist-7}, for each $\varepsilon\in(0,1)$ we may write
\begin{align}\label{adedewd}
&\int_\varepsilon^{\varepsilon^{-1}}\!\!
\int_{\mathbb{R}^{n-1}}\int_{\mathbb{R}^{n-1}}\int_{\mathbb{R}^{n-1}}
|\Phi_t(x'-y'-z')|\,|\Phi_t(z')|\,|f(y')|\,|h(x')|\,dz'\,dy'\,dx'\,\frac{dt}{t}
\nonumber\\[4pt]
&\qquad\leq C\,\int_\varepsilon^{\varepsilon^{-1}}\!\!
\int_{\mathbb{R}^{n-1}}\int_{\mathbb{R}^{n-1}}\int_{\mathbb{R}^{n-1}}
P^\Delta_t(x'-y'-z')\,P^\Delta_t(z')\,|f(y')|\,|g(x')|\,dz'\,dy'\,dx'\,\frac{dt}{t}
\nonumber\\[4pt]
&\qquad\leq C_{\varepsilon}\,
\int_{\mathbb{R}^{n-1}}\int_{\mathbb{R}^{n-1}}P^\Delta(x'-y')\,|f(y')|\,|g(x')|\,dy'\,dx'
\nonumber\\[4pt]
&\qquad\leq C_\varepsilon\,\Big(\int_{\mathbb{R}^{n-1}} (1+|x'|^n)\,|g(x')|\,dx'\Big)\, 
\Big(\int_{\mathbb{R}^{n-1}}\frac{|f(y')|}{1+|y'|^n}\,dx'\Big)<\infty,
\end{align}
where for the last inequality we have used the fact that 
 $1+|y'|\leq  (1+|x'|)\,(1+|x'-y'|)$ for every $x',y'\in{\mathbb{R}}^{n-1}$,
while the finiteness of the rightmost term in \eqref{adedewd} follows from
our assumptions on $f$ and $g$. Thus, recalling the definition of 
$\Psi_{\varepsilon,\varepsilon^{-1}}$ from \eqref{def:Psi-a-b} we have that
\begin{align}\label{hsrewss}
&\int_{\mathbb{R}^{n-1}}\big\langle\big(\Psi_{\varepsilon,\varepsilon^{-1}}*f\big)(x'),
h(x')\big\rangle\,dx'
\nonumber\\[4pt]
&\hskip 0.50in
=\int_\varepsilon^{\varepsilon^{-1}}
\int_{\mathbb{R}^{n-1}}\int_{\mathbb{R}^{n-1}}\int_{\mathbb{R}^{n-1}}
\big\langle\Phi_t(x'-y'-z')\,\Phi_t(z')\,f(y'),h(x')\big\rangle 
\,dz'\,dy'\,dx'\,\frac{dt}{t}
\end{align}
is an absolutely convergent integral. Here and elsewhere we use the notation
\begin{equation}\label{hsreccxs}
\langle\lambda,\lambda'\rangle:=\sum_{\alpha=1}^M\lambda_\alpha\,\lambda_\alpha',
\qquad
\lambda=(\lambda_\alpha)_{1\leq\alpha\le M},\,\,
\lambda'=(\lambda_\alpha')_{1\leq\alpha\leq M}\in\mathbb{C}^M.
\end{equation}
To continue, we introduce the (matrix valued) functions 
\begin{equation}\label{i7g5f}
\widetilde{\Phi}(x'):=\Phi^{\top}(-x'),\quad 
\widetilde{\Psi}_{\varepsilon,\varepsilon^{-1}}(x')
:=\Psi_{\varepsilon,\varepsilon^{-1}}^\top(-x'),\,\,\text{ and }\,\, 
\widetilde{P}^L(x'):=(P^L)^{\top}(-x'), 
\end{equation}
defined for every $x'\in\mathbb{R}^{n-1}$. Then, for every $\varepsilon>0$, we may write
\begin{align}\label{Phi-Tent}
\left|\int_{\mathbb{R}^{n-1}}\big\langle f(x'),
\big(\widetilde{\Psi}_{\varepsilon,\varepsilon^{-1}}*h\big)(x')\big\rangle\,dx'\right|
&=\left|\int_{\mathbb{R}^{n-1}} 
\big\langle\big(\Psi_{\varepsilon,\varepsilon^{-1}}*f\big)(x'),h(x')\big\rangle\,dx'\right|
\nonumber\\[4pt]
&=\left|\int_{\varepsilon}^{\varepsilon^{-1}}\!\!\int_{\mathbb{R}^{n-1}} 
\big\langle\big(\Phi_t*\Phi_t* f\big)(x'),h(x')\big\rangle\,dx'\,\frac{dt}{t}\right|
\nonumber\\[4pt]
&=\left|\int_{\varepsilon}^{\varepsilon^{-1}}\!\!\int_{\mathbb{R}^{n-1}} 
\big\langle\big(\Phi_t* f\big)(x'),\big(\widetilde{\Phi}_t*h\big)(x')\big\rangle
\,dx'\,\frac{dt}{t}\right|
\nonumber\\[4pt]
&=\left|\int_{\varepsilon}^{\varepsilon^{-1}}\!\!\int_{\mathbb{R}^{n-1}} 
\langle F(x',t),H(x',t)\rangle\,dx'\,\frac{dt}{t}\right|
\nonumber\\[4pt]
&\leq\int_{\mathbb{R}_+^{n}}  
\big|\langle F(x',t),H(x',t)\rangle\big|\,dx'\,\frac{dt}{t}
\end{align}
where $F(x',t):=(\Phi_t*f)(x')$ and $H(x',t):=(\widetilde{\Phi}_t*h)(x')$
for every $(x',t)\in{\mathbb{R}}^n_{+}$. 
Denote by $(F_\alpha)_{1\leq\alpha\leq M}$ and $(H_\alpha)_{1\leq\alpha\leq M}$ the 
scalar components of $F$ and $H$, respectively. Note that \eqref{eq:def:CF},
the definition of $F$, and \eqref{Car-BMO:Car} imply
\begin{equation}\label{mhsf}
\|\mathcal{C}F_\alpha\|_{L^\infty({\mathbb{R}}^{n-1})}
\leq\|\mu\|_{\mathcal{C}(\mathbb{R}_+^{n})}^{\frac{1}{2}}<\infty,
\qquad\forall\,\alpha\in\{1,\dots,M\}.
\end{equation}
Also, \eqref{jxgsrr}, \eqref{eq:def:AF}, and Lemma~\ref{lemma:SFE-H1} permit us to write 
\begin{align}\label{jdfrae}
\|\mathcal{A}H_\alpha\|_{L^1({\mathbb{R}}^{n-1})}
&\leq\|S_{\widetilde{\Phi}}h\|_{L^1(\mathbb{R}^{n-1})}
\leq C\|h\|_{H^1(\mathbb{R}^{n-1},\mathbb{C}^M)}
\nonumber\\[4pt]
&=C\|g\|_{H^1(\mathbb{R}^{n-1})},
\qquad\forall\,\alpha\in\{1,\dots,M\}.
\end{align}
Consequently, Lemma~\ref{lemma:tent}, \eqref{mhsf}, and \eqref{jdfrae} allow us to estimate
\begin{align}\label{RWWQD-3}
\int_{\mathbb{R}_+^{n}}\big|\langle F(x',t),H(x',t)\rangle\big|\,dx'\,\frac{dt}{t}
&\leq\sum_{\alpha=1}^M\int_{\mathbb{R}_+^{n}}  
\big|F_\alpha(x',t)\,H_\alpha(x',t)\big|\,dx'\,\frac{dt}{t}
\nonumber\\[4pt]
&\leq C\,\sum_{\alpha=1}^M
\int_{\mathbb{R}^{n-1}}\mathcal{C}F_\alpha(x')\,\mathcal{A}H_\alpha(x')\,dx'
\nonumber\\[4pt]
&\leq C\,\sum_{\alpha=1}^M\|\mathcal{C}F_\alpha\|_{L^\infty(\mathbb{R}^{n-1})}\,
\|\mathcal{A}H_\alpha\|_{L^1(\mathbb{R}^{n-1})}
\nonumber\\[4pt]
&=C\,\|\mu\|_{\mathcal{C}(\mathbb{R}_+^{n})}^\frac12\,
\|g\|_{H^1(\mathbb{R}^{n-1})}.
\end{align}

At this point we make the claim that 
\begin{equation}\label{fg-lim}
\lim_{\varepsilon\to 0^+}\int_{\mathbb{R}^{n-1}}\big\langle f(x'),
(\widetilde{\Psi}_{\varepsilon,\varepsilon^{-1}}\ast h)(x')\big\rangle\,dx'
=\int_{\mathbb{R}^{n-1}}\big\langle f(x'),h(x')\big\rangle\,dx'.
\end{equation}
The idea is to show that Lebesgue's Dominated Convergence Theorem applies in our setting.
With this goal in mind, first observe that by part {\it (5)} in Theorem~\ref{kkjbhV}, 
for every multi-index $\alpha\in {\mathbb{N}}_0^{n-1}$, we have
\begin{equation}\label{derv-Phi}
|\partial^\alpha\Phi(x')|=|\partial^\alpha\partial_n K^L(x',1)|
\leq C_\alpha|(x',1)|^{-n-|\alpha|},
\end{equation}
hence $\widetilde{\Phi}$ satisfies the hypotheses of Lemma~\ref{lemma:phit-g}. 
Moreover, by parts {\it (1)} and {\it (5)} in Theorem~\ref{kkjbhV} we also have that
$\widetilde{P}^L$ satisfies the hypotheses of Lemma~\ref{lemma:phit-g}. 
Hence, Lemma~\ref{lemma:phit-g} and \eqref{jxgsrr} give 
\begin{equation}\label{phihcdyds}
\sup_{t>0}\Big|\big(\widetilde{\Phi}_t*h\big)(x')\Big|
+\sup_{t>0}\Big|\big(\widetilde{P}^L_t*h\big)(x')\Big|
\leq\frac{C_h}{1+|x'|^n},\qquad\mbox{for every }x'\in\mathbb{R}^{n-1}.
\end{equation}
In light of \eqref{Phi-CRF}, the latter yields
\begin{equation}\label{hysewQ}
\sup_{0<\varepsilon<1}\big|\big(\widetilde{\Psi}_{\varepsilon,\varepsilon^{-1}}*h\big)(x')\big|
\leq\frac{C_h}{1+|x'|^n},\qquad\mbox{for every }x'\in\mathbb{R}^{n-1}.
\end{equation}
From this and the fact that 
$f\in L^1\Big({\mathbb{R}}^{n-1}\,,\,\frac{1}{1+|x'|^n}\,dx'\Big)^M$
we arrive at the conclusion that
\begin{equation}\label{hysewQ-bgsf}
\sup_{0<\varepsilon<1}\Big|
\big\langle f,\widetilde{\Psi}_{\varepsilon,\varepsilon^{-1}}\ast h\big\rangle\Big|
\in L^1({\mathbb{R}}^{n-1}).
\end{equation}
Next, we focus on the pointwise convergence of the functions under the integral
in the left hand-side of \eqref{fg-lim}. 
First, by \eqref{eq:IG6gy}, \eqref{exTGFVC.2s:app} in Lemma~\ref{lennii},
and \eqref{eq:IG6gy.2} in Theorem~\ref{kkjbhV} we obtain
\begin{equation}\label{eqn:Pt:t0}
\lim_{s\to 0^+}\big(\widetilde{P}_s^L*h\big)(x')
=\Big(\int_{\mathbb{R}^{n-1}}\widetilde{P}^L(y')\,dy'\Big)\,h(x')
=h(x'),\qquad\mbox{for a.e. } x'\in\mathbb{R}^{n-1}.
\end{equation}
Second, using a suitable change of variables, the properties of $h$, 
and Lebesgue's Dominated Convergence Theorem we have
\begin{equation}\label{eqn:Pt:tinfty}
\lim_{s\to\infty}\big(\widetilde{P}_s^L*h\big)(x')
=\lim_{s\to\infty}\int_{\mathbb{R}^{n-1}}\widetilde{P}^L(y')\,h(x'-sy')\,dy'=0.
\end{equation}
Third, by \eqref{XXdgsr} for every $t>0$ we have
\begin{equation}\label{RWWQD}
\int_{\mathbb{R}^{n-1}} (\partial_n K)(x',t)\,dx'=0,
\quad\forall\,t>0.
\end{equation}
which when specialized to $t=1$ yields 
\begin{equation}\label{Phi-vanish-int}
\int_{\mathbb{R}^{n-1}}\widetilde{\Phi}(x')\,dx'
=\Big(\int_{\mathbb{R}^{n-1}}\Phi(-x')\,dx'\Big)^{\top}
=\Big(\int_{\mathbb{R}^{n-1}}\Phi(x')\,dx'\Big)^{\top}=0.
\end{equation}
This, \eqref{derv-Phi}, and Lemma~\ref{lennii} applied to $\widetilde{\Phi}$ then give that
\begin{equation}\label{eqn:Phit:t0}
\lim_{s\to 0^+}\big(\widetilde{\Phi}_s*h\big)(x')
=\Big(\int_{\mathbb{R}^{n-1}}\widetilde{\Phi}(y')\,dy'\Big)\,h(x')=0
\quad\mbox{for a.e. }\,x'\in\mathbb{R}^{n-1}.
\end{equation}
Fourth, a suitable change of variables, the properties of $h$, 
and Lebesgue's Dominated Convergence Theorem also yield
\begin{equation}\label{eqn:Phit:tinfty}
\lim_{s\to\infty}\big(\widetilde{\Phi}_s*h\big)(x')
=\lim_{s\to\infty}\int_{\mathbb{R}^{n-1}}\widetilde{\Phi}(y')\,h(x'-sy')\,dy'=0.
\end{equation}
In concert, \eqref{eqn:Pt:t0}, \eqref{eqn:Pt:tinfty}, \eqref{eqn:Phit:t0}, 
\eqref{eqn:Phit:tinfty}, and \eqref{Phi-CRF} imply the pointwise convergence
\begin{equation}\label{RWWQD-2}
\lim_{\varepsilon\to 0^+}\big(\widetilde{\Psi}_{\varepsilon,\varepsilon^{-1}}*h\big)(x')
=h(x')\quad\mbox{for a.e. }\,x'\in\mathbb{R}^{n-1}.
\end{equation}
Having dispensed of \eqref{hysewQ-bgsf} and \eqref{RWWQD-2}, we may
apply Lebesgue's Dominated Theorem to write
\begin{align}\label{TCD-Psi}
\lim_{\varepsilon\to 0^+}\int_{\mathbb{R}^{n-1}} 
\big\langle f(x'),\big(\widetilde{\Psi}_{\varepsilon,\varepsilon^{-1}}*h\big)(x')\big\rangle\,dx'
&=\int_{\mathbb{R}^{n-1}}\Big\langle f(x'),\lim_{\varepsilon\to 0^+}
\big(\widetilde{\Psi}_{\varepsilon,\varepsilon^{-1}}\ast h\big)(x')\Big\rangle\,dx'
\nonumber\\[4pt]
&=\int_{\mathbb{R}^{n-1}}\langle f(x'),h(x')\rangle\,dx',
\end{align}
finishing the proof of the claim in \eqref{fg-lim}.

From the definition of $h$, \eqref{fg-lim}, \eqref{Phi-Tent}, and \eqref{RWWQD-3}
we may conclude that
\begin{equation}\label{RWWQD-345}
\Big|\int_{\mathbb{R}^{n-1}} f_{\alpha_0}(x')\,g(x')\,dx'\Big|
\leq C\,\|\mu\|_{\mathcal{C}(\mathbb{R}_+^{n})}^\frac12\,
\|g\|_{H^1(\mathbb{R}^{n-1})}.
\end{equation}
In particular, if we define the functional 
$\Lambda_f^{\alpha_0}:H^1_a(\mathbb{R}^{n-1})\to{\mathbb{C}}$ by setting 
\begin{equation}\label{RWWQD-3SD}
\Lambda_f^{\alpha_0}(g):=\int_{\mathbb{R}^{n-1}}f_{\alpha_0}g\,d{\mathscr{L}}^{n-1}
\,\,\text{ for every }\,\,g\in H^1_a(\mathbb{R}^{n-1}), 
\end{equation}
then \eqref{RWWQD-345} implies $\Lambda_f^{\alpha_0}\in\big(H^1(\mathbb{R}^{n-1})\big)^*$
(here we also used \eqref{h1adense}). Recalling \eqref{jcgsfSEW}, 
it follows that 
\begin{equation}\label{RWWQD-4}
\begin{array}{c}
\text{there exists $b_{\alpha_0}\in{\mathrm{BMO}}(\mathbb{R}^{n-1})$ such that
$\|b_{\alpha_0}\|_{\mathrm{BMO}(\mathbb{R}^{n-1})}\leq
C\,\|\mu\|_{\mathcal{C}(\mathbb{R}_+^{n})}^\frac12$}
\\[8pt]
\text{and }\,\,\displaystyle
\Lambda_f^{\alpha_0}(g)=
\int_{\mathbb{R}^{n-1}}b_{\alpha_0}\,g\,d{\mathscr{L}}^{n-1}\,\, 
\text{ for every function }\,\,g\in H^1_a(\mathbb{R}^{n-1}).
\end{array}
\end{equation}
Thus, 
\begin{equation}\label{RWWQD-5}
\int_{\mathbb{R}^{n-1}}\big(b_{\alpha_0}-f_{\alpha_0}\big)g\,d{\mathscr{L}}^{n-1}=0,
\quad\forall\,g\in H^1_a(\mathbb{R}^{n-1}).
\end{equation}
Hence, if we set $v_{\alpha_0}:=b_{\alpha_0}-f_{\alpha_0}$, then \eqref{RWWQD-5}
implies that 
\begin{equation}\label{RWWQD-5hhb}
\text{$v_{\alpha_0}$ is constant almost everywhere in ${\mathbb{R}}^{n-1}$}. 
\end{equation}
Indeed, if the latter were not true, one could find two bounded Lebesgue measurable sets 
$E^+$, $E^-$ in ${\mathbb{R}}^{n-1}$ such that $0<|E^\pm|<\infty$ and 
$v_{\alpha_0}(x')\leq a<b\le v_{\alpha_0}(y')$ 
for all $x'\in E^-$, $y'\in E^+$. Then the function
\begin{equation}\label{RWWQD-7}
g:=\frac{{\mathbf 1}_{E^+}}{|E_+|}-\frac{{\mathbf 1}_{E^-}}{|E_-|}\,\,\text{ belongs to }\,\, 
H^1_a(\mathbb{R}^{n-1})
\end{equation}
and, when used in \eqref{RWWQD-5}, forces
\begin{equation}\label{RWWQD-8.a}
0=\int_{\mathbb{R}^{n-1}}v_{\alpha_0}g\,d{\mathscr{L}}^{n-1}\geq b-a>0.
\end{equation}
This contradiction proves \eqref{RWWQD-5hhb}. In summary, we have shown that 
$b_{\alpha_0}-f_{\alpha_0}$ is constant almost everywhere in ${\mathbb{R}}^{n-1}$. 
Upon recalling the first line in \eqref{RWWQD-4}, it follows that 
$f_{\alpha_0}\in{\mathrm{BMO}}(\mathbb{R}^{n-1})$ with
\begin{equation}\label{RWWQD-8}
\|f_{\alpha_0}\|_{\mathrm{BMO}(\mathbb{R}^{n-1})}\le
C\,\|\mu\|_{\mathcal{C}(\mathbb{R}_+^{n})}^\frac12.
\end{equation}
Since $\alpha_0\in\{1,\dots,M\}$ is arbitrary, we may further conclude that 
$f\in{\mathrm{BMO}}(\mathbb{R}^{n-1},\mathbb{C}^M)$ 
and satisfies \eqref{Car-BMO:BMO}, as wanted.
\end{proof}

In turn, Proposition~\ref{prop:Car->BMO} is one of the main ingredients in the 
proof of the fact that the boundary traces of vertical shifts of a smooth null-solution
of $L$ satisfying a Carleson measure condition in the upper-half space belong to 
{\rm BMO}, uniformly with respect to the shift. 

\begin{lemma}\label{lemma:feps-BMO}
Let $L$ be an $M\times M$ elliptic system with constant complex coefficients as in
\eqref{L-def}-\eqref{L-ell.X} and consider $P^L$, the associated Poisson kernel 
for $L$ in $\mathbb{R}^{n}_+$ from Theorem~\ref{kkjbhV}. 
Suppose $u\in\mathscr{C}^\infty(\mathbb{R}^n_+,\mathbb{C}^M)$ 
satisfies $Lu=0$ in $\mathbb{R}_+^n$ and $\|u\|_{**}<\infty$. 
For each $\varepsilon>0$, set $u_\varepsilon(x',t):=u(x', t+\varepsilon)$, for every 
$(x',t)\in\mathbb{R}^n_+$ and $f_\varepsilon:=u_\varepsilon\bigl|_{\partial\mathbb{R}^{n}_{+}}$. 
Then for each $\varepsilon>0$ we have 
$f_\varepsilon\in{\mathrm{BMO}}(\mathbb{R}^{n-1},\mathbb{C}^M)$ and
\begin{equation}\label{feps-BMO}
\|f_\varepsilon\|_{\mathrm{BMO}(\mathbb{R}^{n-1},\mathbb{C}^M)}\leq C\,\|u\|_{**},
\end{equation}
for some $C\in(0,\infty)$ independent of $\varepsilon$.
\end{lemma}

\begin{proof}
We are going to apply Proposition~\ref{prop:Car->BMO} to $f_\varepsilon$. Note first that 
by part {\it (d)} in Lemma~\ref{lemma:u-lift:props} we have
$f_\varepsilon\in L^1\Big({\mathbb{R}}^{n-1},\frac{1}{1+|x'|^n}\,dx'\Big)^M\cap
\mathscr{C}^\infty({\mathbb{R}}^{n-1},{\mathbb{C}}^M)$. 
Hence we may define $\mu_\varepsilon$ as in \eqref{Car-BMO:mu} 
associated with $f_\varepsilon$, where we recall that 
$\Phi(x')=\partial_n K^L(x',1)$ for every $x'\in{\mathbb{R}}^{n-1}$ and 
$K^L(x',t)=t^{1-n}P^L(x'/t)$ for every $x'\in{\mathbb{R}}^n_{+}$.
Also, Lemma~\ref{lemma:u-lift:uniq} and \eqref{jsgag-bb} imply
\begin{equation}\label{RWWQD-9}
t\,\partial_t u_\varepsilon(x',t)
=t\,\partial_t (P_t^L * f_\varepsilon)(x')
=\big(\Phi_t* f_\varepsilon\big)(x'),
\quad\forall\,(x',t)\in{\mathbb{R}}^n_{+}.
\end{equation}
Thus part {\it (b)} in Lemma~\ref{lemma:u-lift:props} yields
\begin{align}\label{RWWQD-aa}
\|\mu_\varepsilon\|_{\mathcal{C}(\mathbb{R}_+^{n})}
&=\sup_{Q\subset\mathbb{R}^{n-1}}\frac{1}{|Q|}\int_{0}^{\ell(Q)}
\int_Q |\Phi_t*f_\varepsilon(x')|^2\,\frac{dx'\,dt}{t}
\nonumber\\[4pt]
&=\sup_{Q\subset\mathbb{R}^{n-1}}\frac{1}{|Q|}\int_{0}^{\ell(Q)}
\int_Q |\partial_t u_\varepsilon(x',t)|^2\,t\,dx'\,dt
\nonumber\\[4pt]
&\leq\|u_\varepsilon\|_{**}^2
\leq C\,\|u\|_{**}^2<\infty.
\end{align}
Consequently, we may invoke Proposition~\ref{prop:Car->BMO} to conclude that 
$f_\varepsilon\in{\mathrm{BMO}}(\mathbb{R}^{n-1},\mathbb{C}^M)$ and
\begin{equation}\label{bxfdea}
\|f_\varepsilon\|_{\mathrm{BMO}(\mathbb{R}^{n-1},\mathbb{C}^M)}
\leq C\,\|\mu_\varepsilon\|_{\mathcal{C}(\mathbb{R}_+^{n})}^\frac12
\leq C\,\|u\|_{**},
\end{equation}
as wanted.
\end{proof}

The aim in Lemma~\ref{lemma:weak-*} below is to show that derivatives of the kernel function 
$K^L$ are multiples of molecules in the sense of Hardy space theory. To make this precise, let us recall the 
definition of $L^2(\mathbb{R}^{n-1})$-molecules for the Hardy space $H^1(\mathbb{R}^{n-1})$. 
Specifically, given $\varepsilon>0$ and a ball $B\subset\mathbb{R}^{n-1}$, 
a function $m\in L^1(\mathbb{R}^{n-1})$ is said to be an 
$(L^2(\mathbb{R}^{n-1}),\varepsilon)$-molecule relative to $B$ provided
\begin{equation}\label{RWWQD-bb}
\int_{\mathbb{R}^{n-1}} m(x')\,dx'=0
\end{equation}
and
\begin{equation}\label{RWWQD-cc}
\|m\|_{L^2(B)}\leq |B|^{-\frac12},
\qquad\|m\|_{L^2(2^{k}\,B\setminus 2^{k-1}\,B)}
\leq |2^{k}\,B|^{-\frac12}2^{-k\,\varepsilon},\quad\forall\,k\in{\mathbb{N}}.
\end{equation}

\begin{lemma}\label{lemma:weak-*}
Let $L$ be an $M\times M$ elliptic system with constant complex coefficients as in
\eqref{L-def}-\eqref{L-ell.X} and consider $P^L$, the associated Poisson kernel for 
$L$ in $\mathbb{R}^{n}_+$ from Theorem~\ref{kkjbhV}. Then there exists 
a constant $C\in(0,\infty)$ such that for any fixed $t>0$, the components of
$C\,t\nabla K^L(\cdot,t)$ are $(L^2(\mathbb{R}^{n-1}),1)$-molecules
relative to $B_{n-1}(0',t)$. In particular,
\begin{equation}\label{grefr}
\sup_{t>0}\big\|t\nabla K^L(\cdot,t)\big\|_{H^1(\mathbb{R}^{n-1})}<\infty.
\end{equation}
Consequently, if $f\in\mathrm{BMO}(\mathbb{R}^{n-1},\mathbb{C}^M)$ and 
the sequence $\{f_k\}_{k\in{\mathbb{N}}}\subset\mathrm{BMO}(\mathbb{R}^{n-1},\mathbb{C}^M)$ 
is such that $[f_k]\to [f]$ in the weak-* topology on
$\widetilde{\mathrm{BMO}}(\mathbb{R}^{n-1},\mathbb{C}^M)$ as $k\to\infty$, i.e.,
\begin{equation}\label{Hfdzw}
\lim_{k\to\infty}\int_{{\mathbb{R}}^{n-1}}f_kg\,d{\mathscr{L}}^{n-1}
=\int_{{\mathbb{R}}^{n-1}}fg\,d{\mathscr{L}}^{n-1}
\qquad\forall\,g\in H^1(\mathbb{R}^{n-1},\mathbb{C}^M),
\end{equation}
then for every $(x',t)\in\mathbb{R}_+^n$ fixed one has
\begin{equation}\label{mole-weak*}
\lim_{k\to\infty}\int_{\mathbb{R}^{n-1}}t\,\nabla K^L(x'-y',t)\,f_k(y')\,dy'
=\int_{\mathbb{R}^{n-1}}t\,\nabla K^L(x'-y',t)\,f(y')\,dy'.
\end{equation}
\end{lemma}

\begin{proof}
Fix $t>0$, set $B_t:=B_{n-1}(0',t)$, and write $m(x')=t\nabla K^L(x',t)$ 
for every $x'\in{\mathbb{R}}^{n-1}$. We have already shown in 
\eqref{est-theta-vanish} that
\begin{equation}\label{RWWQD-cc.2}
\int_{\mathbb{R}^{n-1}} m(x')\,dx'=0.
\end{equation}
Also, by part {\it (5)} in Theorem~\ref{kkjbhV} we have
\begin{align}\label{RWWQD-dd}
\int_{B_t} |m(x')|^2\,dx'
&\leq C\,\int_{|x'|<t}\frac{t^2}{(t+|x'|)^{2\,n}}\,dx'
\leq C\,\int_{|x'|<t}\frac{t^2}{t^{2n}}\,dx'
\nonumber\\[4pt]
&=C\,t^{1-n}\leq C_0^2\,|B_t|^{-1},
\end{align}
and, for every $k\ge 1$,
\begin{align}\label{RWWQD-ee}
\int_{2^k\,B_t\setminus 2^{k-1}\,B_t}|m(x')|^2\,dx'
&\leq C\,\int_{2^{k-1}\,t<|x'|<2^k\,t}\frac{t^2}{(t+|x'|)^{2\,n}}\,dx'
\nonumber\\[4pt]
&\leq C\,\int_{2^k\,B_t}\frac{t^2}{(2^k\,t)^{2n}}\,dx'
\nonumber\\[4pt]
&=C\,2^{-2\,k}(2^k\,t)^{1-n}\leq C_0^2\,2^{-2\,k}\,|2^k\,B_t|^{-1},
\end{align}
for some constant $C_0\in(0,\infty)$ independent of $k$, $x'$, and $t$.
All these give that $C_0^{-1}m$ is an $(L^2(\mathbb{R}^{n-1}),1)$-molecule relative 
to $B_t$ and \eqref{grefr} follows from the molecular characterization of 
$H^1(\mathbb{R}^{n-1})$, see \cite{AM}.

In addition, for each $x'\in\mathbb{R}^{n-1}$ fixed, the function 
$C_0^{-1}m(x'-\cdot)$ is an $(L^2(\mathbb{R}^{n-1}),1)$-molecule relative to 
$B_{n-1}(x',t)$ and therefore belongs to $H^1(\mathbb{R}^{n-1})$. Hence, \eqref{mole-weak*} 
follows from the definition of the weak-* convergence.
\end{proof}

We now have all the ingredients to prove Proposition~\ref{prop-Dir-BMO:uniq-A}:

\vskip 0.08in
\begin{proof}[Proof of Proposition~\ref{prop-Dir-BMO:uniq-A}]
Under the notation of Lemma~\ref{lemma:feps-BMO}, from \eqref{feps-BMO} we know that 
the sequence $\big\{\,[f_\varepsilon]\,\big\}_{0<\varepsilon<1}$ is bounded in 
the Banach space $\widetilde{\mathrm{BMO}}(\mathbb{R}^{n-1},\mathbb{C}^M)$. 
By eventually passing to a subsequence, Alaoglu's theorem guarantees that there 
is no loss of generality in assuming that
$\big\{\,[f_{\varepsilon}]\,\big\}_{0<\varepsilon<1}$ converges weakly
in $\widetilde{\mathrm{BMO}}(\mathbb{R}^{n-1},\mathbb{C}^M)$ to some 
$[g]\in\widetilde{\mathrm{BMO}}(\mathbb{R}^{n-1},\mathbb{C}^M)$ (where 
$g\in{\mathrm{BMO}}(\mathbb{R}^{n-1},\mathbb{C}^M)$) satisfying
\begin{equation}\label{bsew}
\big\|\,[g]\,\big\|_{\widetilde{\mathrm{BMO}}(\mathbb{R}^{n-1},\mathbb{C}^M)}
\leq C\,\|u\|_{**},
\end{equation}
for some finite constant $C>0$ which does not depend on $u$.
Applying Lemma~\ref{lemma:weak-*}, for every $(x',t)\in\mathbb{R}_+^{n}$ fixed
we may conclude with the help of \eqref{jcgsfSEW} that
\begin{align}\label{adexce}
\lim_{\varepsilon\to 0^{+}}
\nabla\big[\big(P_t^L*f_{\varepsilon}\big)(x')\big]
& =\lim_{\varepsilon\to 0^{+}}
\int_{{\mathbb{R}}^{n-1}}(\nabla K^L)(x'-y',t)f_{\varepsilon}(y')\,dy'
\nonumber\\[4pt]
& =\int_{{\mathbb{R}}^{n-1}}(\nabla K^L)(x'-y',t)g(y')\,dy'
\nonumber\\[4pt]
& =\nabla\big[\big(P_t^L*g\big)(x')\big].
\end{align}
On the other hand, by Lemma~\ref{lemma:u-lift:uniq} we have
\begin{align}\label{gtgt}
\nabla u(x',t+\varepsilon)=\nabla u_{\varepsilon}(x',t)
=\nabla\big[\big(P_t^L*f_{\varepsilon}\big)(x')\big],
\quad\forall\,(x',t)\in{\mathbb{R}}^n_{+}.
\end{align}
Together, \eqref{adexce} and \eqref{gtgt} give
(keeping in mind part {\it (a)} in Lemma~\ref{lemma:u-lift:props})
\begin{align}\label{RWWQD-vd}
\nabla u(x',t)=\lim_{\varepsilon\to 0^{+}}\nabla u(x', t+\varepsilon)
=\nabla\big[\big(P_t^L*g\big)(x')\big],
\quad\forall\,(x',t)\in{\mathbb{R}}^n_{+}.
\end{align}
Consequently, there exists $C\in{\mathbb{C}}^M$ with the property that
\begin{equation}\label{hdfs}
u(x',t)=(P_t^L*g)(x')+C,\qquad\forall\,(x',t)\in{\mathbb{R}}^n_{+}.
\end{equation}
Then $f:=g+C\in{\mathrm{BMO}}(\mathbb{R}^{n-1},\mathbb{C}^M)$ satisfies
(thanks to \eqref{bsew} and \eqref{defi-BMO-nbgxcr})
\begin{equation}\label{bsew-vxts}
\|f\|_{\mathrm{BMO}(\mathbb{R}^{n-1},\mathbb{C}^M)}
=\|g\|_{\mathrm{BMO}(\mathbb{R}^{n-1},\mathbb{C}^M)}
=\big\|\,[g]\,\big\|_{\widetilde{\mathrm{BMO}}(\mathbb{R}^{n-1},\mathbb{C}^M)}
\leq C\,\|u\|_{**},
\end{equation}
where $C>0$ is a finite constant which does not depend on $u$.
Moreover, \eqref{eq:aaAabgr-22} ensures that 
$f\in L^1\Big({\mathbb{R}}^{n-1},\frac{1}{1+|x'|^n}\,dx'\Big)^M$, 
while formula \eqref{hdfs} becomes, in light of \eqref{eq:IG6gy.2PPP},
precisely \eqref{eqn-Dir-BMO:u-A}. Granted this, the first conclusion in 
Proposition~\ref{prop-Dir-BMO:exis} implies that $f$ is the only function
in $L^1\Big({\mathbb{R}}^{n-1},\frac{1}{1+|x'|^n}\,dx'\Big)^M$ for which 
the representation formula \eqref{eqn-Dir-BMO:u-A} holds, 
$u\big|_{\partial\mathbb{R}^{n}_{+}}^{{}^{\rm n.t.}}$ exists at
a.e. point in ${\mathbb{R}}^{n-1}$, and
$f=u\big|_{\partial\mathbb{R}^{n}_{+}}^{{}^{\rm n.t.}}$.
To conclude the proof of Proposition~\ref{prop-Dir-BMO:uniq-A} there remains to observe
that \eqref{jxdgsvgf} is a consequence of \eqref{bsew-vxts}, \eqref{eqn-Dir-BMO:u-A},
and \eqref{exist:u2-carleso}.
\end{proof}

\section{Proofs of Theorems~\ref{them:BMO-Dir}, \ref{thm:fatou-ADEEDE}, 
\ref{thm:fatou-VMO}, \ref{thm:FEFF}, \ref{THMVMO.i}, \ref{ndyRE}, \ref{THMVMO.CCC}, 
\ref{ndyRE-NNN}, and \ref{jhdwtRD}}
\setcounter{equation}{0}
\label{Pf-mainThms}

We begin by presenting the proof of the Fatou type result stated in Theorem~\ref{thm:fatou-ADEEDE}. 
The main ingredients involved are Proposition~\ref{prop-Dir-BMO:exis}, 
Proposition~\ref{prop-Dir-BMO:uniq-A}, and Proposition~\ref{prop-Dir-BMO:uniq}.

\vskip 0.08in
\begin{proof}[Proof of Theorem~\ref{thm:fatou-ADEEDE}]
The implication in \eqref{Tafva.BMO} is seen directly from Proposition~\ref{prop-Dir-BMO:uniq-A}.
Proposition~\ref{prop-Dir-BMO:uniq-A} also guarantees the right-to-left inclusion in \eqref{eq:tr-sols}. 
The left-to-right inclusion in \eqref{eq:tr-sols} is a consequence of Proposition~\ref{prop-Dir-BMO:exis}.
Going further, it is clear from definitions that ${\mathrm{LMO}}({\mathbb{R}}^n_{+})$ is a linear space 
on which $\|\cdot\|_{**}$ is a seminorm with null-space ${\mathbb{C}}^M$.
The linear mapping in \eqref{eq:tr-OP} is bounded (by the estimate in \eqref{Tafva.BMO}), 
injective (by Proposition~\ref{prop-Dir-BMO:uniq}), and surjective (by Proposition~\ref{prop-Dir-BMO:exis}). 
Moreover, another reference to the estimate in \eqref{Tafva.BMO} shows that the inverse
of the mapping \eqref{eq:tr-OP} is also bounded. Given that 
$\widetilde{\mathrm{BMO}}({\mathbb{R}}^n_{+})$ is complete, it follows that
the quotient space ${\mathrm{LMO}}({\mathbb{R}}^n_{+})\big/{\mathbb{C}}^M$ 
is also complete when equipped with $\|\cdot\|_{**}$.
\end{proof}

Anticipating the proof of Theorem~\ref{thm:fatou-VMO}, below we isolate a key result
to the effect that any smooth null-solution of $L$ satisfying a vanishing Carleson measure condition 
in the upper-half space converges vertically to its nontangential boundary trace in {\rm BMO}. 

\begin{lemma}\label{lemma:VNO-BMO-convergence}
Let $L$ be an $M\times M$ elliptic system with constant complex coefficients as in
\eqref{L-def}-\eqref{L-ell.X} and consider $P^L$, the associated Poisson kernel 
for $L$ in $\mathbb{R}^{n}_+$ from Theorem~\ref{kkjbhV}. 
Suppose $u\in\mathscr{C}^\infty(\mathbb{R}^n_+,\mathbb{C}^M)$ 
satisfies $Lu=0$ in $\mathbb{R}_+^n$ and $\|u\|_{**}<\infty$ and use 
Theorem~\ref{thm:fatou-ADEEDE} to write 
\begin{equation}\label{ky6gg}
f:=u\big|^{{}^{\rm n.t.}}_{\partial{\mathbb{R}}^n_{+}}\in 
{\mathrm{BMO}}(\mathbb{R}^{n-1},\mathbb{C}^M).
\end{equation}
For each number $\varepsilon>0$, define $u_\varepsilon(x',t):=u(x',t+\varepsilon)$ for every 
$(x',t)\in\mathbb{R}^n_+$ and consider 
$f_\varepsilon:=u_\varepsilon\big|_{\partial\mathbb{R}^{n}_{+}}
\in{\mathrm{BMO}}(\mathbb{R}^{n-1},\mathbb{C}^M)$ {\rm (}see Lemma~\ref{lemma:feps-BMO}\,{\rm )}. 
Then 
\begin{equation}\label{eqn:conv-BMO}
\left.
\begin{array}{l}
\big|\nabla u(x',t)\big|^2\,t\,dx'dt\,\,\mbox{is a vanishing}
\\[4pt]
\text{Carleson measure in }\,\,\mathbb{R}^{n}_+
\end{array}
\right\}
\Longrightarrow
\lim_{\varepsilon\to 0^+}\|f_\varepsilon-f\|_{\mathrm{BMO}(\mathbb{R}^{n-1},\mathbb{C}^M)}=0.
\end{equation}
\end{lemma}

\begin{proof}
By Theorem~\ref{thm:fatou-ADEEDE} we have $u(x',t)=(P_t^L*f)(x')$ for every 
$ (x',t)\in{\mathbb{R}}^n_{+}$. Also, Lemma~\ref{lemma:u-lift:uniq} implies 
$u_{\varepsilon}(x',t)=(P_t^L*f_\varepsilon)(x')$ for every $ (x',t)\in{\mathbb{R}}^n_{+}$
and each $\varepsilon>0$. To proceed, for every $ (x',t)\in{\mathbb{R}}^n_{+}$ set
\begin{align}\label{eq:fcaewe}
v_\varepsilon(x',t)
: &=(P_t^L*(f_\varepsilon-f))(x',t)
=(P_t^L* f)(x',t)- (P_t^L*f_\varepsilon)(x',t)
\nonumber\\[4pt]
&=u_\varepsilon(x',t)-u(x',t).
\end{align}
Given that for each parameter $\varepsilon>0$ the function $v_\varepsilon$ 
satisfies the hypotheses of Theorem~\ref{thm:fatou-ADEEDE} and almost everywhere
$v_\varepsilon\big|^{{}^{\rm n.t.}}_{\partial{\mathbb{R}}^n_{+}} =f_\varepsilon-f\in{\mathrm{BMO}}(\mathbb{R}^{n-1})$ 
it follows that
\begin{equation}\label{eq:ythed}
\|f_\varepsilon-f\|_{\mathrm{BMO}(\mathbb{R}^{n-1},\mathbb{C}^M)}
\leq C\,\|v_\varepsilon\|_{**}
=C\,\|u_\varepsilon-u\|_{**}
\end{equation}
for every $\varepsilon>0$. Hence, to complete the proof of \eqref{eqn:conv-BMO} it suffices to show that 
\begin{equation}\label{JGctasw}
\lim_{\varepsilon\to 0^+}\|u_\varepsilon-u\|_{**}=0.
\end{equation}
To this end, for each $r>0$ we introduce
\begin{align}\label{eq:Car-r-restri.a}
\|u\|_{**,r}:=\sup_{Q\subset\mathbb{R}^{n-1},\,\ell(Q)\leq r} 
\left(\int_{0}^{\ell(Q)}\aver{Q}|\nabla u(x',t)|^2\,t\,dx'dt\right)^{\frac12}.
\end{align}
Note that 
\begin{align}\label{kdhfrew-1}
\|u\|_{**,r}\leq\|u\|_{**,s}\leq\|u\|_{**}\quad\text{ if }\,\,r\leq s,
\end{align}
and the fact that $\big|\nabla u(x',t)\big|^2\,t\,dx'dt$ is a vanishing Carleson measure 
in $\mathbb{R}^{n}_+$ (recall \eqref{defi-CarlesonVan}) may be rephrased as 
\begin{equation}\label{kdhfrew-3}
\|u\|_{**,r}\to 0\quad\text{as}\quad r\to 0^+.
\end{equation}

We now make the claim that there exists a constant $C=C(n,L)\in(0,\infty)$ such that
\begin{align}\label{eqn:frefer}
\|u-u_{\varepsilon}\|_{**}
\leq C\,\big(\|u\|_{**,4\,\max\{r,\varepsilon\}}+\|u\|_{**}\,\min\{\varepsilon/r,1\}\big),
\qquad\forall\,r,\varepsilon\in(0,\infty).
\end{align}
Assume the claim for now and based on \eqref{eqn:frefer}, for every $0<r<\infty$, we may write
\begin{align}\label{ndysfr}
0 &\leq\limsup_{\varepsilon\to 0^+} 
\|u-u_{\varepsilon}\|_{**}
\nonumber\\[4pt]
&\leq C\limsup_{\varepsilon\to 0^+}  \|u\|_{**,4\,\max\{r,\varepsilon\}}
+C\|u\|_{**}\,\limsup_{\varepsilon\to 0^+}\big[\min\{\varepsilon/r,1\}\big)\big]
\nonumber\\[4pt]
&=C\|u\|_{**,4\,r}.
\end{align}
Recalling now \eqref{kdhfrew-3}, we may further take the limit as $r\to 0^{+}$ 
in the resulting inequality in \eqref{ndysfr} and conclude that 
$\limsup_{\varepsilon\to 0^+}\|u-u_{\varepsilon}\|_{**}=0$. 
This proves \eqref{JGctasw} as wanted. 

To finish the proof of the lemma we are left with showing the claim. To do so, we first note that
in light of the notation in \eqref{eq:Car-r-restri.a}, the reasoning in \eqref{Twazvee} 
(corresponding to $|\alpha|=1$) yields
\begin{align}\label{adweadfwe-aa}
t\,|\nabla u(x',t)|
&\leq C\,\Big(\frac1{|Q_{x'}|}\int_{t/2}^{3\,t/2}\,\int_{Q_{x'}} 
|\nabla u(y',s)|^2\,s\,dy'ds\Big)^\frac12
\nonumber\\[4pt]
&\leq C\,\|u\|_{**, 2t}
\end{align}
for each $(x',t)\in{\mathbb{R}}^n_{+}$, where $Q_{x'}$ denotes the cube in 
$\mathbb{R}^{n-1}$ centered at $x'$ with side-length $t$. 

Next, fix a cube $Q\subset\mathbb{R}^{n-1}$ and numbers $r,\varepsilon\in(0,\infty)$
and proceed by analyzing the following possible three cases.

\medskip

\noindent\textbf{Case 1: $\ell(Q)\leq\varepsilon.$} Under this assumption, recalling also 
\eqref{adweadfwe-aa} and \eqref{kdhfrew-1}, we obtain
\begin{align}\label{trADE}
& \left(\int_{0}^{\ell(Q)}\aver{Q} 
|\nabla u_\varepsilon(x',t)-\nabla u(x',t)|^2\,t\,dx'dt\right)^{\frac12}
\nonumber\\[4pt]
&\hskip 0.50in
\leq \left(\int_{\varepsilon}^{\ell(Q)+\varepsilon}\aver{Q} 
|\nabla u(x',t)|^2\,t\,dx'dt\right)^{\frac12}
+\left(\int_0^{\ell(Q)}\aver{Q} 
|\nabla u(x',t)|^2\,t\,dx'dt\right)^{\frac12}
\nonumber\\[4pt]
&\hskip 0.50in
\leq\left(\int_{\varepsilon}^{2\,\varepsilon}\aver{Q} 
|\nabla u(x',t)|^2\,t\,dx'dt\right)^{\frac12}+\|u\|_{**,\varepsilon}
\nonumber\\[4pt]
&\hskip 0.50in
\leq C\,\left(\int_{\varepsilon}^{2\,\varepsilon}\frac{\|u\|_{**, 2t}^2}{t}\,dt\right)^{\frac12}
+\|u\|_{**,4\varepsilon}
\leq C\,\|u\|_{**,4\varepsilon}
\nonumber\\[4pt]
&\hskip 0.50in
\leq C\,\|u\|_{**,4\max\{r,\varepsilon\}},
\end{align}
for some constant $C=C(n,L)\in(0,\infty)$ independent of $u$, $\varepsilon$, and $r$.
\medskip

\noindent\textbf{Case 2: $\varepsilon<\ell(Q)\leq r$.} 
Note that in this case $r=\max\{r,\varepsilon\}$ and using again 
\eqref{adweadfwe-aa} and \eqref{kdhfrew-1} we have
\begin{align}\label{deahy}
&\left(\int_{0}^{\ell(Q)}\aver{Q} 
|\nabla u_\varepsilon(x',t)-\nabla u(x',t)|^2\,t\,dx'dt\right)^{\frac12}
\nonumber\\[4pt]
&\qquad\qquad\leq
\left(\int_{\varepsilon}^{\ell(Q)+\varepsilon}\aver{Q} 
|\nabla u(x',t)|^2\,t\,dx'dt\right)^{\frac12}
+\|u\|_{**,r}
\nonumber\\[4pt]
&\qquad\qquad\leq
\left(\int_{0}^{2\,\ell(Q)}\aver{Q} 
|\nabla u(x',t)|^2\,t\,dx'dt\right)^{\frac12}
+\|u\|_{**,\max\{r,\varepsilon\}}
\nonumber\\[4pt]
&\qquad\qquad\le
C\,\|u\|_{**,4\max\{r,\varepsilon\}},
\end{align}
for some constant $C=C(n,L)\in(0,\infty)$ independent of $u$, $\varepsilon$, and $r$.

\medskip

\noindent\textbf{Case 3: $\ell(Q)>\max\{r,\varepsilon\}$.} In this case, set
$\eta:=\max\{r,\varepsilon\}$ and write
\begin{align}\label{gertg}
&\left(\frac1{|Q|}\int_{0}^{\ell(Q)}\int_Q 
|\nabla u_\varepsilon(x',t)-\nabla u(x',t)|^2\,t\,dx'dt\right)^{\frac12}
\nonumber\\[4pt]
&\quad\qquad
\leq\left(\frac1{|Q|}\int_{0}^{\eta}\int_Q|\nabla u_\varepsilon(x',t)-\nabla u(x',t)|^2\,t\,dx'dt\right)^{\frac12}
\nonumber\\[4pt]
&\qquad\qquad
+\left(\frac1{|Q|}\int_{\eta}^{\ell(Q)}
\int_Q|\nabla u_\varepsilon(x',t)-\nabla u(x',t)|^2\,t\,dx'dt\right)^{\frac12}
=:I+II.
\end{align}
To analyze $I$, let $k_0$ be the nonnegative integer such that 
$2^{k_0}\,\eta<\ell(Q)\leq 2^{k_0+1}\,\eta$. Also, consider $\{Q_j\}_{j\in{\mathbb{N}}}$, the collection 
of subcubes of $Q$ with pairwise disjoint interiors, satisfying $\ell(Q_j)=2^{-k_0}\,\ell(Q)$ for each 
$j\in{\mathbb{N}}$ and $\bigcup\limits_{j\in{\mathbb{N}}}Q_j=Q$. Then $\ell(Q_j)\in(\eta,2\eta]$ for every 
$j\in{\mathbb{N}}$. Bearing this in mind and using the fact that $\varepsilon\leq\eta$, we may then estimate
\begin{align}\label{trtrtr}
I\leq
&\left(\frac1{|Q|}\int_{\varepsilon}^{\eta+\varepsilon}\int_Q 
|\nabla u(x',t)|^2\,t\,dx'dt\right)^{\frac12}
+\left(\frac1{|Q|}\int_{0}^{\eta} 
\int_Q|\nabla u(x',t)|^2\,t\,dx'dt\right)^{\frac12}
\nonumber\\[4pt]
& \leq 2\left(\frac1{|Q|}\int_{0}^{2\,\eta} 
\int_Q|\nabla u(x',t)|^2\,t\,dx'dt\right)^{\frac12}
\nonumber\\[4pt]
& \leq 2\left(\frac1{|Q|}\sum_{j\in{\mathbb{N}}}\int_{0}^{2\ell(Q_j)} 
\int_{Q_j}|\nabla u(x',t)|^2\,t\,dx'dt\right)^{\frac12}
\nonumber\\[4pt]
& \leq 2\left(\frac1{|Q|}\sum_{j\in{\mathbb{N}}}\|u\|_{**,2\ell(Q_j)}^2\,|2Q_j|\right)^{\frac12}
\leq 2^{\frac{n+1}{2}}\,\|u\|_{**,4\eta}
\nonumber\\[4pt]
&=2^{\frac{n+1}{2}}\,\|u\|_{**,4\max\{r,\varepsilon\}}.
\end{align}

Up to this point our treatment involved estimating $u_\varepsilon$ and $u$ separately, 
without exploiting any potential cancellations generated by the fact that we are dealing 
with their difference. However, in the task of estimating $II$, having the difference 
$u_\varepsilon-u$ plays a crucial role, as seen next. Given $(x',t)\in\mathbb{R}^n_+$,
from Lemma~\ref{lemma:decay-Du:Carl} we conclude that
\begin{align}\label{thdydh}
|\nabla^2 u(x',t)| \leq\,C\,\|u\|_{**}\,t^{-2}.
\end{align}
In light of this, the Mean Value Theorem implies that for some $\theta\in (0,1)$ there holds 
\begin{align}\label{XXdtse}
|\nabla u_\varepsilon(x',t)-\nabla u(x',t)|
&=|\nabla u(x',t+\varepsilon)-\nabla u(x',t)|
\leq\varepsilon\,|\nabla^2 u(x',t+\theta\,\varepsilon)|
\nonumber\\[4pt]
&\leq C\,\varepsilon\,\|u\|_{**}\,(t+\theta\,\varepsilon)^{-2}
\nonumber\\[4pt]
&\leq C\,\varepsilon\,\|u\|_{**}\,t^{-2}.
\end{align}
Having established \eqref{XXdtse}, we may turn to estimating $II$ as follows
\begin{align}\label{bdgsg-1}
II &=\left(\frac{1}{|Q|}\int_{\eta}^{\ell(Q)}\int_Q 
|\nabla u_\varepsilon(x',t)-\nabla u(x',t)|^2\,t\,dx'dt\right)^{\frac12}
\nonumber\\[4pt]
&\leq C\,\varepsilon\,\|u\|_{**}\,\left(\int_{\eta}^{\ell(Q)} t^{-3}\,dt\right)^{\frac12}
\leq C\,\varepsilon\,\|u\|_{**}\,\eta^{-1}
\nonumber\\[4pt]
&=C\,\varepsilon\,\|u\|_{**}\,\big(\max\{r,\varepsilon\}\big)^{-1}
=C\,\|u\|_{**}\,\min\{\varepsilon/r,1\}.
\end{align}
In concert, \eqref{gertg}, \eqref{trtrtr}, and \eqref{bdgsg-1}, allow us to conclude that, 
under the current assumption $\ell(Q)>\max\{r,\varepsilon\}$, we have
\begin{align}\label{farfr}
&\left(\int_{0}^{\ell(Q)}\aver{Q} 
|\nabla u_\varepsilon(x',t)-\nabla u(x',t)|^2\,t\,dx'dt\right)^{\frac12}
\nonumber\\[4pt]
& \hskip3cm
\leq C\,\big(\|u\|_{**,4\,\max\{r,\varepsilon\}}+\|u\|_{**}\,\min\{\varepsilon/r,1\}\big).
\end{align}

\medskip

Combining \eqref{trADE}, \eqref{deahy}, and \eqref{farfr}, we obtain that the estimate in \eqref{farfr}
actually holds for arbitrary cubes $Q$ in ${\mathbb{R}}^{n-1}$. In turn, the latter yields \eqref{eqn:frefer}
after taking the supremum over all cubes $Q$ in ${\mathbb{R}}^{n-1}$.
With this the proof of the lemma is completed. 
\end{proof}

Having proved the convergence result in Lemma~\ref{lemma:VNO-BMO-convergence}, we are now prepared 
to present the proof of Theorem~\ref{thm:fatou-VMO}. 

\vskip 0.08in
\begin{proof}[Proof of Theorem~\ref{thm:fatou-VMO}]
We start by noticing that since $u$ satisfies the conditions in \eqref{Dir-BVP-VMOq1},
the conclusions in \eqref{Tafva.BMO} hold. Hence if we set 
$f:=u\big|^{{}^{\rm n.t.}}_{\partial{\mathbb{R}}^n_{+}}$, we have that
$f$ exists almost everywhere in ${\mathbb{R}}^{n-1}$ and belongs to 
${\mathrm{BMO}}(\mathbb{R}^{n-1},\mathbb{C}^M)$. To proceed in showing that
$f\in{\mathrm{VMO}(\mathbb{R}^{n-1},\mathbb{C}^M)}$, 
for each $\varepsilon>0$ define $u_\varepsilon(x',t):=u(x',t+\varepsilon)$
for every $(x',t)\in\mathbb{R}^n_+$, and $f_\varepsilon(x'):=u(x',\varepsilon)$ 
for every $x'\in{\mathbb{R}}^{n-1}$. Then from Lemma~\ref{lemma:feps-BMO} and 
part {\it (d)} in Lemma~\ref{lemma:u-lift:props} we obtain
$f_\varepsilon\in\mathrm{BMO}(\mathbb{R}^{n-1},\mathbb{C}^M)
\cap\mathscr{C}^\infty({\mathbb{R}}^{n-1},{\mathbb{C}}^M)$. In addition,
for every $\varepsilon>0$, based on Lemma~\ref{lemma:decay-Du:Carl} we obtain
\begin{equation}\label{eq:avtrdrht}
\sup_{x'\in\mathbb{R}^{n-1}}|\nabla_{x'} f_{\varepsilon}(x')|
=\sup_{x'\in\mathbb{R}^{n-1}}|\nabla_{x'} u(x',\varepsilon)|\le
C\,\varepsilon^{-1}\,\|u\|_{**}<\infty.
\end{equation}
Fix $r\in(0,\infty)$ and let $Q\subset\mathbb{R}^{n-1}$ be a cube in ${\mathbb{R}}^n_{+}$
with $\ell(Q)\le r$. Then using \eqref{eq:avtrdrht} we may estimate
\begin{align}\label{bdfat}
\aver{Q}\big|f(x')-f_Q\big|\,dx'
&\leq\aver{Q}\big|(f-f_\varepsilon)(x')- (f-f_\varepsilon)_Q\big|\,dx'
+\aver{Q}\big|f_\varepsilon(x')- (f_\varepsilon)_Q\big|\,dx'
\\[4pt]
&\leq\|f_\varepsilon-f\|_{\mathrm{BMO}(\mathbb{R}^{n-1},\mathbb{C}^M)}
+\sup_{x'\in\mathbb{R}^{n-1}}|\nabla_{x'} f_{\varepsilon}(x')|\,\sqrt{n-1}\,\ell(Q)
\\[4pt]
&\leq\|f_\varepsilon-f\|_{\mathrm{BMO}(\mathbb{R}^{n-1},\mathbb{C}^M)}
+C\,r\,\varepsilon^{-1}\,\|u\|_{**}.
\end{align}
Hence,
\begin{align}\label{hseqew}
\sup_{Q\subset\mathbb{R}^{n-1},\,\ell(Q)\leq r}\,\,\aver{Q}\big|f(x')-f_Q\big|\,dx'
\leq\|f_\varepsilon-f\|_{\mathrm{BMO}(\mathbb{R}^{n-1},\mathbb{C}^M)}
+C\,r\,\varepsilon^{-1}\,\|u\|_{**}.
\end{align}
Letting $r\to 0^{+}$ first, then sending $\varepsilon\to 0^+$ in \eqref{hseqew} 
and recalling that since  
$\big|\nabla u(x',t)\big|^2\,t\,dx'dt$ is a vanishing Carleson measure 
in $\mathbb{R}^{n}_+$ implication \eqref{eqn:conv-BMO} holds, 
we conclude that 
\begin{align}\label{hseqew-2}
\lim_{r\to 0^+}\sup_{Q\subset\mathbb{R}^{n-1},\,\ell(Q)\leq r}
\aver{Q}\big|f(x')-f_Q\big|\,dx'=0.
\end{align}
Hence, $f\in{\mathrm{VMO}(\mathbb{R}^{n-1},\mathbb{C}^M)}$, as wanted. 

To complete the proof, there remains to establish \eqref{eq:tr-sols-VMO}. 
However, the right-to-left inclusion follows from \eqref{Dir-BVP-VMOq2}, 
while the opposite inclusion is a consequence of 
Proposition~\ref{prop-Dir-BMO:exis}.
\end{proof}

We continue by presenting the proof of Theorem~\ref{thm:FEFF}.

\vskip 0.08in
\begin{proof}[Proof of Theorem~\ref{thm:FEFF}]
Consider first the equivalence in item {\it (1)} of Theorem~\ref{thm:FEFF}. 
The fact that $f\in{\rm BMO}({\mathbb{R}}^{n-1};\mathbb{C}^{M})$ implies $\|u\|_{\ast\ast}<\infty$
is part {\it (e)} of Proposition~\ref{prop-Dir-BMO:exis} and \eqref{eq:aaAabgr-22}, whereas the 
converse follows from Proposition~\ref{prop-Dir-BMO:uniq-A}. Regarding the equivalence in item 
{\it (2)} of Theorem~\ref{thm:FEFF}, to see that $f\in{\rm VMO}({\mathbb{R}}^{n-1};\mathbb{C}^{M})$ 
implies $|\nabla u(x',t)|^2\,t\,dx'dt$ is a vanishing Carleson measure in ${\mathbb{R}}^n_{+}$ 
we use what we just proved in item {\it (1)} (bearing in mind \eqref{defi-VMO}) combined with part 
{\it (f)} of Proposition~\ref{prop-Dir-BMO:exis}. The converse follows from \eqref{Dir-BVP-VMOq2}.
\end{proof}

Having dealt with the Fatou type results from 
Theorem~\ref{thm:fatou-ADEEDE} and Theorem~\ref{thm:fatou-VMO}, 
we turn our attention to the proof of Theorem~\ref{them:BMO-Dir}.  

\vskip 0.08in
\begin{proof}[Proof of Theorem~\ref{them:BMO-Dir}]
The fact that the function $u$ defined in \eqref{eqn-Dir-BMO:u} solves 
the $\mathrm{BMO}$-Dirichlet boundary value problem \eqref{Dir-BVP-BMO}
follows from Proposition~\ref{prop-Dir-BMO:exis}. By Proposition~\ref{prop-Dir-BMO:uniq}, 
this is the only solution of \eqref{Dir-BVP-BMO}. 
Next, the double estimate in \eqref{Dir-BVP-BMO-Car} is a direct consequence 
of \eqref{eqn-Dir-BMO:u} and \eqref{jxdgsvgf}. The uniform {\rm BMO} estimate 
in \eqref{feps-BTTGB} is seen straight from Lemma~\ref{lemma:feps-BMO}.

Moving on, the left-pointing implication in \eqref{eqn:conv-Bfed}
follows from Lemma~\ref{lemma:VNO-BMO-convergence}. For the opposite
implication, invoke part {\it (d)} in Lemma~\ref{lemma:u-lift:props}
together with \eqref{feps-BTTGB} to conclude that $f$ is the limit in 
${\mathrm{BMO}}({\mathbb{R}}^{n-1},{\mathbb{C}}^M)$ of the sequence 
$\{u(\cdot,\varepsilon)\}_{\varepsilon>0}\subset{\mathrm{UC}}({\mathbb{R}}^{n-1},{\mathbb{C}}^M)
\cap{\mathrm{BMO}}({\mathbb{R}}^{n-1},{\mathbb{C}}^M)$. This places $f$ in 
${\mathrm{VMO}}({\mathbb{R}}^{n-1},{\mathbb{C}}^M)$ (cf. \eqref{ku6ffcfc}).
Having established this, part {\it (f)} in Proposition~\ref{prop-Dir-BMO:exis} 
gives that $\big|\nabla u(x',t)\big|^2\,t\,dx'dt$ is a vanishing Carleson measure 
in $\mathbb{R}^{n}_{+}$. Going further, the equivalence in \eqref{Dir-BVP-Reg} 
is seen from \eqref{Dir-BVP-VMOq2} and the last part in Proposition~\ref{prop-Dir-BMO:exis}. 

As regards the equivalence in \eqref{Dir-BVP-Reg.TTT}, let us first observe that, as is
apparent from its definition in \eqref{ustarstar}, the seminorm $\|\cdot\|_{**}$ 
is invariant to horizontal translations. That is, for every 
$u\in{\mathscr{C}}^1({\mathbb{R}}^n_{+},{\mathbb{C}}^M)$ we have
\begin{equation}\label{8h65rf.UJ}
\big\|\tau_{(z',0)}u\big\|_{**}=\|u\|_{**}\,\,\text{ for every }\,\,z'\in{\mathbb{R}}^{n-1}. 
\end{equation}
Given $f\in{\mathrm{VMO}}({\mathbb{R}}^{n-1},{\mathbb{C}}^M)$, the right-pointing 
implication in \eqref{Dir-BVP-Reg} ensures that 
\begin{equation}\label{nyGVyh}
\big|\nabla u(x',t)\big|^2\,t\,dx'dt\,\,\mbox{is a vanishing Carleson measure in }\,\,\mathbb{R}^{n}_{+}.
\end{equation}
For each $z=(z',s)\in{\mathbb{R}}^n_{+}$ we may write, thanks to \eqref{8h65rf.UJ} 
and the estimate in \eqref{Tafva.BMO},  
\begin{align}\label{jytgfg-j64}
\|\tau_z u-u\|_{**} &\leq\big\|\tau_z u-\tau_{(z',0)}u\big\|_{**}
+\big\|\tau_{(z',0)}u-u\big\|_{**}
\nonumber\\[4pt]
&=\big\|\tau_{(0,s)}u-u\big\|_{**}+\big\|\tau_{(z',0)}u-u\big\|_{**}
\nonumber\\[4pt]
&\leq C\|u(\cdot,s)-f\|_{{\mathrm{BMO}}({\mathbb{R}}^{n-1},{\mathbb{C}}^M)}
+C\|\tau_{z'}f-f\|_{{\mathrm{BMO}}({\mathbb{R}}^{n-1},{\mathbb{C}}^M)},
\end{align}
for some constant $C=C(n,L)\in(0,\infty)$. In light of \eqref{nyGVyh}, 
the left-pointing implication in \eqref{eqn:conv-Bfed}, and \eqref{defi-VMO-SSS}, 
we then conclude from \eqref{jytgfg-j64} that
\begin{align}\label{jytgfg-j64.ii}
\lim_{{\mathbb{R}}^n_{+}\ni z\to 0}\|\tau_z u-u\|_{**}=0,
\end{align}
as wanted. Conversely, suppose now that \eqref{jytgfg-j64.ii} holds. 
Specializing this to the case when $z:=(0',\varepsilon)$ with $\varepsilon>0$ then yields,
on account of the estimate in \eqref{Tafva.BMO}, 
\begin{align}\label{jytgfg-j64.ii-pR}
\|u(\cdot,\varepsilon)-f\|_{{\mathrm{BMO}}({\mathbb{R}}^{n-1},{\mathbb{C}}^M)}\leq C
\big\|\tau_{(0',\varepsilon)}u-u\big\|_{**}\to 0\,\,\text{ as }\,\,\varepsilon\to 0^{+}.
\end{align}
Hence, $\|u(\cdot,\varepsilon)-f\|_{{\mathrm{BMO}}({\mathbb{R}}^{n-1},{\mathbb{C}}^M)}\to 0$ 
as $\varepsilon\to 0^{+}$ which, by virtue of \eqref{eqn:conv-Bfed}-\eqref{Dir-BVP-Reg}, 
implies that $f\in{\mathrm{VMO}}({\mathbb{R}}^{n-1},{\mathbb{C}}^M)$. This finishes the 
proofs of the equivalences in part {\it (iv)} of the statement.

Finally, all claims about the $\mathrm{VMO}$-Dirichlet boundary value problem 
\eqref{Dir-BVP-VMO} are direct consequences of what we have proved up to this point.
\end{proof}

Going further, we present the proof of the quantitative characterization of {\rm VMO} 
from Theorem~\ref{THMVMO.i}. 

\vskip 0.08in
\begin{proof}[Proof of Theorem~\ref{THMVMO.i}]
We shall establish all claims stated with $n-1$ in place on $n$. 
Fix a modulus of continuity $\Upsilon$ satisfying $\Upsilon_{\!\#}\leq C\Upsilon$
on $[0,\infty)$ for some finite constant $C>0$. This implies that
\begin{equation}\label{UpU-8i8i-a}
{\mathscr{C}}^{\Upsilon_{\!\#}}({\mathbb{R}}^{n-1})
\subseteq{\mathscr{C}}^\Upsilon({\mathbb{R}}^{n-1}).
\end{equation}
Consider next an arbitrary function $f\in{\mathrm{VMO}(\mathbb{R}^{n-1})}$ and 
define $u\in{\mathscr{C}}^\infty({\mathbb{R}}^n_{+})$ by setting
$u(x',t):=(P^\Delta_t\ast f)(x')$ for $(x',t)\in{\mathbb{R}}^n_{+}$.
Then from item {\it (d)} in Lemma~\ref{lemma:u-lift:props}, Theorem \ref{them:BMO-Dir} part {\it (iii)},
and \eqref{eqn:conv-BfEE} we conclude that 
the sequence of functions $\{f_\varepsilon\}_{\varepsilon>0}$ defined 
for every $\varepsilon>0$ by $f_\varepsilon:=u(\cdot,\varepsilon)$ in ${\mathbb{R}}^{n-1}$ 
satisfies, for each $\varepsilon>0$, 
\begin{equation}\label{UpU-8i8i}
\begin{array}{c}
f_\varepsilon\in{\mathscr{C}}^\Upsilon({\mathbb{R}}^{n-1})\cap{\mathscr{C}}^\infty({\mathbb{R}}^{n-1})
\cap{\mathrm{BMO}}({\mathbb{R}}^{n-1})\,\,\text{ and}
\\[6pt]
\partial^{\alpha'}f\in {\mathscr{C}}^\Upsilon({\mathbb{R}}^{n-1})\cap
L^\infty({\mathbb{R}}^{n-1})\,\text{ for every }\,\alpha'\in{\mathbb{N}}_0^{n-1}
\,\text{ with }\,|\alpha'|\geq 1,
\end{array}
\end{equation}
as well as 
\begin{equation}\label{UpU-8i8i.2}
\|f-f_\varepsilon\|_{{\mathrm{BMO}}({\mathbb{R}}^{n-1})}\longrightarrow 0
\,\,\text{ as }\,\,\varepsilon\to 0^{+}.
\end{equation}
This establishes \eqref{UpUpUp.2}, as well as the stronger claim made in \eqref{iy65ffvgH}.
\end{proof}

Going further, we provide the proof of Theorem~\ref{ndyRE}.

\vskip 0.10in
\begin{proof}[Proof of Theorem~\ref{ndyRE}]
First note that condition \eqref{Bgstwy-2-new2} implies that 
$\varphi$ is continuous on ${\mathbb{R}}^{n}\setminus\{0\}$. 
As such, $\varphi$ is a Lebesgue measurable function $\mathbb{R}^{n}$ which, in turn, 
ensures that condition \eqref{Bgstwy} is meaningful.  

To proceed, observe that if $f\in L^1\Big({\mathbb{R}}^{n}\,,\,\frac{dx}{1+|x|^{n+\varepsilon}}\Big)^M$
then for each $x\in{\mathbb{R}}^n$ we have
\begin{align}\label{Bgstwy-4}
\int_{{\mathbb{R}}^n}|f(y)||\varphi(x-y)|\,dy
& \leq C\int_{{\mathbb{R}}^n}\frac{|f(y)|}{(1+|y|)^{n+\varepsilon}}\cdot
\frac{(1+|y|)^{n+\varepsilon}}{(1+|x-y|)^{n+\varepsilon}}\,dy
\nonumber\\[4pt]
& \leq C(1+|x|)^{n+\varepsilon}\int_{{\mathbb{R}}^n}\frac{|f(y)|}{(1+|y|)^{n+\varepsilon}}\,dy<\infty.
\end{align}
In light of \eqref{eq:aaAabgr-22} (used here with $n+1$ in place of $n$), this implies 
that for every $t>0$ the convolution $\varphi_t\ast f$ is well-defined via an absolutely convergent 
integral whenever the function $f$ belongs to ${\mathrm{BMO}}\big({\mathbb{R}}^{n},{\mathbb{C}}^M\big)$.
In particular, this is the case whenever $f\in{\mathrm{VMO}}\big({\mathbb{R}}^{n},{\mathbb{C}}^M\big)$.

Next, fix $t>0$ and define 
\begin{equation}\label{u6gvV-km}
T_t f:=\varphi_t\ast f\,\,\text{ for every }\,
f\in{\mathrm{BMO}}\big({\mathbb{R}}^{n},{\mathbb{C}}^M\big).
\end{equation}
We first claim that there exists some constant $C\in(0,\infty)$ independent of $t$ such that
\begin{equation}\label{Bgstwy-5}
\|T_tf\|_{{\mathrm{BMO}}({\mathbb{R}}^{n},{\mathbb{C}}^M)}
\leq C\|f\|_{{\mathrm{BMO}}({\mathbb{R}}^{n},{\mathbb{C}}^M)}
\quad\text{ for all }\,\,f\in{\mathrm{BMO}}\big({\mathbb{R}}^{n},{\mathbb{C}}^M\big).
\end{equation}
To prove this claim, fix $f\in{\mathrm{BMO}}\big({\mathbb{R}}^{n},{\mathbb{C}}^M\big)$
and an arbitrary cube $Q$ in ${\mathbb{R}}^n$ with center $x_Q$, then decompose
\begin{equation}\label{Bgstwy-6}
f=(f-f_Q){\mathbf{1}}_{\lambda Q}+(f-f_Q){\mathbf{1}}_{{\mathbb{R}}^n\setminus\lambda Q}+f_Q,
\,\,\text{ where }\,\,\lambda:=2\sqrt{n}.
\end{equation}
Thus, using \eqref{Bgstwy} we have
\begin{equation}\label{Bgstwy-7}
(T_t f)(x)=T_t[(f-f_Q){\mathbf{1}}_{\lambda Q}](x)
+T_t[(f-f_Q){\mathbf{1}}_{{\mathbb{R}}^n\setminus(\lambda Q)}](x)+f_Q
\quad\forall\,x\in{\mathbb{R}}^n,
\end{equation}
and if we set
\begin{equation}\label{Bgstwy-8}
c_Q:=T_t[(f-f_Q){\mathbf{1}}_{{\mathbb{R}}^n\setminus(\lambda Q)}](x_Q)+f_Q
\in{\mathbb{C}}^M
\end{equation}
it follows that
\begin{align}\label{Bgstwy-9}
\aver{Q}\big|(T_t f)(x)-c_Q\big|\,dx
& \leq \aver{Q}\big|T_t[(f-f_Q){\mathbf{1}}_{\lambda Q}](x)\big|\,dx
\nonumber\\[4pt]
&\quad+\aver{Q}\big|T_t[(f-f_Q){\mathbf{1}}_{{\mathbb{R}}^n\setminus\lambda Q}](x)
-T_t[(f-f_Q){\mathbf{1}}_{{\mathbb{R}}^n\setminus\lambda Q}](x_Q)\big|\,dx
\nonumber\\[4pt]
&=:I+II.
\end{align}
Since $f\in{\mathrm{BMO}}\big({\mathbb{R}}^{n},{\mathbb{C}}^M\big)$ we have
$(f-f_Q){\mathbf{1}}_{\lambda Q}\in L^1\big({\mathbb{R}}^{n},{\mathbb{C}}^M\big)$. 
On the other hand, assumption \eqref{Bgstwy-2-new} implies that $T_t$ is bounded
in $L^1\big({\mathbb{R}}^{n},{\mathbb{C}}^M\big)$ uniformly in $t$. 
In concert with \eqref{aver-fq-BBBB}, this permits us to estimate 
\begin{align}\label{Bgstwy-10}
I &=\frac{1}{|Q|}\big\|T_t[(f-f_Q){\mathbf{1}}_{\lambda Q}]
\big\|_{L^1({\mathbb{R}}^{n},{\mathbb{C}}^M)}
\nonumber\\[4pt]
&\leq\frac{C}{|Q|}\big\|(f-f_Q){\mathbf{1}}_{\lambda Q}\big\|_{L^1({\mathbb{R}}^{n},{\mathbb{C}}^M)}
\leq C\|f\|_{{\mathrm{BMO}}({\mathbb{R}}^{n},{\mathbb{C}}^M)},
\end{align}
for some $C\in(0,\infty)$ independent of $f$, $Q$, and $t$.
To treat $II$, first we derive a pointwise estimate. For each $x\in Q$ we have
\begin{align}\label{Bgstwy-11}
&\big|T_t[(f-f_Q){\mathbf{1}}_{{\mathbb{R}}^n\setminus\lambda Q}](x)
-T_t[(f-f_Q){\mathbf{1}}_{{\mathbb{R}}^n\setminus\lambda Q}](x_Q)\big|
\nonumber\\[4pt]
&\hskip 0.50in
\leq t^{-n}\int_{{\mathbb{R}}^{n}\setminus\lambda Q}
|f(y)-f_Q|\Big|\varphi\Big(\frac{x-y}{t}\Big)-\varphi\Big(\frac{x_Q-y}{t}\Big)\Big|\,dy.
\end{align}
Next, pick some arbitrary $x\in Q$ and $y\in{\mathbb{R}}^{n}\setminus\lambda Q$, 
then consider $z:=(x_Q-y)/t\in{\mathbb{R}}^n\setminus\{0\}$ and $h:=(x-x_Q)/t\in{\mathbb{R}}^n$. 
Since in view of the choice of $\lambda$ in \eqref{Bgstwy-6} we have
\begin{align}\label{Bgstwy-12bEE}
|h|\leq\frac{\sqrt{n}\ell(Q)}{2t}=\frac{\lambda\ell(Q)}{4t}\leq\frac{|z|}{2},
\end{align}
it follows from \eqref{Bgstwy-2-new2} that 
\begin{align}\label{Bgstwy-12}
\Big|\varphi\Big(\frac{x-y}{t}\Big)
-\varphi\Big(\frac{x_Q-y}{t}\Big)\Big|
&=|\varphi(z+h)-\varphi(z)|
\leq\frac{C|h|^\varepsilon}{|z|^{n+\varepsilon}}
\nonumber\\[4pt]
&\leq\frac{C\ell(Q)^\varepsilon t^n}{|y-x_Q|^{n+\varepsilon}}
\leq\frac{C\ell(Q)^\varepsilon t^n}{(\ell(Q)+|y-x_Q|)^{n+\varepsilon}}.
\end{align}
Combining \eqref{Bgstwy-11}-\eqref{Bgstwy-12} with \eqref{eqn:BMO-decay.88}
(used here with $n+1$ in place of $n$) and part {\it (c)} in Lemma~\ref{jsfsQAXT} it follows that 
\begin{align}\label{Bgstwy-13}
&\big|T_t[(f-f_Q){\mathbf{1}}_{{\mathbb{R}}^n\setminus\lambda Q}](x)
-T_t[(f-f_Q){\mathbf{1}}_{{\mathbb{R}}^n\setminus\lambda Q}](x_Q)\big|
\nonumber\\[4pt]
&\hskip 1.00in
\leq C\ell(Q)^\varepsilon\int_{{\mathbb{R}}^{n}}
\frac{|f(y)-f_Q|}{(\ell(Q)+|y-x_Q|)^{n+\varepsilon}}\,dy
\nonumber\\[4pt]
&\hskip 1.00in
\leq C\int_{1}^\infty{\rm osc}_1\big(f;\lambda\ell(Q)\big)\,\frac{d\lambda}{\lambda^{1+\varepsilon}}
\nonumber\\[4pt]
&\hskip 1.00in
\leq C\|f\|_{{\mathrm{BMO}}({\mathbb{R}}^{n},{\mathbb{C}}^M)},
\qquad\forall\,x\in Q,
\end{align}
where $C\in(0,\infty)$ is independent of $f,Q$ and $t$. 
From \eqref{Bgstwy-13} and \eqref{Bgstwy-9} we obtain
\begin{align}\label{Bgstwy-10B}
II\leq C\|f||_{{\mathrm{BMO}}({\mathbb{R}}^{n},{\mathbb{C}}^M)}
\end{align}
for some $C\in(0,\infty)$ independent of $f$, $Q$, and $t$.
In concert, \eqref{Bgstwy-9}, \eqref{Bgstwy-10}, and \eqref{Bgstwy-10B} yield
\begin{align}\label{Bgstwy-9BB}
\aver{Q}\big|(T_t f)(x)-c_Q\big|\,dx
\leq C\|f||_{{\mathrm{BMO}}({\mathbb{R}}^{n},{\mathbb{C}}^M)}
\end{align}
with $c_Q\in{\mathbb{C}}^M$ as in \eqref{Bgstwy-8}. In view of \eqref{aver-fq-Cava}, this 
ultimately implies the claim in \eqref{Bgstwy-5}.

The second claim we make is that there exists some constant $C\in(0,\infty)$ 
with the property that for every $t>0$ and every $\eta\in(0,\varepsilon)$ there holds
\begin{equation}\label{Bgstwy-5BC}
\|T_t g-g\|_{L^\infty({\mathbb{R}}^{n},{\mathbb{C}}^M)}
\leq Ct^\eta\|g\|_{\dot{\mathscr{C}}^\eta({\mathbb{R}}^{n},{\mathbb{C}}^M)}
\qquad\text{ for all }\,\,g\in\dot{\mathscr{C}}^\eta\big({\mathbb{R}}^{n},{\mathbb{C}}^M\big).
\end{equation}
To prove this claim, fix $t>0$, $\eta\in(0,\varepsilon)$, 
$g\in\dot{\mathscr{C}}^\eta\big({\mathbb{R}}^{n},{\mathbb{C}}^M\big)$, and 
for $x\in{\mathbb{R}}^n$ arbitrary estimate 
\begin{align}\label{Jdrwe}
\big|(T_t g)(x)-g(x)\big|
& \leq\int_{{\mathbb{R}}^n}|g(x-y)-g(x)|\big|\varphi_t(y)\big|\,dy
\nonumber\\[4pt]
&\leq t^\eta\|g\|_{\dot{\mathscr{C}}^\eta({\mathbb{R}}^{n},{\mathbb{C}}^M)}
\int_{{\mathbb{R}}^n}\frac{|y|^\eta}{t^\eta}\,\big|\varphi_t(y)\big|\,dy
\nonumber\\[4pt]
&\leq t^\eta\|g\|_{\dot{\mathscr{C}}^\eta({\mathbb{R}}^{n},{\mathbb{C}}^M)}
\int_{{\mathbb{R}}^n}|z|^\eta\,\big(1+|z|)^{-n-\varepsilon}\,dz
\nonumber\\[4pt]
&\leq Ct^\eta\|g\|_{\dot{\mathscr{C}}^\eta({\mathbb{R}}^{n},{\mathbb{C}}^M)},
\end{align}
for some constant $C=C(\varepsilon,\eta,n,\varphi)\in(0,\infty)$ independent of $t$ and $g$. 
The first inequality in \eqref{Jdrwe} relies on \eqref{Bgstwy}, for the
third one we have used \eqref{Bgstwy-2-new} and the change of variables $z=y/t$,
while the last one is a consequence of having $\eta\in(0,\varepsilon)$.

Here is the argument involved in the endgame of the proof of Theorem~\ref{ndyRE}. Fix $\eta\in(0,\varepsilon)$
and given $f\in{\mathrm{VMO}}\big({\mathbb{R}}^{n},{\mathbb{C}}^M\big)$ pick
$g\in\dot{\mathscr{C}}^\eta\big({\mathbb{R}}^{n},{\mathbb{C}}^M\big)\cap
{\mathrm{BMO}}({\mathbb{R}}^{n},{\mathbb{C}}^M)$. Then
for each $t>0$, we use \eqref{Bgstwy-5} and \eqref{Bgstwy-5BC} to estimate
\begin{align}\label{Jdrwe-2}
\|T_t f-f\|_{{\mathrm{BMO}}({\mathbb{R}}^{n},{\mathbb{C}}^M)}
&\leq \|T_t(f-g)\|_{{\mathrm{BMO}}({\mathbb{R}}^{n},{\mathbb{C}}^M)}
+\|T_t g-g\|_{{\mathrm{BMO}}({\mathbb{R}}^{n},{\mathbb{C}}^M)}
\nonumber\\[4pt]
& \quad
+\|g-f\|_{{\mathrm{BMO}}({\mathbb{R}}^{n},{\mathbb{C}}^M)}
\nonumber\\[4pt]
& \leq C\|g-f\|_{{\mathrm{BMO}}({\mathbb{R}}^{n},{\mathbb{C}}^M)}
+2\|T_t g-g\|_{L^\infty({\mathbb{R}}^{n},{\mathbb{C}}^M)}
\nonumber\\[4pt]
& \leq C\|g-f\|_{{\mathrm{BMO}}({\mathbb{R}}^{n},{\mathbb{C}}^M)}
+Ct^\eta\|g\|_{\dot{\mathscr{C}}^\eta({\mathbb{R}}^{n},{\mathbb{C}}^M)}.
\end{align}
Thus,
\begin{align}\label{Jdrwe-3}
\limsup_{t\to 0^{+}}
\|T_t f-f\|_{{\mathrm{BMO}}({\mathbb{R}}^{n},{\mathbb{C}}^M)}
\leq C\|g-f\|_{{\mathrm{BMO}}({\mathbb{R}}^{n},{\mathbb{C}}^M)}.
\end{align}
Now \eqref{Bgstwy-3} follows from \eqref{Jdrwe-3} in light of the density result
recorded in \eqref{UpUpUp.2c}.

To prove the very last claim in the statement of Theorem~\ref{ndyRE}, let
$\varphi\in{\mathscr{C}}^1\big({\mathbb{R}}^{n},{\mathbb{C}}^{M\times M}\big)$ be a
function satisfying \eqref{Bgstwy-2aaa}. Then for each $x\in\mathbb{R}^{n}\setminus\{0\}$ 
and $h\in{\mathbb{R}}^n$ with $|h|<|x|/2$ the Mean Value Theorem permits us to estimate 
\begin{align}\label{Bgstwy-aaBB}
|\varphi(x+h)-\varphi(x)| &\leq|h|\sup_{\xi\in[x,x+h]}\big|(\nabla\varphi)(\xi)\big| 
\nonumber\\[4pt]
& \leq C|h|\sup_{\xi\in[x,x+h]}(1+|\xi|)^{-n-1} 
\leq\frac{C|h|}{|x|^{n+1}}.
\end{align}
Hence, both \eqref{Bgstwy-2-new} and \eqref{Bgstwy-2-new2} hold with $\varepsilon=1$ in this case, 
so the left-pointing implication in \eqref{Bgstwy-3TRG} is a consequence of \eqref{Bgstwy-3}.

As regards the right-pointing implication in \eqref{Bgstwy-3TRG}, let us first 
observe that from \eqref{Bgstwy} and \eqref{Bgstwy-2aaa} we have 
\begin{align}\label{i7ggf-AA.2}
\int_{{\mathbb{R}}^n}(\partial_j\varphi)\big((x-y)/t\big)\,dy=0,
\qquad\forall\,x\in{\mathbb{R}}^n,\,\,\forall\,j\in\{1,\dots,n\}.
\end{align}
Next, given a function $f\in{\mathrm{BMO}}({\mathbb{R}}^{n},{\mathbb{C}}^M)$, 
fix $x\in{\mathbb{R}}^n$ and $t>0$ arbitrary and denote by $Q_{x,t}$ the cube in 
${\mathbb{R}}^n$ centered at $x$ and of sidelength $t$. As usual, abbreviate
$f_{Q_{x,t}}:=\aver{Q_{x,t}}f(y)\,dy$. On account of \eqref{i7ggf-AA.2}, \eqref{Bgstwy-2aaa}, 
\eqref{eqn:BMO-decay.88} (used here with $\varepsilon=1$ and $n$ in place of $n-1$), 
and \eqref{jsfd-1} (used with $p=1$ and $n$ in place of $n-1$), for each $j\in\{1,\dots,n\}$ 
we may then estimate 
\begin{align}\label{i7ggf-AA.3}
\big|\partial_j(\varphi_t\ast f)(x)\big| 
&=t^{-n-1}\left|\int_{{\mathbb{R}}^n}(\partial_j\varphi)\Big(\frac{x-y}{t}\Big)f(y)\,dy\right|
\nonumber\\[6pt]
&=t^{-n-1}\left|\int_{{\mathbb{R}}^n}(\partial_j\varphi)\Big(\frac{x-y}{t}\Big)\Big[f(y)-f_{Q_{x,t}}\Big]\,dy\right|
\nonumber\\[6pt]
&\leq C\int_{{\mathbb{R}}^n}\frac{\big|f(y)-f_{Q_{x,t}}\big|}{\big[t+|x-y|\big]^{n+1}}\,dy
\leq Ct^{-1}\|f\|_{{\mathrm{BMO}}({\mathbb{R}}^n,{\mathbb{C}}^M)},
\end{align}
for some constant $C\in(0,\infty)$ independent of $f,x,t$. 
In concert with \eqref{Bgstwy-5}, this proves that
\begin{equation}\label{u6lNBBa}
\varphi_t\ast f\in{\mathrm{BMO}}\big({\mathbb{R}}^{n},{\mathbb{C}}^M\big)
\cap{\mathrm{Lip}}\big({\mathbb{R}}^{n},{\mathbb{C}}^M\big)\,\,\text{ for each }\,\,t>0.
\end{equation}
With this in hand, the right-pointing implication in \eqref{Bgstwy-3TRG} readily follows 
(compare with \eqref{UpUpUp.2b}), finishing the proof of Theorem~\ref{ndyRE}.
\end{proof}

The proof of the negative result stated in Theorem~\ref{THMVMO.CCC} is discussed next.

\vskip 0.08in
\begin{proof}[Proof of Theorem~\ref{THMVMO.CCC}]
From \cite[Proposition~9,\,p.\,1208]{Bour} we know that there exists 
$f\in{\mathscr{C}}^\infty({\mathbb{R}}^{n})$ such that 
\begin{equation}\label{8uy-AA-6tg.1}
\partial^\alpha f\in{\mathrm{BMO}}({\mathbb{R}}^{n}),\qquad\forall\,\alpha\in{\mathbb{N}}_0^n,
\end{equation}
and
\begin{equation}\label{8uy-AA-6tg.2}
\inf\big\{\|f-g\|_{{\mathrm{BMO}}({\mathbb{R}}^{n})}:\,g\in L^\infty({\mathbb{R}}^{n})\big\}>0.
\end{equation}
In concert with \cite[Lemme~6,\,p.\,1211]{Bour}, property \eqref{8uy-AA-6tg.1} (used for multi-indices
$\alpha\in{\mathbb{N}}^n_0$ with $|\alpha|=1$) entails $f\in{\mathrm{UC}}({\mathbb{R}}^{n})$. 
By once again using \eqref{8uy-AA-6tg.1} (with $|\alpha|=0$), this proves that
$f\in{\mathrm{UC}}({\mathbb{R}}^{n})\cap{\mathrm{BMO}}({\mathbb{R}}^{n})$, hence 
$f\in{\mathrm{VMO}}({\mathbb{R}}^{n})$. On the other hand, \eqref{8uy-AA-6tg.2} implies
that $f$ does not belong to the closure of $L^\infty({\mathbb{R}}^{n})$ in 
${\mathrm{BMO}}({\mathbb{R}}^{n})$, hence also $f$ does not belong to the closure of 
${\mathrm{UC}}({\mathbb{R}}^{n})\cap L^\infty({\mathbb{R}}^{n})$ in 
${\mathrm{BMO}}({\mathbb{R}}^{n})$. Ultimately, this proves that the space 
${\mathrm{UC}}({\mathbb{R}}^{n})\cap L^\infty({\mathbb{R}}^{n})$ is not dense in 
${\mathrm{VMO}}({\mathbb{R}}^{n})$. 
\end{proof}

The second to the last proof in this section is that of Theorem~\ref{ndyRE-NNN}.

\vskip 0.08in
\begin{proof}[Proof of Theorem~\ref{ndyRE-NNN}]
That for each $f\in{\mathrm{BMO}}({\mathbb{R}}^{n})$ the measure $\mu_f$ associated with $f$ as in 
\eqref{Bgstwy-3TRG-NNN} satisfies Carleson's condition 
\begin{equation}\label{defi-Carleson-NNN}
\|\mu_f\|_{\mathcal{C}(\mathbb{R}_{+}^{n+1})}=\sup_{Q\subset\mathbb{R}^{n}} 
\frac{1}{|Q|}\int_{0}^{\ell(Q)}\int_Q|(\psi_t\ast f)(x)|^2\frac{dx\,dt}{t}
\leq C\|f\|_{{\mathrm{BMO}}({\mathbb{R}}^{n})}^2
\end{equation}
for some constant $C\in(0,\infty)$ which depends only on the dimension $n$ and the constant in 
\eqref{Bgstwy-2aaa-NNN}, is fairly standard. Specifically, having fixed an arbitrary cube 
$Q\subset\mathbb{R}^{n}$, decompose $f=f_0+f_\infty+f_{2Q}$ where $f_0:=(f-f_{2Q}){\mathbf{1}}_{2Q}$ 
and $f_\infty:=(f-f_{2Q}){\mathbf{1}}_{\mathbb{R}^{n}\setminus 2Q}$. On account of the 
cancellation property of $\psi$, we may write $\psi_t\ast f=\psi_t\ast f_0+\psi_t\ast f_\infty$. 
Then, on the one hand, 
\begin{align}\label{defi-Carleson-NNN.2}
\frac{1}{|Q|}\int_{0}^{\ell(Q)}\int_Q|(\psi_t\ast f_0)(x)|^2\frac{dx\,dt}{t}
&\leq\frac{1}{|Q|}\int_{\mathbb{R}^{n+1}_{+}}|(\psi_t\ast f_0)(x)|^2\frac{dx\,dt}{t}
\nonumber\\[6pt]
&\leq C|Q|^{-1}\|f_0\|_{L^2({\mathbb{R}}^{n})}^2\leq C\|f\|_{{\mathrm{BMO}}({\mathbb{R}}^{n})}^2,
\end{align}
thanks to the square-function estimate \eqref{hdgswf} in Proposition~\ref{prop:SFE-early} 
(used with $n$ replaced by $n+1$ and the kernel $\theta(x,t;y):=\psi_t(x-y)$ for each $x,y\in{\mathbb{R}}^n$, $t>0$), 
and \eqref{jsfd-1}. On the other hand, for each $x\in Q$ and $t\in(0,\ell(Q))$ we may estimate 
\begin{align}\label{defi-Carleson-NNN.3}
|(\psi_t\ast f_\infty)(x)| &\leq\int_{{\mathbb{R}}^{n}\setminus 2Q}t^{-n}
\Big|\psi\Big(\frac{x-y}{t}\Big)\Big||f(y)-f_{2Q}|\,dy
\nonumber\\[6pt]
&\leq Ct\int_{{\mathbb{R}}^{n}\setminus 2Q}\frac{|f(y)-f_{2Q}|}{\big[t+|x-y|\big]^{n+1}}\,dy
\leq Ct\int_{{\mathbb{R}}^{n}\setminus 2Q}\frac{|f(y)-f_{2Q}|}{|x_Q-y|^{n+1}}\,dy
\nonumber\\[6pt]
&\leq Ct\int_{{\mathbb{R}}^{n}}\frac{|f(y)-f_{2Q}|}{\big[\ell(Q)+|x_Q-y|\big]^{n+1}}\,dy
\leq\frac{Ct}{\ell(Q)}\|f\|_{{\mathrm{BMO}}({\mathbb{R}}^{n})},
\end{align}
by virtue of \eqref{eqn:BMO-decay.88} in Lemma~\ref{lemma:BMO-decay} 
(used with $n$ replaced by $n+1$ and $\varepsilon=1$). 
Combining \eqref{defi-Carleson-NNN.2} with \eqref{defi-Carleson-NNN.3} 
then readily yields \eqref{defi-Carleson-NNN}.

Let us next observe that if $g\in\dot{\mathscr{C}}^\eta({\mathbb{R}}^n)$ for some 
$\eta\in(0,1)$ then for each $x\in{\mathbb{R}}^n$ and $t>0$ we may estimate, 
on account of \eqref{Bgstwy-2aaa-NNN},
\begin{align}\label{defi-Carleson-NNN.4}
|(\psi_t\ast g)(x)| &=\Big|\int_{{\mathbb{R}}^{n}}\psi_t(y)\big(g(x-y)-g(x)\big)\,dy\Big|
\nonumber\\[6pt]
&\leq\|g\|_{\dot{\mathscr{C}}^\eta({\mathbb{R}}^n)}\int_{{\mathbb{R}}^{n}}|\psi_t(y)||y|^\eta\,dy
\nonumber\\[6pt]
&\leq Ct^\eta\|g\|_{\dot{\mathscr{C}}^\eta({\mathbb{R}}^n)}
\int_{{\mathbb{R}}^{n}}\frac{|y|^\eta}{(1+|y|)^{n+1}}\,dy
=Ct^\eta\|g\|_{\dot{\mathscr{C}}^\eta({\mathbb{R}}^n)}.
\end{align}

Assume now that some function $f\in{\mathrm{BMO}}({\mathbb{R}}^{n})$ has been fixed. Pick $\eta\in(0,1)$ 
and choose $g\in\dot{\mathscr{C}}^\eta({\mathbb{R}}^n)\cap{\mathrm{BMO}}({\mathbb{R}}^{n})$ arbitrary. 
Then, making use of \eqref{defi-Carleson-NNN} and \eqref{defi-Carleson-NNN.4}, 
for each cube $Q\subseteq{\mathbb{R}}^n$ we may bound
\begin{align}\label{defi-Carleson-NNN.5}
&\hskip -0.20in
\frac{1}{|Q|}\int_{0}^{\ell(Q)}\int_Q|(\psi_t\ast f)(x)|^2\frac{dx\,dt}{t}
\nonumber\\[6pt]
&\hskip 1.00in
\leq\frac{2}{|Q|}\int_{0}^{\ell(Q)}\int_Q\big|\big(\psi_t\ast(f-g)\big)(x)\big|^2\frac{dx\,dt}{t}
\nonumber\\[6pt]
&\hskip 1.00in
\quad+\frac{2}{|Q|}\int_{0}^{\ell(Q)}\int_Q|(\psi_t\ast g)(x)|^2\frac{dx\,dt}{t}
\nonumber\\[6pt]
&\hskip 1.00in
\leq C\|f-g\|_{{\mathrm{BMO}}({\mathbb{R}}^{n})}^2+C\ell(Q)^{2\eta}\|g\|_{\dot{\mathscr{C}}^\eta({\mathbb{R}}^n)}^2.
\end{align}
In turn, \eqref{defi-Carleson-NNN.5} allows us to conclude that
\begin{align}\label{defi-Carleson-NNN.6}
\lim_{r\to 0^{+}}\Bigg\{\sup\limits_{\substack{Q\subset\mathbb{R}^{n}\\ \ell(Q)\leq r}}
\frac{1}{|Q|}\int_{0}^{\ell(Q)}\int_Q|(\psi_t\ast f)(x)|^2\frac{dx\,dt}{t}\Bigg\}
\leq C\|f-g\|_{{\mathrm{BMO}}({\mathbb{R}}^{n})}^2
\end{align}
which, after taking the infimum over all 
$g\in\dot{\mathscr{C}}^\eta({\mathbb{R}}^n)\cap{\mathrm{BMO}}({\mathbb{R}}^{n})$ and bearing in mind
the density result in \eqref{UpUpUp.2c}, yields \eqref{defi-Carleson-Niii}.
\end{proof}

We conclude this section by giving the proof of Theorem~\ref{jhdwtRD}.

\begin{proof}[Proof of Theorem~\ref{jhdwtRD}]
Fix $f\in{\mathrm{BMO}}({\mathbb{R}}^{n},\mathbb{C}^M)$ and let $u$ be the 
unique solution $u$ of the $\mathrm{BMO}$-Dirichlet boundary value 
problem \eqref{Dir-BVP-BMO} for $L$ in $\mathbb{R}^{n}_+$ with boundary datum $f$. 
By \eqref{eqn-Dir-BMO:u} in Theorem~\ref{them:BMO-Dir}, we have 
(with $P^L$ denoting the Poisson kernel for $L$ in $\mathbb{R}^{n}_+$ from
Theorem~\ref{kkjbhV})
\begin{equation}\label{RsEW}
u(x',t)=(P_t^L*f)(x')=\int_{{\mathbb{R}}^{n-1}_{+}}K^L(x'-y',t)f(y')\,dy',
\,\,\text{ for }\,\,(x',t)\in{\mathbb{R}}^n_{+},
\end{equation}
where $K^L$ is as in \eqref{eq:Gvav7g5}. Consider now 
\begin{equation}\label{RsEW-2}
\psi(z'):=(\psi_1,\dots,\psi_n):=
\Big((\partial_j K^L)(z',1)\Big)_{1\leq j\leq n}
\,\,\text{ for each }\,\,z'\in{\mathbb{R}}^{n-1}.
\end{equation}
Then, from item {\it (4)} and \eqref{eq:Kest} in Theorem~\ref{kkjbhV} we deduce that 
$\psi_j\in{\mathscr{C}}^\infty\big({\mathbb{R}}^{n-1},\mathbb{C}^{M\times M}\big)$ 
for each $j\in\{1,\dots,n\}$ and there exists some constant $C\in(0,\infty)$ such that
\begin{equation}\label{RsEW-3}
|\psi(z')|\leq \frac{C}{(1+|z'|)^n}\,\,\text{ and }\,\,
|\nabla\psi(z')|\leq \frac{C}{(1+|z'|)^{n+1}}
\,\,\text{ for each }\,\,z'\in{\mathbb{R}}^{n-1}.
\end{equation}
We also claim that 
\begin{equation}\label{RsEW-4}
\int_{{\mathbb{R}}^{n-1}}\psi_j(z')\,dz'=0\in\mathbb{C}^{M\times M}
\,\,\text{ for each }\,\,j\in\{1,\dots,n\}.
\end{equation}
To see why \eqref{RsEW-4} is true note that based on \eqref{RsEW-2} 
and \eqref{eq:Gvav7g5} we have 
\begin{equation}\label{RsEW-5}
\psi_j(z')=\partial_jP^L(z')\,\,
\text{ for all }\,\,z'\in{\mathbb{R}}^{n-1}\,\,\text{ and each }\,\,j\in\{1,\dots,n-1\},
\end{equation}
while
\begin{equation}\label{RsEW-6}
\psi_n(z')=(1-n)P^L(z')-z'\cdot \nabla P^L(z')\,\,\text{ for all }\,\,z'\in{\mathbb{R}}^{n-1}.
\end{equation}
Now \eqref{RsEW-4} follows from \eqref{RsEW-5}-\eqref{RsEW-6} and \eqref{eq:IG6gy.2}
via integration by parts.

Next, for each $x'\in{\mathbb{R}}^{n-1}$ and $t>0$ set $\psi_t(x'):=t^{1-n}\psi(x'/t)$. 
Then from item {\it (5)} in Theorem~\ref{kkjbhV} it follows that 
$\nabla K^L$ is homogeneous of order $-n$, thus
\begin{equation}\label{RsEW-7}
\psi_t(x')=t^{1-n}(\nabla K^L)(x'/t,1)=t(\nabla K^L)(x',t)\,\,
\text{ for each }\,\,(x',t)\in{\mathbb{R}}^{n-1}_{+}. 
\end{equation}
Combining \eqref{RsEW} and \eqref{RsEW-7} yields
\begin{equation}\label{RsEW-8}
t(\nabla u)(x',t)=\int_{{\mathbb{R}}^{n-1}}t(\nabla K^L)(x'-y',t)f(y')\,dy'
=(\psi_t\ast f)(x') 
\end{equation}
for each $x'\in{\mathbb{R}}^{n-1}$ and each $t>0$. Consequently,
\begin{equation}\label{RsEW-9}
|(\psi_t\ast f)(x')|^2\frac{dx'\,dt}{t}
=t|(\nabla u)(x',t)|^2\,dx'\,dt.
\end{equation}
In light of \eqref{RsEW-3}-\eqref{RsEW-4} we see that Theorem~\ref{ndyRE-NNN} applies
component-wise in the current setting (with $n$ replaced by $n-1$) and yields
a constant $C$ for which \eqref{defi-Carleson-Niii} holds. The latter becomes \eqref{XeeTT}
by invoking \eqref{RsEW-9} and finishes the proof of the theorem.
\end{proof}

\section{Proof of the well-posedness of the Morrey-Campanato-Dirichlet problem}
\setcounter{equation}{0}
\label{S-4}

This section is devoted to presenting the proof of Theorem~\ref{them:BMO-Dir-frac}.
Throughout fix $p,q\in[1,\infty)$.
We divide the proof into several steps, the starting point being the following claim:

\vskip 0.08in
\noindent{\tt Step~1.} {\it There exists a constant $C=C(n,L,\eta)\in(0,\infty)$ such 
that if $f\in\mathscr{E}^{\eta,p}(\mathbb{R}^{n-1},\mathbb{C}^M)$ 
then the function $u$ given at every point $(x',t)\in{\mathbb{R}}^n_{+}$ 
by $u(x',t):=(P^L_t\ast f)(x')$ is well-defined {\rm (}via an absolutely convergent 
integral{\rm )} and satisfies $u\in{\mathscr{C}}^\infty({\mathbb{R}}^n_{+},{\mathbb{C}}^M)$, 
$Lu=0$ in ${\mathbb{R}}^n_{+}$, $u\big|^{{}^{\rm n.t.}}_{\partial{\mathbb{R}}^n_{+}}=f$ 
a.e. in ${\mathbb{R}}^{n-1}$, as well as}
\begin{equation}\label{eqn:BMO-decay-EEE.2}
\sup_{(x',t)\in{\mathbb{R}}^n_{+}}\Big[
t^{1-\eta}\big|(\nabla u)(x',t)\big|\Big]\leq C\|f\|^{(\eta,p)}_{*}.
\end{equation}

The fact that $u$ is well-defined and is a smooth null-solution of $L$ in the 
upper-half space whose nontangential trace matches $f$ a.e. in ${\mathbb{R}}^{n-1}$ 
follows from \eqref{eq:aaAabgr-22BB} with $\varepsilon=1$ and 
item {\it (7)} in Theorem~\ref{kkjbhV}. To proceed, fix an arbitrary point 
$(x',t)\in{\mathbb{R}}^n_{+}$ and, making use of \eqref{eqn:B-knB}
and \eqref{jsfd-3vcgfd}, estimate 
\begin{align}\label{eqn:B-knB.35f}
\big|(\nabla u)(x',t)\big|\leq\frac{C}{t}
\int_{1}^\infty{\rm osc}_1\big(f;\lambda\,t\big)\,\frac{d\lambda}{\lambda^{2}}
\leq\frac{C}{t^{1-\eta}}\|f\|_{\ast}^{(\eta,p)},
\end{align}
from which \eqref{eqn:BMO-decay-EEE.2} readily follows. 

\vskip 0.08in
\noindent{\tt Step~2.} {\it For every function 
$u\in{\mathscr{C}}^1({\mathbb{R}}^n_{+},{\mathbb{C}}^M)$ there holds}
\begin{equation}\label{eqn:BMO-decay-EEE.8T}
\|u\|^{(\eta,q)}_{**}\leq (2\eta)^{-1/2}
\sup_{(x',t)\in{\mathbb{R}}^n_{+}}\Big[t^{1-\eta}\big|(\nabla u)(x',t)\big|\Big].
\end{equation}

This is readily seen from \eqref{ustarstar-222}. 

\vskip 0.08in
\noindent{\tt Step~3.} {\it There exists a constant $C=C(n,L,\eta,q)\in(0,\infty)$ 
with the property that for every function 
$u\in{\mathscr{C}}^\infty({\mathbb{R}}^n_{+},{\mathbb{C}}^M)$ satisfying $Lu=0$ 
in ${\mathbb{R}}^n_{+}$ there holds}
\begin{equation}\label{eqn:BMO-decay-EEE.8T2}
\sup_{(x',t)\in{\mathbb{R}}^n_{+}}\Big[t^{1-\eta}\big|(\nabla u)(x',t)\big|\Big]
\leq C\|u\|^{(\eta,q)}_{**}.
\end{equation}

For each fixed point $(x',t)\in{\mathbb{R}}^n_{+}$ use Theorem~\ref{ker-sbav}
and repeated applications of H\"older's inequality in order to estimate  
(recall that $Q_{x',t}$ is the cube in ${\mathbb{R}}^{n-1}$ centered at 
$x'$ and of side-length $t$)
\begin{align}\label{eqn:BMO-decay-EEE.3Si}
\big|(\nabla u)(x',t)\big| &\leq C\aver{Q_{x',t}\times(t/2,3t/2)}|(\nabla u)(y',s)|\,dy'ds
\nonumber\\[4pt]
&=C\aver{Q_{x',t}}\Big(\aver{(t/2,3t/2)}|(\nabla u)(y',s)|\,ds\Big)dy'
\nonumber\\[4pt]
&\leq C\Big(\,\aver{Q_{x',t}}\Big(\aver{(t/2,3t/2)}|(\nabla u)(y',s)|^2\,ds\Big)^{\frac{q}{2}}dy'
\Big)^{\frac{1}{q}}
\nonumber\\[4pt]
&\leq Ct^{-1/2}\Big(\,\aver{Q_{x',t}}
\Big(\aver{(t/2,3t/2)}|(\nabla u)(y',s)|^2 s\,ds\Big)^{\frac{q}{2}}dy'\Big)^{\frac{1}{q}}
\nonumber\\[4pt]
&\leq Ct^{-1}\Big(\,\aver{Q_{x',t}}
\Big(\int_{0}^{3t/2}|(\nabla u)(y',s)|^2 s\,ds\Big)^{\frac{q}{2}}dy'\Big)^{\frac{1}{q}}
\nonumber\\[4pt]
&\leq Ct^{-1}\Big(\frac{1}{|(3/2)Q_{x',t}|}\int_{(3/2)Q_{x',t}}
\Big(\int_{0}^{\ell((3/2)Q_{x',t})}
|(\nabla u)(y',s)|^2 s\,ds\Big)^{\frac{q}{2}}dy'\Big)^{\frac{1}{q}}
\nonumber\\[4pt]
&\leq Ct^{\eta-1}\|u\|^{(\eta,q)}_{**},
\end{align}
where the last inequality is a consequence of \eqref{ustarstar-222}. 
With this in hand, \eqref{eqn:BMO-decay-EEE.8T2} follows. 

\vskip 0.08in
\noindent{\tt Step~4.} {\it For every function 
$u\in{\mathscr{C}}^1({\mathbb{R}}^n_{+},{\mathbb{C}}^M)$ one has}
\begin{eqnarray}\label{tavav-y3.BAk}
\sup_{\stackrel{x,y\in{\mathbb{R}}^n_{+}}{x\not=y}}\frac{|u(x)-u(y)|}{|x-y|^\eta}
\leq\big(1+2/\eta\big)\sup_{(x',t)\in{\mathbb{R}}^n_{+}}\Big[
t^{1-\eta}\big|(\nabla u)(x',t)\big|\Big].
\end{eqnarray}

{\it In fact, the opposite inequality holds for smooth null-solutions of $L$ in the upper-half space.
Specifically, there exists a constant $C=C(n,L,\eta)\in(0,\infty)$ 
with the property that for every function 
$u\in{\mathscr{C}}^\infty({\mathbb{R}}^n_{+},{\mathbb{C}}^M)$ satisfying $Lu=0$ 
in ${\mathbb{R}}^n_{+}$ there holds}
\begin{equation}\label{eqn:BMO-decay-EEE.8T56-vv}
\sup_{(x',t)\in{\mathbb{R}}^n_{+}}\Big[t^{1-\eta}\big|(\nabla u)(x',t)\big|\Big]
\leq C\sup_{\stackrel{x,y\in{\mathbb{R}}^n_{+}}{x\not=y}}\frac{|u(x)-u(y)|}{|x-y|^\eta}.
\end{equation}

To justify \eqref{tavav-y3.BAk}, abbreviate
\begin{eqnarray}\label{tavav-y3.B}
C_{u,\eta}:=\sup_{(x',t)\in{\mathbb{R}}^n_{+}}\Big[
t^{1-\eta}\big|(\nabla u)(x',t)\big|\Big].
\end{eqnarray}
Pick two arbitrary distinct points $x=(x',t)\in{\mathbb{R}}^n_{+}$, 
$y=(y',s)\in{\mathbb{R}}^n_{+}$, and set $r:=|x-y|>0$. Then 
\begin{eqnarray}\label{tavav-y3.A.1}
r^{-\eta}|u(x)-u(y)|\leq I+II+III,
\end{eqnarray}
where 
\begin{eqnarray}\label{tavav-y3.A.2}
\begin{array}{l}
I:=r^{-\eta}\big|u(x',t)-u(x',t+r)\big|,
\\[12pt]
II:=r^{-\eta}\big|u(x',t+r)-u(y',s+r)\big|,
\\[12pt]
III:=r^{-\eta}\big|u(y',s+r)-u(y',s)\big|.
\end{array}
\end{eqnarray}
Then by the Fundamental Theorem of Calculus and the assumption on $u$, 
\begin{align}\label{tavav-y3.A.3}
I= &\,r^{-\eta}\big|u(x',t)-u(x',t+r)\big|
=r^{-\eta}\Big|\int_0^r(\partial_n u)(x',t+\xi)\,d\xi\Big|
\nonumber\\[4pt]
&\,\leq C_{u,\eta}r^{-\eta}\int_0^r (t+\xi)^{\eta-1}\,d\xi
\leq C_{u,\eta}r^{-\eta}\int_0^r\xi^{\eta-1}\,d\xi
\nonumber\\[4pt]
&\,=C_{u,\eta}r^{-\eta}\eta^{-1}r^{\eta}=C_{u,\eta}/\eta.
\end{align}
Moreover, $III$ may be estimated in a similar manner 
(with the same bound $C_{u,\eta}/\eta$), while 
\begin{align}\label{tavav-y3.A.4}
II &=r^{-\eta}\big|u(x',t+r)-u(y',s+r)\big|
\nonumber\\[4pt]
&=r^{-\eta}\Big|\int_0^1\frac{d}{d\theta}
\big[u\big(\theta(x',t+r)+(1-\theta)(y',s+r)\big)\big]\,d\theta\Big|
\nonumber\\[4pt]
&=r^{-\eta}\Big|\int_0^1(x'-y',t-s)\cdot
(\nabla u)\big(\theta(x',t+r)+(1-\theta)(y',s+r)\big)\,d\theta\Big|
\nonumber\\[4pt]
&\leq C_{u,\eta}r^{-\eta}|x-y|\int_0^1
\big[{\rm dist}\,\big(\theta(x',t+r)+(1-\theta)(y',s+r),
{\partial{\mathbb{R}}^n_{+}}\big)\big]^{\eta-1}\,d\theta
\nonumber\\[4pt]
&\,\leq C_{u,\eta}r^{-\eta}r\int_0^1\big[(1-\theta)s+\theta t+r\big]^{\eta-1}\,d\theta
\leq C_{u,\eta}r^{-\eta}\,r\,r^{\eta-1}=C_{u,\eta}.
\end{align}
Now \eqref{tavav-y3.BAk} follows from \eqref{tavav-y3.A.1}-\eqref{tavav-y3.A.4}.

Consider next \eqref{eqn:BMO-decay-EEE.8T56-vv}. 
Recall \eqref{eqn:BMO-decay-EEE.8T56}.
Fix a point $x=(x',t)\in{\mathbb{R}}^n_{+}$ and write $R_x$ for the cube in $\mathbb{R}^{n}$  
centered at $x$ with side-length $t/2$. Using that the function $u(\cdot)-u(x)$ 
is a null-solution of the system $L$, we may apply Theorem~\ref{ker-sbav} 
(with $\ell=1$ and $p=1$) to write 
\begin{align}\label{Twazvee-tfF}
t\big|(\nabla u)(x',t)\big| &\leq C\,\aver{R_x}|u(y)-u(x)|\,dy
\nonumber\\[4pt]
&\leq C\|u\|_{\dot{\mathscr{C}}^\eta(\mathbb{R}^{n}_+,{\mathbb{C}}^M)}
\,\aver{R_x}|x-y|^\eta\,dy\leq C\|u\|_{\dot{\mathscr{C}}^\eta(\mathbb{R}^{n}_+,{\mathbb{C}}^M)}\,t^{\eta}.
\end{align}
This readily implies \eqref{eqn:BMO-decay-EEE.8T56-vv}. 

\vskip 0.08in
\noindent{\tt Step~5.} {\it There exists a constant $C=C(n,\eta)\in(0,\infty)$
such that for every continuous function $f:{\mathbb{R}}^{n-1}\to{\mathbb{C}}^M$ one has}
\begin{eqnarray}\label{tavUgbvv-y5}
\|f\|^{(\eta,p)}_{*}\leq C
\sup_{\stackrel{x',y'\in{\mathbb{R}}^{n-1}}{x'\not=y'}}\frac{|f(x')-f(y')|}{|x'-y'|^\eta}.
\end{eqnarray}
{\it In particular, the inclusion}
\begin{eqnarray}\label{tavUgbvv-y5.PPP}
\dot{\mathscr{C}}^\eta({\mathbb{R}}^{n-1},{\mathbb{C}}^M)\hookrightarrow
{\mathscr{E}}^{\eta,p}({\mathbb{R}}^{n-1},{\mathbb{C}}^M)
\quad\text{is continuous}.
\end{eqnarray}

This is a direct consequence of \eqref{defi-BMO.2b}. 

\vskip 0.08in
\noindent{\tt Step~6.} {\it Given $f\in{\mathscr{E}}^{\eta,p}(\mathbb{R}^{n-1},\mathbb{C}^M)$, 
the function $u$ defined as in \eqref{eqn-Dir-BMO:u} solves the 
Dirichlet boundary value problem \eqref{Dir-BVP-BMO-frac} and obeys the estimates
in \eqref{Dir-BVP-BMO-Car-frac}. Moreover, 
$u\in\dot{\mathscr{C}}^\eta(\overline{{\mathbb{R}}^n_{+}},{\mathbb{C}}^M)$ and
\eqref{Dir-BVP-BMO-Car-frac22} holds as well.}

\vskip 0.08in
Fix an arbitrary function $f\in{\mathscr{E}}^{\eta,p}(\mathbb{R}^{n-1},\mathbb{C}^M)$.
From Step~1 we know that $u$ given as in \eqref{eqn-Dir-BMO:u} is well-defined, 
$u\in\mathscr{C}^\infty(\mathbb{R}^n_{+},{\mathbb{C}}^M)$, 
$Lu=0$ in ${\mathbb{R}}^n_{+}$, $f=u\big|^{{}^{\rm n.t.}}_{\partial{\mathbb{R}}^n_{+}}$ 
a.e. in ${\mathbb{R}}^n$, and satisfies \eqref{eqn:BMO-decay-EEE.2}. 
To proceed, observe that when used in concert, 
\eqref{eqn:BMO-decay-EEE.2} and \eqref{eqn:BMO-decay-EEE.8T} imply that
\begin{equation}\label{eqn:BMO-decay-EEE.8Tiii}
\|u\|^{(\eta,q)}_{**}\leq C\|f\|^{(\eta,p)}_{*}.
\end{equation}
Hence, $\|u\|^{(\eta,q)}_{**}<\infty$. On the other hand, 
combining the results proved in Step~3 and Step~4 establishes the membership of $u$ to 
$\dot{\mathscr{C}}^\eta(\mathbb{R}^n_{+},{\mathbb{C}}^M)
=\dot{\mathscr{C}}^\eta(\overline{\mathbb{R}^n_{+}},{\mathbb{C}}^M)$ 
(cf. \eqref{eqn:BMO-decay-EEE.8T56.ww}) along with the estimate
\begin{equation}\label{eqn:BMO-decay-EEE.8Ti44-aaa}
\|u\|_{\dot{\mathscr{C}}^\eta(\mathbb{R}^n_{+},{\mathbb{C}}^M)}\leq C\|u\|^{(\eta,q)}_{**}.
\end{equation}
Thanks to \eqref{eqn:BMO-decay-EEE.8Tiii}-\eqref{eqn:BMO-decay-EEE.8Ti44-aaa}
and \eqref{eqn:BMO-decay-EEE.8T56.ww}, we therefore have
$u\in\dot{\mathscr{C}}^\eta(\overline{\mathbb{R}^n_{+}},{\mathbb{C}}^M)$ and
\begin{align}\label{tavUgbvv-y5iaa}
\|f\|_{\dot{\mathscr{C}}^\eta(\mathbb{R}^{n-1},{\mathbb{C}}^M)}
&=\big\|u\big|^{{}^{\rm n.t.}}_{\partial{\mathbb{R}}^n_{+}}\big\|
_{\dot{\mathscr{C}}^\eta(\mathbb{R}^{n-1},{\mathbb{C}}^M)}
=\big\|u\big|_{\partial{\mathbb{R}}^n_{+}}\big\|
_{\dot{\mathscr{C}}^\eta(\mathbb{R}^{n-1},{\mathbb{C}}^M)}
\nonumber\\[4pt]
&\leq\|u\|_{\dot{\mathscr{C}}^\eta(\overline{\mathbb{R}^n_{+}},{\mathbb{C}}^M)}
=\|u\|_{\dot{\mathscr{C}}^\eta(\mathbb{R}^n_{+},{\mathbb{C}}^M)}
\leq C\|u\|^{(\eta,q)}_{**}
\nonumber\\[4pt]
&\leq C\|f\|^{(\eta,p)}_{*}.
\end{align}
Using \eqref{tavUgbvv-y5} and recycling part of the above estimate then yields
\begin{eqnarray}\label{tavUgbvv-y5iii}
\|f\|^{(\eta,p)}_{*}\leq C\|f\|_{\dot{\mathscr{C}}^\eta(\mathbb{R}^{n-1},{\mathbb{C}}^M)}
\leq C\|u\|^{(\eta,q)}_{**}.
\end{eqnarray}
At this stage, all desired properties of $u$ have been established.

\vskip 0.08in
\noindent{\tt Step~7.} {\it Assume that
$u\in{\mathscr{C}}^\infty({\mathbb{R}}^n_{+},{\mathbb{C}}^M)\cap
\dot{\mathscr{C}}^\eta({\mathbb{R}}^n_{+},{\mathbb{C}}^M)$ for some $\eta\in(0,1)$ satisfies 
$Lu=0$ in ${\mathbb{R}}^n_{+}$. Then}
\begin{eqnarray}\label{ayaf-tDCC.15aaa}
u\in\dot{\mathscr{C}}^\eta\big(\overline{{\mathbb{R}}^n_{+}},{\mathbb{C}}^M\big),\qquad
u\big|_{\partial{\mathbb{R}}^n_{+}}\in\dot{\mathscr{C}}^\eta\big({\mathbb{R}}^{n-1},{\mathbb{C}}^M\big)
\subset L^1\Big({\mathbb{R}}^{n-1}\,,\,\frac{1}{1+|x'|^n}\,dx'\Big)^M,
\end{eqnarray}
{\it and}
\begin{eqnarray}\label{ayaf-tDCC.15}
u(x',t)=\Big(P^L_t\ast\big(u\big|_{\partial{\mathbb{R}}^n_{+}}\big)\Big)(x'),\qquad
\forall\,(x',t)\in{\mathbb{R}}^n_{+}.
\end{eqnarray}

To justify this, observe that the two memberships listed in \eqref{ayaf-tDCC.15aaa}
are direct consequences of \eqref{eqn:BMO-decay-EEE.8T56.ww} while the 
inclusion in \eqref{ayaf-tDCC.15aaa} has been proved earlier (see \eqref{Gsyus}).

%
For each fixed $\varepsilon>0$ consider now the function 
\begin{eqnarray}\label{ayaf-tDCC.2bis}
u_\varepsilon(\cdot):=u(\cdot+\varepsilon e_n)\,\,\mbox{ in }\,\,{\mathbb{R}}^n_{+},
\end{eqnarray}
which satisfies 
\begin{eqnarray}\label{ayaf-tDCC.2}
\begin{array}{c}
u_\varepsilon\in{\mathscr{C}}^\infty\big(\overline{{\mathbb{R}}^n_{+}},{\mathbb{C}}^M\big),\qquad
Lu_\varepsilon=0\,\,\text{ in }\,\,{\mathbb{R}}^n_{+},\,\,\text{ and} 
\\[6pt]
u_\varepsilon\in\dot{\mathscr{C}}^\eta(\overline{{\mathbb{R}}^n_{+}},{\mathbb{C}}^M)
\,\,\text{ with }\,\,
\|u_\varepsilon\|_{\dot{\mathscr{C}}^\eta(\overline{{\mathbb{R}}^n_{+}},{\mathbb{C}}^M)}
\leq\|u\|_{\dot{\mathscr{C}}^\eta({\mathbb{R}}^n_{+},{\mathbb{C}}^M)}.
\end{array}
\end{eqnarray}
%
%
%
These and \eqref{eqn:BMO-decay-EEE.8T56-vv} yield
\begin{eqnarray}\label{ayaf-tDCC.3}
\sup_{x\in{\mathbb{R}}^n_{+}}|(\nabla u_\varepsilon)(x)|
\leq C(L,\eta,\varepsilon)
\|u\|_{\dot{\mathscr{C}}^\eta({\mathbb{R}}^n_{+},{\mathbb{C}}^M)}.
\end{eqnarray}
In light of \eqref{ayaf-tDCC.2} (which implies that $u_\varepsilon$ is bounded on bounded 
subsets of $\overline{{\mathbb{R}}^n_{+}}$\,), \eqref{ayaf-tDCC.3} allows us to conclude that
\begin{eqnarray}\label{ayaf-tDCC.4}
u_\varepsilon\in W^{1,2}_{\rm bd}({\mathbb{R}}^n_{+},{\mathbb{C}}^M).
\end{eqnarray}
Going further, set $f_\varepsilon(x'):=u(x',\varepsilon)$ for 
each $x'\in{\mathbb{R}}^{n-1}$. Then, on the one hand, 
\begin{eqnarray}\label{ayaf-tDCC.6}
|f_\varepsilon(x')-f_\varepsilon(y')|=|u(x',\varepsilon)-u(y',\varepsilon)|\leq
\|u\|_{\dot{\mathscr{C}}^\eta({\mathbb{R}}^n_{+},{\mathbb{C}}^M)}|x'-y'|^\eta,
\qquad\forall\,x',y'\in{\mathbb{R}}^{n-1}.
\end{eqnarray}
On the other hand, for all $x',y'\in{\mathbb{R}}^{n-1}$ we have
(with $\nabla'$ denoting the gradient in the first $n-1$ variables 
in ${\mathbb{R}}^{n-1}$)
\begin{align}\label{ayaf-tDCC.7}
|f_\varepsilon(x')-f_\varepsilon(y')|=&\,|u(x',\varepsilon)-u(y',\varepsilon)|
\leq|x'-y'|\sup_{z'\in[x',y']}|(\nabla'u)(z',\varepsilon)|
\nonumber\\[4pt]
=&\,|x'-y'|\sup_{z'\in[x',y']}|(\nabla'u_{\varepsilon/2})(z',\varepsilon/2)|
\nonumber\\[4pt]
\leq &\,|x'-y'|\,C(L,\eta,\varepsilon/2)
\|u\|_{\dot{\mathscr{C}}^\eta({\mathbb{R}}^n_{+},{\mathbb{C}}^M)},
\end{align}
where the last inequality uses \eqref{ayaf-tDCC.3}
(written for $u_{\varepsilon/2}$ and for $x=(z',\varepsilon/2)$). 
A logarithmically convex combination of \eqref{ayaf-tDCC.6}-\eqref{ayaf-tDCC.7} 
then proves that for every $\theta\in[\eta,1]$ there exists a finite constant 
$C(\theta,L,\varepsilon,u)>0$ such that
\begin{eqnarray}\label{ayaf-tDCC.8}
|f_\varepsilon(x')-f_\varepsilon(y')|\leq C(\theta,L,\varepsilon,u)|x'-y'|^\theta,
\qquad\forall\,x',y'\in{\mathbb{R}}^{n-1}.
\end{eqnarray}
Hence,
\begin{eqnarray}\label{ayaf-tDCC.9}
f_\varepsilon\in\bigcap_{\eta\leq\theta<1}\dot{\mathscr{C}}^\theta({\mathbb{R}}^{n-1},{\mathbb{C}}^M).
\end{eqnarray}
Combining \eqref{ayaf-tDCC.9}, \eqref{tavUgbvv-y5.PPP}, and Step~6 then shows that
the function 
\begin{eqnarray}\label{ayaf-tDCC.10}
w_\varepsilon(x',t):=(P^L_t\ast f_\varepsilon)(x'),\qquad
\forall\,(x',t)\in{\mathbb{R}}^n_{+}
\end{eqnarray}
satisfies 
\begin{eqnarray}\label{ayaf-tDCC.9ww}
w_\varepsilon\in{\mathscr{C}}^\infty({\mathbb{R}}^n_{+},{\mathbb{C}}^M),\qquad
Lw_\varepsilon=0\,\,\text{ in }\,\,{\mathbb{R}}^n_{+},\qquad
w_\varepsilon\in\bigcap_{\eta\leq\theta<1}
\dot{\mathscr{C}}^\theta\big(\overline{{\mathbb{R}}^n_{+}},{\mathbb{C}}^M\big).
\end{eqnarray}
In addition, from \eqref{ayaf-tDCC.8}-\eqref{ayaf-tDCC.10}, Step~5, and Step~1,
we see that $w_\varepsilon$ has the property that for each $\theta\in[\eta,1)$ 
there exists a finite constant $C(\theta,L,\varepsilon,u)>0$ such that
\begin{eqnarray}\label{ayaf-tDCC.11}
\Big[{\rm dist}\,\big(x,\partial{\mathbb{R}}^n_{+}\big)\Big]^{1-\theta}
\big|(\nabla w_\varepsilon)(x)\big|\leq C(\theta,L,\varepsilon,u),\qquad
\forall\,x\in{\mathbb{R}}^n_{+}.
\end{eqnarray}
In particular, choosing $\theta\in(\max\{\eta,1/2\},1)$, the latter property 
allows us to estimate for every $R>0$ 
\begin{align}\label{ayaf-tDCC.12}
\int_{B(0,R)\cap{\mathbb{R}}^n_{+}}|(\nabla w_\varepsilon)(x)|^2\,dx
\leq & C(\theta,L,\varepsilon,u)\int_{B(0,R)\cap{\mathbb{R}}^n_{+}}
\Big[{\rm dist}\,\big(x,\partial{\mathbb{R}}^n_{+}\big)\Big]^{2(\theta-1)}\,dx
\nonumber\\[4pt]
= &\,C(\theta,L,\varepsilon,R,u)<+\infty.
\end{align}
In concert with the last property in \eqref{ayaf-tDCC.9ww} (which goes to show that 
$w_\varepsilon$ is bounded on bounded subsets of $\overline{{\mathbb{R}}^n_{+}}$\,), 
this permits us to conclude that 
\begin{eqnarray}\label{ayaf-tDCC.11bis}
w_\varepsilon\in W^{1,2}_{\rm bd}({\mathbb{R}}^n_{+},{\mathbb{C}}^M).
\end{eqnarray}
From \eqref{ayaf-tDCC.2}, \eqref{ayaf-tDCC.4}, \eqref{ayaf-tDCC.10}, 
\eqref{ayaf-tDCC.9ww}, \eqref{ayaf-tDCC.11bis}, we then conclude that the function 
$v_\varepsilon:=u_\varepsilon-w_\varepsilon$ belongs to 
${\mathscr{C}}^\infty({\mathbb{R}}^n_{+},{\mathbb{C}}^M)$ and satisfies
\begin{eqnarray}\label{ayaf-tDCC.11bis.2}
v_\varepsilon\in W^{1,2}_{\rm bd}({\mathbb{R}}^n_{+},{\mathbb{C}}^M)
\cap\dot{\mathscr{C}}^\eta\big(\overline{{\mathbb{R}}^n_{+}},{\mathbb{C}}^M\big),\quad 
L v_\varepsilon=0\,\,\mbox{ in }\,\,{\mathbb{R}}^n_{+},\quad
v_\varepsilon\Big|_{\partial{\mathbb{R}}^n_{+}}=0.
\end{eqnarray}
Moreover, the H\"older property entails the growth estimate 
\begin{eqnarray}\label{ayaf-tDCC.15bbb-ii}
|v_\varepsilon(x)|\leq C(1+|x|^\eta),\qquad\forall\,x\in{\mathbb{R}}^n_{+},
\end{eqnarray}
where $C:=\max\big\{\|v_\varepsilon\|_{\dot{\mathscr{C}}^\eta({\mathbb{R}}^n_{+},{\mathbb{C}}^M)}\,,\,
|v_\varepsilon(0)|\big\}\in(0,\infty)$. 

The estimates near the boundary from Proposition~\ref{c1.2} then imply 
(by sending $\rho\to\infty$) that $v_\varepsilon\equiv 0$. This ultimately translates into saying that
for each $\varepsilon>0$ we have
\begin{eqnarray}\label{ayaf-tDCC.13}
u(x',t+\varepsilon)=(P_t^L\ast f_\varepsilon)(x'),\qquad
\forall\,(x',t)\in{\mathbb{R}}^n_{+}.
\end{eqnarray}
Let us also note that for each $\varepsilon>0$,
\begin{eqnarray}\label{ayaf-tDCC.14}
\sup_{y'\in{\mathbb{R}}^{n-1}}|f_\varepsilon(y')-u(y',0)|
=\sup_{y'\in{\mathbb{R}}^{n-1}}|u(y',\varepsilon)-u(y',0)|\leq
\varepsilon^\eta\|u\|_{\dot{\mathscr{C}}^\eta(\overline{{\mathbb{R}}^n_{+}},{\mathbb{C}}^M)}.
\end{eqnarray}
Hence, $f_\varepsilon\to u\big|_{\partial{\mathbb{R}}^n_{+}}$ as $\varepsilon\to 0^{+}$,
uniformly in ${\mathbb{R}}^{n-1}$. Since $P_t^L$ is absolutely integrable 
in ${\mathbb{R}}^{n-1}$, formula \eqref{ayaf-tDCC.15} then readily follows by passing
to limit $\varepsilon\to 0^{+}$ in \eqref{ayaf-tDCC.13}.

\vskip 0.08in
\noindent{\tt Step~8.} {\it Assume that}
\begin{equation}\label{Dir-BVP-BMO-frTT}
u\in{\mathscr{C}}^\infty(\mathbb{R}^{n}_{+},{\mathbb{C}}^M),\quad
Lu=0\,\,\mbox{ in }\,\,\,\,\mathbb{R}^{n}_{+},\quad\,
\|u\|^{(\eta,q)}_{**}<\infty,\quad\,
u\big|^{{}^{\rm n.t.}}_{\partial{\mathbb{R}}^{n}_{+}}=0.
\end{equation}
{\it Then necessarily $u\equiv 0$ in ${\mathbb{R}}^n_{+}$.} 

\vskip 0.08in
This is a consequence of Steps~3, 4, and 7. 

\vskip 0.08in
\noindent{\tt Step~9.} {\it The end-game in the proof of Theorem~\ref{them:BMO-Dir-frac}.}

\vskip 0.08in
Existence for the Dirichlet boundary value problem \eqref{Dir-BVP-BMO-frac} 
follows from Step~6. Uniqueness of the Dirichlet boundary value 
problem \eqref{Dir-BVP-BMO-frac} is seen from Step~8.

\section{Calder\'on-Zygmund operators on {\rm VMO}}
\setcounter{equation}{0}
\label{NEWSSS}

The main goal of this section is to develop the machinery which eventually permits us to 
prove Theorem~\ref{i87hbBV}. 

We begin by recalling (cf., e.g., \cite[Theorem~1, p.\,91]{Stein93}) that for each 
$q\in(0,\infty)$, the Hardy space $H^q({\mathbb{R}}^n)$ consists of tempered distributions 
$g$ in ${\mathbb{R}}^n$ with the property that their radial maximal function, defined as 
$({\mathcal{M}}_{\rm rad}\,g)(x):=\sup_{t>0}|(\Phi_t\ast g)(x)|$ for each 
$x\in{\mathbb{R}}^n$ (where $\Phi$ is a fixed background Schwartz function in 
${\mathbb{R}}^n$ with $\int_{{\mathbb{R}}^n}\Phi\,d{\mathscr{L}}^n\not=0$
and $\Phi_t(x):=t^{-n}\Phi(x/t)$ for each $t>0$ and $x\in{\mathbb{R}}^n$), 
satisfies
\begin{equation}\label{iygFFF}
\|g\|_{H^q({\mathbb{R}}^n)}:=\|{\mathcal{M}}_{\rm rad}\,g\|_{L^q({\mathbb{R}}^n)}<+\infty.
\end{equation}
It is then well-known that 
\begin{equation}\label{iygFFF.teee}
H^q({\mathbb{R}}^n)=L^q({\mathbb{R}}^n)\,\,\text{ if }\,\,1<q<\infty.
\end{equation}
Another classical result in harmonic analysis 
(cf., e.g., \cite[Theorem~2, p.\,107]{Stein93}, or \cite[Theorem~4.10, p.\,283]{GCRF85}) 
is the fact that distributions belonging to $H^q({\mathbb{R}}^n)$ with $q\in(0,1]$ 
admit atomic decompositions. To elaborate on this aspect, having fixed $r\in[1,\infty]$, 
call a Lebesgue measurable function $a:\mathbb{R}^{n}\rightarrow\mathbb{C}$ a $(q,r)$-atom 
provided there exists a cube $Q\subset\mathbb{R}^{n}$ such that the following localization, normalization, and cancellation properties hold:
\begin{equation}\label{defi-atom-NEW}
\mathrm{supp}\,a\subseteq Q,\qquad\|a\|_{L^r(\mathbb{R}^{n})}
\leq |Q|^{(1/r)-(1/q)},\quad\text{ and }\,\,\int_{\mathbb{R}^{n}}x^\alpha a(x)\,dx=0,
\end{equation}
for every multi-index $\alpha\in{\mathbb{N}}_0^n$ with $|\alpha|\leq n\big(\frac{1}{q}-1\big)$.
Then, given $q\in(0,1]$ and $r\in[1,\infty]$ with $q<r$, any $g\in H^q({\mathbb{R}}^n)$
may be written as $g=\sum_{j\in{\mathbb{N}}}\lambda_j a_j$ in $H^q({\mathbb{R}}^n)$
for a numerical sequence $\{\lambda_j\}_{j\in{\mathbb{N}}}$ satisfying 
$\big(\sum_{j\in{\mathbb{N}}}|\lambda_j|^q\big)^{1/q}\approx\|g\|_{H^q({\mathbb{R}}^n)}$ 
and with each $a_j$ a $(q,r)$-atom. 
In particular, this implies that if for each $q\in(0,1]$ and $r\in[1,\infty]$ 
with $q<r$ we let $H^{q,r}_{\rm fin}({\mathbb{R}}^n)$ stand for the vector space consisting 
of all finite linear combinations of $(q,r)$-atoms, then 
\begin{equation}\label{dkegfs.7}
\begin{array}{c}
H^{q,r}_{\rm fin}({\mathbb{R}}^n)=\Big\{f\in L^r_{\rm comp}({\mathbb{R}}^n):\,
\int_{\mathbb{R}^{n}}x^\alpha f(x)\,dx=0\,\,\text{ if }\,\,|\alpha|\leq n\big(\frac{1}{q}-1\big)
\Big\},
\\[10pt]
\text{$H^{q,r}_{\rm fin}({\mathbb{R}}^n)\subset H^q({\mathbb{R}}^n)$ densely,
and $H^{s,r}_{\rm fin}({\mathbb{R}}^n)\subseteq H^{q,r}_{\rm fin}({\mathbb{R}}^n)$
if $0<s\leq q$.}
\end{array}
\end{equation}
It turns out that if a given distribution $g\in H^q({\mathbb{R}}^n)$ with $0<q\leq 1$ 
additionally belongs to a Lebesgue space, or another Hardy space, then one may perform an
atomic decomposition which converges to $g$ simultaneously in all the said spaces. 
This is made precise in the lemma below. 

\begin{lemma}\label{LL-Hhna}
Suppose $0<p<\infty$, $0<q\leq 1$, $r\in(1,\infty)$ with $r\geq p$, 
and $0<s\leq\min\{p,q\}$ are given. 
Then for any $g\in H^q({\mathbb{R}}^n)\cap H^p({\mathbb{R}}^n)$ one can find 
a sequence $\{g_N\}_{N\in{\mathbb{N}}}\subset H^{s,r}_{\rm fin}({\mathbb{R}}^n)$ 
which converges to $g$ both in $H^q({\mathbb{R}}^n)$ and in $H^p({\mathbb{R}}^n)$.
\end{lemma}

\begin{proof}
Following the suggestion on \cite[p.\,948]{PiVe92} 
(where the treatment in the case $p=2$ and $q=1$ is outlined), we revisit the technology used 
to perform atomic decompositions of distributions in $H^q({\mathbb{R}}^n)$ presented 
in \cite[pp.\,345-348]{Torchinsky}, from which we borrow notation and results
(cf. also the proof of \cite[Theorem~4.6, pp.\,278-282]{GCRF85}). 
The starting point is the consideration of a function $\psi$ as in 
\cite[Lemma~1.7, p.\,345]{Torchinsky}. Among other things, 
\begin{equation}\label{7h7ggg.333}
\psi\in{\mathscr{C}}^\infty_0({\mathbb{R}}^n),\quad
\int_{{\mathbb{R}}^n}x^\alpha\psi(x)\,dx=0\,\,\text{ if }\,\,
|\alpha|\leq n\big(\tfrac{1}{s}-1\big),\,\,\text{ and $\psi$ is radial}.
\end{equation}
The latter condition implies that $\widehat{\psi}$, the Fourier transform 
of $\psi$ (normalized as in \cite{DM}), is also radial. Hence, there exists a
a real-valued function $\widetilde{\psi}$ defined on $[0,\infty)$ such that 
$\widehat{\psi}(x)=\widetilde{\psi}(|x|)$ for each $x\in{\mathbb{R}}^n$. 
Note that $\widetilde{\psi}$ necessarily satisfies
\begin{equation}\label{7h7ggg.333.bbb}
\widetilde{\psi}\in{\mathscr{C}}^\infty\big([0,\infty)\big),\quad
\widetilde{\psi}(0)=0,\,\,\text{ and $\widetilde{\psi}$ has rapid decay at infinity}.
\end{equation}
For each $t>0$ define $\psi_t(x):=t^{-n}\psi(x/t)$ for every $x\in{\mathbb{R}}^n$.

Fix now an arbitrary distribution $g\in H^q({\mathbb{R}}^n)$.
From \cite[Proposition~1.9, p.\,346]{Torchinsky} and the formula at the bottom 
of page 347 in \cite{Torchinsky} we know that there exists 
\begin{equation}\label{86YGG}
\parbox{8.30cm}{a partition $\{T_{j,k}\}_{j,k}$
of ${\mathbb{R}}^{n+1}_{+}$ consisting of pairwise disjoint measurable sets 
which depend only on $g$}
\end{equation}
such that, if $P^\Delta$ is the Poisson kernel for the Laplacian in ${\mathbb{R}}^{n+1}$
(cf. \eqref{Uah-TTT} with $n$ replaced by $n+1$) 
and $P^{\Delta}_t(x):=t^{-n}P^{\Delta}(x/t)$ for each $x\in{\mathbb{R}}^n$ and $t>0$, 
then the following properties hold:
\begin{enumerate}
\item[(a)] For each $j,k$, the function
\begin{equation}\label{7h7ggg.444}
a_{j,k}(x):=\int_{T_{j,k}}\partial_t\big(P^{\Delta}_t\ast g)(y)\psi_t(y-x)\,dy\,dt,
\qquad x\in{\mathbb{R}}^n,
\end{equation}
is a multiple of an $(s,r)$-atom. 

\item[(b)] Moreover, each $a_{j,k}$ is also a multiple of an $(q,r)$-atom, and if we write 
\begin{equation}\label{u76gggfWW}
\text{$a_{j,k}=\lambda_{j,k}\widetilde{a}_{j,k}$ for some $\lambda_{j,k}\in{\mathbb{C}}$ 
and $\widetilde{a}_{j,k}$ a genuine $(q,r)$-atom,}
\end{equation}
then there exists a constant $C>0$, independent of $g$, with the property that 
\begin{equation}\label{7h7ggg.5tEE}
\Big(\sum_{j,k}|\lambda_{j,k}|^q\Big)^{1/q}\leq C\|g\|_{H^q({\mathbb{R}}^n)}.
\end{equation}

\item[(c)] One has
\begin{equation}\label{7h7ggg.555}
g=\sum_{j,k}a_{j,k}\,\,\text{ in }\,\,H^q({\mathbb{R}}^n).
\end{equation}
\end{enumerate}

If we now set 
\begin{equation}\label{7h7ggg.555.hhbb}
g_N:=\sum_{j+k\leq N}a_{j,k}\,\,\text{ for each }\,\,N\in{\mathbb{N}},
\end{equation}
then each $g_N$ belongs to 
$H^{s,r}_{\rm fin}({\mathbb{R}}^n)\subset H^{q,r}_{\rm fin}({\mathbb{R}}^n)$, 
and \eqref{7h7ggg.555} implies 
\begin{equation}\label{7h7ggg.555.hhbb.2211}
\lim_{N\to\infty}g_N=g\,\,\text{ in }\,\,H^q({\mathbb{R}}^n). 
\end{equation}

Next, if $0<p\leq 1$ and $g\in H^q({\mathbb{R}}^n)\cap H^p({\mathbb{R}}^n)$,
then running the same argument as in  \eqref{86YGG}-\eqref{7h7ggg.555.hhbb.2211}
(in which we now view $g$ as a distribution in $H^p({\mathbb{R}}^n)$) leads to 
the conclusion that the sequence 
$\{g_N\}_{N\in{\mathbb{N}}}\subset H^{s,r}_{\rm fin}({\mathbb{R}}^n)$ constructed earlier 
in \eqref{7h7ggg.555.hhbb} also satisfies
\begin{equation}\label{7h7ggg.555.hhbb.2211.CCC}
\lim_{N\to\infty}g_N=g\,\,\text{ in }\,\,H^p({\mathbb{R}}^n). 
\end{equation}
The lemma is therefore established in the case when $p\in(0,1]$.

Henceforth, consider the case when $1<p<\infty$, i.e., assume 
$g\in H^q({\mathbb{R}}^n)\cap L^p({\mathbb{R}}^n)$ (cf. \eqref{iygFFF.teee}). 
The goal is to show that, with $g_N$ as in \eqref{7h7ggg.555.hhbb}, we also have
\begin{equation}\label{55.hhbb.2211-GGG}
\lim_{N\to\infty}g_N=g\,\,\text{ in }\,\,L^p({\mathbb{R}}^n). 
\end{equation}
This requires some preparation. Since the radial maximal function of $g$ is 
pointwise dominated by a multiple of the Hardy-Littlewood maximal function 
of $g$ (cf., e.g., \cite[(16), p.\,57]{Stein93}), it follows that 
${\mathcal{M}}_{\rm rad}\,g\in L^p({\mathbb{R}}^n)\cap L^q({\mathbb{R}}^n)$.
Given that in the current case $q\leq 1<p$, 
this forces ${\mathcal{M}}_{\rm rad}\,g\in L^1({\mathbb{R}}^n)$, hence 
$g\in H^1({\mathbb{R}}^n)$. With this in hand, the same reasoning that has led to 
\eqref{7h7ggg.555.hhbb.2211} now gives $\lim_{N\to\infty}g_N=g$ in $H^1({\mathbb{R}}^n)$.
This further implies $\lim_{N\to\infty}g_N=g$ in $L^1({\mathbb{R}}^n)$, hence also 
(by eventually restricting the index $N$ to a subsequence of ${\mathbb{N}}$)
\begin{equation}\label{Yrsf}
\lim_{N\to\infty}g_N(x)=g(x)\quad\text{ for a.e. $x\in{\mathbb{R}}^n$.}
\end{equation}
Consequently, if we set 
\begin{equation}\label{7h7ggg.555.hhbb.2}
D_N:=\bigcup_{j+k\leq N}T_{j,k}\,\,\text{ for each }\,\,N\in{\mathbb{N}},
\end{equation}
then for each $M,N\in{\mathbb{N}}$ with $N<M$ we have
\begin{align}\label{7h7ggg.666-resxC}
g_M(x)-g_N(x)=\int_{D_M\setminus D_N}\partial_t\big(P^{\Delta}_t\ast g)(y)\psi_t(y-x)\,dy\,dt,
\quad x\in{\mathbb{R}}^n.
\end{align}
Hence, if $p'$ is such that $1/p+1/p'=1$, for each function $h\in L^{p'}({\mathbb{R}}^n)$ 
and $M,N\in{\mathbb{N}}$ such that $N<M$ we may write
\begin{align}\label{7h7ggg.666-XXX}
\int_{{\mathbb{R}}^n}(g_M-g_N)(x)h(x)\,dx
=\int_{D_M\setminus D_N}\partial_t\big(P^{\Delta}_t\ast g)(y)(\psi_t\ast h)(y)\,dy\,dt.
\end{align}
Next, define 
\begin{align}\label{7h7g6-tf-Af.1}
\begin{array}{c}
G(y,t):=t\,\partial_t\big(P^{\Delta}_t\ast g)(y),\quad
F(y,t):=(\psi_t\ast h)(y),
\\[4pt]
\text{and }\,\,G_N(y,t):={\mathbf{1}}_{D_N}(y,t)\cdot G(y,t),
\end{array}
\end{align}
for each $(y,t)\in{\mathbb{R}}^{n+1}_{+}$ and $N\in{\mathbb{N}}$. 
With the Lusin area-function ${\mathcal{A}}$ defined as in \eqref{eq:def:AF}
(with $n$ replaced by $n+1$), from \eqref{7h7ggg.666-XXX}, Lemma~\ref{lexbdg} 
(used with $n$ replaced by $n+1$), and H\"older's inequality we see that
\begin{align}\label{eqn:CMS-XXX}
\Big|\int_{{\mathbb{R}}^n}(g_M-g_N)(x)h(x)\,dx\Big|
\leq C\|\mathcal{A}F\|_{L^{p'}({\mathbb{R}}^n)}\|\mathcal{A}(G_M-G_N)\|_{L^{p}({\mathbb{R}}^n)}.
\end{align}
We claim that there exists a finite constant $C>0$, independent of $h$, such that 
\begin{align}\label{eqn:CMS-XXX.a}
\|\mathcal{A}F\|_{L^{p'}({\mathbb{R}}^n)}\leq C\|h\|_{L^{p'}({\mathbb{R}}^n)},
\end{align}
and that 
\begin{align}\label{eqn:CMS-XXX.b}
\mathcal{A}(G-G_N)\to 0\,\,\text{ in }\,\,L^{p}({\mathbb{R}}^n)\,\,\text{ as }\,\,N\to\infty.
\end{align}
Granted these, we may then conclude from \eqref{eqn:CMS-XXX} that 
\begin{align}\label{7h7ggg.666.eee.ddd}
\|g_M-g_N\|_{L^p({\mathbb{R}}^n)} &=
\sup_{h\in L^{p'}({\mathbb{R}}^n),\,\|h\|_{L^{p'}({\mathbb{R}}^n)}\leq 1}
\Big|\int_{{\mathbb{R}}^n}(g_M-g_N)(x)h(x)\,dx\Big|
\nonumber\\[4pt]
&\leq C\|\mathcal{A}(G_M-G_N)\|_{L^{p}({\mathbb{R}}^n)}
\\[4pt]
&\leq C\|\mathcal{A}(G_M-G)\|_{L^{p}({\mathbb{R}}^n)}+C\|\mathcal{A}(G_N-G)\|_{L^{p}({\mathbb{R}}^n)}
\to 0\,\,\text{ as }\,\,M,N\to\infty,
\nonumber
\end{align}
thus, $\{g_N\}_{N\in{\mathbb{N}}}$ is Cauchy in $L^{p}({\mathbb{R}}^n)$. The latter combined with
\eqref{Yrsf} yields \eqref{55.hhbb.2211-GGG}. 

Turning our attention to \eqref{eqn:CMS-XXX.a} we first observe that 
\begin{align}\label{eqn:CMS-XXX.a.2}
\|\mathcal{A}F\|_{L^{p'}({\mathbb{R}}^n)}
=C\|S_\Theta h\|_{L^{p'}({\mathbb{R}}^n)}
\end{align}
where $S_\Theta$ is as in \eqref{bhxdhyswt} (with $n$ replaced by $n+1$) corresponding to
\begin{align}\label{eqn:CMS-XXX.a.3}
(\Theta h)(y,t):=\int_{{\mathbb{R}}^n}\psi_t(y-z)h(z)\,dz,
\qquad\forall\,(y,t)\in{\mathbb{R}}^{n+1}_{+}.
\end{align}
Since the kernel $\theta(y,t;z):=\psi_t(y-z)$ of the operator $\Theta$ 
satisfies (with $\varepsilon=1$ and $n$ replaced by $n+1$) 
\eqref{SFE-est-theta}, \eqref{SFE-vanish-theta}, and \eqref{SFE-est-theta-bis},
the hypotheses of Proposition~\ref{prop:SFE-early} are satisfied, and 
\eqref{hdgswf.DD} gives that 
$\|S_\Theta h\|_{L^{p'}({\mathbb{R}}^n)}\leq C\|h\|_{L^{p'}({\mathbb{R}}^n)}$. 
The estimate claimed in \eqref{eqn:CMS-XXX.a} now follows from this and 
\eqref{eqn:CMS-XXX.a.2}. 

Finally, consider the claim made in \eqref{eqn:CMS-XXX.b}. For starters, observe that
\begin{align}\label{eqn:CMS-XXX.b.1}
0\leq\mathcal{A}G_N\leq\mathcal{A}G\,\,\text{ in }\,\,{\mathbb{R}}^n,
\,\,\text{ for each }\,\,N\in{\mathbb{N}}.
\end{align}
Also, 
\begin{align}\label{eqn:CMS-XXX.a.2-TT}
\|\mathcal{A}G\|_{L^{p}({\mathbb{R}}^n)}=\|S_\Theta g\|_{L^{p}({\mathbb{R}}^n)}
\end{align}
where now the operator $\Theta$ is taken to be  
\begin{align}\label{eqn:CMS-XXX.a.3-TT}
(\Theta g)(y,t):=\int_{{\mathbb{R}}^n}t\partial_t\big(P^\Delta_t(y-z)\big)g(z)\,dz,
\qquad\forall\,(y,t)\in{\mathbb{R}}^{n+1}_{+}.
\end{align}
Since its kernel, $\theta(y,t;z):=t\partial_t\big(P^\Delta_t(y-z)\big)$ once again 
satisfies (with $\varepsilon=1$ and $n$ replaced by $n+1$) \eqref{SFE-est-theta}, 
\eqref{SFE-vanish-theta}, and \eqref{SFE-est-theta-bis}, Proposition~\ref{prop:SFE-early} 
applies and \eqref{hdgswf.DD} guarantees that  
$\|S_\Theta g\|_{L^{p}({\mathbb{R}}^n)}\leq C\|g\|_{L^{p}({\mathbb{R}}^n)}$. 
Together with \eqref{eqn:CMS-XXX.a.2-TT}, this shows that 
\begin{align}\label{eqn:CMS-XXX.a.2-TT.YY}
\mathcal{A}G\in L^{p}({\mathbb{R}}^n).
\end{align}
In particular, there exists a Lebesgue measurable set $E\subseteq{\mathbb{R}}^n$ 
satisfying 
\begin{align}\label{eqn:CMS-XXX.a.2-TT.YY.2}
\text{${\mathscr{L}}^n(E)=0$ and $(\mathcal{A}G)(x)<+\infty$
for each $x\in{\mathbb{R}}^n\setminus E$}.
\end{align}
For each fixed $x\in{\mathbb{R}}^n\setminus E$, we have
\begin{align}\label{eqn:CMS-XXX.a.3-TT.jj}
\big(\mathcal{A}(G-G_N)\big)(x)=\Big(\int_{\Gamma_\kappa(x)}
{\mathbf{1}}_{{\mathbb{R}}^{n+1}_{+}\setminus D_N}(y,t)|G(y,t)|^2\frac{dy\,dt}{t^{n+1}}\Big)^{1/2},
\end{align}
and the fact that $(\mathcal{A}G)(x)<+\infty$ implies that 
\begin{align}\label{eqn:CMS-XXX.a.3-TT.jj.2}
0\leq{\mathbf{1}}_{{\mathbb{R}}^{n+1}_{+}\setminus D_N}|G|
\leq|G|\in L^2\Big(\Gamma_\kappa(x)\,,\,\frac{dy\,dt}{t^{n+1}}\Big)
\end{align}
Since, clearly, ${\mathbf{1}}_{{\mathbb{R}}^{n+1}_{+}\setminus D_N}|G|$ converges pointwise 
to zero as $N\to\infty$, Lebesgue's Dominated Convergence Theorem applies and gives that 
$\big(\mathcal{A}(G-G_N)\big)(x)\to 0$ as $N\to\infty$. With this in hand, one more application 
of Lebesgue's Dominated Convergence Theorem (bearing in mind \eqref{eqn:CMS-XXX.a.2-TT.YY},
\eqref{eqn:CMS-XXX.b.1}, and the fact that ${\mathscr{L}}^n(E)=0$) proves \eqref{eqn:CMS-XXX.b}.
This completes the proof of Lemma~\ref{LL-Hhna}.
\end{proof}

Having disposed of Lemma~\ref{LL-Hhna}, we now proceed to show that 
the $\widetilde{\rm BMO}$-$H^1$ duality pairing is compatible with integral pairing 
for dual Lebesgue spaces, as made precise in the next lemma.
As a preamble, we recall the specific nature of the duality pairing 
$\langle\cdot,\cdot\rangle$ between $\widetilde{\rm BMO}({\mathbb{R}}^n)$ and 
$H^1({\mathbb{R}}^n)$. Concretely, \cite[Theorem~1, p.\,142]{Stein93} gives that 
for each $r\in(1,\infty]$ 
\begin{equation}\label{dk-TTFF}
\langle[f],g\rangle=\int_{{\mathbb{R}}^n}fg\,d{\mathscr{L}}^n,\qquad
\forall\,f\in{\rm BMO}({\mathbb{R}}^n),\,\,\forall\,g\in H^{1,r}_{\rm fin}({\mathbb{R}}^n),
\end{equation}
which further implies that whenever $f\in{\rm BMO}({\mathbb{R}}^n)$, 
$g\in H^1({\mathbb{R}}^n)$, and $\{g_N\}_{N\in{\mathbb{N}}}\subseteq 
H^{1,r}_{\rm fin}({\mathbb{R}}^n)$ is such that $\lim_{N\to\infty}g_N=g$ in 
$H^1({\mathbb{R}}^n)$, then 
\begin{equation}\label{dk-TTFF.2}
\lim_{N\to\infty}\int_{{\mathbb{R}}^n}fg_N\,d{\mathscr{L}}^n
\,\,\text{ exists and equals }\,\,\langle[f],g\rangle.
\end{equation}
As a consequence, whenever $f\in{\rm BMO}({\mathbb{R}}^n)$, and
$g\in H^1({\mathbb{R}}^n)$ may be written as $g=\sum_{j\in{\mathbb{N}}}\lambda_j a_j$ 
in $H^1({\mathbb{R}}^n)$ for a numerical sequence $\{\lambda_j\}_{j\in{\mathbb{N}}}$ satisfying 
$\sum_{j\in{\mathbb{N}}}|\lambda_j|<\infty$ and with each $a_j$ a $(1,r)$-atom, we may write
\begin{equation}\label{dk-TTFF.2-pj}
\langle[f],g\rangle=\sum_{j=1}^\infty\lambda_j\int_{{\mathbb{R}}^n}f a_j\,d{\mathscr{L}}^n.
\end{equation}

\begin{lemma}\label{U-gGg}
Consider $f\in{\rm BMO}({\mathbb{R}}^n)\cap L^{p'}({\mathbb{R}}^n)$ and 
$g\in H^1({\mathbb{R}}^n)\cap L^p({\mathbb{R}}^n)$ where 
$p,p'\in(1,\infty)$ are such that $1/p+1/p'=1$. Then, with $\langle\cdot,\cdot\rangle$ 
denoting the $\widetilde{\rm BMO}$-$H^1$ duality bracket, one has
\begin{equation}\label{dkegfs.8.96544332}
\langle[f],g\rangle=\int_{{\mathbb{R}}^n}fg\,d{\mathscr{L}}^n.
\end{equation}
\end{lemma}

\begin{proof}
Having picked $r\in[p,\infty)$, Lemma~\ref{LL-Hhna} guarantees the existence of a 
sequence $\{g_N\}_{N\in{\mathbb{N}}}\subseteq H^{1,r}_{\rm fin}({\mathbb{R}}^n)$
such that $\lim_{N\to\infty}g_N=g$ both in $H^1({\mathbb{R}}^n)$ and in 
$L^p({\mathbb{R}}^n)$. Then, thanks to \eqref{dk-TTFF.2} and the $L^p$-$L^{p'}$ 
duality, we have
\begin{equation}\label{dkegfs.8.96544332.AAA}
\langle[f],g\rangle=\lim_{N\to\infty}\int_{{\mathbb{R}}^n}fg_N\,d{\mathscr{L}}^n
=\int_{{\mathbb{R}}^n}fg\,d{\mathscr{L}}^n,
\end{equation}
which establishes \eqref{dkegfs.8.96544332}. 
\end{proof}

Recall from \cite[Theorem~5.30, p.\,307]{GCRF85} that 
\begin{equation}\label{WSWS-8htggg.1}
\big(H^q({\mathbb{R}}^n)\big)^\ast=\dot{\mathscr{C}}^\eta({\mathbb{R}}^n)\big/_{\sim},\quad
\frac{n}{n+1}<q<1,\quad\eta=n\Big(\frac{1}{q}-1\Big)\in(0,1).
\end{equation}
The manner in which the H\"older-Hardy duality is understood in \eqref{WSWS-8htggg.1} is 
similar to \eqref{dk-TTFF}-\eqref{dk-TTFF.2}. Specifically, with $(\cdot,\cdot)$ 
denoting the said H\"older-Hardy duality bracket, $q,\eta$ as in \eqref{WSWS-8htggg.1},
and with $r\in(1,\infty]$ fixed, we have
\begin{equation}\label{dk-TTFF-HHH}
([f],g)=\int_{{\mathbb{R}}^n}fg\,d{\mathscr{L}}^n,\qquad
\forall\,f\in\dot{\mathscr{C}}^\eta({\mathbb{R}}^n),\,\,
\forall\,g\in H^{q,r}_{\rm fin}({\mathbb{R}}^n).
\end{equation}
This further implies that whenever $f\in\dot{\mathscr{C}}^\eta({\mathbb{R}}^n)$, 
$g\in H^q({\mathbb{R}}^n)$, and $\{g_N\}_{N\in{\mathbb{N}}}\subseteq 
H^{q,r}_{\rm fin}({\mathbb{R}}^n)$ is such that $\lim_{N\to\infty}g_N=g$ in 
$H^q({\mathbb{R}}^n)$, then 
\begin{equation}\label{dk-TTFF.2-HHH}
\lim_{N\to\infty}\int_{{\mathbb{R}}^n}fg_N\,d{\mathscr{L}}^n
\,\,\text{ exists and equals }\,\,([f],g).
\end{equation}
In particular, whenever $f\in\dot{\mathscr{C}}^\eta({\mathbb{R}}^n)$, and
$g\in H^q({\mathbb{R}}^n)$ may be written as $g=\sum_{j\in{\mathbb{N}}}\lambda_j a_j$ 
in $H^q({\mathbb{R}}^n)$ for a numerical sequence $\{\lambda_j\}_{j\in{\mathbb{N}}}$ satisfying 
$\sum_{j\in{\mathbb{N}}}|\lambda_j|^q<\infty$ and with each $a_j$ a $(q,r)$-atom, we have
\begin{equation}\label{dk-TTFF.2-pj-uu}
([f],g)=\sum_{j=1}^\infty\lambda_j\int_{{\mathbb{R}}^n}f a_j\,d{\mathscr{L}}^n.
\end{equation}

In a parallel fashion to Lemma~\ref{U-gGg} we have the following compatibility result.

\begin{lemma}\label{U-gGg-HHH}
Suppose $f\in\dot{\mathscr{C}}^\eta({\mathbb{R}}^n)\cap L^{p'}({\mathbb{R}}^n)$ and 
$g\in H^q({\mathbb{R}}^n)\cap L^p({\mathbb{R}}^n)$ where 
$p,p'\in(1,\infty)$ are such that $1/p+1/p'=1$, while
$q\in\big(\tfrac{n}{n+1}\,,\,1\big)$ and $\eta=n\big(\tfrac{1}{q}-1\big)\in(0,1)$.
Then, with $(\cdot,\cdot)$ denoting the $\dot{\mathscr{C}}^\eta\big/_{\sim}$-$H^q$ 
duality bracket, there holds
\begin{equation}\label{dkegfs.8.96544332-HHH}
([f],g)=\int_{{\mathbb{R}}^n}fg\,d{\mathscr{L}}^n.
\end{equation}
\end{lemma}

\begin{proof}
Choose some $r\in[p,\infty)$. From Lemma~\ref{LL-Hhna} we then know that there exists 
a sequence $\{g_N\}_{N\in{\mathbb{N}}}\subseteq H^{q,r}_{\rm fin}({\mathbb{R}}^n)$
such that $\lim_{N\to\infty}g_N=g$ both in $H^q({\mathbb{R}}^n)$ and in 
$L^p({\mathbb{R}}^n)$. By virtue of \eqref{dk-TTFF.2-HHH} and the $L^p$-$L^{p'}$ 
duality we may then write
\begin{equation}\label{dkegfs.8.96544332.AAA-fg}
([f],g)=\lim_{N\to\infty}\int_{{\mathbb{R}}^n}fg_N\,d{\mathscr{L}}^n
=\int_{{\mathbb{R}}^n}fg\,d{\mathscr{L}}^n,
\end{equation}
which proves \eqref{dkegfs.8.96544332-HHH}.
\end{proof}

There is another compatibility result, discussed in the next lemma, 
which is going to be relevant for us shortly. 

\begin{lemma}\label{U-gGg-HHH.2}
Suppose $f\in\dot{\mathscr{C}}^\eta({\mathbb{R}}^n)\cap{\rm BMO}({\mathbb{R}}^n)$ 
and $g\in H^q({\mathbb{R}}^n)\cap H^1({\mathbb{R}}^n)$ where 
$q\in\big(\tfrac{n}{n+1}\,,\,1\big)$ and $\eta\in(0,1)$ are related via
$\eta=n\big(\tfrac{1}{q}-1\big)$. Then, with $(\cdot,\cdot)$ and 
$\langle\cdot,\cdot\rangle$ denoting, respectively,
the $\dot{\mathscr{C}}^\eta\big/_{\sim}$-$H^q$ and $\widetilde{\rm BMO}$-$H^1$ 
duality brackets, there holds
\begin{equation}\label{dkegfs.8.96544332-HHH.2}
([f],g)=\langle[f],g\rangle.
\end{equation}
\end{lemma}

\begin{proof}
Fix some $r\in(1,\infty)$ and once again invoke Lemma~\ref{LL-Hhna} to produce a sequence 
$\{g_N\}_{N\in{\mathbb{N}}}\subseteq H^{q,r}_{\rm fin}({\mathbb{R}}^n)$
such that $\lim_{N\to\infty}g_N=g$ both in $H^1({\mathbb{R}}^n)$ and in 
$H^q({\mathbb{R}}^n)$. Then
\begin{align}\label{dkegfs.8.96544332-HHH.2-XXX}
([f],g)=\lim_{N\to\infty}\int_{{\mathbb{R}}^n}fg_N\,d{\mathscr{L}}^n
=\langle[f],g\rangle,
\end{align}
where the first equality is provided by \eqref{dk-TTFF.2-HHH} and 
the second equality is given by \eqref{dk-TTFF.2}.
\end{proof}

Finally, we record a compatibility result for the H\"older-Hardy duality bracket 
considered for two choices of the parameters involved in the definitions of these spaces. 

\begin{lemma}\label{U-gGg-HHH.2-grf}
Assume $f\in\dot{\mathscr{C}}^{\eta_1}({\mathbb{R}}^n)\cap
\dot{\mathscr{C}}^{\eta_2}({\mathbb{R}}^n)$ and 
$g\in H^{q_1}({\mathbb{R}}^n)\cap H^{q_2}({\mathbb{R}}^n)$ where 
$q_j\in\big(\tfrac{n}{n+1}\,,\,1\big)$ and $\eta_j\in(0,1)$ are related via
$\eta_j=n\big(\tfrac{1}{q_j}-1\big)$ for $j=1,2$. Then, if for each $j=1,2$ one denotes by 
$(\cdot,\cdot)_j$ the $\dot{\mathscr{C}}^{\eta_j}\big/_{\sim}$-$H^{q_j}$ duality bracket, 
there holds
\begin{equation}\label{dkegfs.8.96544332-HHH.2-ttt}
([f],g)_1=([f],g)_2.
\end{equation}
\end{lemma}

\begin{proof}
Pick some $r\in(1,\infty)$ and introduce $q:=\min\{q_1,q_2\}$. By once more 
invoking Lemma~\ref{LL-Hhna}, we can find a sequence 
$\{g_N\}_{N\in{\mathbb{N}}}\subseteq H^{q,r}_{\rm fin}({\mathbb{R}}^n)$
such that $\lim_{N\to\infty}g_N=g$ both in $H^{q_1}({\mathbb{R}}^n)$ and in 
$H^{q_2}({\mathbb{R}}^n)$. Bearing in mind that each $g_N$ belongs to both 
$H^{q_1,r}_{\rm fin}({\mathbb{R}}^n)$ and $H^{q_2,r}_{\rm fin}({\mathbb{R}}^n)$
(cf. \eqref{dkegfs.7}), we may write 
\begin{align}\label{dkegfs.8.96544332-HHH.2-XXX-uu}
([f],g)_1=\lim_{N\to\infty}\int_{{\mathbb{R}}^n}fg_N\,d{\mathscr{L}}^n=([f],g)_2,
\end{align}
where both equalities are implied by \eqref{dk-TTFF-HHH}.
\end{proof}

In the proposition below we elaborate on a standard duality procedure according 
to which one associates a certain bounded mapping on ${\rm BMO}$ with any given 
Calder\'on-Zygmund operator which annihilates constants; cf., e.g.,
\cite[Corollaire, p.\,239]{Meyer}, \cite[p.\,156]{Stein93}, \cite[Corollary~2, p.\,151]{FS}.
The goal is to prove that the mappings induced by such a Calder\'on-Zygmund operator on a 
variety of spaces (Lebesgue, Hardy, {\rm BMO}, H\"older) are all compatible with one another, 
and to provide norm estimates in cases of interest. To state this result in precise terms, 
recall that the class ${\rm SCZ}(n,\gamma)$ has been introduced in Definition~\ref{h6f43sSSSA}.

\begin{proposition}\label{jhsdgf}
Fix $n\in{\mathbb{N}}$, $\gamma\in(0,1]$, and let $T\in{\rm SCZ}(n,\gamma)$ satisfy $T(1)=0$. 
Then the following statements are true.
\begin{enumerate}
\item[(i)] For each $p\in[2,\infty)$ the operator $T$, originally considered 
on $L^p({\mathbb{R}}^n)\cap L^2({\mathbb{R}}^n)$, extends uniquely to a linear 
and bounded mapping 
\begin{equation}\label{dkegfs.2}
T:L^p({\mathbb{R}}^n)\longrightarrow L^p({\mathbb{R}}^n).
\end{equation}
Moreover, the operators defined as above for any two arbitrary 
choices of $p$ in $[2,\infty)$ act in a compatible fashion with one another.

\vskip 0.08in
\item[(ii)] For each $p'\in(1,2]$ the operator $T^\top$, originally considered 
on $L^{p'}({\mathbb{R}}^n)\cap L^2({\mathbb{R}}^n)$, extends uniquely to a linear and bounded mapping 
\begin{equation}\label{dkegfs.2b}
T^\top:L^{p'}({\mathbb{R}}^n)\longrightarrow L^{p'}({\mathbb{R}}^n).
\end{equation}
Moreover, the operators defined as above for any two arbitrary 
choices of $p'$ in $(1,2]$ act in a compatible fashion with one another, and 
whenever $p\in[2,\infty)$ and $p'\in(1,2]$ are such that $1/p+1/p'=1$
the transposed of \eqref{dkegfs.2} is precisely \eqref{dkegfs.2b}.

\vskip 0.08in
\item[(iii)] The operator \eqref{dkegfs.2b} further extends uniquely to a well-defined, 
linear and bounded mapping in the context of Hardy spaces. Specifically, 
whenever $\frac{n}{n+\gamma}<q\leq 1$ there exists a unique linear and bounded operator 
\begin{equation}\label{dkegfs.3}
T^\top:H^q({\mathbb{R}}^n)\longrightarrow H^q({\mathbb{R}}^n).
\end{equation}
which acts in a compatible fashion with \eqref{dkegfs.2b}. Moreover, 
the operators in \eqref{dkegfs.3}, considered for two arbitrary choices of $q$, 
are compatible with one another. Also, for each $p\in[2,\infty)$ there exist 
$\theta\in(0,1)$ and $c\in(0,\infty)$ depending only on $n,\gamma,q,p$ such that,
with $C''$ as in \eqref{ytgfff}, 
\begin{equation}\label{887hytg-ff}
\big\|T^\top\big\|_{{\mathscr{B}}(H^q({\mathbb{R}}^n))}
\leq c\|T\|_{{\mathscr{B}}(L^{p}({\mathbb{R}}^n))}^{1-\theta}
\big(C''+\|T\|_{{\mathscr{B}}(L^{p}({\mathbb{R}}^n))}\big)^\theta.
\end{equation}

\vskip 0.08in
\item[(iv)] The operator 
\begin{equation}\label{dkegfs.4}
\widetilde{T}:\widetilde{\rm BMO}({\mathbb{R}}^n)
\longrightarrow\widetilde{\rm BMO}({\mathbb{R}}^n)
\end{equation}
defined by setting {\rm (}with $\langle\cdot,\cdot\rangle$ denoting the 
$\widetilde{\rm BMO}$-$H^1$ duality pairing{\rm )}
\begin{equation}\label{dkegfs.5}
\big\langle\widetilde{T}[f],g\big\rangle:=\big\langle[f],T^\top g\big\rangle,
\quad\forall\,[f]\in\widetilde{\rm BMO}({\mathbb{R}}^n),\,\,\forall\,g\in H^1({\mathbb{R}}^n),
\end{equation}
is well-defined, linear, bounded, and compatible with \eqref{dkegfs.2} in the sense that 
for each $p\in[2,\infty)$ one has
\begin{equation}\label{dkegfs.6}
\widetilde{T}[f]=[Tf],\qquad\forall\,f\in{\rm BMO}({\mathbb{R}}^n)\cap L^p({\mathbb{R}}^n).
\end{equation}
Moreover, for each $p\in[2,\infty)$ there exist $\theta\in(0,1)$ and $c\in(0,\infty)$ 
depending only on $n,\gamma,p$ such that, with $C''$ as in \eqref{ytgfff}, 
\begin{equation}\label{dkegfs.5-AAA}
\big\|\widetilde{T}\big\|_{{\mathscr{B}}(\,\widetilde{\rm BMO}({\mathbb{R}}^n))}
\leq c\|T\|_{{\mathscr{B}}(L^{p}({\mathbb{R}}^n))}^{1-\theta}
\big(C''+\|T\|_{{\mathscr{B}}(L^{p}({\mathbb{R}}^n))}\big)^\theta.
\end{equation} 

\item[(v)] Given any $\eta\in(0,\gamma)$, the operator 
\begin{equation}\label{dkegfs.4NNN}
\widehat{T}:\dot{\mathscr{C}}^\eta({\mathbb{R}}^n)\big/_{\sim}\,\,
\longrightarrow\dot{\mathscr{C}}^\eta({\mathbb{R}}^n)\big/_{\sim}
\end{equation}
defined by setting, with $q:=n/(n+\eta)\in\big(\tfrac{n}{n+\gamma}\,,\,1\big)$
and $(\cdot,\cdot)$ denoting the $\dot{\mathscr{C}}^\eta/_{\sim}$-$H^q$ duality pairing, 
\begin{equation}\label{dkegfs.5NNN}
\big(\widehat{T}[f],g\big):=\big([f],T^\top g\big),
\quad\forall\,[f]\in\dot{\mathscr{C}}^\eta({\mathbb{R}}^n)\big/_{\sim},\,\,
\forall\,g\in H^q({\mathbb{R}}^n),
\end{equation}
is well-defined, linear, bounded, and compatible with \eqref{dkegfs.2} 
and \eqref{dkegfs.4}, in the sense that for each $p\in[2,\infty)$ one has
\begin{align}\label{dkegfs.6-NNN.1}
& \widehat{T}[f]=[Tf],\qquad\forall\,f\in\dot{\mathscr{C}}^\eta({\mathbb{R}}^n)
\cap L^p({\mathbb{R}}^n),
\\[4pt]
& \widehat{T}[f]=\widetilde{T}[f],\qquad
\forall\,f\in\dot{\mathscr{C}}^\eta({\mathbb{R}}^n)\cap{\rm BMO}({\mathbb{R}}^n).
\label{dkegfs.6-NNN.2}
\end{align}
In addition, the operators in \eqref{dkegfs.4NNN}, considered for two arbitrary 
choices of $\eta$, are also compatible with one another. 
\end{enumerate}
\end{proposition}

Of course, if actually $T\in{\rm CZ}(n,\gamma)$ then we may take $p,p'\in(1,\infty)$ arbitrary
(retaining condition $1/p+1/p'=1$ in the second part of item {\it (ii)} though) throughout the 
statement of Proposition~\ref{jhsdgf}.

\vskip 0.08in
\begin{proof}[Proof of Proposition~\ref{jhsdgf}]
Working with $T^\top$ which, by design, is a bounded operator on $L^2({\mathbb{R}}^n)$ and
whose kernel $K^\top\in L^1_{\rm loc}({\mathbb{R}}^n\times{\mathbb{R}}^n\setminus\text{\rm diag})$ 
has the property that there exist $C'_{K},C''_{K}\in(0,\infty)$ such that, for every $x,y\in{\mathbb{R}}^n$ 
with $x\not=y$ and each $z\in{\mathbb{R}}^n$ with $|x-z|<\tfrac12|x-y|$, 
\begin{align}\label{ytgfff-123456}
|K^\top(x,y)|\leq\frac{C'_{K}}{|x-y|^n}\,\,\text{ and }\,\,
|K^\top(y,x)-K^\top(y,z)|\leq C''_{K}\frac{|x-z|^\gamma}{|x-y|^{n+\gamma}},
\end{align}
and relying on the Calder\'on-Zygmund Lemma in the usual fashion, it follows that 
$T^\top$ induces a well-defined linear and bounded mapping
\begin{align}\label{ytgfff-123456.BBB}
T^\top:L^{1}({\mathbb{R}}^n)\longrightarrow L^{1,\infty}({\mathbb{R}}^n).
\end{align}
Hence, via Marcinkiewicz's interpolation theorem, we conclude that 
$T^\top:L^2({\mathbb{R}}^n)\to L^2({\mathbb{R}}^n)$ has a unique extension 
to linear and bounded operator from $L^{p'}({\mathbb{R}}^n)$ into itself
for each $p'\in(1,2]$. From \cite[Theorem~1.17, p.15]{Rudin} it follows that 
\begin{equation}\label{i76tGFG}
\parbox{10.00cm}{given $p_1,p_2\in(1,\infty)$ and 
$f\in L^{p_1}({\mathbb{R}}^n)\cap L^{p_2}({\mathbb{R}}^n)$,
there exists a sequence $\{s_j\}_{j\in{\mathbb{N}}}$ of simple functions in ${\mathbb{R}}^n$ 
which converges to $f$ simultaneously in $L^{p_1}({\mathbb{R}}^n)$ 
and in $L^{p_2}({\mathbb{R}}^n)$.}
\end{equation}
In turn, this readily implies that the operators in \eqref{dkegfs.2b}, considered 
for any two arbitrary choices of $p'$ in $(1,2]$, act in a compatible fashion 
with one another. Consider next $p\in[2,\infty)$ such that $1/p+1/p'=1$. 
Since for each $f\in L^p({\mathbb{R}}^n)\cap L^2({\mathbb{R}}^n)$ and 
$g\in L^{p'}({\mathbb{R}}^n)\cap L^2({\mathbb{R}}^n)$ we may estimate 
\begin{align}\label{i76tGFG.2-hhihb}
\Big|\int_{{\mathbb{R}}^n}(Tf)g\,d{\mathscr{L}}^n\Big|
&=\Big|\int_{{\mathbb{R}}^n}f(T^\top g)\,d{\mathscr{L}}^n\Big|
\leq\|f\|_{L^{p}({\mathbb{R}}^n)}\|T^\top g\|_{L^{p'}({\mathbb{R}}^n)}
\nonumber\\[6pt]
&\leq C\|f\|_{L^{p}({\mathbb{R}}^n)}\|g\|_{L^{p'}({\mathbb{R}}^n)},
\end{align}
and since, generally speaking, 
\begin{equation}\label{i76tGFG-ugg}
\text{if }\,\,h\in L^2({\mathbb{R}}^n)\,\,\text{ then }\,\,\|h\|_{L^p({\mathbb{R}}^n)}
=\sup\limits_{\substack{g\in L^{p'}({\mathbb{R}}^n)\cap L^2({\mathbb{R}}^n)\\ \|g\|_{L^{p'}({\mathbb{R}}^n)}\leq 1}}
\Big|\int_{{\mathbb{R}}^n}hg\,d{\mathscr{L}}^n\Big|,
\end{equation}
we conclude that there exists $C\in(0,\infty)$ such that 
$\|Tf\|_{L^p({\mathbb{R}}^n)}\leq C\|f\|_{L^p({\mathbb{R}}^n)}$
for every function $f\in L^p({\mathbb{R}}^n)\cap L^2({\mathbb{R}}^n)$. By density
it follows that $T$, originally considered on $L^p({\mathbb{R}}^n)\cap L^2({\mathbb{R}}^n)$, 
extends uniquely to a linear and bounded mapping as in \eqref{dkegfs.2}. By once again appealing to
\eqref{i76tGFG} we see that the operators in \eqref{dkegfs.2}, considered 
for any two arbitrary choices of $p$ in $[2,\infty)$, act in a compatible fashion 
with one another. Finally, granted the continuity properties established above, the identity 
\begin{align}\label{i76tGFG.2-hhi-iu7tt}
\int_{{\mathbb{R}}^n}(Tf)g\,d{\mathscr{L}}^n=\int_{{\mathbb{R}}^n}f(T^\top g)\,d{\mathscr{L}}^n,\quad
f\in L^p({\mathbb{R}}^n)\cap L^2({\mathbb{R}}^n),\,\,
g\in L^{p'}({\mathbb{R}}^n)\cap L^2({\mathbb{R}}^n),
\end{align}
further extends by density to 
\begin{equation}\label{i76tGFG.2}
\int_{{\mathbb{R}}^n}(Tf)g\,d{\mathscr{L}}^n
=\int_{{\mathbb{R}}^n}f(T^\top g)\,d{\mathscr{L}}^n,\qquad
\forall\,f\in L^p({\mathbb{R}}^n),\,\,\forall\,g\in L^{p'}({\mathbb{R}}^n),
\end{equation}
where $T$ is as in \eqref{dkegfs.2} and $T^\top$ is as in \eqref{dkegfs.2b}.
This finishes the proofs of the claims in items {\it (i)}-{\it (ii)}. 

Consider next the claims made in item {\it (iii)}. Throughout, fix an exponent $p'\in(1,2]$, 
set $p:=p'/(p'-1)\in[2,\infty)$, take $r\in(p',\infty)$, and pick $q\in\big(\tfrac{n}{n+\gamma}\,,\,1\big]$ 
arbitrary. Since these choices entail $(n+\gamma)/n-1/p'>1/q-1/p'$, it is possible to select 
\begin{equation}\label{887hytg-ytr}
\theta\in(0,1)\,\,\text{ such that }\,\,(n+\gamma)/n-1/p'>\big(1/q-1/p'\big)/\theta.
\end{equation}
We first claim that 
\begin{equation}\label{887hytg}
\begin{array}{c}
\text{for each $(q,r)$-atom $a$ in ${\mathbb{R}}^n$ we have }\,\,T^\top a\in H^q({\mathbb{R}}^n)\,\,\text{ and}
\\[4pt]
\big\|T^\top a\big\|_{H^q({\mathbb{R}}^n)}
\leq C:=c\|T\|_{{\mathscr{B}}(L^{p}({\mathbb{R}}^n))}^{1-\theta}
\big(C_K''+\|T\|_{{\mathscr{B}}(L^{p}({\mathbb{R}}^n))}\big)^\theta,
\end{array}
\end{equation}
where $c\in(0,\infty)$ depends only on $n,\gamma,q,p$, and where $C''_K$ is as in \eqref{ytgfff-123456}.
To see that this is the case, fix some $(q,r)$-atom $a$ as in \eqref{defi-atom-NEW} 
and observe that, since $a\in L^{p'}({\mathbb{R}}^n)$, the function $m:=T^\top a$ 
is meaningfully defined (cf. \eqref{dkegfs.2b}) and satisfies (thanks to \eqref{dkegfs.2b})
\begin{equation}\label{defi-atom-NEW.Y1}
\|m\|_{L^{p'}({\mathbb{R}}^n)}\leq\|T^\top\|_{{\mathscr{B}}(L^{p'}({\mathbb{R}}^n))}
\|a\|_{L^{p'}({\mathbb{R}}^n)}\leq\|T\|_{{\mathscr{B}}(L^{p}({\mathbb{R}}^n))}|Q|^{(1/p')-(1/q)}.
\end{equation}
In addition, the vanishing moment condition of the atom in concert with the second estimate 
for the kernel $K^\top$ of $T^\top$ in \eqref{ytgfff-123456} and the size estimate for the 
atom yield the decay property
\begin{equation}\label{defi-atom-NEW.Y2}
|m(x)|\leq\frac{c_nC''_K\ell(Q)^\gamma}{|x-x_Q|^{n+\gamma}}|Q|^{1-(1/q)}\,\,\,\text{ 
for each }\,\,x\in{\mathbb{R}}^{n}\setminus(2Q),
\end{equation}
where $c_n\in(0,\infty)$ is a purely dimensional constant and $C''_K$ is as in \eqref{ytgfff-123456}.
Let us also observe that since any $(q,r)$-atom is a multiple of some 
$(1,r)$-atom, we have that $a\in H^1({\mathbb{R}}^n)$. Granted this, from \eqref{ytgfff.9ytrdf.455}
and the fact that $T(1)=0$ we conclude that (see \eqref{ytgfff.9ytrdf})
\begin{equation}\label{defi-atom-NEW.Y3}
m\in L^1({\mathbb{R}}^n)\,\,\,\text{ and }\,\,\,\int_{{\mathbb{R}}^n}m(x)\,dx=0.
\end{equation}
In turn, from the estimates recorded in \eqref{defi-atom-NEW.Y1}-\eqref{defi-atom-NEW.Y2} one may 
readily check that if we now introduce $b:=\big(1/q-1/p'\big)/\theta\in\big(1/q-1/p'\,,\,\infty\big)$
we have 
\begin{align}\label{defi-atom-NEW.Y4.a}
&\|m\|_{L^{p'}({\mathbb{R}}^n)}^{1-\theta}
\big\||\cdot-x_Q|^{nb}m\big\|_{L^{p'}({\mathbb{R}}^n\setminus 2Q)}^{\theta}
\leq c\|T\|_{{\mathscr{B}}(L^{p}({\mathbb{R}}^n))}^{1-\theta}\big(C_K''\big)^\theta
\\[6pt]
&\text{and }\,\,\|m\|_{L^{p'}({\mathbb{R}}^n)}^{1-\theta}
\big\||\cdot-x_Q|^{nb}m\big\|_{L^{p'}(2Q)}^{\theta}
\leq c\|T\|_{{\mathscr{B}}(L^{p}({\mathbb{R}}^n))},
\label{defi-atom-NEW.Y4.b}
\end{align}
where $c\in(0,\infty)$ depends only on $n,\gamma,q,p$, 
and where $C''_K$ is as in \eqref{ytgfff-123456}. In the language of \cite[Definition~7.13, p.\,328]{GCRF85}, 
\eqref{defi-atom-NEW.Y3}-\eqref{defi-atom-NEW.Y4.b} amount to saying that 
$m$ is a $(q,p',b)$-molecule centered at $x_Q$. Having established this, 
we may invoke \cite[Theorem~7.16, p.\,330]{GCRF85} to conclude that 
$m\in H^q({\mathbb{R}}^n)$ and $\|m\|_{H^q({\mathbb{R}}^n)}\leq 
c\|T\|_{{\mathscr{B}}(L^{p}({\mathbb{R}}^n))}^{1-\theta}
\big(C_K''+\|T\|_{{\mathscr{B}}(L^{p}({\mathbb{R}}^n))}\big)^\theta$. 
This proves \eqref{887hytg}.
 
We next claim that 
\begin{equation}\label{887hytg.Z1}
\parbox{11.10cm}{for each given $g\in L^{p'}({\mathbb{R}}^n)\cap H^q({\mathbb{R}}^n)$, 
the function $T^\top g$, originally regarded in $L^{p'}({\mathbb{R}}^n)$ 
by considering the operator $T^\top$ as in \eqref{dkegfs.2b}, actually belongs to 
$H^q({\mathbb{R}}^n)$ and satisfies the estimate 
$\big\|T^\top g\big\|_{H^q({\mathbb{R}}^n)}\leq C\|g\|_{H^q({\mathbb{R}}^n)}$
with $C$ of the same format as in \eqref{887hytg}.}
\end{equation}
With this goal in mind, from \eqref{7h7ggg.555.hhbb}-\eqref{7h7ggg.555.hhbb.2211.CCC} 
and items {\it (a)}-{\it (c)} in the proof of Lemma~\ref{LL-Hhna} we conclude that 
there exist a constant $c=c_{n,p,q,r}\in(0,\infty)$, along with $(q,r)$-atoms 
$\{a_j\}_{j\in{\mathbb{N}}}$ and numbers $\{\lambda_j\}_{j\in{\mathbb{N}}}$, such that 
\begin{equation}\label{7h7ggg.5tEE-yrrr}
\Big(\sum_{j\in{\mathbb{N}}}|\lambda_j|^q\Big)^{1/q}\leq c\|g\|_{H^q({\mathbb{R}}^n)},
\end{equation}
and if 
\begin{equation}\label{u76gggfWW-iii}
g_N:=\sum_{j=1}^N\lambda_j a_j\,\,\text{ for each }\,\,N\in{\mathbb{N}}
\end{equation}
then 
\begin{equation}\label{u76gggfWW-iii.b}
\lim_{N\to\infty}g_N=g\,\,\text{ both in }\,\,H^q({\mathbb{R}}^n)
\,\,\text{ and in }\,\,L^{p'}({\mathbb{R}}^n).
\end{equation}
Note that whenever $N,M\in{\mathbb{N}}$ are such that $N<M$, we may rely on 
\eqref{887hytg} to conclude that 
\begin{equation}\label{887hytg.iii}
T^\top g_N,\,T^\top g_M\in H^q({\mathbb{R}}^n)\,\,\text{ and }\,\,
\big\|T^\top g_N-T^\top g_M\big\|_{H^q({\mathbb{R}}^n)}\leq C
\Big(\sum_{j=N+1}^{M}|\lambda_j|^q\Big)^{1/q}.
\end{equation}
Given that $\{\lambda_j\}_{j\in{\mathbb{N}}}\in\ell^q$, this proves that 
the sequence $\big\{T^\top g_N\big\}_{N\in{\mathbb{N}}}$ is 
Cauchy in $H^q({\mathbb{R}}^n)$. Since the latter is a quasi-Banach space, 
it follows that there exists some $h\in H^q({\mathbb{R}}^n)$ such that 
$\lim_{N\to\infty}T^\top g_N=h$ in $H^q({\mathbb{R}}^n)$. On the other hand, 
from \eqref{u76gggfWW-iii.b} and \eqref{dkegfs.2b} we conclude that 
$\lim_{N\to\infty}T^\top g_N=T^\top g$ in $L^{p'}({\mathbb{R}}^n)$.
Hence, necessarily, $T^\top g=h$ as distributions in ${\mathbb{R}}^n$.
This goes to show that $T^\top g\in H^q({\mathbb{R}}^n)$, and we may also estimate 
\begin{align}\label{h5ffDEDX}
\big\|T^\top g\big\|_{H^q({\mathbb{R}}^n)}
&=\|h\|_{H^q({\mathbb{R}}^n)}
=\lim_{N\to\infty}\big\|T^\top g_N\big\|_{H^q({\mathbb{R}}^n)}
\nonumber\\[4pt]
&\leq C\limsup_{N\to\infty}\Big(\sum_{j=1}^N|\lambda_j|^q\Big)^{1/q}
=C\Big(\sum_{j=1}^\infty|\lambda_j|^q\Big)^{1/q}
\leq C\|g\|_{H^q({\mathbb{R}}^n)},
\end{align}
where the constant $C$ has the same format as in \eqref{887hytg}. Above, the second equality
uses the fact that $\|\cdot\|_{H^q({\mathbb{R}}^n)}$ is a $q$-norm which defines the topology 
on $H^q({\mathbb{R}}^n)$, the subsequent inequality is a 
consequence of \eqref{u76gggfWW-iii}, \eqref{887hytg}, and the sub-additivity of 
$\|\cdot\|^q_{H^q({\mathbb{R}}^n)}$, while the last inequality comes from 
\eqref{7h7ggg.5tEE-yrrr}. This finishes the proof of \eqref{887hytg.Z1}.

Moving on, consider now an arbitrary $g\in H^q({\mathbb{R}}^n)$. 
Since $L^{p'}({\mathbb{R}}^n)\cap H^q({\mathbb{R}}^n)$ is dense in 
$H^q({\mathbb{R}}^n)$, there exists a sequence 
$\{g_j\}_{j\in{\mathbb{N}}}\subset L^{p'}({\mathbb{R}}^n)\cap H^q({\mathbb{R}}^n)$
such that $\lim_{j\to\infty}g_j=g$ in $H^q({\mathbb{R}}^n)$. 
From \eqref{887hytg.Z1} it follows that $\{T^\top g_j\}_{j\in{\mathbb{N}}}$ 
is Cauchy in $H^q({\mathbb{R}}^n)$. 
Define $T^\top g$ to be the limit of $\{T^\top g_j\}_{j\in{\mathbb{N}}}$ 
in $H^q({\mathbb{R}}^n)$. By interlacing sequences, it may shown that 
the limit defining $T^\top g$ does not depend on the actual choice of the 
sequence $\{g_j\}_{j\in{\mathbb{N}}}$. In turn, this implies that
$T^\top:H^q({\mathbb{R}}^n)\to H^q({\mathbb{R}}^n)$ is well-defined, linear, 
and compatible with the action of $T^\top$ on $L^{p'}({\mathbb{R}}^n)$.  
To see that the operator just defined is also bounded, if $g$ and $\{g_j\}_{j\in{\mathbb{N}}}$
are as before write 
\begin{align}\label{h5ffDEDX-ytt}
\big\|T^\top g\big\|_{H^q({\mathbb{R}}^n)}
=\lim_{j\to\infty}\big\|T^\top g_j\big\|_{H^q({\mathbb{R}}^n)}
\leq C\limsup_{j\to\infty}\|g_j\|_{H^q({\mathbb{R}}^n)}
=C\|g\|_{H^q({\mathbb{R}}^n)},
\end{align}
where the constant $C$ has the same format as in \eqref{887hytg}.
In \eqref{h5ffDEDX-ytt}, we have used the definition of $T^\top$ on $H^q({\mathbb{R}}^n)$, 
the fact that $\lim_{j\to\infty}g_j=g$ in $H^q({\mathbb{R}}^n)$,
the estimate in \eqref{887hytg.Z1}, and the fact that
$\|\cdot\|_{H^q({\mathbb{R}}^n)}$ is a $q$-norm which defines the topology 
on $H^q({\mathbb{R}}^n)$
(in the first and last equalities in \eqref{h5ffDEDX-ytt}). 

In summary, for each $q\in\big(\tfrac{n}{n+\gamma}\,,\,1\big]$, 
we have succeeded in producing a linear and bounded operator 
$T^\top:H^q({\mathbb{R}}^n)\to H^q({\mathbb{R}}^n)$ which acts in a compatible fashion 
with $T^\top$ in \eqref{dkegfs.2b} and which satisfies the estimate in \eqref{887hytg-ff}.
There remains to show that these newly produced operators 
are also compatible with one another as $q$ varies through $\big(\tfrac{n}{n+\gamma}\,,\,1\big]$.
To this end, fix $q_1,q_2\in\big(\tfrac{n}{n+\gamma}\,,\,1\big]$ and consider 
some arbitrary $g\in H^{q_1}({\mathbb{R}}^n)\cap H^{q_2}({\mathbb{R}}^n)$. 
Also, fix $p'\in(1,2]$, choose $r\in(1,\infty)$ with $r\geq p'$, and 
set $s:=\min\{q_1,q_2\}$. Then Lemma~\ref{LL-Hhna} ensures that there exists some 
sequence $\{g_N\}_{N\in{\mathbb{N}}}\subset H^{s,r}_{\rm fin}({\mathbb{R}}^n)
\subset L^{p'}({\mathbb{R}}^n)$ which converges to $g$ both in 
$H^{q_1}({\mathbb{R}}^n)$ and in $H^{q_2}({\mathbb{R}}^n)$. Then, 
with $T^\top$ considered in the sense of \eqref{dkegfs.2b}, the sequence 
$\big\{T^\top g_N\big\}_{N\in{\mathbb{N}}}$ converges both in 
$H^{q_1}({\mathbb{R}}^n)$ and in $H^{q_2}({\mathbb{R}}^n)$. In light of the manner 
in which the extension to Hardy spaces has been defined earlier, this shows
that the operator $T^\top:H^{q_1}({\mathbb{R}}^n)\to H^{q_1}({\mathbb{R}}^n)$ 
acting on $g$, viewed in $H^{q_1}({\mathbb{R}}^n)$, agrees with the operator 
$T^\top:H^{q_2}({\mathbb{R}}^n)\to H^{q_2}({\mathbb{R}}^n)$ 
acting on $g$ now viewed as a distribution in $H^{q_2}({\mathbb{R}}^n)$. 
This concludes the justification of the claims made in item {\it (iii)}.

Going further, the well-definiteness, linearity, and boundedness 
of $T^\top$ in \eqref{dkegfs.3}, together with Fefferman's basic duality result 
$\big(H^1({\mathbb{R}}^n)\big)^\ast=\widetilde{\rm BMO}({\mathbb{R}}^n)$, ensure that 
$\widetilde{T}$ defined as in \eqref{dkegfs.5} is a well-defined, 
linear and bounded operator in the context of \eqref{dkegfs.4}. 
To prove the compatibility condition described in \eqref{dkegfs.6}, fix some 
$p\in[2,\infty)$ along with an arbitrary function
$f\in{\rm BMO}({\mathbb{R}}^n)\cap L^p({\mathbb{R}}^n)$. 
Then, if $p'\in(1,2]$ is such that $1/p+1/p'=1$, 
for each function $g\in H^1({\mathbb{R}}^n)\cap L^{p'}({\mathbb{R}}^n)$ we may compute
\begin{equation}\label{dkegfs.10}
\big\langle\widetilde{T}[f],g\big\rangle=\big\langle[f],T^\top g\big\rangle
=\int_{{\mathbb{R}}^n}f(T^\top g)\,d{\mathscr{L}}^n
=\int_{{\mathbb{R}}^n}(Tf)g\,d{\mathscr{L}}^n.
\end{equation}
Above, the first equality is simply \eqref{dkegfs.5}, the second equality 
is implied by the fact that $T^\top g\in H^1({\mathbb{R}}^n)\cap L^{p'}({\mathbb{R}}^n)$
(cf. \eqref{dkegfs.2b}, \eqref{dkegfs.3}) and Lemma~\ref{U-gGg}, while the last 
equality is seen from the fact that the adjoint of \eqref{dkegfs.2} is \eqref{dkegfs.2b}.
Let us now select a representative $h\in{\rm BMO}({\mathbb{R}}^n)$ of the 
equivalence class $\widetilde{T}[f]\in\widetilde{\rm BMO}({\mathbb{R}}^n)$, 
and specialize \eqref{dkegfs.10} to the case when $g$ is an $(1,r)$-atom for some $r\in(1,\infty)$.
On account of \eqref{dk-TTFF}, this yields
\begin{equation}\label{dkegfs.11}
\int_{{\mathbb{R}}^n}h\,a\,d{\mathscr{L}}^n
=\int_{{\mathbb{R}}^n}(Tf)a\,d{\mathscr{L}}^n\,\,\text{ for each $(1,r)$-atom $a$}.
\end{equation}
It is not difficult to see that, generally speaking,  
\begin{equation}\label{7h7ggg.222}
\parbox{11.50cm}{if $q\in\big(\tfrac{n}{n+1}\,,\,1\big]$ and $r,r'\in[1,\infty]$ 
are such that $1/r+1/r'=1$ and $q<r$, then a function  
$\phi\in L^{r'}_{\rm loc}({\mathbb{R}}^n)$ satisfying 
$\int_{{\mathbb{R}}^n}\phi\,a\,d{\mathscr{L}}^n=0$ for each $(q,r)$-atom $a$
is necessarily constant in ${\mathbb{R}}^n$.}
\end{equation}
This may be seen by considering scalar multiples of $(q,r)$-atoms of the form 
\begin{equation}\label{7h7ggg.222.2}
a={\mathbf{1}}_{B(x,R)}/{\mathscr{L}}^n(B(x,R))-{\mathbf{1}}_{B(0,1)}/{\mathscr{L}}^n(B(0,1))
\end{equation}
with $x\in{\mathbb{R}}^n$ and $R>0$ arbitrary, then letting $R\to 0^{+}$ and 
invoking Lebesgue's Differentiation Theorem. In concert, \eqref{dkegfs.11}
and \eqref{7h7ggg.222} then prove that $h$ and $Tf$ differ by a constant. 
Hence, $\widetilde{T}[f]=[h]=[Tf]$, finishing the proof of \eqref{dkegfs.6}.
Finally, the estimate recorded in \eqref{dkegfs.5-AAA} is obtain by noting that 
\eqref{dkegfs.5} and the quantitative aspect of the $\widetilde{\rm BMO}$-$H^1$ duality yield 
\begin{equation}\label{dkegfs.5-ABB}
\big\|\widetilde{T}\big\|_{{\mathscr{B}}(\,\widetilde{\rm BMO}({\mathbb{R}}^n))}\leq c_n
\big\|T^\top\big\|_{{\mathscr{B}}(H^1({\mathbb{R}}^n))},
\end{equation}
and then combining this with \eqref{887hytg-ff} (used here with $q=1$).

Moving on, from the well-definiteness, linearity, and boundedness of $T^\top$ 
in \eqref{dkegfs.3}, together with the duality result recorded in \eqref{WSWS-8htggg.1}  
we conclude that $\widehat{T}$ defined in \eqref{dkegfs.5NNN} is a well-defined, 
linear and bounded operator in the context of \eqref{dkegfs.4NNN}. Next, 
the compatibility condition \eqref{dkegfs.6-NNN.1} is proved much like \eqref{dkegfs.6}, 
this time making use of Lemma~\ref{U-gGg-HHH} instead of Lemma~\ref{U-gGg}.

Consider next the compatibility condition in \eqref{dkegfs.6-NNN.2}.
With this in mind, select an arbitrary function
$f\in\dot{\mathscr{C}}^\eta({\mathbb{R}}^n)\cap{\rm BMO}({\mathbb{R}}^n)$. 
Then for each $g\in H^1({\mathbb{R}}^n)\cap H^q({\mathbb{R}}^n)$ we have
\begin{equation}\label{dkegfs.10-ZZZ}
\big(\widehat{T}[f],g\big)=\big([f],T^\top g\big)
=\big\langle[f],T^\top g\big\rangle
=\big\langle\widetilde{T}[f],g\big\rangle.
\end{equation}
Here, the first equality is based on \eqref{dkegfs.5NNN}, the second equality 
takes into account the fact that $T^\top g\in H^1({\mathbb{R}}^n)\cap H^q({\mathbb{R}}^n)$
(cf. \eqref{dkegfs.3}) and uses Lemma~\ref{U-gGg-HHH.2}, whereas the last 
equality is implied by \eqref{dkegfs.3} and \eqref{dkegfs.5}. Pick a 
representative $\widetilde{h}$ of $\widetilde{T}[f]\in\widetilde{\rm BMO}({\mathbb{R}}^n)$
along with a representative $\widehat{h}$ of 
$\widehat{T}[f]\in\dot{\mathscr{C}}^\eta({\mathbb{R}}^n)\big/_{\sim}$. If we now 
fix $r\in(1,\infty)$ and specialize the equality of the most extreme sides of 
\eqref{dkegfs.10-ZZZ} to the case when $g$ is an arbitrary $(q,r)$-atom we arrive at the 
conclusion that 
\begin{equation}\label{dkegfs.10-ZZZ.2}
\int_{{\mathbb{R}}^n}\widehat{h}\,a\,d{\mathscr{L}}^n
=\int_{{\mathbb{R}}^n}\widetilde{h}\,a\,d{\mathscr{L}}^n
\,\,\,\text{ for each $(q,r)$-atom $a$.}
\end{equation}
On account of this and \eqref{7h7ggg.222} we may then conclude that the functions
$\widehat{h}$ and $\widetilde{h}$ differ by a constant, which ultimately goes
to show that \eqref{dkegfs.6-NNN.2} holds. 

At this stage, there remains to prove that the operators in \eqref{dkegfs.4NNN} 
considered for two arbitrary choices of the smoothness parameter are 
compatible with one another. To this end, pick arbitrary 
$f\in\dot{\mathscr{C}}^{\eta_1}({\mathbb{R}}^n)\cap
\dot{\mathscr{C}}^{\eta_2}({\mathbb{R}}^n)$ and 
$g\in H^{q_1}({\mathbb{R}}^n)\cap H^{q_2}({\mathbb{R}}^n)$, where 
$q_j\in\big(\tfrac{n}{n+1}\,,\,1\big)$ and $\eta_j\in(0,1)$ are related via
$\eta_j=n\big(\tfrac{1}{q_j}-1\big)$ for $j=1,2$. For each $j=1,2$, we agree to denote
the $\dot{\mathscr{C}}^{\eta_j}\big/_{\sim}$-$H^{q_j}$ duality bracket by $(\cdot,\cdot)_j$.
Then
\begin{equation}\label{4332-HHH.2-ttt.TTT}
\big(\widehat{T}[f],g\big)_1=\big([f],T^\top g\big)_1
=\big([f],T^\top g\big)_2=\big(\widehat{T}[f],g\big)_2,
\end{equation}
where the first and last equalities are based on \eqref{dkegfs.5NNN}
while the middle equality is a consequence of Lemma~\ref{U-gGg-HHH.2-grf}.
Specializing the coincidence of the most extreme terms in \eqref{4332-HHH.2-ttt.TTT} 
to the case when $g$ is a $(q,r)$-atom for some $r\in(1,\infty)$ and $q:=\min\{q_1,q_2\}$ 
then yields, on account of \eqref{dk-TTFF-HHH}, 
\begin{equation}\label{dkbdj-D}
\int_{{\mathbb{R}}^n}h_1\,a\,d{\mathscr{L}}^n
=\int_{{\mathbb{R}}^n}h_2a\,d{\mathscr{L}}^n\,\,\text{ for each $(q,r)$-atom $a$},
\end{equation}
where $h_j\in\dot{\mathscr{C}}^{\eta_j}({\mathbb{R}}^n)$ is a representative 
of $\widehat{T}[f]\in\dot{\mathscr{C}}^{\eta_j}({\mathbb{R}}^n)\big/_{\sim}$, for $j=1,2$.
At this point we invoke \eqref{7h7ggg.222} to conclude that 
$h_1-h_2$ is constant in ${\mathbb{R}}^n$ from which the very 
last claim in Proposition~\ref{jhsdgf} follows. 
The proof of Proposition~\ref{jhsdgf} is therefore complete.
\end{proof}

Having dealt with Proposition~\ref{jhsdgf} we are now ready to present 
the proof of Theorem~\ref{i87hbBV}.

\vskip 0.08in
\begin{proof}[Proof of Theorem~\ref{i87hbBV}]
Fix $n\in{\mathbb{N}}$ along with $\gamma\in(0,1]$ and suppose $T\in{\rm SCZ}(n,\gamma)$. 
Pick $\eta\in(0,\gamma)$ arbitrary. By Proposition~\ref{jhsdgf}, the operator $T$ extends to 
a bounded linear mapping $\widetilde{T}$ from $\widetilde{\mathrm{BMO}}({\mathbb{R}}^n)$ 
into itself and to a bounded linear mapping $\widehat{T}$ from 
$\dot{\mathscr{C}}^\eta({\mathbb{R}}^n)\big/_{\sim}$ into itself. In addition, 
these extensions are compatible in the sense of \eqref{dkegfs.6-NNN.2}. 
From these we deduce that $\widetilde{T}$ maps the linear subspace 
$X:=\Big(\dot{\mathscr{C}}^\eta({\mathbb{R}}^n)\big/_{\sim}\Big)\cap
\widetilde{\mathrm{BMO}}({\mathbb{R}}^n)$ of $\widetilde{\mathrm{BMO}}({\mathbb{R}}^n)$
into $X$. Since $\widetilde{T}$ is continuous on $\widetilde{\mathrm{BMO}}({\mathbb{R}}^n)$,
it follows that $\widetilde{T}$ maps the closure of $X$ in 
$\widetilde{\mathrm{BMO}}({\mathbb{R}}^n)$ linearly and boundedly into itself. 
Corollary~\ref{Cbna-j77h-TTT} tells us that the said closure is simply 
$\widetilde{\mathrm{VMO}}({\mathbb{R}}^n)$, so we ultimately conclude that 
$\widetilde{T}$ maps $\widetilde{\mathrm{VMO}}({\mathbb{R}}^n)$ 
linearly and boundedly into itself. Keeping in mind that the action of 
$\widetilde{T}$ in this setting is compatible with that of the original 
operator $T$ (cf. \eqref{dkegfs.6}), the desired conclusion follows. 
\end{proof}

Theorem~\ref{i87hbBV} is the main ingredient in the proof of Theorem~\ref{i87hbBV-MY}, discussed next.

\vskip 0.08in
\begin{proof}[Proof of Theorem~\ref{i87hbBV-MY}]
According to \cite[\S9]{Meyer} (cf. also \cite[Theorem~5, p.\,231]{Meyer85})
\begin{equation}\label{ijBBa-hyTVv}
{\mathscr{A}}^0_{{}_{\rm CZ}}:=\bigcup_{0<\gamma\leq 1}\big\{T\in{\rm CZ}(n,\gamma):\,T(1)=T^\top(1)=0\big\}
\end{equation}
is the largest sub-algebra of ${\mathscr{B}}\big(L^2({\mathbb{R}}^n)\big)$ consisting of Calder\'on-Zygmund 
operators in ${\mathbb{R}}^n$. Since ${\mathscr{A}}^0_{{}_{\rm CZ}}$ is invariant under transposition, we 
conclude from Proposition~\ref{jhsdgf} and Theorem~\ref{i87hbBV} that ${\mathscr{A}}^0_{{}_{\widetilde{\rm CZ}}}$ 
is indeed a sub-algebra of ${\mathscr{B}}\big(\,\widetilde{\mathrm{VMO}}({\mathbb{R}}^n)\big)$.
\end{proof}

Next, we present the proof of Theorem~\ref{i87hbBV-ALG} which, once again, makes essential use of Theorem~\ref{i87hbBV}.

\vskip 0.08in
\begin{proof}[Proof of Theorem~\ref{i87hbBV-ALG}]
Proposition~\ref{jhsdgf} ensures that each principal-value convolution type operator $T_\Theta$ associated 
as in \eqref{ytgfff.9ytrdf.ygfg} with a function $\Theta$ as in \eqref{ytgfff.9ytrdf.ygfg.222222}
induces a well-defined linear and bounded mapping $\widetilde{T}_{\Theta}$ on 
$\widetilde{\mathrm{BMO}}({\mathbb{R}}^n)$. From Theorem~\ref{i87hbBV} we also know that
$\widetilde{T}_{\Theta}\big|_{{}_{\rm VMO}}$, the restriction of $\widetilde{T}_{\Theta}$
to $\widetilde{\mathrm{VMO}}({\mathbb{R}}^n)$, is a well-defined linear and bounded operator from the space 
$\widetilde{\mathrm{VMO}}({\mathbb{R}}^n)$ into itself. Hence, ${\mathscr{A}}_{{}_{\widetilde{\rm SIO}}}$ 
defined in \eqref{gabbIHka} is a subset of ${\mathscr{B}}\big(\,\widetilde{\mathrm{VMO}}({\mathbb{R}}^n)\big)$. 
Proving that ${\mathscr{A}}_{{}_{\widetilde{\rm SIO}}}$ is actually a commutative sub-algebra of 
${\mathscr{B}}\big(\,\widetilde{\mathrm{VMO}}({\mathbb{R}}^n)\big)$ requires some preparations. 

Regarding the relationship between a kernel $\Theta$ as in \eqref{ytgfff.9ytrdf.ygfg.222222}
and its associated symbol $m_\Theta$ as in \eqref{PhBB-1}, two features are particularly 
significant for us here. First, from \eqref{PhBB-1gd} we know that 
\begin{equation}\label{ijBBa-hyf}
\parbox{9.10cm}{if $\Theta$ is as in \eqref{ytgfff.9ytrdf.ygfg.222222}, then 
$m_{\Theta}$ given by \eqref{PhBB-1} is positive homogeneous of degree zero
and of class ${\mathscr{C}}^\infty$ in ${\mathbb{R}}^n\setminus\{0\}$.}
\end{equation}
Second, from \cite[Theorem~6, p.\,75]{St70} 
(or \cite[Proposition~2.4.7 on p.\,128, and Proposition~4.2.3 on p.\,267]{Grafakos}) 
it follows that 
\begin{equation}\label{ijBBa-hyf.baba}
\parbox{9.50cm}{given any function $m\in{\mathscr{C}}^\infty({\mathbb{R}}^n\setminus\{0\})$ 
which is positive homogeneous of degree zero, there exist some unique function $\Theta$ as in 
\eqref{ytgfff.9ytrdf.ygfg.222222} and some unique number $c\in{\mathbb{C}}$ 
such that $m=c+m_{\Theta}$ (actually $c=\int_{S^{n-1}}m(\omega)\,d\omega\in{\mathbb{C}}$).}
\end{equation}

Consider next two functions $\Theta_1,\Theta_2$ as in \eqref{ytgfff.9ytrdf.ygfg.222222} and associate with 
them $m_{\Theta_1}$, $m_{\Theta_2}$ as in \eqref{PhBB-1}. Since then their product $m_{\Theta_1}m_{\Theta_2}$ 
belongs to ${\mathscr{C}}^{\infty}({\mathbb{R}}^n\setminus\{0\})$ (thanks to \eqref{ijBBa-hyf}) 
and is positive homogeneous of degree zero (given that both $m_{\Theta_1}$ and $m_{\Theta_2}$ are),  
we may invoke \eqref{ijBBa-hyf.baba} to conclude that  
\begin{equation}\label{ijBBa}
\begin{array}{c}
\text{there exists a function $\Theta$ as in \eqref{ytgfff.9ytrdf.ygfg.222222} with the property that}
\\[6pt]
m_{\Theta_1}m_{\Theta_2}=c+m_{\Theta}\,\,\text{ in }\,\,{\mathbb{R}}^n\setminus\{0\},
\,\,\text{ where }\,\,c:=\int_{S^{n-1}}m_{\Theta_1}(\omega)m_{\Theta_2}(\omega)\,d\omega.
\end{array}
\end{equation}
If ${\mathcal{F}}^{-1}_{\xi\to x}$ denotes the inverse Fourier transform (taking functions in 
the variable $\xi$ into functions in the variable $x$), then for each $f\in L^2({\mathbb{R}}^n)$ 
we may write 
\begin{align}\label{TFvav-1234.a}
(T_{\Theta_1}\circ T_{\Theta_2})f(x)
&={\mathcal{F}}^{-1}_{\xi\to x}\big[m_{\Theta_1}(\xi)m_{\Theta_2}(\xi)\widehat{f}(\xi)\big]
\nonumber\\[6pt]
&={\mathcal{F}}^{-1}_{\xi\to x}\big[\big(c+m_{\Theta}(\xi)\big)\widehat{f}(\xi)\big]
=cf(x)+(T_\Theta f)(x),\qquad x\in{\mathbb{R}}^n.
\end{align}
Hence, $T_{\Theta_1}\circ T_{\Theta_2}=cI+T_{\Theta}$ as operators from the space $L^2({\mathbb{R}}^n)$ into itself.
Also, 
\begin{align}\label{TFvav-1234.a-b2}
(T_{\Theta_1}\circ T_{\Theta_2})f(x)
&={\mathcal{F}}^{-1}_{\xi\to x}\big[m_{\Theta_1}(\xi)m_{\Theta_2}(\xi)\widehat{f}(\xi)\big]
\nonumber\\[6pt]
&={\mathcal{F}}^{-1}_{\xi\to x}\big[m_{\Theta_2}(\xi)m_{\Theta_1}(\xi)\widehat{f}(\xi)\big]
=(T_{\Theta_2}\circ T_{\Theta_1})f(x),\qquad x\in{\mathbb{R}}^n,
\end{align}
thus $T_{\Theta_1}\circ T_{\Theta_2}=T_{\Theta_2}\circ T_{\Theta_1}$ on $L^2({\mathbb{R}}^n)$.
In turn, given that $H^1({\mathbb{R}}^n)\cap L^2({\mathbb{R}}^n)$ is dense in $L^2({\mathbb{R}}^n)$ 
(see \eqref{dkegfs.7}) and since $T_{\Theta_1},T_{\Theta_2},T_{\Theta}$ map $H^1({\mathbb{R}}^n)$ into itself 
boundedly and in a compatible fashion with their action on $L^2({\mathbb{R}}^n)$ (cf. Proposition~\ref{jhsdgf}), 
we may conclude that
\begin{equation}\label{jhsfrwfr-H1-ttt-iub}
\begin{array}{c}
T_{\Theta_1}\circ T_{\Theta_2}=T_{\Theta_2}\circ T_{\Theta_1}\,\,\text{ and }\,\,
T_{\Theta_1}\circ T_{\Theta_2}=cI+T_{\Theta}\,\,\text{ on }\,\,H^1({\mathbb{R}}^n),
\\[4pt]
\text{whenever $c$, $\Theta$ are related to $\Theta_1$, $\Theta_2$ as in \eqref{ijBBa}.}
\end{array}
\end{equation}

Going further, fix $\Theta_1$, $\Theta_2$, $\Theta$ as in \eqref{ytgfff.9ytrdf.ygfg.222222}.
With $\langle\cdot,\cdot\rangle$ denoting the $\widetilde{\rm BMO}$-$H^1$ duality bracket, 
from Proposition~\ref{jhsdgf} and \eqref{ytgfff.9ytrdf.ygfg} it follows that  
$T_{\Theta_1},T_{\Theta_2},T_{\Theta}$ induce linear and bounded operators 
$\widetilde{T}_{\Theta_1},\widetilde{T}_{\Theta_2},\widetilde{T}_{\Theta}$ from 
$\widetilde{\mathrm{BMO}}({\mathbb{R}}^n)$ into itself according to 
\begin{equation}\label{jhsfrw-ttt-iub}
\begin{array}{c}
\big\langle\widetilde{T}_{\Theta_j}[f],g\big\rangle=\big\langle[f],T_{\widetilde{\Theta}_j}g\big\rangle,
\qquad\forall\,f\in{\mathrm{BMO}}({\mathbb{R}}^n),\,\,\forall\,g\in H^1({\mathbb{R}}^n),\,\,\forall\,j\in\{1,2\},
\\[8pt]
\text{and }\,\,\,\big\langle\widetilde{T}_{\Theta}[f],g\big\rangle=\big\langle[f],T_{\widetilde{\Theta}}g\big\rangle,
\qquad\forall\,f\in{\mathrm{BMO}}({\mathbb{R}}^n),\,\,\forall\,g\in H^1({\mathbb{R}}^n),
\end{array}
\end{equation}
where $\widetilde{\Theta}_j(x):=\Theta_j(-x)$ for $j\in\{1,2\}$, and 
$\widetilde{\Theta}(x):=\Theta(-x)$, for each $x\in{\mathbb{R}}^n\setminus\{0\}$.
Retaining the symbol $I$ for the identity operator on $\widetilde{\mathrm{BMO}}({\mathbb{R}}^n)$, 
we claim that these extensions satisfy
\begin{equation}\label{jhsfrwfr-BMO-ttt-iub}
\begin{array}{c}
\widetilde{T}_{\Theta_1}\circ\widetilde{T}_{\Theta_2}=\widetilde{T}_{\Theta_2}\circ\widetilde{T}_{\Theta_1}
\,\,\text{ and }\,\,\widetilde{T}_{\Theta_1}\circ\widetilde{T}_{\Theta_2}=cI+\widetilde{T}_{\Theta}
\,\,\text{ on }\,\,\widetilde{\mathrm{BMO}}({\mathbb{R}}^n)
\\[4pt]
\text{provided $m_{\widetilde{\Theta}_1}m_{\widetilde{\Theta}_2}=c+m_{\widetilde{\Theta}}$ in 
${\mathbb{R}}^n\setminus\{0\}$ for some $c\in{\mathbb{C}}$.}
\end{array}
\end{equation}
Indeed, for each $f\in{\mathrm{BMO}}({\mathbb{R}}^n)$ and $g\in H^1({\mathbb{R}}^n)$ 
based on \eqref{jhsfrw-ttt-iub} and \eqref{jhsfrwfr-H1-ttt-iub} 
(applied to $\widetilde{\Theta}_1$, $\widetilde{\Theta}_2$ in place of $\Theta_1$, $\Theta_2$)
we may write  
\begin{align}\label{jhsfrwfr-ttt-iub.qa}
\big\langle\widetilde{T}_{\Theta_1}\widetilde{T}_{\Theta_2}[f]\,,\,g\big\rangle
=\big\langle[f]\,,\,T_{\widetilde{\Theta}_2}T_{\widetilde{\Theta}_1}g\big\rangle
=\big\langle[f]\,,\,T_{\widetilde{\Theta}_1}T_{\widetilde{\Theta}_2}g\big\rangle
=\big\langle\widetilde{T}_{\Theta_2}\widetilde{T}_{\Theta_1}[f]\,,\,g\big\rangle
\end{align}
which, in view of the fact that $\widetilde{\mathrm{BMO}}({\mathbb{R}}^n)$ is the dual 
of $H^1({\mathbb{R}}^n)$, establishes the first formula in \eqref{jhsfrwfr-BMO-ttt-iub}.
As regards the second formula in \eqref{jhsfrwfr-BMO-ttt-iub}, for each 
$f\in{\mathrm{BMO}}({\mathbb{R}}^n)$ and $g\in H^1({\mathbb{R}}^n)$ 
using \eqref{jhsfrw-ttt-iub} and \eqref{jhsfrwfr-H1-ttt-iub} 
(applied to $\widetilde{\Theta}_1$, $\widetilde{\Theta}_2$ in place of $\Theta_1$, $\Theta_2$)
we may compute 
\begin{align}\label{jhsfrwfr-ttt-iub}
\big\langle\widetilde{T}_{\Theta_1}\widetilde{T}_{\Theta_2}[f]\,,\,g\big\rangle
&=\big\langle[f]\,,\,T_{\widetilde{\Theta}_2}T_{\widetilde{\Theta}_1}g\big\rangle
=\big\langle[f]\,,\,T_{\widetilde{\Theta}_1}T_{\widetilde{\Theta}_2}g\big\rangle
\nonumber\\[4pt]
&=\big\langle[f]\,,\,(cI+T_{\widetilde{\Theta}})g\big\rangle
=\big\langle(cI+\widetilde{T}_{\Theta})[f]\,,\,g\big\rangle.
\end{align}
The third equality above is provided by the second formula in \eqref{jhsfrwfr-H1-ttt-iub}, 
written for $\widetilde{\Theta}_1$, $\widetilde{\Theta}_2$, $\widetilde{\Theta}$ in place of 
$\Theta_1$, $\Theta_2$, $\Theta$ (whose validity is ensured by the assumptions we  
make on $c\in{\mathbb{C}}$ and $\Theta$ in \eqref{jhsfrwfr-BMO-ttt-iub}).
By once again relying on the fact that $\widetilde{\mathrm{BMO}}({\mathbb{R}}^n)$ 
is the dual of $H^1({\mathbb{R}}^n)$, the second formula in \eqref{jhsfrwfr-BMO-ttt-iub} follows
from \eqref{jhsfrwfr-ttt-iub}. Having established \eqref{jhsfrwfr-BMO-ttt-iub}, we may now conclude
(with the help of Theorem~\ref{i87hbBV}) that ${\mathscr{A}}_{{}_{\widetilde{\rm SIO}}}$ defined as 
in \eqref{gabbIHka} is a commutative unital sub-algebra of the algebra of all linear and bounded operators 
from the space $\widetilde{\mathrm{VMO}}({\mathbb{R}}^n)$ into itself. Also, the fact that if $c\in{\mathbb{C}}$ 
and the functions $\Theta_1,\dots,\Theta_N,{\Theta'}_{\!\!1},\dots,{\Theta'}_{\!\!N},\Theta$ are as in 
\eqref{ytgfff.9ytrdf.ygfg.222222} and satisfy \eqref{u7GBB.Za.5.XXX} then \eqref{u7GBB.Za.8.YYY}
holds is established in a similar fashion to the second formula in \eqref{jhsfrwfr-BMO-ttt-iub}. 

Consider next the claim made in item {\it (b)}. For starters, the right-to-left inclusion in 
\eqref{gab-ii-GGVa} is clear from definitions. As regards the opposite inclusion in \eqref{gab-ii-GGVa}, 
it suffices to show that ${\mathscr{A}}_{{}_{\widetilde{\rm SIO}}}\subseteq
\overline{\text{span}}\,\big\{\widetilde{R}_j\big|_{{}_{\rm VMO}}\big\}_{1\leq j\leq n}$.
Since \eqref{u7GBB.Za.5.XXX} holds with $c=-1$, $\Theta=0$, and ${\Theta'}_{\!\!j}=\Theta_j=K_j$, 
defined in \eqref{R-87ygbg}, for each $j\in\{1,\dots,n\}$, we conclude from \eqref{u7GBB.Za.5.XXX} that 
\begin{equation}\label{u7GBB.Za.5.Xa}
\sum_{j=1}^n\big(\widetilde{R}_j\big|_{{}_{\rm VMO}}\big)^2=-I
\,\,\text{ in }\,\,{\mathscr{B}}\big(\,\widetilde{\rm VMO}({\mathbb{R}}^n)\big).
\end{equation}
In particular, this proves that the identity operator $I$ belongs to the sub-algebra spanned by 
$\big\{\widetilde{R}_j\big|_{{}_{\rm VMO}}\big\}_{1\leq j\leq n}$ 
in ${\mathscr{B}}\big(\,\widetilde{\rm VMO}({\mathbb{R}}^n)\big)$. Keeping this in mind, 
formula \eqref{gab-ii-GGVa} is established as soon as we show that 
\begin{equation}\label{gab-ii-GGVa-FFF}
\widetilde{T}_\Theta\in\overline{\text{span}}\,\Big\{\widetilde{R}_j\big|_{{}_{\rm VMO}}\Big\}_{1\leq j\leq n}
\,\,\text{ for each $\Theta$ as in \eqref{ytgfff.9ytrdf.ygfg.222222}}.
\end{equation}
To this end, fix an arbitrary $\Theta$ as in \eqref{ytgfff.9ytrdf.ygfg.222222}. To perform a spherical 
decomposition of $\Theta\big|_{S^{n-1}}$, we bring in some notation and recall some basic results. 
Specifically, define the integers
\begin{equation}\label{D-HarmK} 
H_0:=1,\quad
H_1:=n,\,\,\mbox{ and }\,\,
H_{\ell}:=\Big(\!\!
\begin{array}{c}
n-1+\ell
\\[-2pt]
\ell
\end{array}
\!\!\Big)
-
\Big(\!\!
\begin{array}{c}
n+\ell-3
\\[-2pt]
\ell-2
\end{array}
\!\!\Big)
\,\,\mbox{ if }\,\,\ell\geq 2,
\end{equation} 
and, for each $\ell\in{\mathbb{N}}_0$, let $\bigl\{\Psi_{i\ell}\bigr\}_{1\leq i\leq H_\ell}$ be 
an orthonormal basis for the space of spherical harmonics of degree $\ell$ 
on the $(n-1)$-dimensional sphere $S^{n-1}$ in ${\mathbb{R}}^{n}$. In particular, 
\begin{equation}\label{D-Har-Nr} 
H_{\ell}\leq(\ell+1)\cdot(\ell+2)\cdots(n+\ell-2)\cdot(n+\ell-1)
\leq C_n\,\ell^{n-1}\quad\mbox{ for }\,\,\ell\geq 2
\end{equation} 
and, if $\Delta_{S^{n-1}}$ denotes the Laplace-Beltrami operator on $S^{n-1}$,
then for each $\ell\in{\mathbb{N}}_0$ and $1\leq i\leq H_\ell$, 
\begin{equation}\label{eihen-XS}
\begin{array}{c}
\Delta_{S^{n-1}}\Psi_{i\ell}=-\ell(n+\ell-2)\Psi_{i\ell}\,\,\mbox{ on }\,\,S^{n-1},\,\,\mbox{ and}
\\[6pt]
\displaystyle
\Psi_{i\ell}\Big(\frac{x}{|x|}\Big)=\frac{P_{i\ell}(x)}{|x|^{\ell}}
\,\,\text{ for every }\,\,x\in{\mathbb{R}}^n\setminus\{0\},
\end{array}
\end{equation} 
for some homogeneous harmonic polynomial $P_{i\ell}$ of degree $\ell$ in ${\mathbb{R}}^n$.
Also, 
\begin{equation}\label{eihen-amm-ONB}
\big\{\Psi_{i\ell}\big\}_{\ell\in{\mathbb{N}}_0,\,1\leq i\leq H_\ell}
\,\,\mbox{ is an orthonormal basis for }\,\,L^2(S^{n-1}),
\end{equation} 
hence, 
\begin{equation}\label{eihen-amm}
\|\Psi_{i\ell}\|_{L^2(S^{n-1})}=1
\,\,\mbox{ for each $\ell\in{\mathbb{N}}_0$ and $1\leq i\leq H_\ell$}.
\end{equation} 
More details on these matters may be found in, e.g., 
\cite[pp.\,137--152]{StWe71} and \cite[pp.\,68--75]{Stein93}.
For further reference let us note here that, having fixed
\begin{equation}\label{eihen-amBB}
\text{an even integer $d\in{\mathbb{N}}$ with $d>[(n+1)/2]$,}
\end{equation} 
Sobolev's embedding theorem then gives that for each $\ell\in{\mathbb{N}}_0$ and 
$1\leq i\leq H_\ell$ we have (with $I$ standing for the identity operator on $S^{n-1}$)
\begin{equation}\label{kl-dYU-11}
\|\Psi_{i\ell}\|_{{\mathscr{C}}^1(S^{n-1})}
\leq C_n\big\|(I-\Delta_{S^{n-1}})^{d/2}\Psi_{i\ell}\big\|_{L^2(S^{n-1})}\leq C_n\ell^{\,d},
\end{equation}
where the last inequality is a consequence of \eqref{eihen-XS}-\eqref{eihen-amm} and, generally speaking, 
\begin{equation}\label{kl-dYU-11-XXX}
\|\Psi\|_{{\mathscr{C}}^1(S^{n-1})}:=\|\Psi\|_{L^\infty(S^{n-1})}+\|\nabla_{\rm tan}\Psi\|_{L^\infty(S^{n-1})},
\qquad\forall\,\Psi\in{\mathscr{C}}^1(S^{n-1}),
\end{equation}
with $\nabla_{\rm tan}$ denoting the tangential gradient to $S^{n-1}$. 

At this stage, observe that $\Theta\big|_{S^{n-1}}\in L^2(S^{n-1})$ hence we may expand
\begin{equation}\label{Yvvca-h6gv-1}
\Theta\Big|_{S^{n-1}}=\sum_{\ell=0}^\infty\sum_{i=1}^{H_\ell}\lambda_{i\ell}\Psi_{i\ell}\,\,\text{ in }\,\,L^2(S^{n-1})
\end{equation}
where
\begin{equation}\label{Yvvca-h6gv-2}
\lambda_{i\ell}:=\int\limits_{S^{n-1}}\Theta(\omega)\Psi_{i\ell}(\omega)\,d\omega
\,\,\,\text{ for each $\ell\in{\mathbb{N}}_0$ and $1\leq i\leq H_\ell$.}
\end{equation}
In relation to \eqref{Yvvca-h6gv-2} we claim that $\lambda_{i\ell}$ decays faster than any power of $\ell$, i.e., 
\begin{equation}\label{Yvvca-h6gv-3}
\parbox{8.60cm}{for each $m\in{\mathbb{N}}$ there exists $C_m\in(0,\infty)$ such that 
$|\lambda_{i\ell}|\leq C_m(1+\ell)^{-m}$ for each $\ell\in{\mathbb{N}}_0$ and $1\leq i\leq H_\ell$.}
\end{equation}
Indeed, if $\ell=0$ this is immediate from \eqref{Yvvca-h6gv-2}. 
In the case when $\ell\in{\mathbb{N}}$, then for each $m\in{\mathbb{N}}$ and $i\in\{1,\dots,H_\ell\}$ 
we may estimate 
\begin{align}\label{Yvvca-h6gv-4}
\big|\lambda_{i\ell}[-\ell(n+\ell-2)]^m\big| &
=\Big|\int\limits_{S^{n-1}}\Theta(\omega)[-\ell(n+\ell-2)]^m\Psi_{i\ell}(\omega)\,d\omega\Big|
\nonumber\\[6pt]
&=\Big|\int\limits_{S^{n-1}}\Delta_{S^{n-1}}^m\big(\Theta\big|_{S^{n-1}}\big)(\omega)
\Psi_{i\ell}(\omega)\,d\omega\Big|
\nonumber\\[6pt]
&\leq\Big\|\Delta_{S^{n-1}}^m\big(\Theta\big|_{S^{n-1}}\big)\Big\|_{L^2(S^{n-1})}=:C_m<+\infty,
\end{align}
thanks to \eqref{Yvvca-h6gv-2}, the first formula in \eqref{eihen-XS}, repeated integrations by parts, 
Cauchy-Schwarz inequality, and \eqref{eihen-amm} (bearing in mind that the finiteness of $C_m$ above 
is implied by the smoothness of $\Theta$). Now \eqref{Yvvca-h6gv-3} readily follows from \eqref{Yvvca-h6gv-4}.

To proceed, we recall a basic formula and make some notational conventions. Concretely, it is well-known 
(cf., e.g., \cite[Theorem~5, p.\,73]{St70}) that, in general, 
\begin{equation}\label{hyBBv.ygg}
\begin{array}{c}
\text{if $P_k$ is a harmonic homogeneous polynomial of degree $k\in{\mathbb{N}}$ in ${\mathbb{R}}^n$ then}
\\[6pt]
\displaystyle
{\mathcal{F}}\left({\rm P.V.}\,\frac{P_k(x)}{|x|^{n+k}}\right)(\xi)=(-i)^k\pi^{n/2}
\frac{\Gamma\big(\frac{k}{2}\big)}{\Gamma\big(\frac{k+n}{2}\big)}\frac{P_k(\xi)}{|\xi|^{k}},\qquad
\xi\in{\mathbb{R}}^n\setminus\{0\}.
\end{array}
\end{equation}
Also, for each multi-index $\alpha=(\alpha_1,\dots,\alpha_n)\in{\mathbb{N}}_0$ we agree to abbreviate 
\begin{equation}\label{hyBBv.ygg-uhb.1}
\begin{array}{c}
R^\alpha:=R_1^{\alpha_1}\circ\cdots\circ R_n^{\alpha_n}
\,\text{ in }\,{\mathscr{B}}\big(\,L^2({\mathbb{R}}^n)\big),
\\[6pt]
\widetilde{R}^\alpha:=\widetilde{R}_1^{\alpha_1}\circ\cdots\circ\widetilde{R}_n^{\alpha_n}
\,\text{ in }\,{\mathscr{B}}\big(\,\widetilde{\rm BMO}({\mathbb{R}}^n)\big),\,\,\text{ and}
\\[6pt]
\big(\widetilde{R}\big|_{\rm VMO}\big)^\alpha:=\big(\widetilde{R}_1\big|_{\rm VMO}\big)^{\alpha_1}\circ
\cdots\circ\big(\widetilde{R}_n\big|_{\rm VMO}\big)^{\alpha_n}
\,\text{ in }\,{\mathscr{B}}\big(\,\widetilde{\rm VMO}({\mathbb{R}}^n)\big),
\end{array}
\end{equation}
then use these abbreviations to define, for each given polynomial 
$P(x)=\sum_{|\alpha|\leq M}c_\alpha x^\alpha$ in ${\mathbb{R}}^n$,
\begin{equation}\label{hyBBv.ygg-uhb.2}
\begin{array}{c}
P(R):=\sum_{|\alpha|\leq M}c_\alpha R^\alpha,\qquad
P\big(\widetilde{R}\big):=\sum_{|\alpha|\leq M}c_\alpha\widetilde{R}^\alpha,
\\[10pt]
\displaystyle
\text{and }\,\,
P\big(\widetilde{R}\big|_{\rm VMO}\big):=\sum_{|\alpha|\leq M}c_\alpha\big(\widetilde{R}\big|_{\rm VMO}\big)^\alpha.
\end{array}
\end{equation}
For further reference, let us also observe that if $A\in{\mathscr{B}}\big(\,\widetilde{\rm BMO}({\mathbb{R}}^n)\big)$
is an operator leaving the space $\widetilde{\rm VMO}({\mathbb{R}}^n)$ invariant then 
$A\big|_{\widetilde{\rm VMO}({\mathbb{R}}^n)}\in{\mathscr{B}}\big(\,\widetilde{\rm VMO}({\mathbb{R}}^n)\big)$ and
\begin{equation}\label{hyBBv.ygg-uhb.3}
\Big\|A\big|_{\widetilde{\rm VMO}({\mathbb{R}}^n)}\Big\|_{{\mathscr{B}}\big(\,\widetilde{\rm VMO}({\mathbb{R}}^n)\big)}
\leq\|A\|_{{\mathscr{B}}\big(\,\widetilde{\rm BMO}({\mathbb{R}}^n)\big)}.
\end{equation}

Returning to the mainstream discussion, we claim that, with the polynomials $P_{i\ell}$ as in \eqref{eihen-XS}
and the $\lambda_{i\ell}$'s as in \eqref{Yvvca-h6gv-2}, we have
\begin{equation}\label{hyBBv.ygg-uhb.4}
\pi^{n/2}\sum_{\ell=0}^N\sum_{i=1}^{H_\ell}\lambda_{i\ell}
\frac{\Gamma\big(\frac{\ell}{2}\big)}{\Gamma\big(\frac{\ell+n}{2}\big)}
P_{i\ell}\big(\widetilde{R}\big)\longrightarrow\widetilde{T}_\Theta\,\,\text{ in }\,\,
{\mathscr{B}}\big(\,\widetilde{\rm BMO}({\mathbb{R}}^n)\big)\,\,\text{ as }\,\,N\to\infty.
\end{equation}
Once this is established, we may conclude with the help of \eqref{hyBBv.ygg-uhb.1}-\eqref{hyBBv.ygg-uhb.3}
that the claim in \eqref{gab-ii-GGVa-FFF} holds. This finishes the proof of \eqref{gab-ii-GGVa}, modulo 
the justification of \eqref{hyBBv.ygg-uhb.4}. 

To facilitate the proof of \eqref{hyBBv.ygg-uhb.4}, for each $N\in{\mathbb{N}}$ introduce 
\begin{equation}\label{hyBBv.ygg-uhb.5}
\Theta_N(x):=\sum_{\ell=0}^N\sum_{i=1}^{H_\ell}\lambda_{i\ell}\frac{P_{i\ell}(x)}{|x|^{n+\ell}}
=\sum_{\ell=0}^N\sum_{i=1}^{H_\ell}\frac{\lambda_{i\ell}}{|x|^n}\Psi_{i\ell}\Big(\frac{x}{|x|}\Big),
\qquad\forall\,x\in{\mathbb{R}}^n\setminus\{0\}.
\end{equation}
Note that \eqref{Yvvca-h6gv-2} implies $\lambda_{10}=0$, given the vanishing moment of 
$\Theta$ and the fact that $\Psi_{10}\big|_{S^{n-1}}$ is a constant (as seen from the second 
line in \eqref{eihen-XS} bearing in mind that the polynomial $P_{10}$ has degree zero). 
Then for each $N\in{\mathbb{N}}$ the function $\Theta_N$ is as in \eqref{ytgfff.9ytrdf.ygfg.222222}.
Bearing this in mind, we may rely on \eqref{PhBB-1}, \eqref{hyBBv.ygg-uhb.5}, \eqref{hyBBv.ygg}, 
and the fact that each $P_{i\ell}$ is a homogeneous harmonic polynomial of degree $\ell$ in ${\mathbb{R}}^n$,
to write
\begin{align}\label{hyBBv.ygg-uhb.5A}
m_{\Theta_N}(\xi) &=\big(\widehat{{\rm P.V.}\Theta_N}\big)(\xi)
=\sum_{\ell=0}^N\sum_{i=1}^{H_\ell}\lambda_{i\ell}\,
{\mathcal{F}}\left({\rm P.V.}\,\frac{P_{i\ell}(x)}{|x|^{n+\ell}}\right)(\xi)
\nonumber\\[6pt]
&=\pi^{n/2}\sum_{\ell=0}^N\sum_{i=1}^{H_\ell}\lambda_{i\ell}\,
\frac{\Gamma\big(\frac{\ell}{2}\big)}{\Gamma\big(\frac{\ell+n}{2}\big)}P_{i\ell}\Big(-i\frac{\xi}{|\xi|}\Big),
\qquad\forall\,\xi\in{\mathbb{R}}^n\setminus\{0\},
\end{align}
for each $N\in{\mathbb{N}}$. In turn, from \eqref{PhBB-1nv-ugg-hgv} and \eqref{hyBBv.ygg-uhb.5A} we see that
for each $N\in{\mathbb{N}}$ and each $f\in L^2({\mathbb{R}}^n)$ we have
\begin{align}\label{ibgTGF6hb54fe3}
\widehat{T_{\Theta_N}f}=m_\Theta\widehat{f}=\pi^{n/2}\sum_{\ell=0}^N\sum_{i=1}^{H_\ell}\lambda_{i\ell}\,
\frac{\Gamma\big(\frac{\ell}{2}\big)}{\Gamma\big(\frac{\ell+n}{2}\big)}\widehat{P_{i\ell}(R)f}.
\end{align}
Thus, for each $N\in{\mathbb{N}}$, 
\begin{align}\label{ibgTGF6hb54fe3.B}
T_{\Theta_N}=\pi^{n/2}\sum_{\ell=0}^N\sum_{i=1}^{H_\ell}\lambda_{i\ell}\,
\frac{\Gamma\big(\frac{\ell}{2}\big)}{\Gamma\big(\frac{\ell+n}{2}\big)}P_{i\ell}(R)
\,\,\text{ in }\,\,{\mathscr{B}}\big(L^2({\mathbb{R}}^n)\big)
\end{align}
which, with the help of Proposition~\ref{jhsdgf}, eventually permits us to conclude that 
\begin{align}\label{ibgTGF6hb54fe3.C}
\widetilde{T}_{\Theta_N}=\pi^{n/2}\sum_{\ell=0}^N\sum_{i=1}^{H_\ell}\lambda_{i\ell}\,
\frac{\Gamma\big(\frac{\ell}{2}\big)}{\Gamma\big(\frac{\ell+n}{2}\big)}P_{i\ell}(\widetilde{R})
\,\,\text{ in }\,\,{\mathscr{B}}\big(\,\widetilde{\rm BMO}({\mathbb{R}}^n)\big)
\,\,\text{ for each }\,\,N\in{\mathbb{N}}.
\end{align}
In view of \eqref{ibgTGF6hb54fe3.C} and \eqref{hyBBv.ygg-uhb.4}, the ultimate goal then becomes proving 
\begin{equation}\label{hyBBv.ygg-uhb.5FFF}
\widetilde{T}_{\Theta_N}\longrightarrow\widetilde{T}_\Theta\,\,\text{ in }\,\,
{\mathscr{B}}\big(\,\widetilde{\rm BMO}({\mathbb{R}}^n)\big)\,\,\text{ as }\,\,N\to\infty.
\end{equation}

With this aim in mind, recall from \eqref{dkegfs.5-AAA} (used with $p=2$) that there exists $\theta\in(0,1)$ 
such that for each $N\in{\mathbb{N}}$ we have
\begin{align}\label{dkegfs.5-AAA-jhh}
\big\|\widetilde{T}_{\Theta}-\widetilde{T}_{\Theta_N}\big\|_{{\mathscr{B}}(\,\widetilde{\rm BMO}({\mathbb{R}}^n))}
&=\big\|\widetilde{T}_{\Theta-\Theta_N}\big\|_{{\mathscr{B}}(\,\widetilde{\rm BMO}({\mathbb{R}}^n))}
\nonumber\\[6pt]
&\leq C_{n}\big\|T_{\Theta-\Theta_N}\big\|_{{\mathscr{B}}(L^2({\mathbb{R}}^n))}^{1-\theta}
\|\nabla\Theta-\nabla\Theta_N\|_{L^\infty(S^{n-1})}^{\theta}
\nonumber\\[6pt]
&\quad+C_n\big\|T_{\Theta-\Theta_N}\big\|_{{\mathscr{B}}(L^2({\mathbb{R}}^n))},
\end{align} 
where the last inequality uses the current format of the constant $C''$ from \eqref{dkegfs.5-AAA}
given in \eqref{ytgfff.9ytrdf.ygfg}.
Next, from \eqref{PhBB-1nv-ugg-hgv} and \eqref{PhBB-1nv-ugg} (used with $p=2$) we deduce that,
for each $N\in{\mathbb{N}}$,
\begin{align}\label{hyBBv.ygg-uhb.7}
\big\|T_{\Theta-\Theta_N}\big\|_{{\mathscr{B}}(L^2({\mathbb{R}}^n))}
\leq C_n\|m_{\Theta-\Theta_N}\|_{L^\infty({\mathbb{R}}^n)}\leq C_{n}\|\Theta-\Theta_N\|_{L^2(S^{n-1})}.
\end{align}
Since \eqref{hyBBv.ygg-uhb.5} and \eqref{Yvvca-h6gv-1} imply
\begin{equation}\label{hyBBv.ygg-uhb.6}
\Theta_N\Big|_{S^{n-1}}=\sum_{\ell=0}^N\sum_{i=1}^{H_\ell}\lambda_{i\ell}\Psi_{i\ell}\longrightarrow
\Theta\Big|_{S^{n-1}}\,\,\text{ in }\,\,L^2(S^{n-1})\,\,\text{ as }\,\,N\to\infty,
\end{equation}
it follows that $\|\Theta-\Theta_N\|_{L^2(S^{n-1})}\to 0$ as $N\to\infty$. Granted this, 
\eqref{hyBBv.ygg-uhb.5FFF} becomes a consequence of \eqref{dkegfs.5-AAA-jhh} and \eqref{hyBBv.ygg-uhb.7}
as soon as we establish that 
\begin{align}\label{dkegfs.5-AAA-jhh-hgfF}
\sup_{N\in{\mathbb{N}}}\|\nabla\Theta_N\|_{L^\infty(S^{n-1})}<+\infty.
\end{align} 

To justify \eqref{dkegfs.5-AAA-jhh-hgfF}, fix $N\in{\mathbb{N}}$ arbitrary and observe that
since $\Theta_N$ is positive homogeneous of degree $-n$, Euler's formula implies
\begin{align}\label{dkegfs.5-AAA-jhh-hgfF.ZA1}
x\cdot(\nabla\Theta_N)(x)=-n\,\Theta_N(x),\qquad\forall\,x\in{\mathbb{R}}^n\setminus\{0\}.
\end{align}
Consequently, 
\begin{align}\label{dkegfs.5-AAA-jhh-hgfF.ZA2}
\nabla_{\rm tan}\big(\Theta_N\big|_{S^{n-1}}\big)(x) &=(\nabla\Theta_N)(x)-\big(x\cdot(\nabla\Theta_N)(x)\big)x
\nonumber\\[4pt]
&=(\nabla\Theta_N)(x)+n\,\Theta_N(x)\,x\,\,\text{ for each }\,\,x\in S^{n-1}
\end{align}
which, in light of \eqref{kl-dYU-11-XXX}, further implies 
\begin{equation}\label{dkegfs.5-AAA-jhh-hgfF.ZA3}
\|\nabla\Theta_N\|_{L^\infty(S^{n-1})}\leq n\|\Theta_N\|_{{\mathscr{C}}^1(S^{n-1})}.
\end{equation}
On the other hand, from \eqref{hyBBv.ygg-uhb.5} we know that 
$\Theta_N=\sum_{\ell=0}^N\sum_{i=1}^{H_\ell}\lambda_{i\ell}\Psi_{i\ell}$ on $S^{n-1}$, hence
for each $m\in{\mathbb{N}}$ there exists $C_m\in(0,\infty)$ such that
\begin{align}\label{dkegfs.5-AAA-jhh-hgfF.ZA4}
\|\nabla\Theta_N\|_{L^\infty(S^{n-1})} &\leq n\|\Theta_N\|_{{\mathscr{C}}^1(S^{n-1})}
\leq n\sum_{\ell=0}^N\sum_{i=1}^{H_\ell}|\lambda_{i\ell}|\|\Psi_{i\ell}\|_{{\mathscr{C}}^1(S^{n-1})}
\nonumber\\[0pt]
&\leq C_m C_n\sum_{\ell=0}^N(1+\ell)^{-m}\ell^d\ell^{n-1},
\end{align}
where the last inequality is based on \eqref{Yvvca-h6gv-3}, \eqref{kl-dYU-11}, and \eqref{D-Har-Nr}.
Choosing $m$ large enough (depending on $n$ and $d$) so that the partial sums above converge, we ultimately see that  
\begin{align}\label{dkegfs.5-AAA-jhh-hgfF.ZA5}
\sup_{N\in{\mathbb{N}}}\|\nabla\Theta_N\|_{L^\infty(S^{n-1})}
\leq C_n C_m\sum_{\ell=0}^\infty(1+\ell)^{-m}\ell^d\ell^{n-1}<+\infty,
\end{align}
which establishes \eqref{dkegfs.5-AAA-jhh-hgfF}. This finishes the proof \eqref{gab-ii-GGVa}.

To deal with item {\it (c)}, assume next that $\Theta$ is as in \eqref{ytgfff.9ytrdf.ygfg.222222} 
and $c$ is as in \eqref{gab-ii-GG-yt55}. Then  
\begin{equation}\label{gab-ii-GG-yt55.bvc.1}
m(\xi):=\big(c+m_{\widetilde{\Theta}}(\xi)\big)^{-1}\,\,\text{ for each }\,\,\xi\in{\mathbb{R}}^n\setminus\{0\}
\end{equation}
is a well-defined function, which belongs to ${\mathscr{C}}^\infty({\mathbb{R}}^n\setminus\{0\})$
and is positive homogeneous of degree zero. As such, \eqref{ijBBa-hyf.baba} guarantees the existence
of a function $\Theta_0$ as in \eqref{ytgfff.9ytrdf.ygfg.222222} with the property that 
$m=c_0+m_{\widetilde{\Theta}_0}$, where $c_0:=\int_{S^{n-1}}m(\omega)\,d\omega\in{\mathbb{C}}$. 
We claim that 
\begin{equation}\label{gab-ii-GG-yt55.bvc.2iii}
m_{\widetilde{\Theta}}\,m_{\widetilde{\Theta}_0}=(1-cc_0)+m_{-c\widetilde{\Theta}_0-c_0\widetilde{\Theta}}.
\end{equation}
This is seen by expanding $m_{-c\widetilde{\Theta}_0-c_0\widetilde{\Theta}}
=-cm_{\widetilde{\Theta}_0}-c_0m_{\widetilde{\Theta}}$ then replacing throughout
$m_{\widetilde{\Theta}_0}$ by $\big(c+m_{\widetilde{\Theta}}\big)^{-1}-c_0$. After some 
simple algebra \eqref{gab-ii-GG-yt55.bvc.2iii} follows. 
By virtue of the second formula in \eqref{jhsfrwfr-BMO-ttt-iub}, the identity in \eqref{gab-ii-GG-yt55.bvc.2iii}
implies
\begin{align}\label{jhsfrwfr-BMO-ttt-iub.LLL}
\widetilde{T}_{\Theta}\circ\widetilde{T}_{\Theta_0}
&=(1-cc_0)I+\widetilde{T}_{-c\widetilde{\Theta}_0-c_0\widetilde{\Theta}}
\nonumber\\[4pt]
&=(1-cc_0)I-c\widetilde{T}_{\widetilde{\Theta}_0}-c_0\widetilde{T}_{\widetilde{\Theta}}
\,\,\text{ on }\,\,\widetilde{\mathrm{BMO}}({\mathbb{R}}^n).
\end{align}
The above formula may be re-cast as 
\begin{align}\label{jhsfrwfr-BMO-ttt-iub.FFF}
(cI+\widetilde{T}_{\Theta})\circ(c_0I+\widetilde{T}_{\Theta_0})=I
\,\,\text{ on }\,\,\widetilde{\mathrm{BMO}}({\mathbb{R}}^n).
\end{align}
In a similar manner we also obtain 
\begin{align}\label{jhsfrwfr-BMO-ttt-iub.FFF.2}
(c_0I+\widetilde{T}_{\Theta_0})\circ(cI+\widetilde{T}_{\Theta})=I
\,\,\text{ on }\,\,\widetilde{\mathrm{BMO}}({\mathbb{R}}^n).
\end{align}
From \eqref{jhsfrwfr-BMO-ttt-iub.FFF}-\eqref{jhsfrwfr-BMO-ttt-iub.FFF.2}
we conclude that $cI+\widetilde{T}_{\Theta}$ is invertible as an operator 
on $\widetilde{\mathrm{BMO}}({\mathbb{R}}^n)$, whose inverse is 
$c_0I+\widetilde{T}_{\Theta_0}\in{\mathscr{B}}\big(\,\widetilde{\mathrm{BMO}}({\mathbb{R}}^n)\big)$.
Since both operators map $\widetilde{\mathrm{VMO}}({\mathbb{R}}^n)$ into itself 
(cf. Theorem~\ref{i87hbBV}), it follows that 
$c_0I+\widetilde{T}_{\Theta_0}\big|_{{}_{\rm VMO}}\in{\mathscr{A}}_{{}_{\widetilde{\rm SIO}}}$ 
is the inverse of $cI+\widetilde{T}_{\Theta}\big|_{{}_{\rm VMO}}$. This concludes the treatment of item {\it (c)}.

Moving on, the first claim made in item {\it (d)}, pertaining to the equivalence stated in \eqref{u7GBB}, 
is seen directly from item {\it (c)} (which yields the left-pointing implication), and 
Theorem~\ref{i87hbBV} (which gives the right-pointing implication).
Consider next the second claim made in item {\it (d)}. To set the stage, pick $N\in{\mathbb{N}}$ 
and assume $\Theta_1,\dots,\Theta_N$ are as in \eqref{ytgfff.9ytrdf.ygfg.222222} while 
$c_1,\dots,c_N$ are as in \eqref{gab-ii-GG-yt55.ZZZ}. If we set 
\begin{equation}\label{u7GBB.Za.1}
Q(\xi):=\sum_{j=1}^N\big|c_j+m_{\widetilde{\Theta}_j}(\xi)\big|^2\,\,\text{ for each }\,\,
\xi\in{\mathbb{R}}^n\setminus\{0\},
\end{equation}
then the present assumptions ensure that $Q$ is a real-valued function which is well-defined, 
of class ${\mathscr{C}}^\infty$, positive homogeneous of degree zero, and never zero in 
${\mathbb{R}}^n\setminus\{0\}$. As such, if for each $j\in\{1,\dots,N\}$ we now introduce 
\begin{equation}\label{u7GBB.Za.2}
m_j(\xi):=\frac{\overline{c_j}+m_{\overline{\Theta}_j}(\xi)}{Q(\xi)}
=\frac{\overline{c_j+m_{\widetilde{\Theta}_j}(\xi)}}{Q(\xi)}
\,\,\text{ for each }\,\,\xi\in{\mathbb{R}}^n\setminus\{0\}
\end{equation}
(where the second equality is a consequence of one of the formulas in \eqref{PhBB-1gd}), 
then each $m_j$ is a complex-valued function which is well-defined, of class ${\mathscr{C}}^\infty$, 
and positive homogeneous of degree zero in ${\mathbb{R}}^n\setminus\{0\}$. According to \eqref{ijBBa-hyf.baba}, 
these properties guarantee the existence of numbers $c'_j\in{\mathbb{C}}$ and functions $\Theta'_{\!j}$ 
as in \eqref{ytgfff.9ytrdf.ygfg.222222} such that 
\begin{equation}\label{u7GBB.Za.3}
m_j=c'_j+m_{\widetilde{\Theta'}_{\!\!j}}\,\,\text{ in }\,\,{\mathbb{R}}^n\setminus\{0\}
\,\,\text{ for each }\,\,j\in\{1,\dots,N\}.
\end{equation}
Since, by design, $\sum_{j=1}^N m_j(\xi)\big(c_j+m_{\widetilde{\Theta}_j}(\xi)\big)=1$ 
for each $\xi\in{\mathbb{R}}^n\setminus\{0\}$, we then conclude that 
\begin{equation}\label{u7GBB.Za.4}
\sum_{j=1}^N\big(c'_j+m_{\widetilde{\Theta'}_{\!\!j}}(\xi)\big)\big(c_j+m_{\widetilde{\Theta}_j}(\xi)\big)=1
\,\,\text{ for each }\,\,\xi\in{\mathbb{R}}^n\setminus\{0\}
\end{equation}
or, equivalently, 
\begin{equation}\label{u7GBB.Za.5}
\sum_{j=1}^N m_{\widetilde{\Theta'}_{\!\!j}}\,m_{\widetilde{\Theta}_j}=c+m_{\widetilde{\Theta}}
\,\,\text{ in }\,\,{\mathbb{R}}^n\setminus\{0\},
\end{equation}
where 
\begin{equation}\label{u7GBB.Za.6}
c:=\Big(1-\sum_{j=1}^N c'_jc_j\Big)\in{\mathbb{C}}\,\,\text{ and }\,\,
\Theta:=-\sum_{j=1}^N\Big\{c'_j{\Theta}_j+c_j{\Theta'}_{\!\!j}\Big\}
\,\,\text{ is as in }\,\,\eqref{ytgfff.9ytrdf.ygfg.222222}.
\end{equation}
Similarly to \eqref{jhsfrwfr-BMO-ttt-iub}, from \eqref{u7GBB.Za.5}-\eqref{u7GBB.Za.6} 
we conclude that 
\begin{equation}\label{u7GBB.Za.7}
\sum_{j=1}^N \widetilde{T}_{{\Theta'}_{\!\!j}}\,\widetilde{T}_{{\Theta}_j}=cI+\widetilde{T}_{\Theta}
=\Big(1-\sum_{j=1}^N c'_jc_j\Big)I-\sum_{j=1}^N\Big\{c'_j\widetilde{T}_{{\Theta}_j}
+c_j\widetilde{T}_{{\Theta'}_{\!\!j}}\Big\}
\end{equation}
which, in turn, implies 
\begin{equation}\label{u7GBB.Za.8}
\sum_{j=1}^N\big(c'_j I+\widetilde{T}_{{\Theta'}_{\!\!j}}\big)
\big(c_j I+\widetilde{T}_{{\Theta}_j}\big)=I\,\,\text{ in }\,\,
{\mathscr{B}}\big(\,\widetilde{\rm BMO}({\mathbb{R}}^n)\big).
\end{equation}

With this in hand, we may turn to the proof of the equivalence recorded in \eqref{u7GBB.ZZZ} in earnest. 
The right-pointing implication is clear from Theorem~\ref{i87hbBV}. As regards the left-pointing 
implication, assume $f\in{\mathrm{BMO}}({\mathbb{R}}^n)$ is such that there exist 
$g_1,\dots,g_N\in{\mathrm{VMO}}({\mathbb{R}}^n)$ with the property that 
\begin{equation}\label{u7GBB.ZZZ.VVV}
[g_j]=(c_jI+\widetilde{T}_{\Theta_j})[f]\,\,\text{ in }\,\,\widetilde{\mathrm{BMO}}({\mathbb{R}}^n)
\,\,\text{ for each }\,\,j\in\{1,\dots,N\}.
\end{equation}
Then \eqref{u7GBB.Za.8} permits to express $[f]\in\widetilde{\rm BMO}({\mathbb{R}}^n)$ as
\begin{equation}\label{u7GBB.Za.8.TT}
[f]=\sum_{j=1}^N\big(c'_j I+\widetilde{T}_{{\Theta'}_{\!\!j}}\big)\big(c_j I+\widetilde{T}_{{\Theta}_j}\big)[f]
=\sum_{j=1}^N\big(c'_j I+\widetilde{T}_{{\Theta'}_{\!\!j}}\big)[g_j]\in\widetilde{\rm VMO}({\mathbb{R}}^n)
\end{equation}
where the membership above is provided by Theorem~\ref{i87hbBV}. Ultimately, from \eqref{u7GBB.Za.8.TT} 
we conclude that $f\in{\mathrm{VMO}}({\mathbb{R}}^n)$, finishing the proof of \eqref{u7GBB.ZZZ}.

Finally, the proofs of the claims in item {\it (e)} closely parallel those in the scalar case, with minor 
natural adjustments, of a purely algebraic nature (designed to accommodate the present matrix-formalism).
\end{proof}

In turn, Theorem~\ref{i87hbBV-ALG} may be specialized as to yield 
Corollaries~\ref{i87hbBV-ALG-CCC}-\ref{nhxtrE-SSS} as indicated below. 

\vskip 0.08in
\begin{proof}[Proof of Corollary~\ref{i87hbBV-ALG-CCC}]
The strategy is to devise a suitable dictionary between the algebra formalism, currently used, 
and the matrix formalism described in item {\it (e)} of Theorem~\ref{i87hbBV-ALG}, which is going 
to yield \eqref{u7GBB-CCC} at once. To get started, fix a linear basis $\{e_1,\dots,e_N\}$ 
in $A$, regarded as a vector space. Then we have a linear isomorphism 
\begin{equation}\label{kj7AA-LLGG.1}
A\ni a=\sum_{j=1}^Na_je_j\longmapsto a^V:=(a_j)_{1\leq j\leq N}\in{\mathbb{C}}^N
\end{equation}
identifying algebra elements $a\in A$ with their vector realizations $a^V\in{\mathbb{C}}^N$.
We shall also need to identify each algebra element $a\in A$ with a certain matrix $a^M\in{\mathbb{C}}^{N\times N}$.
To define this matrix realization, consider the family of complex numbers $\lambda_{\ell k j}$, 
with $1\leq\ell,k,j\leq N$, such that 
\begin{equation}\label{kj7AA-LLGG.2}
e_j\odot e_k=\sum_{\ell=1}^N\lambda_{\ell k j}e_\ell\,\,\text{ for each }\,\,j,k\in\{1,\dots,N\},
\end{equation}
then set 
\begin{equation}\label{kj7AA-LLGG.4}
a^M:=\Big(\sum_{j=1}^N\lambda_{\ell k j}a_j\Big)_{1\leq\ell,k\leq N}\in{\mathbb{C}}^{N\times N},
\qquad\forall\,a=\sum_{j=1}^Na_je_j\in A.
\end{equation}
In relation to these realizations of algebra elements, the following identity holds
\begin{equation}\label{kj7AA-LLGG.5}
a\odot b=c\Longleftrightarrow a^Mb^V=c^V,\qquad\forall\,a,b,c\in A.
\end{equation}
We next claim that 
\begin{equation}\label{kj7AA-LLGG.6}
\parbox{7.00cm}{if $a\in A$ is invertible in $A$ from the right then 
the matrix $a^M$ is invertible in ${\mathbb{C}}^{N\times N}$.}
\end{equation}
To see this, fix $a\in A$ which has an inverse $a^{-1}_R\in A$ from the right, and pick some 
arbitrary $(z_1,\dots,z_N)\in{\mathbb{C}}^N$. Set $c:=\sum_{\ell=1}^Nz_\ell e_\ell\in A$
and consider $b:=a^{-1}_R\odot c\in A$. According to \eqref{kj7AA-LLGG.5}, the fact that $a\odot b=c$ 
then translates into $a^Mb^V=c^V=(z_1,\dots,z_N)$. Since the latter is an arbitrary vector in ${\mathbb{C}}^N$, 
this proves that, as a linear map from ${\mathbb{C}}^N$ into itself, the matrix $a^M$ is surjective, 
hence ultimately, invertible. 

Consider next an $A$-valued kernel $\Theta$ as in \eqref{ytgfff.9ytrdf.ygfg.222222-CCC}.
Then $\Theta=\sum_{j=1}^N\Theta_j e_j$ with each scalar component $\Theta_j$ as 
in \eqref{ytgfff.9ytrdf.ygfg.222222} and, by definition and \eqref{kj7AA-LLGG.2},  
\begin{align}\label{kj7AA-LLGG.7}
\begin{array}{c}
\widetilde{T}_\Theta[f]=\sum_{j,k=1}^N\widetilde{T}_{\Theta_j}[f_k]e_j\odot e_k
=\sum_{j,k,\ell=1}^N\lambda_{\ell kj}\widetilde{T}_{\Theta_j}[f_k]e_\ell,
\\[6pt]
\text{for every }\,\,f=\sum_{k=1}^Nf_ke_k\in{\rm BMO}({\mathbb{R}}^n)\otimes A.
\end{array}
\end{align}
If we also associated with the $A$-valued kernel $\Theta$ the matrix-valued kernel $\Theta^M$ as in 
\eqref{kj7AA-LLGG.4}, we may rewrite \eqref{kj7AA-LLGG.7} simply as
\begin{equation}\label{kj7AA-LLGG.8}
\big(\widetilde{T}_\Theta[f]\big)^V=\widetilde{T}_{\Theta^M}[f]^V,
\qquad\forall\,f\in{\rm BMO}({\mathbb{R}}^n)\otimes A.
\end{equation}
Since \eqref{kj7AA-LLGG.5} also gives
\begin{equation}\label{kj7AA-LLGG.9}
\big(c\odot[f]\big)^V=c^M[f]^V,\qquad\forall\,c\in A,\quad\forall\,f\in{\rm BMO}({\mathbb{R}}^n)\otimes A,
\end{equation}
from \eqref{kj7AA-LLGG.8}-\eqref{kj7AA-LLGG.9} we finally conclude that 
\begin{equation}\label{kj7AA-LLGG.10}
\Big(\big(cI+\widetilde{T}_\Theta\big)[f]\Big)^V=\big(c^M+\widetilde{T}_{\Theta^M}\big)[f]^V,
\qquad\forall\,c\in A,\quad\forall\,f\in{\rm BMO}({\mathbb{R}}^n)\otimes A.
\end{equation}
There remains to observe that, since 
$\big(c+m_{\widetilde{\Theta}}(\xi)\big)^M=c^M+m_{{\widetilde{\Theta}}^M}(\xi)$
for each $\xi\in{\mathbb{R}}^n\setminus\{0\}$, from \eqref{kj7AA-LLGG.6} we have that
\begin{equation}\label{kj7AA-LLGG.11}
\parbox{8.00cm}{if $c$ is as in \eqref{gab-ii-GG-yt55-CCC} then
for each $\xi\in{\mathbb{R}}^n\setminus\{0\}$ the matrix 
$c^M+m_{{\widetilde{\Theta}}^M}(\xi)$ is invertible in ${\mathbb{C}}^{N\times N}$.}
\end{equation}
Then \eqref{kj7AA-LLGG.10}-\eqref{kj7AA-LLGG.11} ensure that item {\it (e)} of Theorem~\ref{i87hbBV-ALG}
applies (with ${\mathscr{V}}:={\mathbb{C}}^N$) which proves \eqref{u7GBB-CCC}.
\end{proof}

\vskip 0.08in
\begin{proof}[Proof of Corollary~\ref{nhxtrE.2D}]
The complex Riesz transform defined in \eqref{R-87ygbg.RRR} as well as the Beurling transform \eqref{R-87ygbg.BEBE} 
are principal value convolution operators of the sort discussed in \eqref{ytgfff.9ytrdf.ygfg}. Specifically, 
\begin{equation}\label{ABC-rrr.1}
R_{{}_{\mathbb{C}}}=T_{\Theta_1}\,\,\text{ with }\,\,\Theta_1(z):=\frac{z}{2\pi|z|^3}
\,\,\text{ for }\,\,z\in{\mathbb{C}}\setminus\{0\},
\end{equation}
and 
\begin{equation}\label{ABC-rrr.2}
S=T_{\Theta_2}\,\,\text{ with }\,\,\Theta_2(z):=-\frac{1}{\pi z^2}=-\frac{(\overline{z})^2}{\pi |z|^4}
\,\,\text{ for }\,\,z\in{\mathbb{C}}\setminus\{0\}.
\end{equation}
Their associated symbols are given by (cf. \eqref{hyBBv.ygg}) 
\begin{equation}\label{ABC-rrr.rd}
\begin{array}{c}
m_{\Theta_1}(\xi)=(\widehat{{\rm P.V.}\Theta_1})(\xi)
=-i\xi/|\xi|\,\,\text{ for }\,\,\xi\in{\mathbb{C}}\setminus\{0\},\,\,\text{ and}
\\[8pt]
m_{\Theta_2}(\xi)=(\widehat{{\rm P.V.}\Theta_2})(\xi)
=(\overline{\xi})^2/|\xi|^2=\overline{\xi}/\xi\,\,\text{ for }\,\,\xi\in{\mathbb{C}}\setminus\{0\}.
\end{array}
\end{equation}
Upon observing that for $j\in\{1,2\}$ we have
\begin{align}\label{ABC-rrr.rd.222}
c\in{\mathbb{C}}\,\,\text{ with }\,\,|c|\not=1\,\Longrightarrow\,
c\in{\mathbb{C}}\setminus\big\{-m_{\widetilde{\Theta}_j}(\xi):\,\xi\in{\mathbb{C}}\setminus\{0\}\big\},
\end{align}
the first part of item {\it (d)} in Theorem~\ref{i87hbBV-ALG} applies and gives that 
{\it (i)}\,$\Leftrightarrow$\,{\it (ii)} as well as {\it (i)}\,$\Leftrightarrow$\,{\it (iii)}. 
This finishes the proof of Corollary~\ref{nhxtrE.2D}.
\end{proof}

\vskip 0.08in
\begin{proof}[Proof of Corollary~\ref{nhxtrE.ND}]
The Clifford-Riesz transform defined in \eqref{R-87ygbg.RRR.NNN} is a principal value convolution operator of form  
$R_{{}_{{\mathcal{C}}\!\ell}}=T_{\Theta}$, where the kernel is the Clifford algebra-valued function 
(see the convention in \eqref{im-e-yyy}) 
\begin{equation}\label{ABC-rrr.1XXX}
\Theta:{\mathbb{R}}^n\setminus\{0\}\to{\mathcal{C}}\!\ell_n
\,\,\text{ given by }\,\,\Theta(x):=\frac{\Gamma\big(\frac{n+1}{2}\big)}{\pi^{\frac{n+1}{2}}}\frac{x}{|x|^{n+1}}
\,\,\text{ for }\,\,x\in{\mathbb{R}}^n\setminus\{0\}.
\end{equation}
Thanks to \eqref{hyBBv.ygg}, its associated symbol may be explicitly identified as 
\begin{equation}\label{ABC-rrr.rdXXX}
m_{\Theta}(\xi)=(\widehat{{\rm P.V.}\Theta})(\xi)
=-i\xi/|\xi|\,\,\text{ for }\,\,\xi\in{\mathbb{R}}^n\setminus\{0\}.
\end{equation}
In particular, if $c\in{\mathcal{C}}\!\ell_{n}$ is such that 
$c+i\omega$ is invertible in ${\mathcal{C}}\!\ell_{n}$ from the right for each vector 
$\omega\in S^{n-1}\subseteq{\mathbb{R}}^n\hookrightarrow{\mathcal{C}}\!\ell_{n}$, then
\begin{align}\label{ABC-rrr.rd.222XXX}
c+m_{\widetilde{\Theta}}(\xi)\,\,\text{ is invertible in ${\mathcal{C}}\!\ell_{n}$ from the right for each }
\,\,\xi\in{\mathbb{R}}^n\setminus\{0\}.
\end{align}
Granted this, Corollary~\ref{i87hbBV-ALG-CCC} applies with $A:={\mathcal{C}}\!\ell_{n}$
and gives the equivalence in \eqref{ytGVVV}.
\end{proof}

\vskip 0.08in
\begin{proof}[Proof of Corollary~\ref{nhxtrE}]
The equivalence stated in \eqref{jhxfdw-R.jhghg} is an immediate consequence of \eqref{u7GBB.ZZZ} 
(used with $N=n$ and $\Theta_j:=K_j$, defined in \eqref{R-87ygbg}, for $1\leq j\leq n$) upon noting 
that condition \eqref{gab-ii-GG-yt55.ZZZ} presently reads $(c_1,\dots,c_n)\in{\mathbb{C}}^n\setminus iS^{n-1}$.
\end{proof}

\vskip 0.08in
\begin{proof}[Proof of Corollary~\ref{nhxtrE-SSS}]
To recast the operator $S_\theta$ in the manner described in \eqref{i7yggg.B4}, fix some arbitrary differential 
form $f\in L^2({\mathbb{R}}^n)\otimes\Lambda$ and, starting with \eqref{8hhBBBa.UUU}-\eqref{8hhBBBa.AAA.xxx}, 
write (bearing in mind that the $j$-th Riesz transform on $L^2({\mathbb{R}}^n)$ is the multiplier with 
symbol $-i\xi_j/|\xi|$)
\begin{align}\label{i7yggg.B1}
\widehat{S_\theta f}(\xi) &=-\theta\sum_{j,k=1}^n
dx_j\wedge\Big(dx_k\vee\Big(\frac{\xi_j\xi_k}{|\xi|^2}\widehat{f}(\xi)\Big)\Big)
\nonumber\\[6pt]
&\quad +\theta^{-1}\sum_{j,k=1}^n dx_j\vee\Big(dx_k\wedge\Big(\frac{\xi_j\xi_k}{|\xi|^2}\widehat{f}(\xi)\Big)\Big).
\end{align}
Granted \eqref{utggvVV.T1}-\eqref{utggvVV.T2}, we may consider the principal-value distribution 
${\rm P.V.}\,\Theta_{jk}$ associated with $\Theta_{jk}$ as in \eqref{iu7HBBam}. 
Upon recalling (cf. \cite[Proposition~4.70, p.\,141]{DM}) that for each pair of indices
$j,k\in\{1,\dots,n\}$ we have (with ${\mathcal{F}}$ used as an alternative notation for 
the Fourier transform `hat' in ${\mathbb{R}}^n$, and with $\delta$ denoting the standard 
Dirac distribution in ${\mathbb{R}}^n$) 
\begin{align}\label{i7yggg.B2}
\frac{\xi_j\xi_k}{|\xi|^2}={\mathcal{F}}\Big(\,\frac{1}{\omega_{n-1}}\big({\rm P.V.}\,\Theta_{jk}\big)
+\frac{1}{n}\delta_{jk}\,\delta\Big)(\xi),
\end{align}
for each $j,k\in\{1,\dots,n\}$ we may express
\begin{align}\label{i7yggg.B3}
\frac{\xi_j\xi_k}{|\xi|^2}\widehat{f}(\xi)
&={\mathcal{F}}\Big(\Big(\frac{1}{\omega_{n-1}}\big({\rm P.V.}\,\Theta_{jk}\big)
+\frac{1}{n}\delta_{jk}\,\delta\Big)\ast f\Big)(\xi)
\nonumber\\[6pt]
&={\mathcal{F}}\Big(\,\frac{1}{\omega_{n-1}}T_{\Theta_{jk}}f+\frac{1}{n}\delta_{jk}f\Big)(\xi).
\end{align}
In turn, from \eqref{i7yggg.B1} and \eqref{i7yggg.B3} we readily conclude that \eqref{i7yggg.B4} holds.

Next, Proposition~\ref{jhsdgf} ensures that $S_\theta$, originally considered as in \eqref{i7yggg.B4}, 
further extends to a well-defined linear and bounded operator from the space 
$H^1({\mathbb{R}}^n)\otimes\Lambda$ into itself. Keeping this in mind, for each 
$[f]\in\widetilde{\rm BMO}({\mathbb{R}}^n)\otimes\Lambda$ and each $g\in H^1({\mathbb{R}}^n)\otimes\Lambda$ 
we may write 
\begin{align}\label{i7yggg.B6666}
\big\langle[f],S_\theta g\big\rangle &=-\frac{\theta}{\omega_{n-1}}\sum_{j,k=1}^n
\Big\langle[f]\,,\,dx_k\wedge\big(dx_j\vee\big(T_{\widetilde{\Theta}_{jk}}g\big)\big)\Big\rangle 
\nonumber\\[6pt]
&\quad+\frac{\theta^{-1}}{\omega_{n-1}}\sum_{j,k=1}^n\Big\langle[f]\,,\,  
dx_k\vee\big(dx_j\wedge\big(T_{\widetilde{\Theta}_{jk}}g\big)\big)\Big\rangle 
\nonumber\\[6pt]
&\quad-\frac{\theta}{n}\sum_{j=1}^n\Big\langle[f]\,,\,dx_j\wedge(dx_j\vee g)\Big\rangle 
\nonumber\\[6pt]
&\quad+\frac{\theta^{-1}}{n}\sum_{j=1}^n\Big\langle[f]\,,\,dx_j\vee(dx_j\wedge g)\Big\rangle 
\nonumber\\[6pt]
&=-\frac{\theta}{\omega_{n-1}}\sum_{j,k=1}^n
\Big\langle dx_j\wedge\big(dx_k\vee\big(\widetilde{T}_{\Theta_{jk}}[f]\big)\big)\,,\,g\Big\rangle 
\nonumber\\[6pt]
&\quad+\frac{\theta^{-1}}{\omega_{n-1}}\sum_{j,k=1}^n\Big\langle  
dx_j\vee\big(dx_k\wedge\big(\widetilde{T}_{\Theta_{jk}}[f]\big)\big)\,,\,g\Big\rangle 
\nonumber\\[6pt]
&\quad-\frac{\theta}{n}\sum_{j=1}^n\Big\langle dx_j\wedge(dx_j\vee[f])\,,\,g\Big\rangle 
\nonumber\\[6pt]
&\quad+\frac{\theta^{-1}}{n}\sum_{j=1}^n\Big\langle dx_j\vee(dx_j\wedge[f])\,,\,g\Big\rangle.
\end{align}
The first equality above uses $\widetilde{\Theta}_{jk}={\Theta}_{jk}={\Theta}_{kj}$ (cf. \eqref{utggvVV.T2}), 
while the second equality is based on the transposition formula \eqref{dkegfs.gG-jgV} and  
the fact that the interior and exterior product of forms are dual to one another. 
On the other hand, since for each $[f]\in\widetilde{\rm BMO}({\mathbb{R}}^n)\otimes\Lambda$ and 
$g\in H^1({\mathbb{R}}^n)\otimes\Lambda$ we have 
\begin{align}\label{i7yggg.BihR.a}
\big\langle\widetilde{R}\wedge[f],g\big\rangle &=\Big\langle\sum_{j=1}^n
dx_j\wedge\widetilde{R}_j[f]\,,\,g\Big\rangle 
=\sum_{j=1}^n\big\langle\widetilde{R}_j[f]\,,\,dx_j\vee g\big\rangle
\nonumber\\[6pt]
&=-\sum_{j=1}^n\big\langle[f]\,,\,dx_j\vee R_jg\big\rangle
=-\big\langle[f],R\vee g\big\rangle,
\end{align}
and, similarly, 
\begin{align}\label{i7yggg.BihR.b}
\big\langle\widetilde{R}\vee[f],g\big\rangle=-\big\langle[f],R\wedge g\big\rangle,
\end{align}
from \eqref{8hhBBBa.UUU.xxxx} and \eqref{i7yggg.BihR.a}-\eqref{i7yggg.BihR.b} we conclude that 
\begin{equation}\label{8hhBBBa.RFcc}
\big\langle\widetilde{S}_\theta[f],g\big\rangle=\big\langle[f],S_\theta g\big\rangle,\quad
\forall\,[f]\in\widetilde{\mathrm{BMO}}({\mathbb{R}}^n)\otimes\Lambda,\,\,
\forall\,g\in H^1({\mathbb{R}}^n)\otimes\Lambda.
\end{equation}
At this stage, by comparing \eqref{i7yggg.B6666} with \eqref{8hhBBBa.RFcc} and keeping in mind the 
$\widetilde{\mathrm{BMO}}$-$H^1$ duality, we conclude that \eqref{i7yggg.B5555} holds.

Let us now turn our attention to the equivalences in the last part of the statement of the corollary. 
As a preamble, for each $\omega=(\omega_1,\dots,\omega_n)\in S^{n-1}$, identified with the differential form 
of degree one $\omega_1\,dx_1+\cdots+\omega_n\,dx_n$ in ${\mathbb{R}}^n$, introduce the 
operators $P_\omega,Q_\omega$ acting on an arbitrary differential form $u\in\Lambda$ according to 
\begin{equation}\label{i7yggg.A1}
P_\omega u:=\omega\wedge(\omega\vee u),\qquad Q_\omega u:=\omega\vee(\omega\wedge u).
\end{equation}
In the same vein, for each $\theta\in{\mathbb{C}}\setminus\{0\}$ and $\omega\in S^{n-1}$ let us also set 
\begin{equation}\label{i7yggg.A1bbbb}
\Omega_{\theta,\omega}\,u:=\theta\omega\wedge u+\theta^{-1}\omega\vee u,\qquad\forall\,u\in\Lambda.
\end{equation}
Then, with $I$ denoting the identity operator on $\Lambda$, for each $\omega\in S^{n-1}$ and 
$\theta\in{\mathbb{C}}\setminus\{0\}$ we have (see \cite[Lemma~2.2, p.\,54]{MMMT16})
\begin{equation}\label{i7yggg.A2}
\begin{array}{c}
P_\omega Q_\omega=Q_\omega P_\omega=0,\quad P_\omega+Q_\omega=I,
\\[4pt]
P_\omega^2=P_\omega,\quad Q_\omega^2=Q_\omega,
\,\,\text{ and }\,\,\Omega_{\theta,\omega}\Omega_{\theta,\omega}=I.
\end{array}
\end{equation}
In this notation, it follows from \eqref{8hhBBBa.AAA}-\eqref{8hhBBBa.1} that
\begin{equation}\label{i7yggg.A4}
\begin{array}{c}
S_\theta:L^2({\mathbb{R}}^n)\otimes\Lambda\longrightarrow L^2({\mathbb{R}}^n)\otimes\Lambda
\,\,\text{ acts on each }\,\,f\in L^2({\mathbb{R}}^n)\otimes\Lambda
\\[6pt]
\text{according to }\,\,
S_\theta f(x)={\mathcal{F}}^{-1}_{\xi\to x}
\Big(\big(-\theta P_{\xi/|\xi|}+\theta^{-1}Q_{\xi/|\xi|}\big)\widehat{f}(\xi)\Big)
\,\,\text{ for a.e. }\,\,x\in{\mathbb{R}}^n.
\end{array}
\end{equation}
Hence, $S_\theta$ is a multiplier operator with symbol given by 
\begin{equation}\label{i7yggg.A4-hfc}
m(\xi):=-\theta P_{\xi/|\xi|}+\theta^{-1}Q_{\xi/|\xi|}\in{\rm Hom}(\Lambda,\Lambda)
\,\,\text{ for }\,\,\xi\in{\mathbb{R}}^n\setminus\{0\}.
\end{equation}
We now claim that 
\begin{equation}\label{i7yggg.A4-hfc.22}
\parbox{8.30cm}{if $\theta\in{\mathbb{C}}\setminus\{0\}$ and 
$c\in{\mathbb{C}}\setminus\big\{\theta\,,\,-\theta^{-1}\big\}$ then $cI+m(\xi)$ is 
invertible in ${\rm Hom}(\Lambda,\Lambda)$ for each $\xi\in{\mathbb{R}}^n\setminus\{0\}$.}
\end{equation}
To see this, assume $\theta$ and $c$ are as above and fix some $\xi\in{\mathbb{R}}^n\setminus\{0\}$ 
arbitrary. Then, based on \eqref{i7yggg.A2} it is easy to see that $cI+m(\xi)\in {\rm Hom}(\Lambda,\Lambda)$ 
and $cI+\theta^{-1}P_{\xi/|\xi|}-\theta Q_{\xi/|\xi|}\in {\rm Hom}(\Lambda,\Lambda)$
commute and their composition is $(c-\theta)(c+\theta^{-1})I$. Hence, \eqref{i7yggg.A4-hfc.22} follows. 
Granted this, we may then conclude from item {\it (e)} of Theorem~\ref{i87hbBV-ALG} 
(applied with ${\mathscr{V}}:=\Lambda$) that the equivalence {\it (i)}\,$\Leftrightarrow$\,{\it (ii)} 
in the last part of Corollary~\ref{nhxtrE-SSS} holds. 

Likewise, it is visible from \eqref{8hhBBBa.VVV} that
\begin{equation}\label{i7yggg.A4+RRR}
\begin{array}{c}
R_\theta:L^2({\mathbb{R}}^n)\otimes\Lambda\longrightarrow L^2({\mathbb{R}}^n)\otimes\Lambda
\,\,\text{ acts on each }\,\,f\in L^2({\mathbb{R}}^n)\otimes\Lambda
\\[6pt]
\text{according to }\,\,
R_\theta f(x)=-{\mathcal{F}}^{-1}_{\xi\to x}\Big(\Omega_{\theta,\,\xi/|\xi|}\widehat{f}(\xi)\Big)
\,\,\text{ for a.e. }\,\,x\in{\mathbb{R}}^n,
\end{array}
\end{equation}
hence $R_\theta$ is a multiplier operator with symbol given by 
\begin{equation}\label{i7yggg.A4-hfc+RRR}
m(\xi):=-\Omega_{\theta,\,\xi/|\xi|}\in{\rm Hom}(\Lambda,\Lambda)
\,\,\text{ for }\,\,\xi\in{\mathbb{R}}^n\setminus\{0\}.
\end{equation}
Since, thanks to the last formula in \eqref{i7yggg.A2}, for each vector $\xi\in{\mathbb{R}}^n\setminus\{0\}$ 
we may write $\big(cI-\Omega_{\theta,\,\xi/|\xi|}\big)\big(cI+\Omega_{\theta,\,\xi/|\xi|}\big)=(c^2-1)I$, 
we conclude that  
\begin{equation}\label{i7yggg.A4-hfc.22-XX}
\parbox{8.30cm}{if $\theta\in{\mathbb{C}}\setminus\{0\}$ and 
$c\in{\mathbb{C}}\setminus\{\pm 1\}$ then $cI+m(\xi)$ is 
invertible in ${\rm Hom}(\Lambda,\Lambda)$ for each $\xi\in{\mathbb{R}}^n\setminus\{0\}$.}
\end{equation}
As such, item {\it (e)} of Theorem~\ref{i87hbBV-ALG} (once again applied with ${\mathscr{V}}:=\Lambda$) proves
the equivalence {\it (i)}\,$\Leftrightarrow$\,{\it (iii)} in the last part of Corollary~\ref{nhxtrE-SSS}.
\end{proof}

\end{document}